\documentclass[reqno]{amsart}%
\usepackage{amsfonts}
\usepackage{amsmath}
\usepackage{amssymb}
\usepackage{graphicx}
\usepackage{amsthm}
\usepackage{xcolor}
\usepackage{hyperref}%
\providecommand{\U}[1]{\protect\rule{.1in}{.1in}}
\hypersetup{
colorlinks=true,     linktoc=all,         linkcolor=blue,  }
\newtheorem{theorem}{Theorem}[section]

\newtheorem{corollary}[theorem]{Corollary}

\newtheorem{definition}[theorem]{Definition}
\newtheorem{example}[theorem]{Example}

\newtheorem{lemma}[theorem]{Lemma}

\newtheorem{proposition}[theorem]{Proposition}
\newtheorem*{remark}{Remark}

\newtheorem{fact}[theorem]{Fact}
\begin{document}
\title{Ribbon invariants I}
\author{Simeon Stefanov}
\address{University of Architecture, Civil Engineering and Geodesy (UACEG), Sofia 1046,
1~Hr. Smirnenski Blvd.}
\email{sim.stef.ri@gmail.com}
\keywords{critical point, gradient flow, ribbon, multiplicity results}
\subjclass[2010]{58K05, 05E99, 26B99}
\maketitle
\tableofcontents
\date{}

\begin{abstract}
A non-negative integer invariant, estimating from below the number of
geometrically different critical points of a smooth function $f$ defined in
the 2-disk, $f:\mathbb{B}^{2}\rightarrow\mathbb{R}$, is considered. (We denote
it by "$\gamma$" in the text.) It depends on combined $C^{0}+C^{1}$ type
conditions on the boundary $\partial(\mathbb{B}^{2})=\mathbb{S}^{1}$, that we
call \textit{ribbons} here. It turns out to be an alternative to the degree of
the gradient map and almost independent from it. Note that the computation of
the degree \textit{does not} guarantee multiple critical points, unlike the
ribbon invariant~$\gamma$. In fact, this invariant is counting the number of
essential components of the critical set, rather than simply the number of
critical points. Various estimates of $\gamma$ are established. Some other
\textit{ribbon type} invariants of geometrical nature are defined and
investigated. All these invariants turn out to be more combinatorial, rather
than algebraic, in nature. Algorithms for the calculation of the ribbon
invariants are presented. Interconnections with some different areas, such as
the theory of immersed curves in the plane or independent domination in
graphs, as well as various geometric applications, are commented. The latter
topics will be investigated in detail in Part II of the present article. At
the end, different questions about $\gamma$ are asked.

\end{abstract}



\section{Introduction}

Let us specify at the beginning that the \textit{ribbons} and \textit{ribbon
invariants} under consideration in the present article \textit{do not refer}
to ribbon knots or graphs, although there is some visual (geometrical)
resemblance with latter objects. Our \textit{ribbons} are simpler and are
mainly related to multiplicity results about critical points. In fact, a big
part of the material in the paper seems elementary, or even trivial sometimes,
and often becomes evident from picture, so we shall omit here and there
annoying technical details and appeal to reader's imagination. On the other
hand, the main problems seem to be quite hard, as they are not susceptible to
any algebraic methods or simple formulas as a rule. Such is for example the
central problem with the computation of the main ribbon invariant $\gamma$.

We shall start with some key observations that will be justified later in the text.

Let $f:\mathbb{B}^{2}\rightarrow\mathbb{R}$ be a smooth function on the unit
disk with a compact set of critical points $\mathrm{Crit}(f)$ not intersecting
the boundary $\mathbb{S}^{1}$. Suppose that the degree of the gradient field
$\nabla f$ along $\mathbb{S}^{1}$ is zero:%
\[
\deg(\nabla f|_{\mathbb{S}^{1}})=0\text{.}%
\]

Then, according to Hopf's Degree Theorem (\cite{b1}), the field $\nabla
f|_{\mathbb{S}^{1}}$ may be extended inside $\mathbb{B}^{2}$ to some field $V$
without zeroes. Now, it is a natural question to ask whether there exists a
\textit{gradient} vector field $V_{1}$ extending $\nabla f|_{\mathbb{S}^{1}}$
without zeroes, in other words, whether the function $f$ may be extended from
some neighbourhood of $\mathbb{S}^{1}$ to its interior without critical
points. It turns out that in general the answer to this question \textit{is
negative}! So, there is a nontrivial problem here! Naturally, it may happen
that critical points free extensions exist, but in other cases the number of
critical points cannot be reduced under some positive integer limit, depending
on the boundary data. This low limit is, roughly speaking, the main
\textquotedblleft ribbon invariant\textquotedblright\ $\gamma$ we deal with in
this article. Of course, this invariant is defined for functions with
arbitrary integer value of $\deg(\nabla f|_{\mathbb{S}^{1}})$, not only zero.
To justify the above observation in this general setting, consider the
following situation.

Let $C=\mathbb{S}^{1}\times\lbrack1,2]$ be an annulus and $f:C\rightarrow
\mathbb{R}$ be a smooth function with a critical set $\mathrm{Crit}(f)$ not
intersecting the boundary $\partial C$. Let $\partial C=C_{1}\cup C_{2}$,
where $C_{1}=\mathbb{S}^{1}\times\{1\}$ and $C_{2}=\mathbb{S}^{1}\times\{2\}$
and suppose that $\deg(\nabla f|_{C_{1}})=\deg(\nabla f|_{C_{2}})$. Then, by
Hopf's Theorem, the fields $\nabla f|_{C_{1}}$, $\nabla f|_{C_{2}}$ are
homotopic by a \textit{stationary points free} homotopy $H:C\rightarrow
\mathbb{R}$. Now, it turns out again that there might be no a
\textit{gradient} stationary points free homotopy $H$ connecting the above
fields, though any such gradient homotopy should have some minimal number of
critical points, which may be greater than zero. Therefore we obtain some
integer invariant somehow measuring the \textit{gradient distance} between the
fields $\nabla f|_{C_{1}}$, $\nabla f|_{C_{2}}$ on $\mathbb{S}^{1}$ (which is
defined, of course, for arbitrary values of $\deg(\nabla f|_{C_{1}})$ and
$\deg(\nabla f|_{C_{2}})$). So, in this situation we have a problem with the
calculation of this minimal number, yet the answer being not obvious or simple
as a procedure at all.

Another curious observation. Let $f:\mathbb{R}^{2}\rightarrow\mathbb{R}$ be a
smooth function with bounded set of critical points $\mathrm{Crit}(f)$.
Suppose that there is a simple closed curve $\lambda$ surrounding
$\mathrm{Crit}(f)$, such that $\deg(\nabla f|_{\lambda})=0$. Now, let us try
to ``kill'' its critical points by a smooth homotopy (perturbation) with
compact support. It turns out again that the number of critical points cannot
be reduced under certain limit, contrary to the expectation that we may kill
all of them, in view of $\deg(\nabla f|_{\lambda})=0$. For example, it is
possible to construct a Morse function $f:\mathbb{R}^{2}\rightarrow\mathbb{R}$
with $\mathrm{Crit}(f)$ consisting of $k$ local extrema and $k$ non-degenerate
saddles, such that the number of critical points cannot be reduced under $2k$
by a smooth homotopy with compact support. This contradicts the natural
expectation that extrema and saddles may ``annihilate'' by pairs. And the
cause again is the ribbon number \textit{at infinity} $\gamma_{\infty}$\ of
function $f$ (see \hyperref[s18]{Section~\ref*{s18}}, which is a modification
of the main invariant $\gamma$.

There is a simple motivating example in dimension 1.

Let $f:[a,b]\rightarrow\mathbb{R}$ be a smooth function. Consider the numbers
$d=f(b)-f(a)$, $\alpha=f^{\prime}(a)$, $\beta=f^{\prime}(b)$ and suppose they
are nonzero. Then, according to the signs of $d$, $\alpha$, $\beta$ one may
predict a minimal number $\Gamma$ of critical points of $f$. For example, if
$d>0$, $\alpha>0$, $\beta>0$, then $\Gamma=0$, while if $d>0$, $\alpha<0$,
$\beta<0$, then $\Gamma=2$ (\hyperref[f1]{Fig.~\ref*{f1}}). Here $d$ may be
treated as a $C^{0}$-boundary datum, while $\alpha$, $\beta$ are $C^{1}%
$-boundary data. If $d=0$ and we neglect the sign of $\alpha$ and $\beta$,
then $\Gamma=1$ and this is, of course, Rolle's Theorem. This simple situation
raises the question what happens in higher dimensions. The most general
question of this kind should be: \textit{ Given a manifold }$M$\textit{ with
boundary }$\partial M$\textit{ and a smooth function }$\varphi$\textit{
defined in a small neighbourhood of }$\partial M$\textit{, which is critical
points free, then what is the biggest number }$\Gamma$\textit{ such that any
(smooth) extension of }$\varphi$\textit{, }$\Phi:M\rightarrow\mathbb{R}%
$\textit{ has at least }$\Gamma$\textit{ distinct critical points in }%
$M$\textit{?}

As such a (finite)\ number always exists, we get some integer\ ``invariant''
$\Gamma=\Gamma(\varphi)$, that gives rise to different questions, such as: how
large can $\Gamma$ be, how to compute it, can we ``discretize'' the boundary
data and, of course, is this invariant ``interesting'', in particular, are
there consistent examples where $\Gamma$ may be, more or less, easily found?
Another keystone is to show that this invariant \textit{does not hold} from
the simple calculation of the degree of the gradient field $\nabla
\varphi|_{\partial M}$. In fact, as we shall see later, these two numbers are
almost independent from each other, except for some obvious inequalities. Yet
another significant difference between these is the fact that our invariant is
subadditive and combinatorial in nature, unlike the degree, which is an
additive algebraic invariant.

\begin{center}
\begin{figure}[ptb]
\includegraphics[width=115mm]{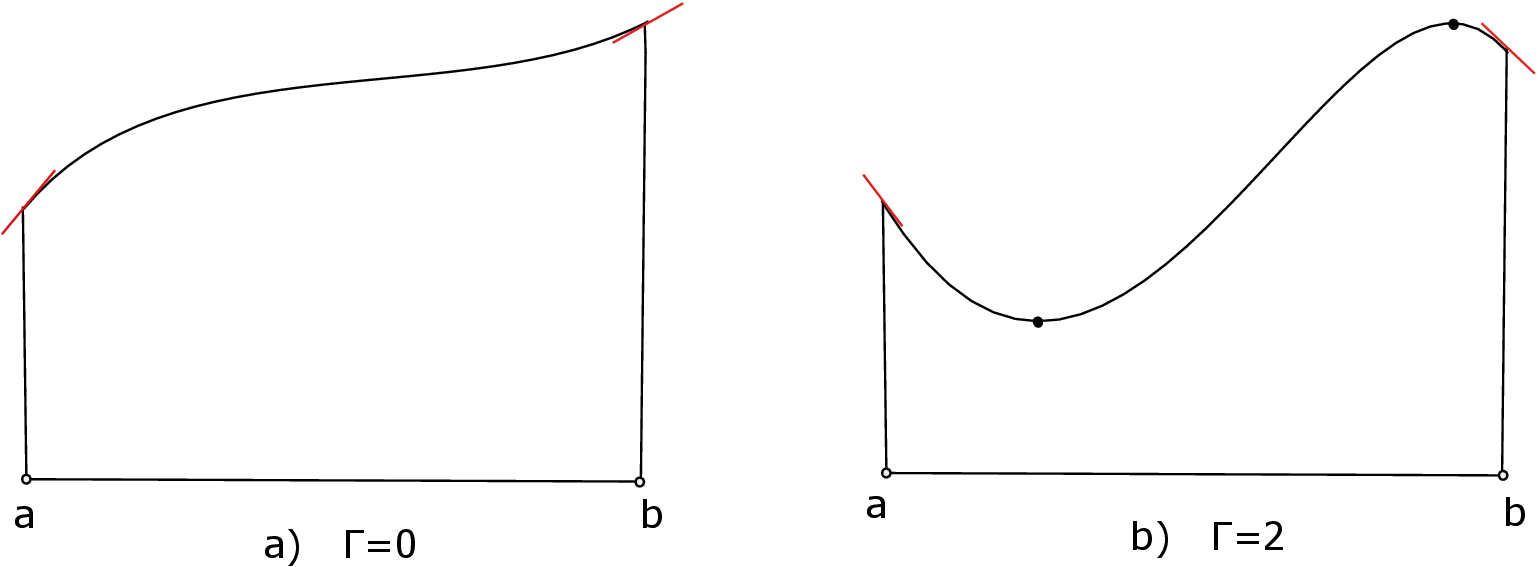}\caption{1-dimensional example: a
Rolle-type situation.}%
\label{f1}%
\end{figure}
\end{center}

Note also that a straightforward generalization of Rolle's Theorem would
presume the function $\varphi$ being defined only on $\partial M$, neglecting
in such a way the $C^{1}$ part of the boundary data. However, as we shall see
later, this almost trivializes the problem, as then multiplicity results are
not available. On the other hand, taking into account only the $C^{1}$-data
and neglecting the $C^{0}$-part of the boundary conditions, may provide us
with nontrivial multiplicity results, in some cases (see \hyperref[s15]%
{Section~\ref*{s15}}).

We shall suppose for simplicity that $\varphi|_{\partial M}$ is a Morse
function of class $C^{\infty}$; this allows us, in some cases, to discretize
the boundary data. The restriction $\varphi|_{\partial M}$ may be treated as
the $C^{0}$-part of the boundary condition, while the restriction of the
gradient $\nabla\varphi|_{\partial M}$ should be its $C^{1}$-part. Of course,
the above question sounds fairly general in this form and we shall focus our
attention on the case $M=\mathbb{B}^{2}$, so

\begin{center}
\textit{ our manifold }$M$\textit{ will be the 2-dimensional disk }%
$\mathbb{B}^{2}$\textit{ from now on.}
\end{center}

This case turns to be difficult and interesting enough for itself.

Let us note that the computation of $\Gamma$ may be of some practical
interest, as we get an estimate from below of the number of \textit{distinct}
critical points only from some boundary conditions. From this point of view it
is easy to imagine that varying the boundary, we obtain different estimates of
this number from below and this allows us in addition to localize the critical
set of a function $f:\mathbb{R}^{2}\rightarrow\mathbb{R}$.

\section{\label{s2} Definition of the main ribbon invariant $\gamma$}

Let $\varphi:\mathbb{S}^{1}\rightarrow\mathbb{R}$ be a Morse function, i.e.
a\ smooth (class $C^{\infty}$) function with finite number of critical points,
all being non-degenerate extrema with different values. Let the extrema be
$p_{1},\dots,p_{n}$, then $n$ is an even number. Henceforth, we shall call
them ``nodes''. We assign to any node $p_{i}$ its ``sign'' $\nu(p_{i}%
)\in\{-1,+1\}$ in an arbitrary manner. If $\nu(p_{i})=+1$, we say that $p_{i}$
$\ $is a \textit{positive node} and if $\nu(p_{i})=-1$, then $p_{i}$ $\ $is a
\textit{negative node}. This information will be the boundary condition of our
problem. Now we consider all smooth extensions of $\varphi$ on $\mathbb{B}%
^{2}$ with prescribed behaviour at $p_{i}$. More precisely, let $f:\mathbb{B}%
^{2}\rightarrow\mathbb{R}$ be a smooth function such that

1) $f|_{\mathbb{S}^{1}}=\varphi$

2) $\mathrm{sign}(\nabla f(p_{i}),p_{i})=\pm\nu(p_{i})$, where ``$+$'' is
taken if $p_{i}$ is a local maximum and ``$-$'' is taken if $p_{i}$ is a local minimum.

Note that condition 2) implies that $\nabla f|_{\mathbb{S}^{1}}\neq0$
everywhere on $\mathbb{S}^{1}$.

\begin{center}
\begin{figure}[ptb]
\includegraphics[width=95mm]{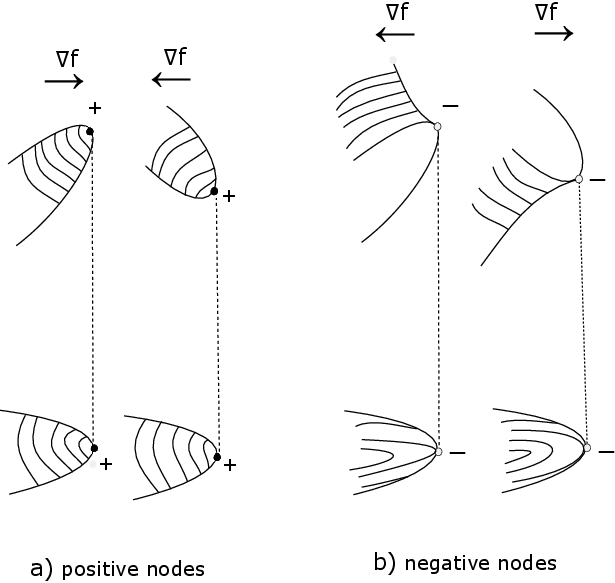}\caption{Behaviour of extensions at
nodes.}%
\label{f2}%
\end{figure}

\begin{figure}[ptb]
\includegraphics[width=75mm]{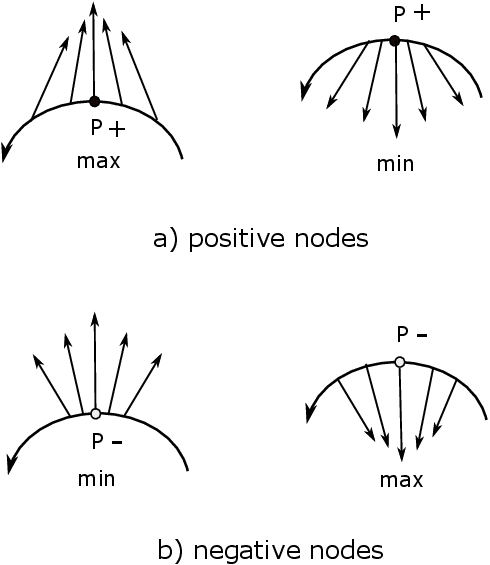}\caption{Behaviour of the
gradient field at nodes.}%
\label{f2.5}%
\end{figure}
\end{center}

\hyperref[f2]{Fig.~\ref*{f2}} illustrates the\ desired behaviour of the
extension $f$ according to the sign of $\nu(p_{i})$, while at \hyperref[f2.5]{Fig.~\ref*{f2.5}} the behaviour
of the gradient field of $f$ is depicted. Observe that in a positive node the
level lines of $f$ are touching the boundary $\mathbb{S}^{1}$ from outside,
while in a negative node, they are touching the boundary from inside.
Henceforth, a line which is touching the boundary from inside and is\ not
containing any critical points will be called a \textit{regular touching
line}, or simply a \textit{touching line}. Clearly, a touching line may be
either a topological segment or a topological circle. So, the boundary
conditions include 2 types of data - the function $\varphi$ ($C^{0}$-data) and
the assignment $p_{i}\rightarrow\nu(p_{i})$ ($C^{1}$-data). The latter means
that we have some function $\nu:P\rightarrow\{-1,+1\}$, where $P$ is the set
of extrema of $\varphi$. For the moment, the boundary condition is the pair
$(\varphi,\nu)$ (later we shall\ discretize it). In this manner, the boundary
condition is a \textquotedblleft ribbon type\textquotedblright\ condition, so
it is convenient to name the pair $(\varphi,\nu)$ a \textquotedblleft
ribbon\textquotedblright. We may think of \textit{ribbon} as a thin
noncritical band situated above $\mathbb{S}^{1}$. Fortunately, this object may
be easily discretized, at least for the goals of this article.

\begin{definition}
\label{d1}The pair $a=(\varphi,\nu)$ is called a \textit{\textquotedblleft
ribbon\textquotedblright}. The set $\mathcal{A}$\textit{ of all ribbons will
be called \textquotedblleft the ribbon space\textquotedblright. The class of
functions }$f:\mathbb{B}^{2}\rightarrow\mathbb{R}$ satisfying 1) and 2) will
be denoted by $\mathcal{F}(a)$. We shall often say that $f$ is an extension of
ribbon~$a$.

A ribbon with all nodes being positive will be called a ``positive ribbon''.
The set of all positive ribbons will be denoted by $\mathcal{A}^{+}$.
Similarly is defined the set $\mathcal{A}^{-}$ of ``negative ribbons''. The
set \textit{of all ribbons with }$n$ nodes will be denoted by $\mathcal{A}%
_{n}$.
\end{definition}

It is convenient to consider ribbons up to translation: $(\varphi,\nu
)\sim(\varphi+C,\nu)$ and rotation: $(\varphi,\nu)\sim(\varphi r_{\alpha}%
,\nu)$, where $C$ is a constant and $r_{\alpha}$ is a rotation at some angle
$\alpha$. We shall also often denote for simplicity the sign $\nu(p_{i})$ of a
node by $+$ or $-$, instead of $+1$, $-1$.

\begin{remark}
It is not difficult to show that any ribbon $a=(\varphi,\nu)\in\mathcal{A}%
$\textit{ is phisically realizable by a thin band above }$\mathbb{S}^{1}$,
thus $\mathcal{F}(a)$ is a nonempty set.
\end{remark}

Now we give our main definition.

\begin{definition}
\label{d2}Let $a=(\varphi,\nu)\in\mathcal{A}$ be a ribbon. Then we shall
denote by $\gamma(a)$ the minimal number of critical points of $f$, where
$f:\mathbb{B}^{2}\rightarrow\mathbb{R}$ varies among all extensions
$f\in\mathcal{F}(a)$. In such a way, we get some map%
\[
\gamma:\mathcal{A}\rightarrow\mathbb{N\cup\{}0\mathbb{\}}\text{,}%
\]

that we shall refer to as a ``ribbon invariant''.
\end{definition}

Although somewhat tautological, let us emphasize the principal property of the
ribbon invariant $\gamma$:

\textit{For a given ribbon }$a=(\varphi,\nu)$\textit{, any its extension
}$f\in F(a)$\textit{ has at least }$\gamma(a)$\textit{ critical points. There
is an extension with exactly }$\gamma(a)$\textit{ critical points.}

Note that all the critical points of an extension realizing $\gamma$ should be
of nonzero index, since a critical point of index $0$ may be ``killed'' by a
small perturbation not disturbing the other critical points of $f$.

Let us make some clarifying remarks. First, it is easy to see that $\gamma(a)$
always exists, as there are extensions satisfying 1) and 2) with finite number
of critical points. In such a way, the above definition of $\gamma$ is more a
notation, rather than a ``true'' definition. Second, $\gamma(a)$ is the
minimal number of \textit{geometrically distinct}\ critical points, without
taking care about their indices. (Note that the calculation of the degree
$\nabla f|_{\mathbb{S}^{1}}$\ may provide us with at most 1 critical point, in
case it is nonzero.) And third, the definition of $\gamma$ makes sense equally
for ribbons, which are not in general position and may have coinciding
critical values. For now, we shall not consider such ribbons, unless the
opposite is specified. It turns out that for such ribbons almost all the
theory of general position ribbons remains valid. These appear in a natural
way when performing a generic homotopy of a ribbon.

From now on, ``$n$'' will stay for the number of nodes of the ribbon under consideration.

Now we list in advance some facts about the ribbon invariant $\gamma$.

a) $\gamma$ takes the same value on any two \textit{similar} ribbons
(\hyperref[s4]{Section~\ref*{s4}}), legitimating in such a way the term
``invariant''. This also\ allows us to ``discretize'' the problem and to
attack it algorithmically.

b) $\gamma$ is attained on a set of quite simple and natural extensions that
we call here \textquotedblleft economic\textquotedblright\ extensions
(\hyperref[s6]{Section~\ref*{s6}}). This is the key tool for its
investigation, as the economic extensions of a given ribbon are finite in
number, up to combinatorial equivalence.

c) the ribbon invariant satisfies some basic inequalities:

1) $0\leq\gamma\leq\frac{n}{2}+1$

Let $s_{+}$ and $s_{-}$ be the number of positive and negative nodes,
respectively. Consider the \textit{signature} $\sigma=s_{+}-$ $s_{-}$ (it is
an even number). Then we have

2) $1-\frac{\sigma}{2}\leq\gamma\leq n-1-\frac{\sigma}{2}$.

Inequalities 1) and 2) are proved in \hyperref[s14]{Section~\ref*{s14}} and
\hyperref[s8]{Section~\ref*{s8}}, respectively. Somewhat paradoxically, 1)
turns out to be much harder than 2). Note also that 1) and 2) immediately
imply that in $\mathcal{A}^{-}$ we have $\gamma=\frac{n}{2}+1$. This follows
from the fact that $\sigma=-n$ in $\mathcal{A}^{-}$. On the other hand, the
class of positive ribbons $\mathcal{A}^{+}$ turns out to be much more
intriguing and complicated, unlike the class of negative ribbons
$\mathcal{A}^{-}$, where everything is clear from point of view of ribbon invariants.

d) no simple formulas or procedures for the computation of $\gamma$ are known
to us. We shall present in \hyperref[s11]{Section~\ref*{s11}} a
\textquotedblleft brute force\textquotedblright\ ramifying algorithm for its
computation, which seems to be quite slow as $n$ grows. Probably, the problem
of computation of $\gamma$ is NP-hard. However, in some extremal cases the
ribbon invariant will be exactly computed by direct estimations.

The calculation of $\gamma$ is hard enough even in $\mathcal{A}^{+}$. In
Part~II of the article we shall present an algorithm based on the reduction of
positive ribbons to either a \textit{ladder }or an\textit{ alternation} by
elementary moves. (These are two extremal opposite cases of ribbons, see next
section.) Then we make use of the fact that the ribbon invariant of ladders
and alternations is quite easily computable, while at each move we may say by
how many $\gamma$ changes. This is much faster than the \textquotedblleft
brute force\textquotedblright\ algorithm.\ Moreover, this method finds the set
of \textit{all} minimal extensions of a ribbon $a\in\mathcal{A}^{+}$.
Furthermore, we shall discuss a similar algorithm in $\mathcal{A}$, where much
more variants are available.

e) the set of discrete ribbons of rank $n$ is quite large (for big $n$). Its
cardinality and asymptotics are computed in \hyperref[s4]{Section~\ref*{s4}}.
This is one of the reasons for the difficulties with the computation of
$\gamma$, at least for the \textquotedblleft brute force\textquotedblright%
\ algorithm described in \hyperref[s11]{Section~\ref*{s11}}. The other two
main obstacles are the \textit{subadditivity} of $\gamma$ and the fact that it
takes only non-negative values, so we can hardly expect it being algebraic in
nature. Consider for example the situation from \hyperref[f3]{Fig.~\ref*{f3}},
where some function $f$ is supposed to be defined in a region of the picture
and let us identify, for a moment, each curve with the \textit{induced ribbon
}in a small region on it. Then it turns out that $\gamma(\lambda)\leq
\gamma(\lambda_{1})+\gamma(\lambda_{2})$, where strict inequality is possible,
while for the degree $i(\lambda)=\deg(\nabla f|_{\lambda})$, one has
$i(\lambda)=i(\lambda_{1})+i(\lambda_{2})$ and this is one of the main
differences between the ribbon invariant and the degree. \ Note also that
strict inequality for $\gamma$ is much more likely than equality.
\vspace{-0.6cm}

\begin{center}
\begin{figure}[ptb]
\includegraphics[width=80mm]{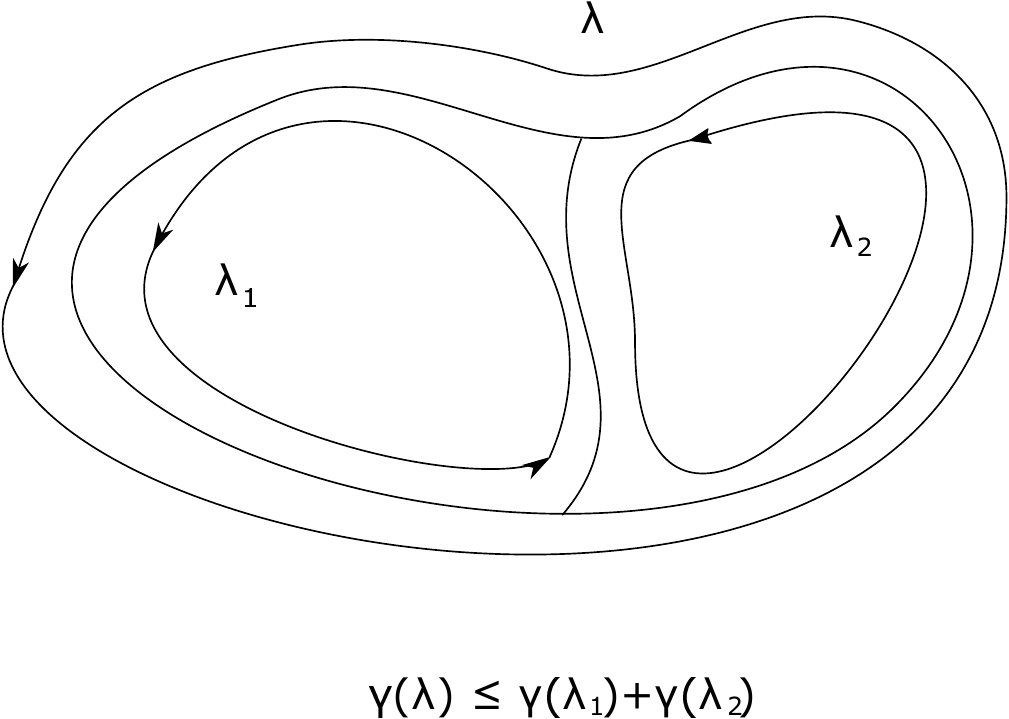}\caption{$\gamma$ is subadditive.}%
\label{f3}%
\end{figure}
\end{center}

f) $\gamma$ gives in fact an estimate from below for the number of essential
\textit{components} of the critical set of an arbitrary extension, rather than
only for the number of critical points. So, any extension of a given ribbon
$a$ has at least $\gamma(a)$ components of its critical set, each one of
nonzero index (\hyperref[s17]{Section~\ref*{s17}}).

g) the invariant $\gamma$ has some ``stability'' properties, for example, it
is true that for any $m>0$ there is $d>0$ such that if $\left\Vert \nabla
f\right\Vert <m$ for an extension $f\in\mathcal{F}(a)$, then there exist
$\gamma$ critical points of $f$, $p_{1},\dots,p_{\gamma}$ such that
$|p_{i}-p_{j}|\geq d$ for $i\neq j$. So we may say that there are $\gamma$
critical points \textit{distant} from each other. The same is true for the
components of the critical set with respect to the Hausdorff distance. Global
stability properties under homotopy of the ribbon invariant are discussed
later (\hyperref[s19]{Section~\ref*{s19}}).

h) there are some relationships of the ribbon invariant $\gamma$ with other
areas, such as the theory of immersed curves in the plane (the problem of
self-overlapping) or independent domination in graphs. In Part~II some
relation between $\gamma$ and the problem of recognition of self-overlapping
curves in the plane (see \cite{b3}) is established. Therein, a geometric
invariant of immersed curves is defined by means of $\gamma$. Furthermore, it
is shown that the economic extensions of a given ribbon $a$ are in one-to-one
correspondence with the maximal independent subsets of some so-called
\textit{critical graph} $G$ associated with $a$, and that the (weighted)
\textit{independent domination number} of $G$ equals the ribbon invariant
$\gamma(a)$. Varying the weight system, we may obtain in such a manner
description of the other ribbon invariants of \textit{geometric} nature.

i) note also that besides $\gamma$, we consider here some other ``ribbon
type'' invariants - $\gamma_{0}$, $\gamma_{\mathtt{\operatorname{ext}}}$,
$\gamma_{\mathtt{\operatorname{sad}}}$ defined in \hyperref[s9]%
{Section~\ref*{s9}}. Like $\gamma$, they are counting critical points, but in
different geometrical context. All these invariants may be computed by one and
the same algorithm (up to some initial normalization).

j) finally, in \hyperref[s20]{Section~\ref*{s20}} we consider two natural
algebraic operations in the class $\mathcal{A}$ of rigid ribbons, getting in
such a way some algebraic object, we refer to as\ the \textit{ribbon
semigroup.} Then we define axiomatically the \textit{algebraic }ribbon
invariants\ similarly to the F-invariant (Lusternik-Schnirelmann) approach to
critical points problems. It is shown by examples that the class of such
invariants is quite large. Moreover, it turns out that $\gamma$ is the
supremum of all algebraic ribbon invariants satisfying the corresponding
normalization conditions. This is equally valid for the other 3 ribbon
invariants of geometric origin, we deal in the present article.

As for the methods for proving things in the article, we should note that
these are elementary in nature and do not involve any complicated machinery.
Among others, there is one central natural method - \textit{splitting} a
ribbon into simpler pieces, which allows one to prove things by induction on
some natural lexicographic order defined in the ribbon class $\mathcal{A}$.

We shall consider further another invariant $\beta$ estimating from below the
number of different \textit{critical values} of an extension, rather than the
number of critical points.

\begin{definition}
\label{d3}Let $a=(\varphi,\nu)\in\mathcal{A}$ be a ribbon. Then we shall
denote by $\beta(a)$ the minimal number of critical values of $f$, where
$f:\mathbb{B}^{2}\rightarrow\mathbb{R}$ varies among all extensions
$f\in\mathcal{F}(a)$.
\end{definition}

Clearly,%
\[
\gamma(a)\geq\beta(a)\text{.}%
\]

It turns out that $\beta$ is a \textit{ribbon invariant} as well, in the sense
that it satisfies some natural subadditivity conditions (\hyperref[s22]%
{Section~\ref*{s22}}). It is clear that in fact $\beta$ is estimating the
number of values of the critical points which are \textit{not} local extrema,
since one may occupy no more than 2 values for all local extrema. So, in the
case of finite number of critical points, $\beta$ is estimating from below the
number of values of saddle points of $f$. On the other hand, in class
$\mathcal{A}^{+}$, $\beta$ is equal to the \textit{cluster number} of the
corresponding ribbon (the latter defined in \hyperref[s7]{Section~\ref*{s7}}).
As in many cases the cluster number may be easily found (estimated), we get an
\textquotedblleft easy\textquotedblright\ estimate from below of $\gamma$ via
the inequality $\gamma\geq\beta$.

Looking for interconnections between our investigations and some well-known
area, we shall note a direct analogy between the ribbon invariant $\gamma$ and
the so-called $\mathit{F}$\textit{-invariant} which is aiming similar
multiplicity results.\smallskip

\begin{center}
\textbf{Parallel with the F-invariant and Lusternik-Schnirelmann
problems.\smallskip}
\end{center}

Recall in brief the definition of the F-invariant (see \cite{b5}) and some
basic facts about it.

Let $M$ be a compact $n$-dimensional smooth manifold (with or without
boundary). Then $F(M)$ is the minimal number of critical points of maps
$f:M\rightarrow\mathbb{R}$, where $f$ varies among the class of all smooth
functions. Clearly, like $\gamma$, this is more a notation rather than a
\textquotedblleft true\textquotedblright\ definition, since this (finite)
number exists by itself. Then come the difficulties with the computation of
$F(M)$ and $\gamma$. It turns out that the exact computation is quite hard, so
it is convenient to look for some consistent estimates from below. In the case
of $F(M)$ there are two nice such invariants - the Lusternik-Schnirelmann
\textit{category} and \textit{cup length}.

For a given space $X$, its Lusternik-Schnirelmann category $\mathrm{cat}(X)$
is the least number $k$, such that $X$ may be decomposed into $k$ (closed)
subsets $F_{1},\dots,F_{k}$ so that the $F_{i}$'s are ambient contractible in
$X$ into a point.

The cup length of $X$ is the greatest number $m$ such that there exist
$a_{1},\dots,a_{m}\in H^{\ast}(X)$ for which $a_{1}\cup\dots\cup a_{m}\neq0$.
(Here $H^{\ast}(X)$ is the cohomologies ring of $X$.) Then we set
$\mathrm{length}(X)=m$. Note that in many cases the cup length may be
effectively calculated.

It is a basic fact that for a closed manifold $M$%
\[
F(M)\geq\mathrm{cat}(M)\geq\mathrm{length}(M)+1\text{,}%
\]

where strong inequalities are possible, as examples show. However, in many
interesting cases all these 3 numbers coincide, so it is, more or less, easy
to calculate $F(M)$. For example, $F(\mathbb{R}P^{n})=F(\mathbb{C}%
P^{n})=F(\mathbb{T}^{n})=n+1$ and $F(M^{2})=3$, for any closed 2-surface of
genus $\geq$1, which follows from the direct calculation of the cup length.

Similarly to the $F$-invariant, the ribbon invariant $\gamma$ is pretty hard
to be exactly calculated, anyway, there are some \textquotedblleft a
priory\textquotedblright\ estimates from below:%
\[
\gamma\geq\delta_{0}\text{, \ }\gamma\geq i=1-\frac{\sigma}{2}\text{ in
}\mathcal{A}\text{, \ }\gamma\geq\delta\text{ in }\mathcal{A}^{+}\text{.}%
\]

(See \hyperref[s7]{Section~\ref*{s7}} for the definition of the cluster
numbers $\delta$ and $\delta_{0}$.) Note that the latter inequalities are not
so good and there might be a large gap between the corresponding values. So,
it would be interesting to find some effectively calculable invariant which
gives a consistent estimate from below of $\gamma$.

As for estimations from above (not so interesting, though), yet there is some
analogy between the $F-$ and the $\gamma-$ invariant.

1) If $M$ is a compact $n$-dimensional manifold, then a) $F(M)\leq n$, if
$\partial(M)=\varnothing$, b) $F(M)\leq n-1$, if $\partial(M)\neq\varnothing$.

2) For any ribbon $a\in\mathcal{A}_{n}$ it holds that $\gamma(a)\leq\frac
{n}{2}+1$ (see \hyperref[s14]{Section~\ref*{s14}}).

\section{\label{s3} Some arguments about $\gamma$}

We shall give here some simple examples and arguments justifying our interest
in the ribbon invariant.

\textbf{0.} First of all, let us see what happens for $n=2$ and $n=4$. This
may be done \textquotedblleft by hand\textquotedblright.

For $n=2$ there are 4 ribbons (up to \textit{similarity}) with 3 possible
values of $\gamma$ equal to 0,1 and 2. The corresponding cases, with the
corresponding minimal solutions, are depicted at \hyperref[f4]{Fig.~\ref*{f4}%
}. The most simple ribbon $(1^{+},2^{+})$ with $\gamma=0$ is called by us a
\textquotedblleft minimal ribbon\textquotedblright\ (\hyperref[f4]%
{Fig.~\ref*{f4}-a}). Albeit elementary, it is important for the \textit{ribbon
semigroup }defined in \hyperref[s20]{Section~\ref*{s20}}. It is the minimal
element in the ordered set of discrete ribbons (see next section).
\vspace{-0.6cm}

\begin{center}
\begin{figure}[ptb]
\includegraphics[width=90mm]{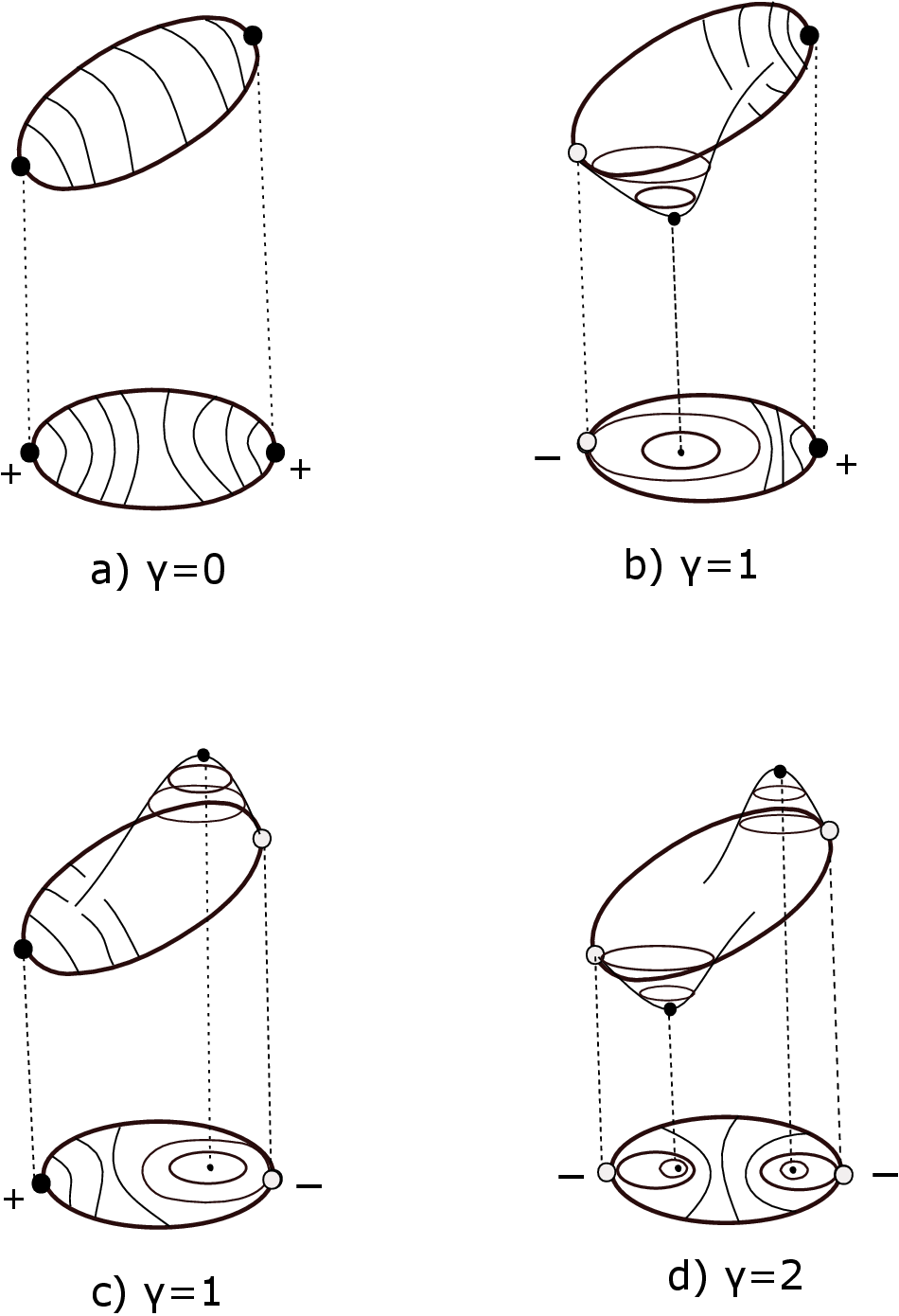}\caption{Ribbons with 2 nodes and the
minimal extensions.}%
\label{f4}%
\end{figure}
\end{center}

For $n=4$ there are 32 ribbons (up to similarity) with possible values of
$\gamma$ equal to 0,~1,~2 and 3. Three of them are shown at \hyperref[f5]%
{Fig.~\ref*{f5}}. At \hyperref[f6]{Fig.~\ref*{f6}} we give the level portrait
of a ``good'' extension of some ribbon $a\in\mathcal{A}_{4}^{-}$ with 5
critical points, that does not realize the ribbon invariant, since later we
show that $\gamma(a)=3$ for any ribbon from $\mathcal{A}_{4}^{-}$. The
corresponding solution is given at \hyperref[f6x]{Fig.~\ref*{f6x}}.
\vspace{-.6cm}

\begin{center}
\begin{figure}[ptb]
\includegraphics[width=90mm]{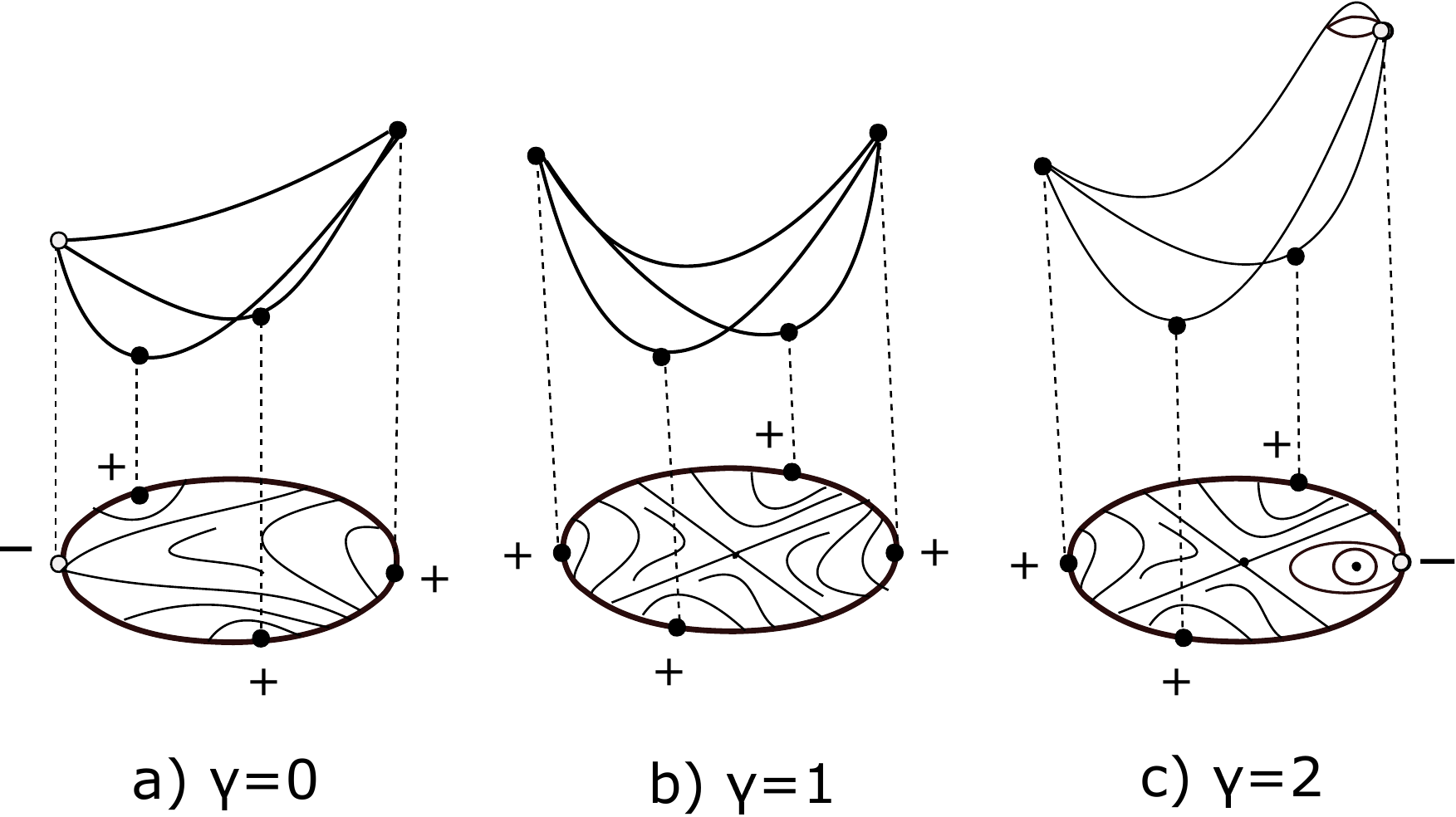}\caption{Some ribbons with 4 nodes.}%
\label{f5}%
\end{figure}

\begin{figure}[ptb]
\includegraphics[width=60mm]{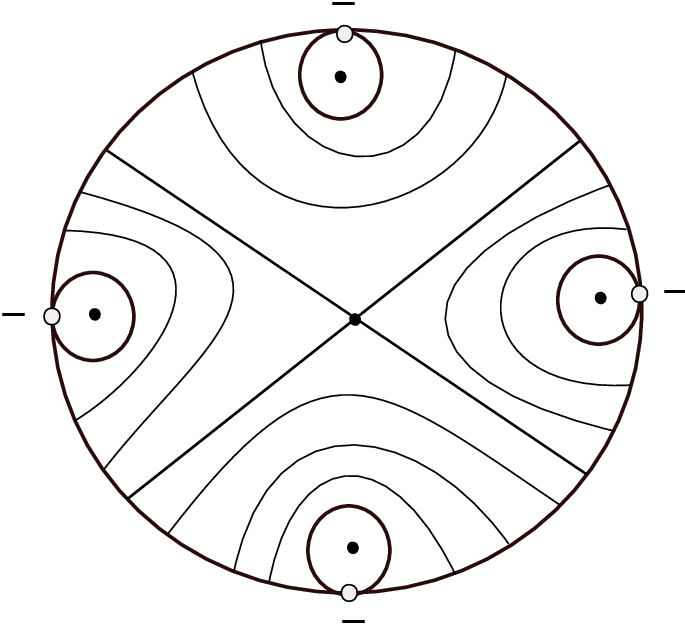}\caption{An extension with 5 critical
points not realizing $\gamma=3$.}%
\label{f6}%
\end{figure}

\begin{figure}[ptb]
\includegraphics[width=60mm]{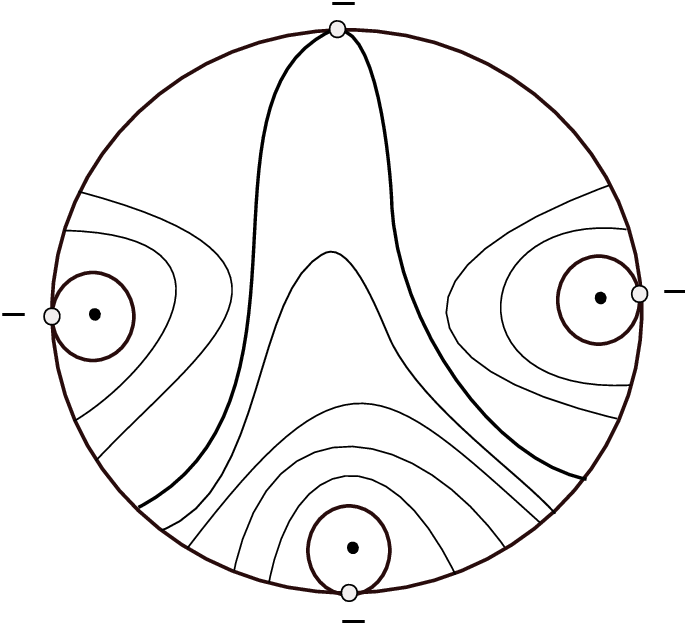}\caption{The solution $\gamma=3$ in
$\mathcal{A}_{4}^{-}$.}%
\label{f6x}%
\end{figure}
\end{center}

Let us make a practical convention: \textit{From now on, in examples, we shall
denote for simplicity a ribbon }$a=(\varphi,\nu)$\textit{ in a following
manner: If }$l_{i}=\varphi(p_{i})$\textit{ are the level nodes, then we shall
write}%
\[
a=(l_{1}^{\nu(p_{1})},\dots,l_{n}^{\nu(p_{n})})\text{,}%
\]

\textit{for example, we shall write things like }$a=(1^{+},3^{-},2^{+},4^{+}%
)$\textit{ and this will be enough for identifying the ribbon.}

\textit{Furthermore, as it is a little bit hard to depict complex ribbons by
their graph like in \hyperref[f5]{Fig.~\ref*{f5}}, we adopt some simpler
schematic way to represent them linearly as in \hyperref[f17]{Fig.~\ref*{f17}%
}. Positive nodes are depicted by black dots, negative - by white ones. In
order to link easily the ribbon with the corresponding level lines portrait of
some extension, we number the nodes monotonically: }$(1,2,3,4,5,6)$\textit{,
and, of course, this is \textbf{not} the zig-zag permutation expressing the
}$C^{0}$\textit{-part of the ribbon.}

\textbf{1.} An almost trivial observation: Let $a=(\varphi,\nu)$ be a ribbon
such that $\varphi$ attains its minimum and maximum at nodes $p$ and $q$,
respectively, and $p$ and $q$ are negative nodes (\hyperref[f7]{Fig.~\ref*{f7}%
}). Then $\gamma(a)\geq2$, as any extension $f$ of $\varphi$ should have both
a global minimum and maximum inside $\mathbb{B}^{2}$. This situation is
analogous to the 1-dimensional one depicted at \hyperref[f1]{Fig.~\ref*{f1}%
-b)}. We shall call, up to some inaccuracy, $p$ and $q$ \textit{minimal} and
\textit{maximal} nodes of the ribbon. It turns out that, in aiming the
calculation of $\gamma$, we may restrict ourselves to the case of positive
minimal and maximal nodes, as we may replace any such negative node by a
positive one and then $\gamma$ decreases by~1. More precisely, if $a^{\prime}$
is the ribbon obtained from $a$ by \textit{making} the minimal and maximal
nodes $p$ and $q$ positive, then%
\[
\gamma(a^{\prime})=\gamma(a)-\varepsilon\text{,}%
\]

where $\varepsilon$ is the number of negative nodes among $p$, $q$
($\varepsilon=0,1,2$). However, we do \textit{not }suppose the minimal and
maximal nodes being positive from now on. This assumption will be useful when
defining the ``connected sum'' of two ribbons (\hyperref[s20]%
{Section~\ref*{s20}}).

\begin{center}
\begin{figure}[ptb]
\includegraphics[width=84mm]{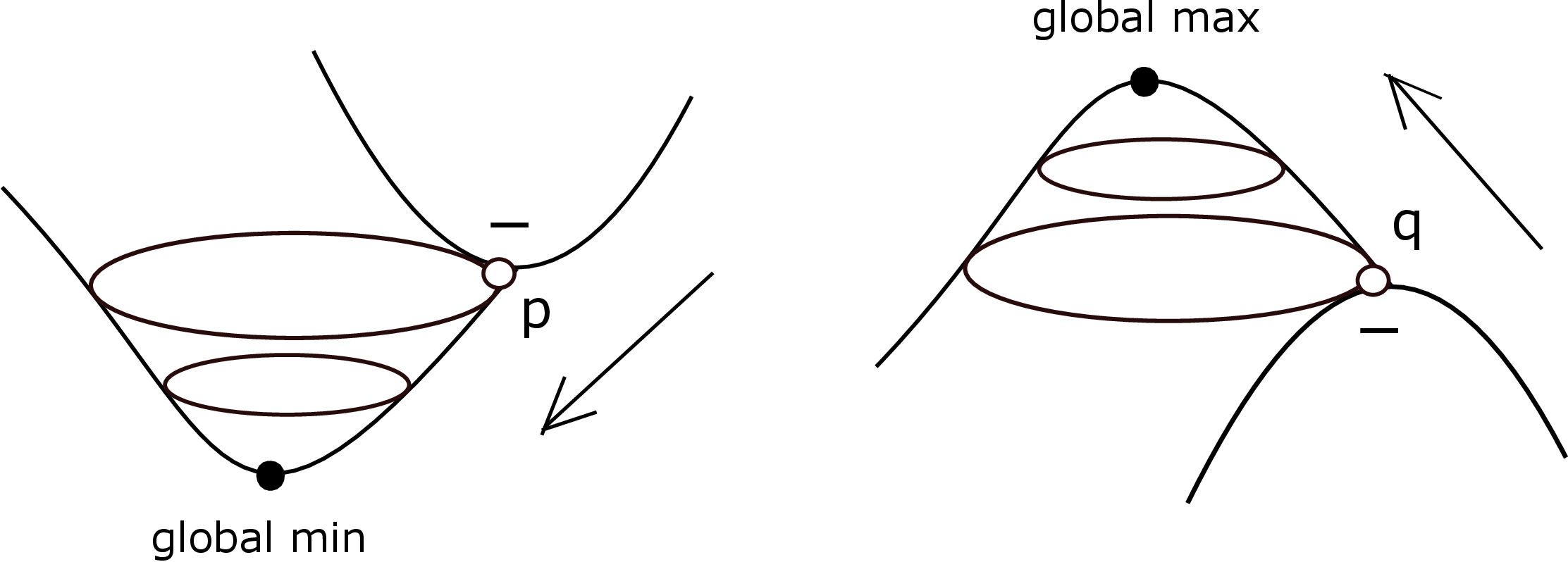}\caption{Situations of a forced
global extremum.}%
\label{f7}%
\end{figure}
\end{center}

\textbf{2.} Another simple, but useful observation. Let $n\geq4$, the nodes
$p_{i}$ are cyclically ordered on $\mathbb{S}^{1}$ and $p_{i}$, $p_{i+1}$ are
two consecutive \textit{positive} nodes. Suppose that the segment
$[\varphi(p_{i}),\varphi(p_{i+1})]$ does not contain any \textit{level}
$\varphi(p_{j})$, where $p_{j}$ varies among all \textit{negative} nodes. Then
any extension $f$ of $\varphi$ has a critical point $x\in\mathbb{B}^{2}$, such
that $f(x)\in\lbrack\varphi(p_{i}),\varphi(p_{i+1})]$. To see this, it
suffices to observe that supposing the contrary, then each level line of $f$
starting from a point of the arc $(p_{i},p_{i+1})$ should finish somewhere on
$\mathbb{S}^{1}$ at a point different from a node (\hyperref[f8]%
{Fig.~\ref*{f8}}) and thus intersecting transversely $\mathbb{S}^{1}$. Now the
simple flow-box and continuity argument implies that there are no other nodes
except $p_{i}$, $p_{i+1}$ which contradicts the assumption $n\geq4$. So, there
is a critical point $x\in\mathbb{B}^{2}$ which is \textquotedblleft
attached\textquotedblright\ to the segment $[p_{i},p_{i+1}]$. Note that $x$
cannot be a local extremum of $f$, so, in case $f$ has a finite number of
critical points, $x$ is either a saddle (possibly degenerated), or a non
essential singularity. The latter will be discussed in \hyperref[s6]%
{Section~\ref*{s6}}. Note also that there might be many different saddles
\textquotedblleft attached\textquotedblright\ to $[p_{i},p_{i+1}]$ (see
\hyperref[f10]{Fig.~\ref*{f10}}).

\begin{center}
\begin{figure}[ptb]
\includegraphics[width=60mm]{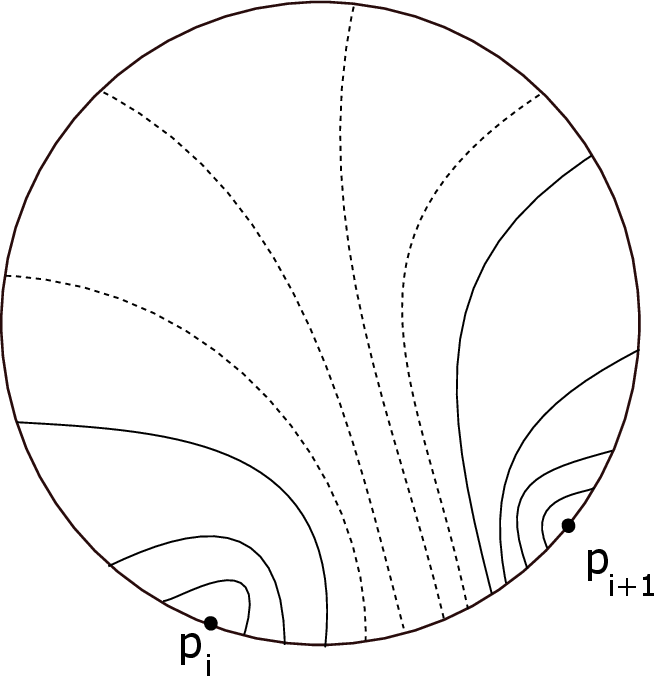}\caption{There should be some critical
point here, as $n\geq4$.}%
\label{f8}%
\end{figure}
\end{center}

The above argument allows one to construct easily ribbons with big ribbon
invariant $\gamma$. Here is a basic example relying on this observation.

$\boldsymbol{3.~Ladders.}$ Consider a ribbon $a\in\mathcal{A}$ of the form%
\[
a=(1^{\pm},3^{\pm},2^{\pm},5^{\pm},4^{\pm},7^{\pm},\dots,(n-2)^{\pm},n^{\pm
})\text{,}%
\]

\begin{center}
\begin{figure}[ptb]
\includegraphics[width=90mm]{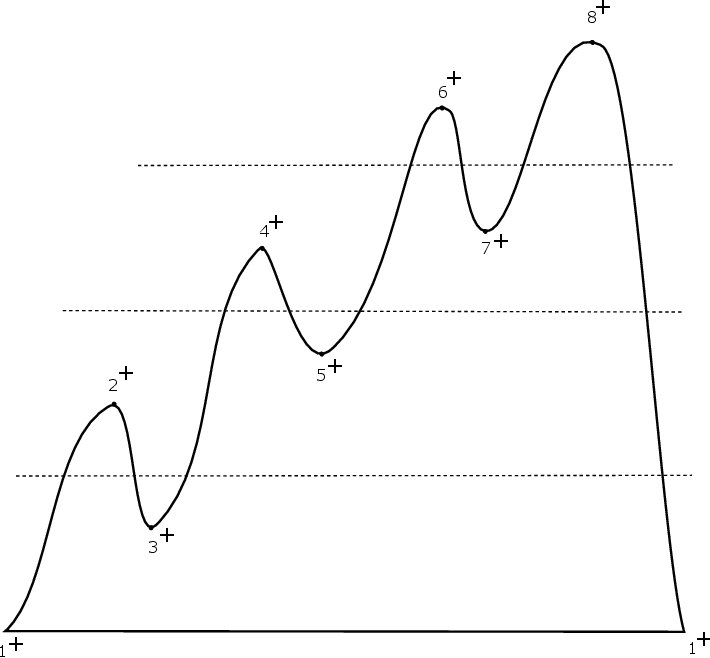}\caption{A positive ladder.}%
\label{f9}\end{figure}
\end{center}

where the signs $\pm$ are arbitrarily chosen from the set $\{+,-\}$. This
situation for $n=8$ is schematically depicted at \hyperref[f9]{Fig.~\ref*{f9}%
}. We shall call such a ribbon a \textit{ladder}. It turns out that the ribbon
invariant $\gamma$ of a ladder is very easy to calculate (see \hyperref[s11]%
{Section~\ref*{s11}}, \hyperref[p11]{Proposition~\ref*{p11}}). If the ladder
$a$ has only positive nodes ($a\in\mathcal{A}^{+}$), we say that $a$ is a
\textit{positive} ladder. Then the above remarks imply that for an arbitrary
extension $f$ of a positive ladder, any of the segments $(2,3)$, $(4,5)$%
,\dots,$(n-2,n-1)$ contains a critical level of $f$. But the number of all
such segments is $\frac{n}{2}-1$ and since they are disjoint, it follows that
$f$ has at least $\frac{n}{2}-1$ distinct critical points. Therefore for this
ribbon we have $\gamma(a)\geq\frac{n}{2}-1$. It is easy to see that in fact%
\[
\gamma(a)=\frac{n}{2}-1.
\]

\begin{center}
\begin{figure}[ptb]
\includegraphics[width=80mm]{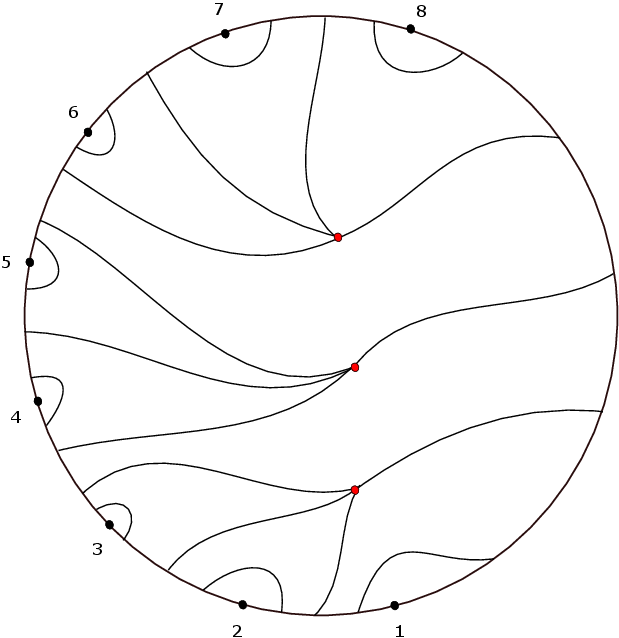}\caption{The solution, $\gamma=3$.}%
\label{f10}%
\end{figure}
\end{center}

The level lines picture of the corresponding extension is shown at
\hyperref[f10]{Fig.~\ref*{f10}}. It has 3 non-degenerate saddles. Later we
shall see that a positive ladder has only one \textquotedblleft
good\textquotedblright\ solution with $\frac{n}{2}-1$\ nondegenerate saddles.
Note that there are different variants of the ladder - see for example
\hyperref[f11]{Fig.~\ref*{f11}}, where a two-sided positive \textquotedblleft
ladder\textquotedblright\ (with missing steps!) is depicted. The level
portrait of an extension realizing $\gamma$ for this ribbon is shown at
\hyperref[f12]{Fig.~\ref*{f12}}. Furthermore, at \hyperref[f13]%
{Fig.~\ref*{f13}} some general ladder is drawn and a corresponding solution
$\gamma=2$ is shown at \hyperref[f14]{Fig.~\ref*{f14}}. For all these
variations, in the positive case, we have $\gamma=\frac{n}{2}-1$. It is not
hard to see that the combinatorial number of positive ladders with $n$ nodes
is $2^{\frac{n}{2}-1}$, whereby the number of all (general) ladders equals
$2^{\frac{3n}{2}-1}$.

\begin{center}
\begin{figure}[ptb]
\includegraphics[width=75mm, height=85mm]{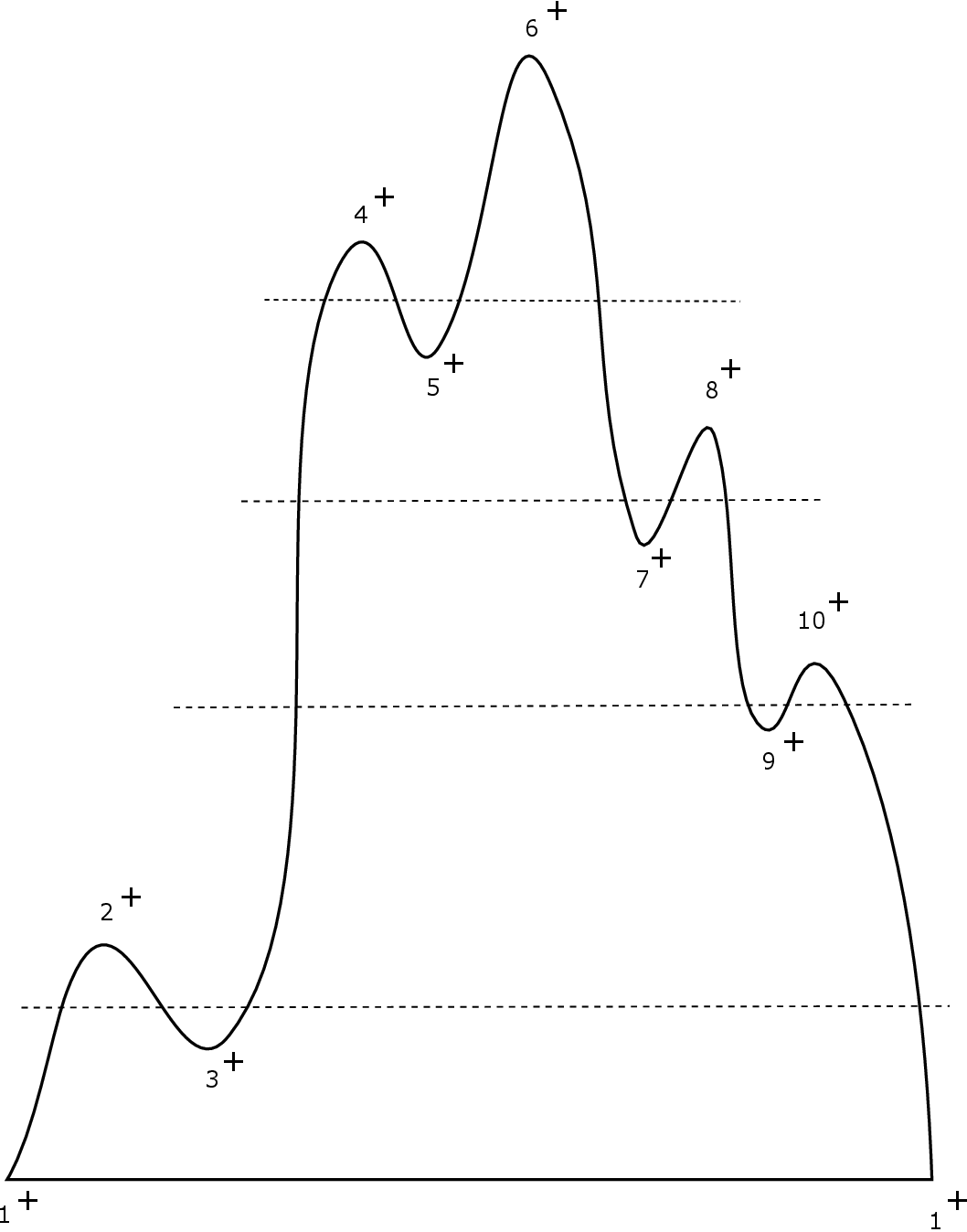}\caption{Two - sided
positive ladder.}%
\label{f11}%
\end{figure}

\begin{figure}[ptb]
\includegraphics[width=90mm]{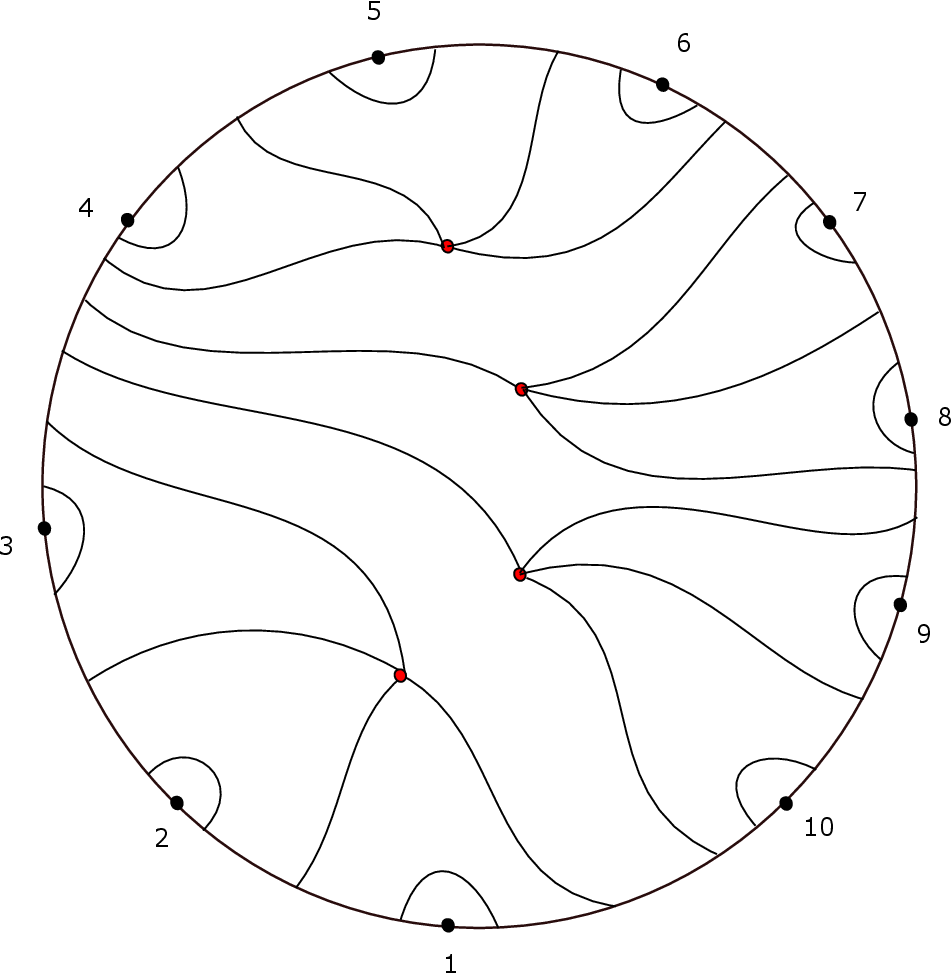}\caption{The solution, $\gamma=4$.}%
\label{f12}%
\end{figure}

\begin{figure}[ptb]
\includegraphics[width=90mm]{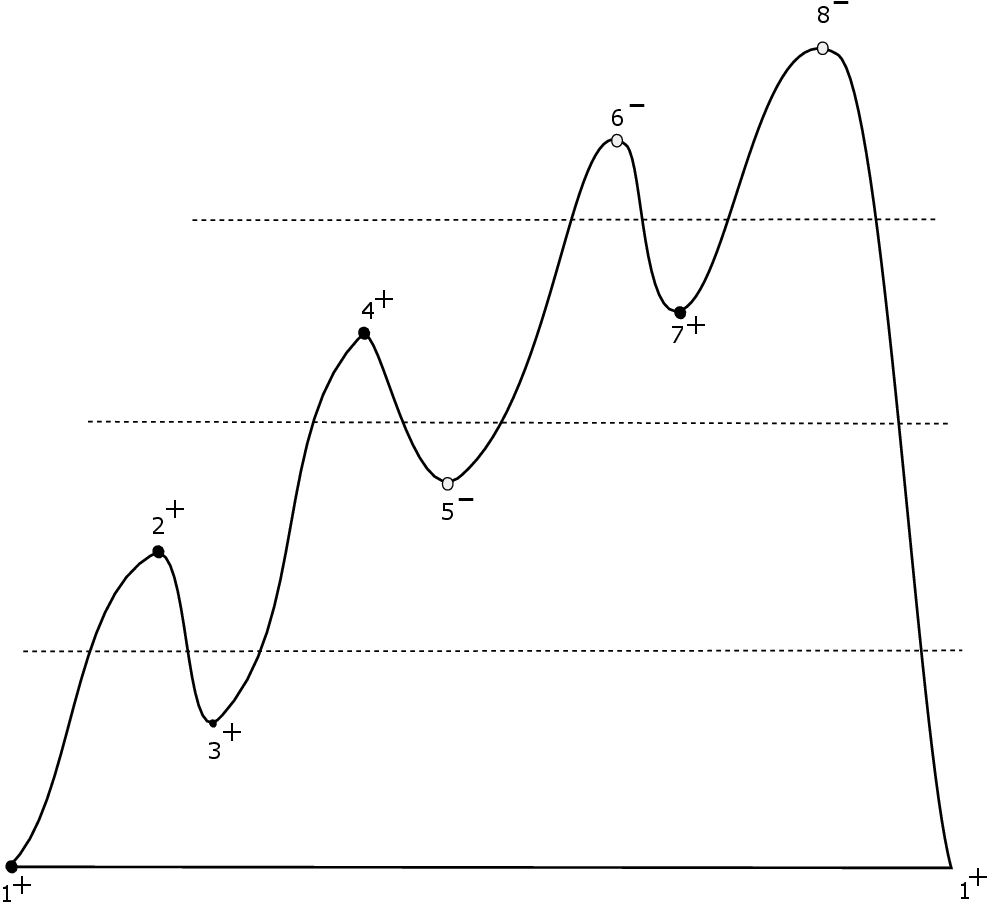}\caption{A general ladder.}%
\label{f13}%
\end{figure}
\end{center}

So, let us give an exact definition of a ladder: A ribbon $a=(\varphi,\nu
)\in\mathcal{A}$ with $n\geq4$ nodes is a \textit{ladder}, if for any critical
value $c$, different from the global minimal and maximal ones, we have%
\[
|\varphi^{-1}(c)|=3\text{.}%
\]

\begin{center}
\begin{figure}[ptb]
\includegraphics[width=80mm]{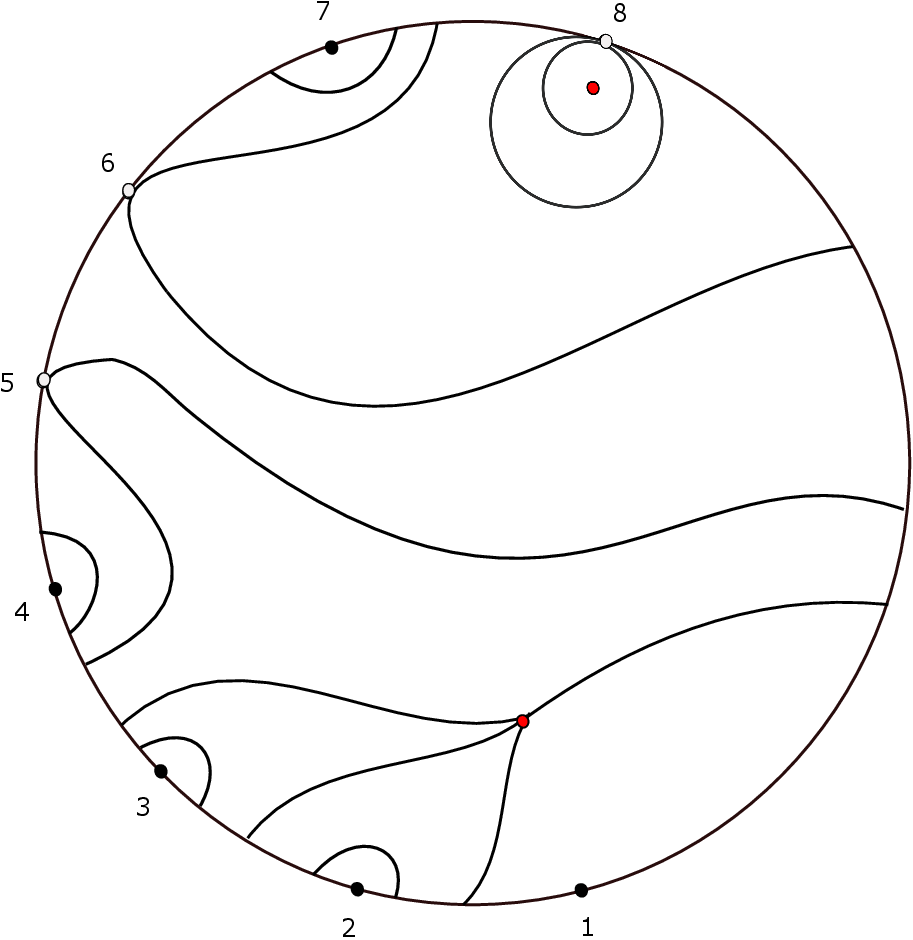}\caption{A solution $\gamma=2$
for the ribbon from \hyperref[f13]{Fig.~\ref*{f13}}.}%
\label{f14}%
\end{figure}
\end{center}

This definition covers all types of \textquotedblleft
ladders\textquotedblright. (An alternative definition is to say that for any
noncritical value $c$ we have $|\varphi^{-1}(c)|\leq4$.) As we noticed above,
the ribbon invariant $\gamma$ of all ladders is completely computable
(\hyperref[p11]{Proposition~\ref*{p11}}).

Ladders have another useful property: Let $a=(\varphi,\nu)\in\mathcal{A}_{n}$
be an arbitrary ribbon and $b=(\varphi^{\prime},\nu)\in\mathcal{A}_{n}$ be a
ladder with the same number of nodes and the same marking $\nu$. Then for any
extension $f$ of $b$ there if an extension $g$ of $a$ with exactly the same
level lines portrait as $f$. Ladders are also crucial for the ``fast''
algorithm for the calculation of $\gamma$ and the other ribbon invariants (Part~II).

$\boldsymbol{4.~Alternations.}$\textit{ }Consider a ribbon $a$ with only
positive nodes ($a\in\mathcal{A}^{+}$) and node levels $l_{i}=\varphi(p_{i})$
such that%
\[
\bigcap\limits_{i=1}^{n-1}[l_{i},l_{i+1}]\neq\varnothing\text{.}%
\]

Let $n\geq4$. We shall call such a ribbon \textit{alternation} (by analogy
with functions oscillating between two values)\textit{.} An example of an
alternation is shown at \hyperref[f15]{Fig.~\ref*{f15}}. It turns out that for
an alternation\textit{ }$a$ we have
\[
\gamma(a)=1.
\]

\begin{center}
\begin{figure}[ptb]
\includegraphics[width=90mm]{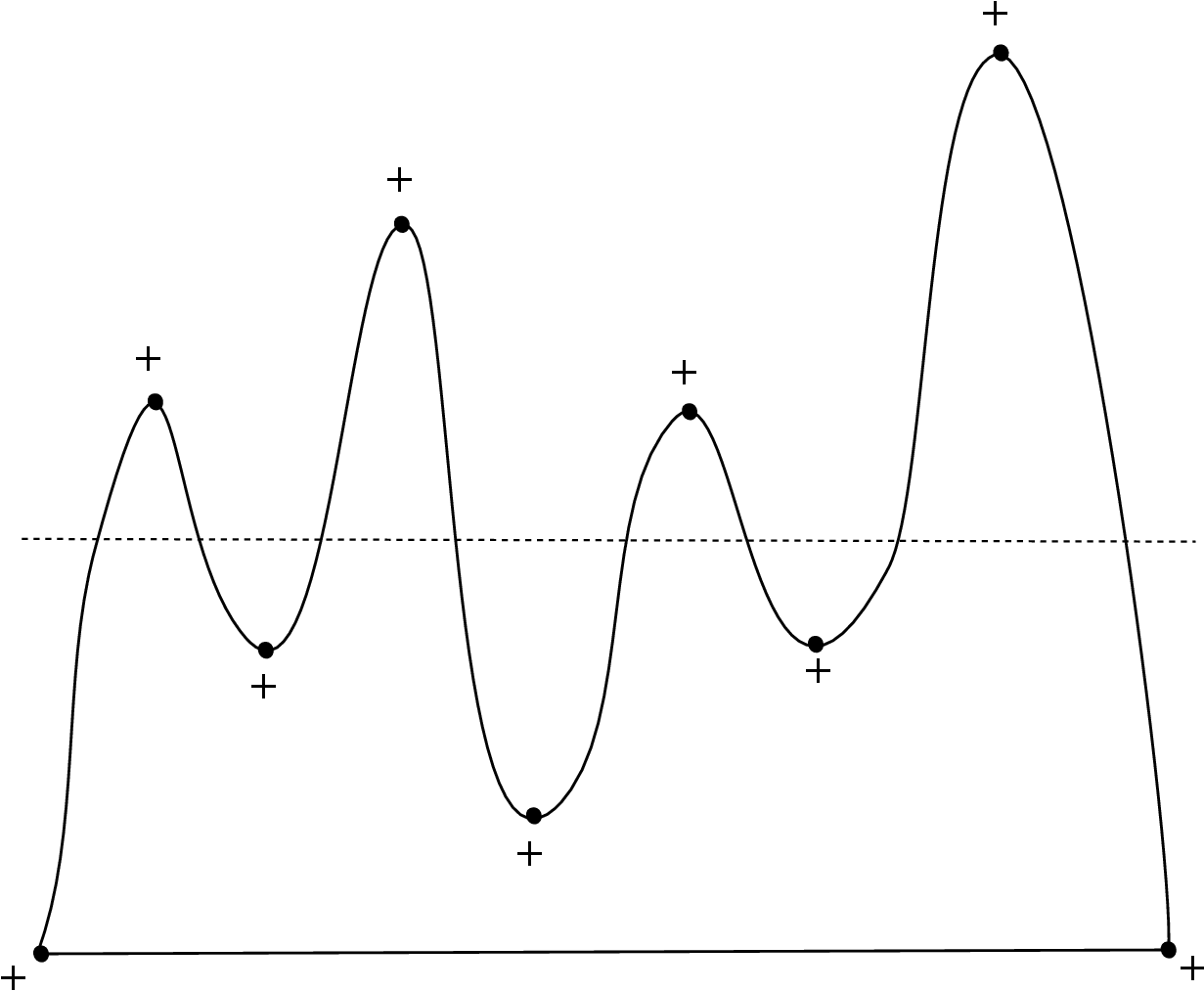}\caption{An alternation (always
supposed positive).}%
\label{f15}%
\end{figure}

\begin{figure}[ptb]
\includegraphics[width=85mm]{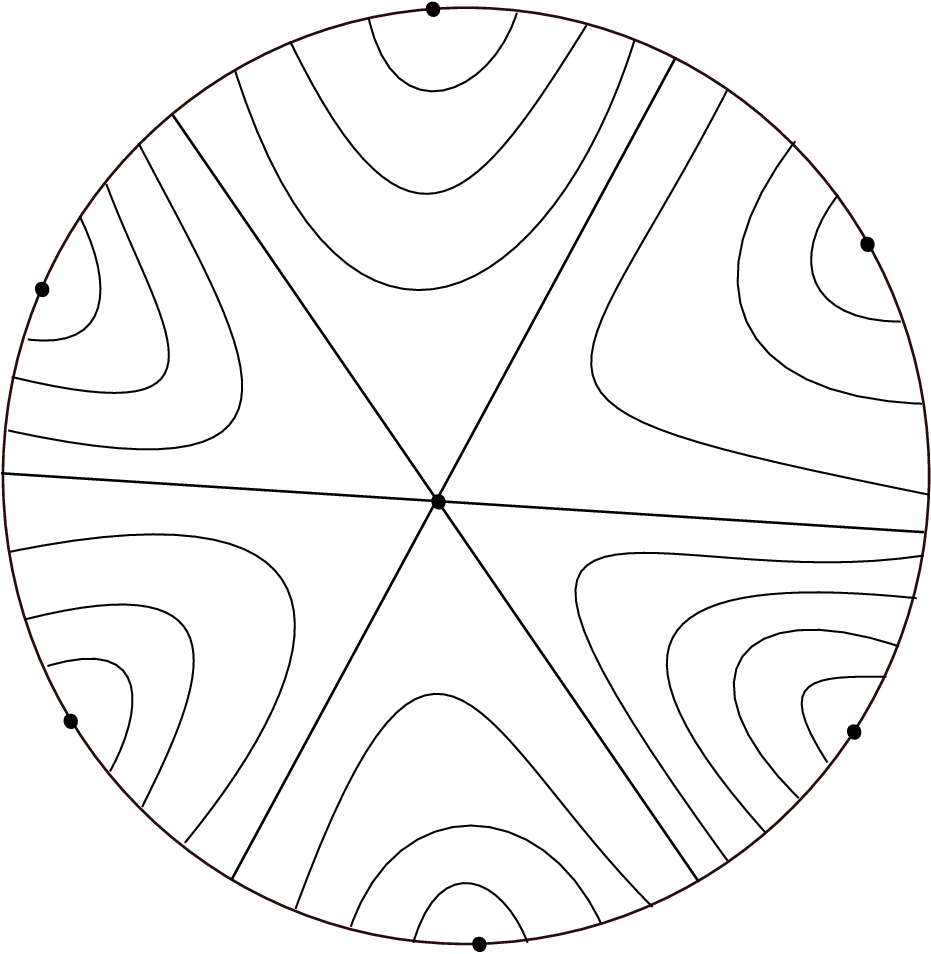}\caption{The solution for an
alternation, $\gamma=1$.}%
\label{f16}%
\end{figure}

\begin{figure}[ptb]
\includegraphics[width=75mm]{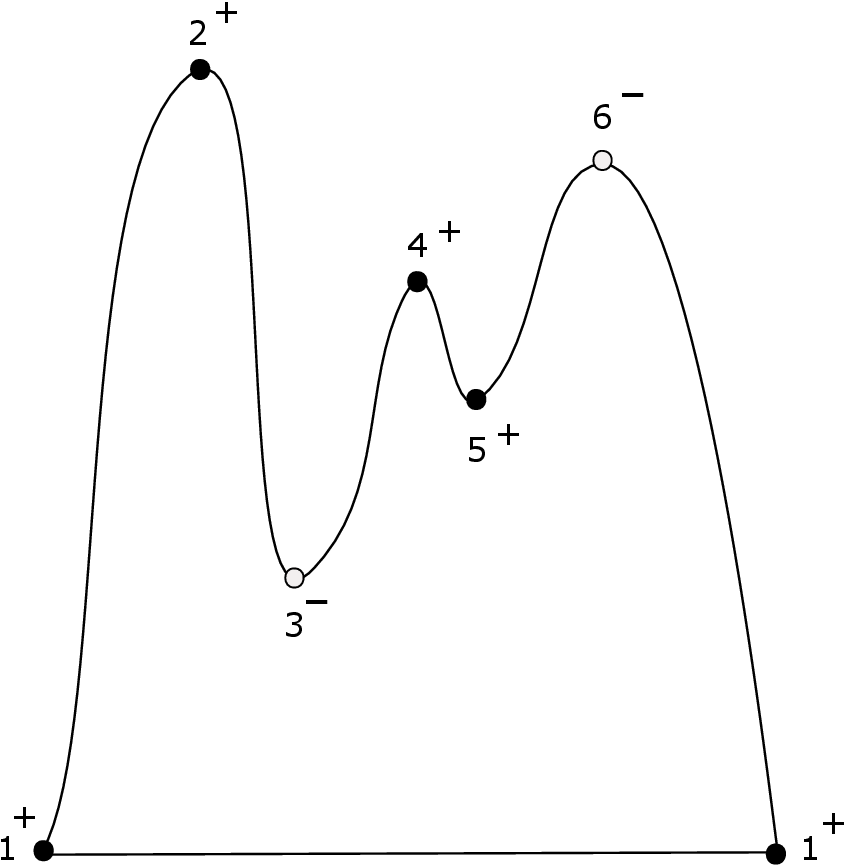}\caption{The ribbon $a=(1^{+}%
,6^{+},2^{-},4^{+},3^{+},5^{-})$.}%
\label{f17}%
\end{figure}

\begin{figure}[tbh]
\includegraphics[width=75mm]{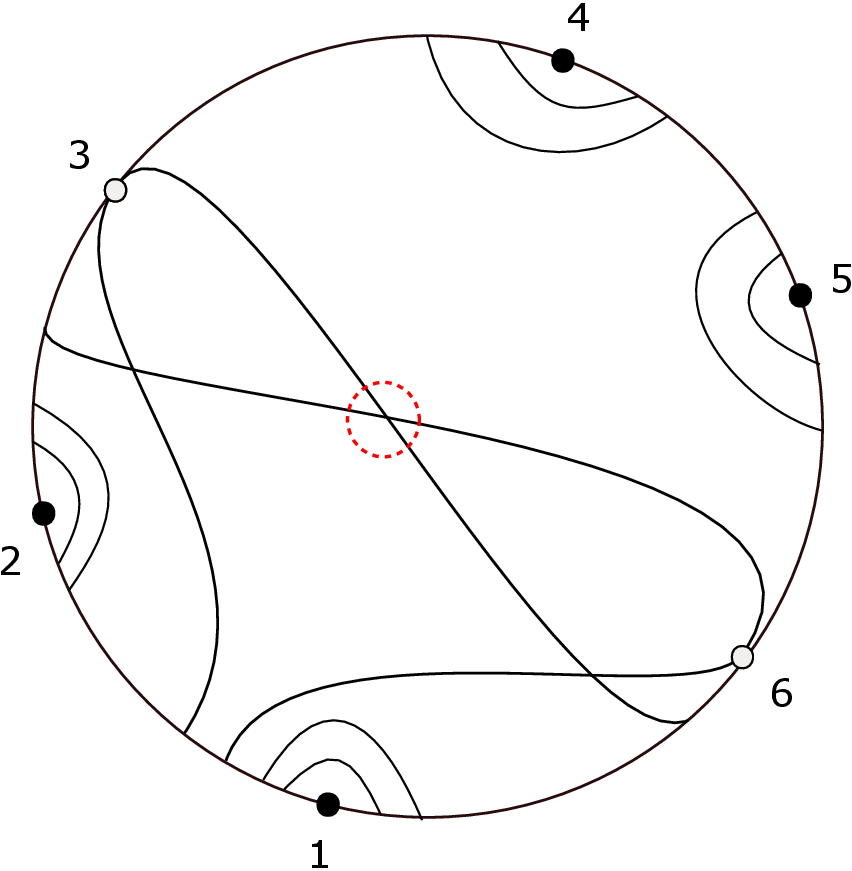}\caption{Nodes 3 and 6 are in
conflict.}%
\label{f18}%
\end{figure}

\begin{figure}[tbh]
\includegraphics[width=75mm]{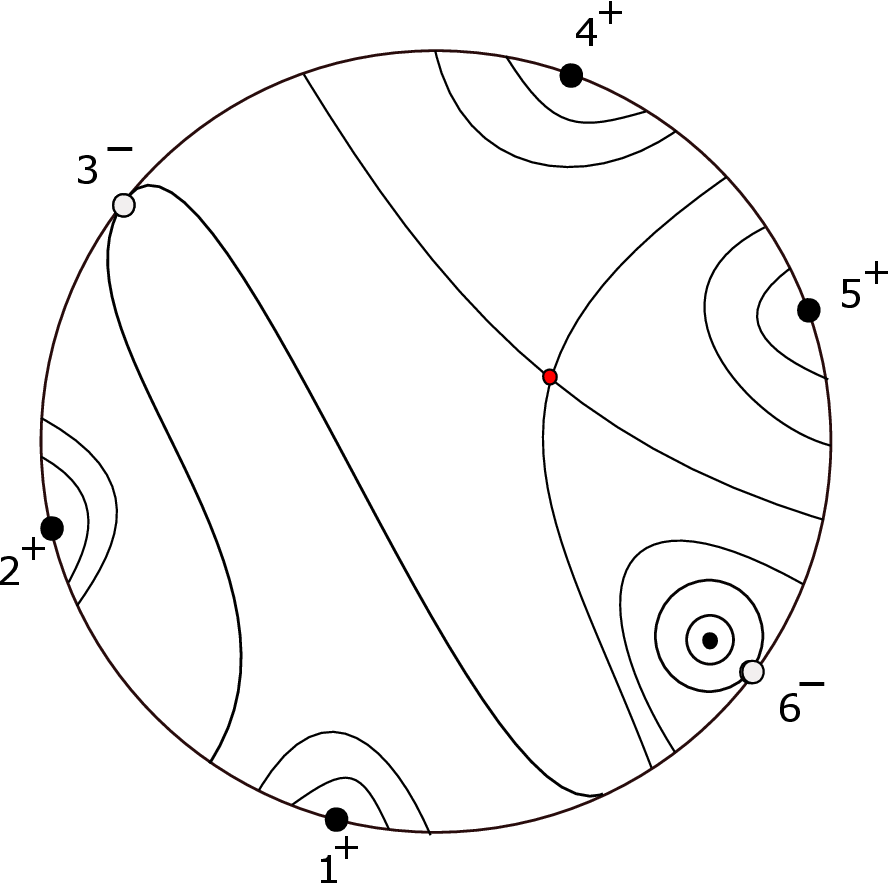}\caption{A solution, $\gamma=2$.}%
\label{f19}%
\end{figure}
\end{center}

Indeed, take some $\alpha\in\bigcap\limits_{i=1}^{n-1}[l_{i},l_{i+1}]$, then
for any maximal node $p_{i}$ we have $l_{i}=\varphi(p_{i})>\alpha$ and for any
minimal $p_{i}$, $l_{i}=\varphi(p_{i})<\alpha$. Now, one finds an extension
$f$ with only one critical level $\alpha$ and an unique saddle-type critical
point $P$ (\hyperref[f16]{Fig.~\ref*{f16}}). For $n=4$ the saddle $P$ is non
degenerate, while for $n>4$, $P$ is a degenerate saddle. As in the case of
ladders, we shall consider \textit{general }alternations, that is,
alternations with arbitrary marking of the nodes. A not very complicated
calculation shows that the number of (positive) alternations with $n$ nodes
equals $\left(  \frac{n}{2}\right)  !\left(  \frac{n}{2}-1\right)  !$, so the
number of general alternations is $2^{n}\left(  \frac{n}{2}\right)  !\left(
\frac{n}{2}-1\right)  !$\ Let us note that for general alternations the
invariant $\gamma$ is not at all easily computable, unlike the case of
ladders. In some sense, in class $\mathcal{A}^{+}$, ladders and alternations
are two opposite extremal cases from point of view of the ribbon invariant
$\gamma$. As we shall see later, all the positive alternations with the same
number of nodes $n$ may be somehow treated as one and the same ribbon
$\beta_{n}$. In particular, many ribbon invariants take one and the same value
on any ribbon from $\beta_{n}$ and belongs to the set of \textit{irreducible}
(elementary) ribbons (see \hyperref[s22]{Section~\ref*{s22}}).

5. The following example illustrates the introductory remarks. Consider the
ribbon $a=(1^{+},6^{+},2^{-},4^{+},3^{+},5^{-})$. It is shown at
\hyperref[f17]{Fig.~\ref*{f17}}. Since the signature $\sigma=s_{+}-$
$s_{-}=4-2=2$, it follows that for any extension $f$ of $a$ we have
$\deg(\nabla f|_{\mathbb{S}^{1}})=0$ and is\ due to the fact that $\deg(\nabla
f|_{\mathbb{S}^{1}})=1-\frac{\sigma}{2}$. So, the expectation is that there is
some extension without critical points, i.e. $\gamma(a)=0$. But it turns in
fact that $\gamma(a)>0$ (actually $\gamma(a)=2$)! Indeed, the segment $[4,3]$
has positive ends and does not contain levels of negative nodes. Then,
according to remark 2 of the present section, $f$ should have a critical
point. There is another purely geometrical argument: Suppose that there is a
critical points free extension and consider the corresponding level lines
portrait. Then the level lines passing through the negative nodes $2^{-}$ and
$5^{-}$ should finish in a single way somewhere on the boundary, but it is
easy to see that, by Jordan's Separation Theorem, they are in conflict (see
\hyperref[f18]{Fig.~\ref*{f18}}) and thus have to intersect, which is a
contradiction. A minimal extension with 2 critical points solving the problem
is shown at \hyperref[f19]{Fig.~\ref*{f19}}. Note that there is yet another
combinatorially different solution.

\section{\label{s4} Discretization of the ribbon space}

It is natural first to consider ribbons up to ``similarity'', as it is almost
clear that $\gamma$ takes the same value on any two \textit{similar} ribbons.

\begin{definition}
\label{d4}Two ribbons $a=(\varphi,\nu)$ and $a^{\prime}=(\varphi^{\prime}%
,\nu^{\prime})$ are said to be similar, if there is an orientation preserving
diffeomorphism $\psi:\mathbb{S}^{1}\rightarrow\mathbb{S}^{1}$ such that

1) $\psi$\ is sending the critical points set of $\varphi$ onto the critical
points set of $\varphi^{\prime}$: $\psi(p_{i})=p_{i}^{\prime}$, preserving
their type (minimum or maximum)

2) $\nu(p_{i})=\nu^{\prime}(p_{i}^{\prime})$

3) $\psi$\ is preserving the critical levels order:%
\[
(\varphi(p_{i})-\varphi(p_{j}))(\varphi^{\prime}(p_{i}^{\prime})-\varphi
^{\prime}(p_{j}^{\prime}))>0\text{ for }i\neq j\text{.}%
\]

\end{definition}

It is clear that two similar ribbons $a,a^{\prime}$ are identical from
combinatorial/topological point of view, so $\gamma(a)=\gamma(a^{\prime})$. We
shall denote the similarity classes by $\mathcal{A}_{0}$\textit{.} We shall
call the elements of $\mathcal{A}_{0}$\ \textquotedblleft
soft\textquotedblright\ ribbons and the elements of $\mathcal{A}$\ -
\textquotedblleft rigid\textquotedblright\ ribbons. It is clear that
$\mathcal{A}_{0}$ is countable and of combinatorial essence. So, it is natural
to identify it with some purely combinatorial structure, discretizing in such
a way the problem.

It turns out that the so called \textit{zig-zag permutations} are the
appropriate combinatorial tool for studying the ribbon invariant, as they
model the alternation \textquotedblleft min-max\textquotedblright\ of the
boundary data.

\begin{definition}
\label{d5}Let $(c_{1},c_{2},\dots,c_{n})$ be a permutation of the numbers
$\{1,2,\dots,n\}$. Then it is called zig-zag permutation, if its entries
alternately rise and descend: $c_{1}<c_{2}>c_{3}<\dots$
\end{definition}

Sometimes such a permutation is called \textit{up-down alternating
permutation} in contrast with \textit{down-up} permutations: $c_{1}%
>c_{2}<c_{3}>\dots$ The number of zig-zag permutations is found by Andr\'{e}
\cite{b2} via generating functions.

Let $A_{n}$ be the number of zig-zag permutations of $n$ elements. Consider
the series $A(x)=\sum_{n=1}^{\infty}A_{n}\frac{x^{n}}{n!}$. Then Andr\'{e}'s
Theorem states that
\[
A(x)=\tan x+\sec x\text{,}%
\]

in other words, $\tan x+\sec x$ is an exponential generating function for
sequence$~A_{n}$.

It should be noted that since $\tan x$ is odd and $\sec x$ is even, these two
functions control the number of odd and even zig-zag permutations separately.

Consider now the ribbon space $\mathcal{A}$ and take some $a=(\varphi,\nu
)\in\mathcal{A}$. Let the extrema of $\varphi$ (the nodes) be $p_{1}%
,\dots,p_{n}$\ , (then $n$ is an even number). Now, since the critical values
$l_{i}=\varphi(p_{i})$ are all different and go up-down, it is clear that we
may model function $\varphi$ by a cyclic zig-zag permutation $(c_{1}%
,c_{2},\dots,c_{n})$. Moreover, it is convenient to suppose $c_{1}=1$ and to
consider linear zig-zag permutations, instead of cyclic ones. For example,
$(1,3,2,4)$, $(1,6,2,4,3,5)$ are two cyclic zig-zag permutations in a linear
notation. Furthermore, we may define the \textit{mark} function $\nu
:P\rightarrow\{-1,+1\}$ setting $\nu(c_{i})=\nu(p_{i})$. In such a way, we get
some marked cyclic zig-zag permutation, which contains all the boundary
information of ribbon $a$.

\begin{definition}
\label{d6}Let $\mathcal{B}$ be the set of all pairs $(t,\nu)$, where
$t=(c_{1},c_{2},\dots,c_{n})$ is a cyclic zig-zag permutation with $n$ even,
and $\nu(c_{i})=\pm1$ is some mark function. We shall refer to $\mathcal{B}$
as the discrete ribbon space (or simply the ribbon space). Its elements
$b=(t,\nu)$ will be called discrete ribbons (or simply ribbons). The set of
ribbons of order $n$ will be denoted by $\mathcal{B}_{n}$.
\end{definition}

It is clear that we get some natural map $j:\mathcal{A\rightarrow B}$ such
that the co-images of discrete ribbons are similarity classes in $\mathcal{A}%
$. Now one defines the ribbon invariant $\gamma$ on $\mathcal{B}$ by setting
for $b\in\mathcal{B}$
\[
\gamma(b)=\gamma(a)\text{, where }a\in j^{-1}(b)\text{.}%
\]

It is clear also that $\mathcal{A}_{0}\mathcal{\approx B}$ by a natural isomorphism.

In such a way, the problem of computing $\gamma$ turns into a purely discrete
one and we shall see further that some inductive algorithm solves it. Of
course, it is not very fast and one of the reasons for this may be the obvious
fact, that the number of ribbons with $n$ nodes grows very fast as
$n\rightarrow\infty$. A better algorithm based on elementary moves and
reduction to ladders is proposed in Part~II.

So, let us see an expression for this number, based on Andr\'{e}'s Theorem.

\begin{proposition}
\label{p0}Let $2\tan2x=\sum_{n=2}^{\infty}B_{n-1}\frac{x^{n-1}}{(n-1)!}$. Then
$B_{n-1}$ is the number of ribbons with $n$ nodes.
\end{proposition}

\begin{proof}
As we noticed above, cyclic zig-zag permutations with $n$ nodes may be
identified with ordinary zig-zag permutations of type $(1,c_{2},\dots,c_{n})$.
But the number of latter equals the number of down-up permutations of order
$n-1$, so it is $A_{n-1}$ (the numbers of down-up and up-down permutations are
equal via the involution $k\rightarrow n-k+1$). Now one has $B_{n-1}%
=2^{n}A_{n-1}$, as there are $2^{n}$ mark functions.\ By Andr\'{e}'s theorem%
\[
\tan x=\sum_{n=2}^{\infty}A_{n-1}\frac{x^{n-1}}{(n-1)!}\text{, so}%
\label{andre}%
\]%
\[
\text{ }2\tan2x=\sum_{n=2}^{\infty}2A_{n-1}\frac{(2x)^{n-1}}{(n-1)!}%
=\sum_{n=2}^{\infty}B_{n-1}\frac{x^{n-1}}{(n-1)!}.
\]

\end{proof}

If one prefers to get as a generating function $\tan2x$ instead of $2\tan2x$,
it suffices to identify a cyclic permutation with the same one, but going in
reverse order (with the same marking), which is convenient from geometric
point of view and does not affect the ribbon invariant $\gamma$. (This is
equivalent to allowing orientation-reversing diffeomorphisms $\psi
:\mathbb{S}^{1}\rightarrow\mathbb{S}^{1}$ in the definition of similar
ribbons). However, we shall distinguish two such permutations furthermore.

Note that $2\tan2x=\frac{4}{1!}x+\frac{32}{3!}x^{3}+\frac{1024}{5!}x^{5}\dots$
that agrees with the fact that there are 4 ribbons with 2 nodes, 32 ribbons
with 4 nodes, etc. It is clear that the number of ribbons grows very fast as
$n\rightarrow\infty$. For example, the number of ribbons with 10 nodes is
$\left\vert \mathcal{A}_{10}\right\vert =1449\,132\,032$.

\begin{corollary}
Combining \hyperref[p0]{Proposition~\ref*{p0}} with the well known relation
between the Taylor expansion of $\tan x$ and Bernoulli numbers, on one hand,
and the asymptotics of Bernoulli numbers on the other, one obtains for the
asymptotics of the number of ribbons with $n$ nodes the following formula%
\[
\left\vert \mathcal{A}_{n}\right\vert \sim\frac{2^{n}(2^{n}-1)(n-1)!}{\pi^{n}%
}\text{, as }n\rightarrow\infty.
\]

\end{corollary}

\begin{remark}
The number of positive ribbons $a\in\mathcal{A}_{n}^{+}$ is computed via
generating function $\tan x$, i.e., if $\tan x=\sum_{n=2}^{\infty}C_{n-1}%
\frac{x^{n-1}}{(n-1)!}$, then $C_{n-1}$ is the cardinality of $\mathcal{A}%
_{n}^{+}$.
\end{remark}

This is a simple corollary of Andr\'{e}'s Theorem and the fact that positive
ribbons may be identified with zig-zag permutations. Of course, this allows
one to get easily the asymptotics of $|\mathcal{A}_{n}^{+}|$, the latter
number coinciding with the number of similarity classes of Morse functions
$\mathbb{\varphi}:\mathbb{S}^{1}\rightarrow\mathbb{R}$,%
\[
|\mathcal{A}_{n}^{+}|\sim\frac{(2^{n}-1)(n-1)!}{2\pi^{n}}\text{, as
}n\rightarrow\infty.
\]

Clearly, $|\mathcal{A}_{n}^{+}|=|\mathcal{A}_{n}^{-}|$ in a trivil way.
Moreover, if $\mathcal{A}\left(  n,\sigma\right)  $ is the set of ribbons with
$n$ nodes and signature $\sigma$ it is easy to see that%
\[
\left\vert \mathcal{A}\left(  n,\sigma\right)  \right\vert =\binom{n}%
{\frac{n+\sigma}{2}}|\mathcal{A}_{n}^{+}|\text{.}%
\]

Therefore, for the asymptotics of $\left\vert \mathcal{A}\left(
n,\sigma\right)  \right\vert $ one gets%
\[
\left\vert \mathcal{A}\left(  n,\sigma\right)  \right\vert \sim\binom{n}%
{\frac{n+\sigma}{2}}\frac{(2^{n}-1)(n-1)!}{2\pi^{n}}\text{, as }%
n\rightarrow\infty.
\]

Here $\sigma$ is mainly assumed to be a fixed constant, but it also may be
thought as a function $\sigma(n)$, such that $\left\vert \sigma(n)\right\vert
\leq n$ and the limit $\lim_{n\rightarrow\infty}\frac{\sigma(n)}{n}$ exists.
For example, $\sigma(n)=\pm n$ covers both the case of positive and negative
ribbons. Clearly, $\sigma=0$ implies the strongest possible asymptotics.

In order to carry out induction in the ribbon space $\mathcal{B}$, it is
convenient to define some linear (lexicographic) order in $\mathcal{B}$.

Let $a=(t,\nu)\in\mathcal{B}$, where $t=(c_{1},c_{2},\dots,c_{n})$. Set
$\nu(t)=(\nu(c_{1}),\dots,\nu(c_{n}))$. Consider the triple $(n,t,\nu(t))$. It
is clear now that we may introduce a lexicographic ordering in these triples,
assuming that $+1\prec-1$. In such a way, $\mathcal{B}$ is totally ordered and
has a minimal element $\alpha_{0}=(1^{+},2^{+})$ - the \textquotedblleft
minimal\textquotedblright\ ribbon. We shall denote it by \textquotedblleft%
$\prec$\textquotedblright\ again and write $a\prec b$. Clearly, this ordering
of $\mathcal{B}$ may be considered as a partial order in $\mathcal{A}$ as
well. Observe that the minimal and the maximal element of $\mathcal{A}_{n}$
are both ladders (an ascending and descending one, respectively), while the
alternations are situated somewhere at \textquotedblleft the
middle\textquotedblright\ of $\mathcal{A}_{n}$ with respect to this ordering.
It should not be confused with the partial order in $\mathcal{A}$ defined in
\hyperref[s21]{Section~\ref*{s21}}, the latter being more geometric and
interesting, in nature.

Let us finally make a practical convention, concerning notations:

\bigskip

\textit{We shall denote by }$\mathcal{A}$\textit{ the space of both rigid and
soft ribbons and the particular case will either be clear from the context, or
will it be mentioned specially.\bigskip}

We shall mainly think of $\mathcal{A}$ as the class of soft/discrete ribbons,
or, equivalently, marked cyclic zig-zag permutations.

\section{\label{s5} Splitting of ribbons and extensions}

Splitting of ribbons along a pair/triple of points and splitting of extensions
along a level line (regular or touching) is the main technical tool to prove
things by induction. Throughout this section $\mathcal{A}$\ will denote the
class of \textit{soft} ribbons, i.e. we consider ribbons up to similarity, or
equivalently - discrete ribbons.\ Let $a=(\varphi,\nu)\in\mathcal{A}$ be a
ribbon, $c$ be a non-critical level of $\varphi$ and $x_{1},x_{2}\in
\mathbb{S}^{1}$ be such that $\varphi(x_{1})=\varphi(x_{2})=c$ and
$\varphi^{\prime}(x_{1})\varphi^{\prime}(x_{2})<0$, where $\varphi^{\prime}$
is the derivative with respect to the natural parameter. So $\varphi$ has
opposite monotonicity at points $x_{1}$ and $x_{2}$ and we may split $a$ into
two other ribbons $a=a_{1}\#a_{2}$, $a_{1}=(\varphi_{1},\nu_{1})$ and
$a_{2}=(\varphi_{2},\nu_{2})$, as follows:

1) if $l_{1}$ and $l_{2}$ are the components of $\mathbb{S}^{1}\backslash
\{x_{1},x_{2}\}$\ (arcs) and $p_{1}:l_{1}\rightarrow\mathbb{S}^{1}$,
$p_{2}:l_{2}\rightarrow\mathbb{S}^{1}\ $are such that $p_{i}$ are one-to-one
on $l_{i}$ $\backslash\{x_{1},x_{2}\}$\ and $p_{i}(x_{1})=p_{i}(x_{2})$, set%
\[
\varphi_{i}(x)=\varphi(p_{i}^{-1}(x))\text{, }i=1,2\text{,}%
\]

in such a way each $\varphi_{i}$ gets a new-born extremum,

2) $\nu_{i}$ inherits the marking from $l_{i}$, while the new-born extrema are
marked as ``positive''.

See \hyperref[f20]{Fig.~\ref*{f20}-a} for an example of a ribbon splitting.
Note that $a_{1}\#a_{2}$ \textit{is not} an algebraic operation here, in
contrast with \hyperref[s20]{Section~\ref*{s20}}, where the \textit{ribbon
semigroup} is defined in the class of \textit{rigid} ribbons.

\begin{center}
\begin{figure}[ptb]
\includegraphics[width=75mm]{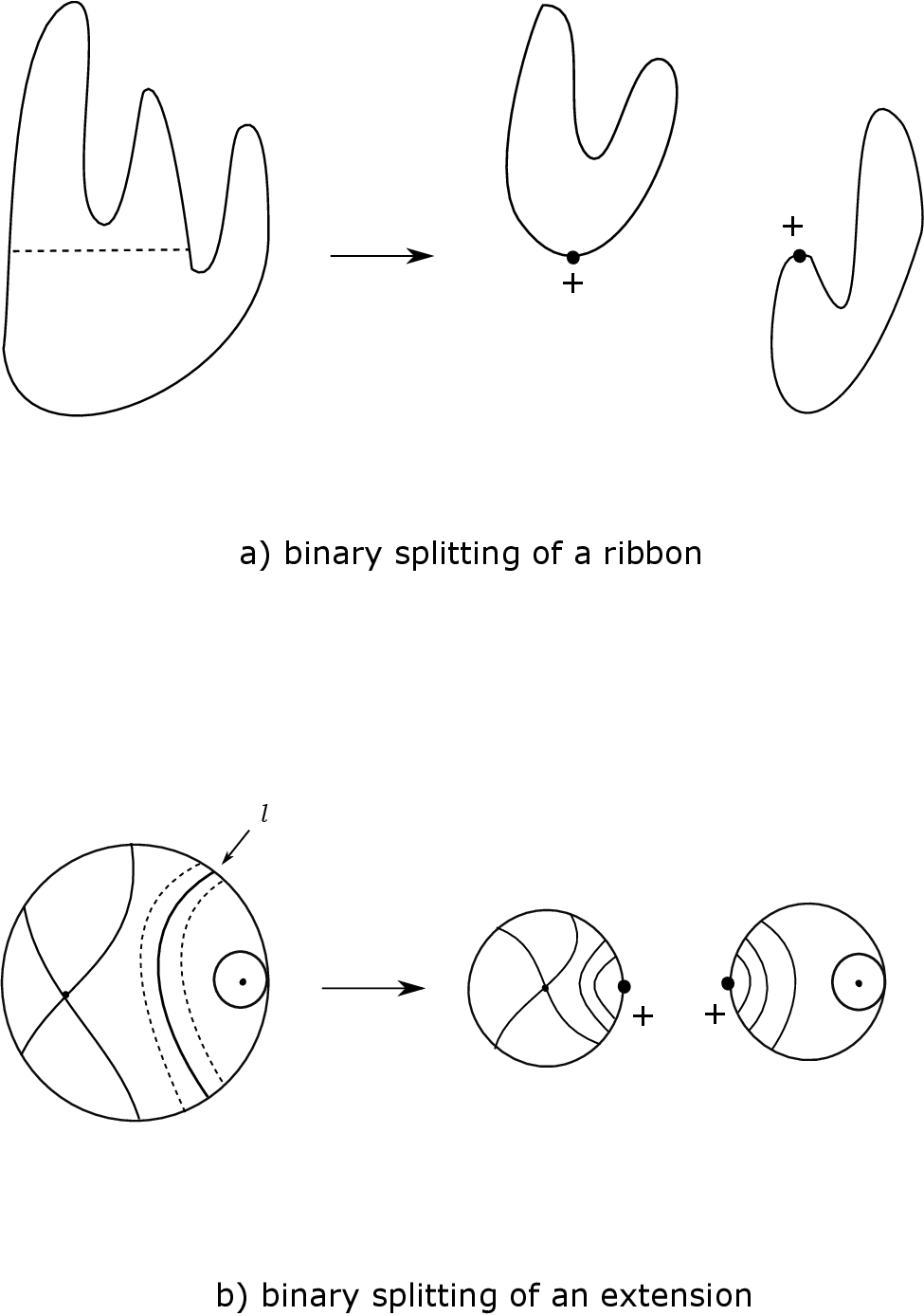}\caption{Binary splittings.}%
\label{f20}%
\end{figure}

\begin{figure}[ptb]
\includegraphics[width=75mm]{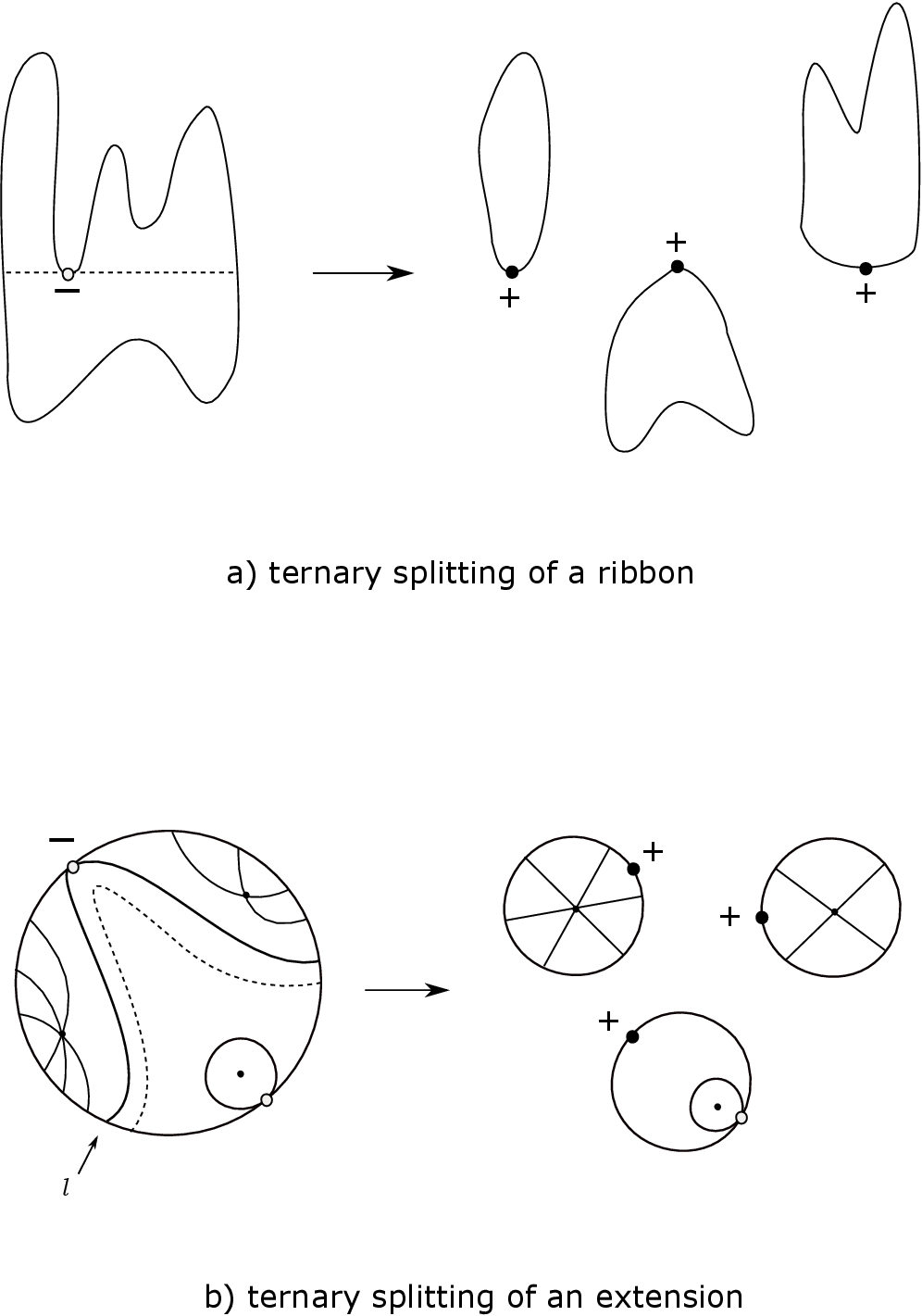}\caption{Ternary splittings.}%
\label{f21}\end{figure}
\end{center}

Let now $a\in\mathcal{A}$ and $f\in\mathcal{F}(a)$ ($f$ is an extension of
$a$). Let $l$ be a non-critical level line of $f$, so $l$ is a topological
segment intersecting transversely the boundary of $\mathbb{B}^{2}$ in points
$x_{1},x_{2}$. We may split $f$ along $l$: $f=f_{1}\vee f_{2}$ as follows. Let
$L_{1},L_{2}$ be the components of $\mathbb{B}^{2}$, take maps $P_{1}%
:\overline{L}_{1}\rightarrow\mathbb{B}^{2}$, $P_{2}:\overline{L}%
_{2}\rightarrow\mathbb{B}^{2}\ $such that $P_{i}(l)=y_{i}$ (a point) and
$P_{i}$ are one-to-one on $L_{i}\backslash l$. Then set%
\[
f_{i}(x)=f(P_{i}^{-1}(x))\text{, }i=1,2\text{.}%
\]

Functions $f_{i}$ are correctly defined, since $f|_{l}=const$.

Note that the ribbon $a$ splits then in a natural way into two other ribbons
$a=a_{1}\#a_{2}$ (along points $x_{1},x_{2}$), such that $f_{i}\in
\mathcal{F}(a_{i})$. See \hyperref[f20]{Fig.~\ref*{f20}-b} for an example by
the level portrait of some $f$.

\textit{Notes. }1) After the splitting we should ``smoothen'' both ribbons and
extensions in order to fall in the smooth class again. This will be explained
in Part~II. Another approach is to consider from the beginning piecewise
smooth (or even piecewise linear) functions, but then we shall lose the
gradient flow, which is important for our further investigations.

2) Splitting is defined for discrete ribbons in a similar way, but we have to
re-label the elements of the new zig-zag permutations. This will be explained
in full details in \hyperref[s11]{Section~\ref*{s11}}, where the algorithm for
calculating the ribbon invariant is described.

It turns out that we need another type of splitting, depending on the negative
nodes (if any). Let $p$ be a negative node and $x_{1},x_{2}\in\mathbb{S}^{1}$
be such that $\varphi(x_{1})=\varphi(x_{2})=\varphi(p)$ and $\varphi^{\prime
}(x_{1})\varphi^{\prime}(x_{2})<0$. (In such a way $\varphi(p)$ is different
from the minimal and maximal value of $\varphi$). Then we may define a
splitting of $a$ into three ribbons: $a=a_{1}\ast a_{2}\ast a_{3}$
analogically to the above splitting $\#$:

1) if $l_{1}$, $l_{2}$ and $l_{3}$ are the components of $\mathbb{S}%
^{1}\backslash\{x_{1},x_{2},p\}$\ and $p_{1}:l_{1}\rightarrow\mathbb{S}^{1}$,
$p_{2}:l_{2}\rightarrow\mathbb{S}^{1}$, $p_{3}:l_{3}\rightarrow\mathbb{S}%
^{1}\ $are such that $p_{i}$ are one-to-one on $l_{i}$ $\backslash
\{x_{1},x_{2},p\}$\ and $p_{i}(x_{1})=p_{i}(x_{2})$, set%
\[
\varphi_{i}(x)=\varphi(p_{i}^{-1}(x))\text{, }i=1,2,3\text{.}%
\]

As above, each $\varphi_{i}$ gets a new-born extremum.

2) $\nu_{i}$ inherits the marking from $l_{i}$, while the new-born extrema are
marked as ``positive''.

Just as above, each extension $f\in\mathcal{F}(a)$ splits $f=f_{1}\circ
f_{2}\circ f_{3}$, splitting $a$ into $a=a_{1}\ast a_{2}\ast a_{3}$, in such a
way that $f_{i}\in\mathcal{F}(a_{i})$. (See \hyperref[f21]{Fig.~\ref*{f21}})

We shall call $a_{1}\#a_{2}$ \textit{binary} splitting and $a_{1}\ast
a_{2}\ast a_{3}$ \textit{ternary} splitting. Splittings are of geometrical
nature: $a_{1}\#a_{2}$ corresponds to a splitting along a regular level line,
while $a_{1}\ast a_{2}\ast a_{3}$ corresponds to a splitting along a level
line touching the boundary from inside in a negative node. It turns out that
the ribbon invariant is subadditive with respect to splittings.

\begin{definition}
\label{d7}The ribbon $a$ is called reducible, if for any extension
$f\in\mathcal{F}(a)$ there is a non-critical level line $l$, either regular,
or touching, such that after splitting $a$ along $l$, $a=a_{1}\#a_{2}$ or
$a=a_{1}\ast a_{2}\ast a_{3}$, we have $a_{i}\prec a$, $\forall i$ (in the
defined above partial order). Otherwise $a$ is called irreducible.
\end{definition}

\begin{lemma}
\label{l1}The only irreducible ribbons are $\alpha_{0}=(1^{+},2^{+})$,
$\alpha_{1}=(1^{+},2^{-})$,\linebreak$\alpha_{2}=(1^{-},2^{+})$, as well as
all the positive alternations $\beta_{n},$ $n=4\dots$
\end{lemma}

\begin{proof}
Note that the second and the third ribbon are in fact identical. It is clear
that the first three are irreducible. Let $a\in\mathcal{A}^{+}$ be an
alternation. Then $\gamma(a)=1$ and if we take the unique $f\in\mathcal{F}(a)$
realizing $\gamma(a)$, (see \hyperref[f16]{Fig.~\ref*{f16}}), it is evident
from picture that $a$ is irreducible. Let $a$ be some ribbon different from
the listed ones. We shall prove that it is reducible. Suppose first that $a$
has a negative node $p$ and let $f\in\mathcal{F}(a)$ be realizing $\gamma(a)$.
Let $l$ be the level line of $f$ passing through $p$. There are two cases: a)
$l$ is closed. Take some level line $l^{\prime}$ close to $l$ situated in the
exterior of $l$. Then $l^{\prime}$ is regular and we may split $a$ along
$l^{\prime}$: $a=a_{1}\#a_{2}$. But now it is easy to see that $a_{i}\prec a$,
$i=1,2$. Indeed, one of them, say $a_{1}$ is equivalent to $(1^{+},2^{-})$,
while then $a_{2}\prec a$, since $a_{2}$ is identical with $a$, except for one
node, which has become ``positive''. b) $l$ is a regular touching line. Then,
after splitting $a$ along $l$: $a=a_{1}\ast a_{2}\ast a_{3}$, it is clear that
$a_{i}\prec a$, $i=1,2,3$, since $a_{i}$ has smaller number of nodes than $a$.
Suppose now that $a\in\mathcal{A}^{+}$. Consider the node levels
$l_{i}=\varphi(p_{i})$ and the system $\omega=\{[l_{i},l_{i+1}]\}$,
$i=1,\dots,n$ (assuming $l_{n+1}=l_{1}$). Take the minimal non-empty
intersections of elements of $\omega$. These are called later
\textit{clusters} (\hyperref[s7]{Section~\ref*{s7}}). It is not hard to see
that each such minimal intersection has the form $C=[\varphi(p_{i}%
),\varphi(p_{j})]$, where $p_{i}$ and $p_{j}$ have opposite type (a minimum
and a maximum) and there are no other critical levels of $\varphi$ in $C$.
Thus, the system $\beta$ of all clusters is disjoint. Since $a$ is not an
alternation, there are at least two elements $C_{1}$, $C_{2}$ of $\beta$. We
may suppose that $C_{1}$, $C_{2}$ are adjacent elements of $\beta$. Take a
non-critical value $c$ of $f$ between $C_{1}$ and $C_{2}$.\ Consider the
pairing $\psi:\varphi^{-1}(c)\rightarrow\varphi^{-1}(c)$ induced by $f$, i.e.
satisfying $f(x)=f(\psi(x))$. Let $\varphi^{-1}(c)=\{x_{1},x_{2},\dots
,x_{k}\}$. We shall show the following property of $\psi$: there is a pair
$(x_{i},\psi(x_{i}))$ which is dividing the set of nodes of $\varphi$ into 2
groups, each of them containing at least 3 nodes. Indeed, suppose the
contrary, then each pair $(x_{i},\psi_{i}(x))$ is dividing the nodes into 2
groups $A_{i}$ and $B_{i}$ such that $|A_{i}|=1$, $|B_{i}|=n-3$. Consider the
set $M=\cup A_{i}$. It is easily seen that the nodes in $M$ have the same
type, suppose that they are all maximums. But then $\varphi$ has no minimums
above level $c$ and we get contradiction with the choice of $c$ since it is
impossible to exist some cluster in $[c,\max\varphi]$ (it has to be of the
form $C=[\varphi(p_{i}),\varphi(p_{j})]$, where $p_{i}$ and $p_{j}$ have
opposite type). Now, let $(x_{i},\psi(x_{i}))$ be a pair which is dividing the
set of nodes into 2 groups, each of them containing at least 3 nodes. Let $l$
be the level line of $f$ joining $x_{i}$ with $\psi(x_{i})$. Split $f$ along
$l$: $f=f_{1}\vee f_{2}$, thus inducing a splitting $a=a_{1}\#a_{2}$. But now
one has $a_{1}\prec a$, $a_{2}\prec a$, as $a_{1}$ and $a_{2}$ have less nodes
than $a$. The lemma is proved.
\end{proof}

This concept is useful when carrying out induction on the lexicographic order,
then we have first to check the assertion for irreducible ribbons. Note also
that later, in \hyperref[s20]{Section~\ref*{s20}} we adopt some more simple
notation for splittings, interpreted as algebraic operations in $\mathcal{A}$.

\section{\label{s6} Economic extensions}

It is natural to look for extensions of a given ribbon $a$, which are
minimizing the number of critical points within a class of simple
\textquotedblleft good\textquotedblright\ extensions. It turns out that there
is such a class, that we call \textquotedblleft economic\textquotedblright%
\ extensions here. (The term \textquotedblleft good function\textquotedblright%
\ is overused in the literature, so we prefer another one instead.) The basic
fact about them is that for any ribbon $a\in\mathcal{A}$ there is an economic
extension with $\gamma(a)$ critical points, and moreover, any general position
extension with $\gamma(a)$ critical points \textit{is} \textit{indispensably}
economic. Before giving the definition, we show some graphical examples of
economic and non-economic extensions at \hyperref[f22]{Fig.~\ref*{f22}} and
\hyperref[f23]{Fig.~\ref*{f23}}. Note that not any Morse extension is an
economic one.

\begin{center}
\begin{figure}[ptb]
\includegraphics[width=85mm]{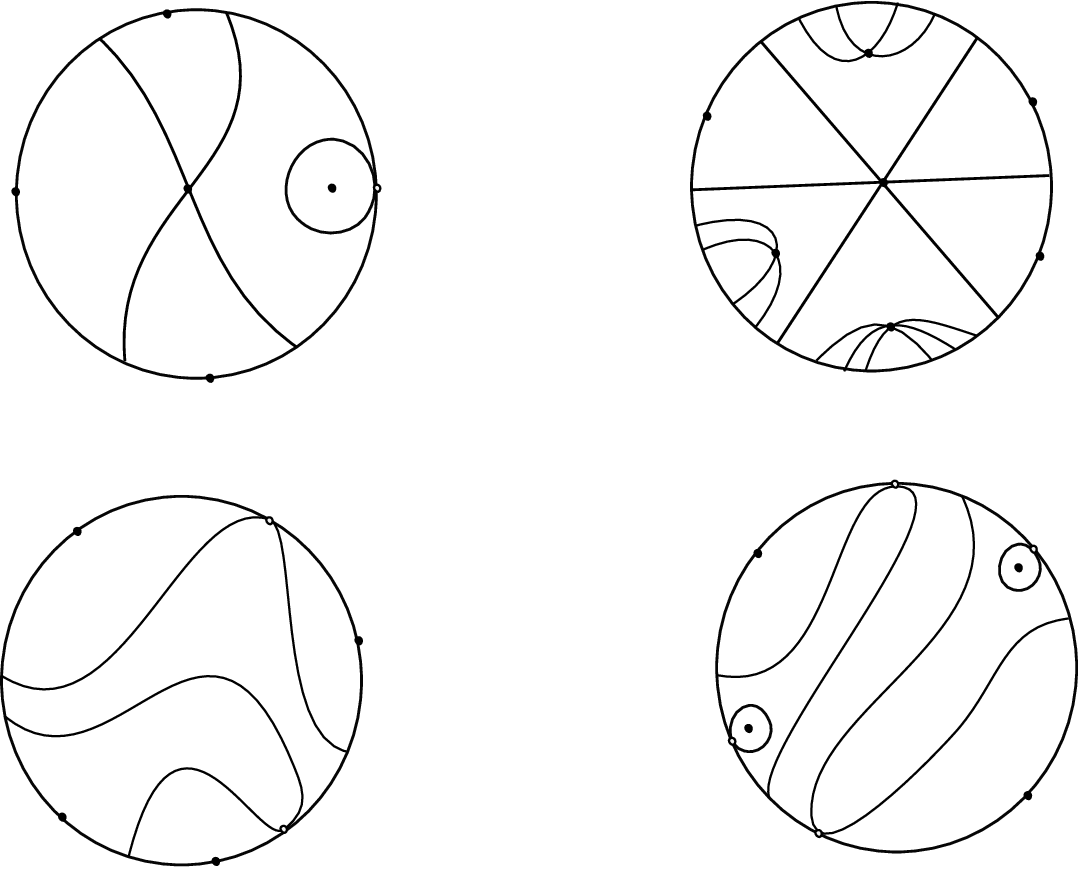}\caption{Economic extensions.}%
\label{f22}%
\end{figure}

\begin{figure}[ptb]
\includegraphics[width=85mm]{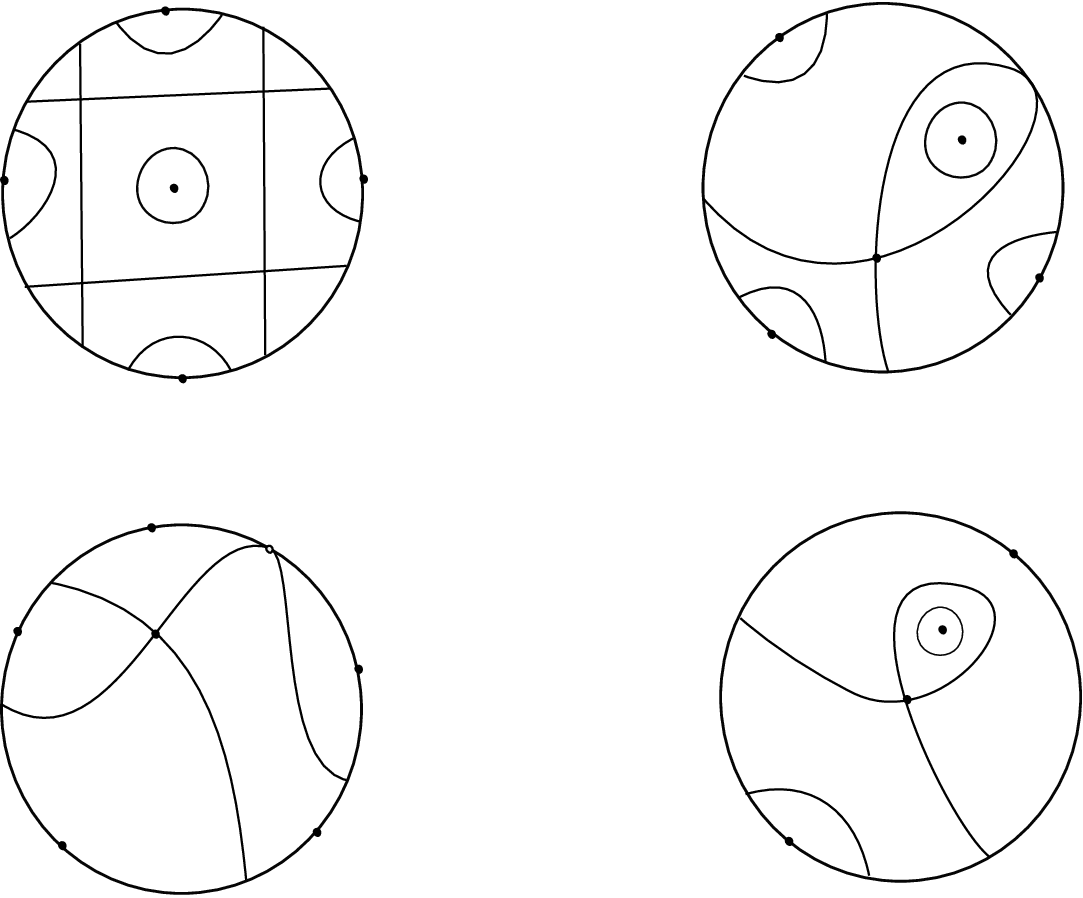}\caption{Non economic extensions.}%
\label{f23}%
\end{figure}
\end{center}

\begin{definition}
\label{d8}Let $a=(\varphi,\nu)\in\mathcal{A}$ be a ribbon and $f:\mathbb{B}%
^{2}\rightarrow\mathbb{R}$ be an extension with the corresponding boundary
data ($f\in\mathcal{F}(a)$). Then $f$ is called ``economic'', if

1) $f$ has a finite number of critical points with different critical values

2) each critical point is either a non degenerate local extremum, or a saddle
point (possibly degenerate)

3) the separatrices of each saddle finish transversely on the boundary

4) the level line passing through a negative node is either a) a topological
segment ending transversely on the boundary, or b) a simple closed curve with
exactly one critical point in its interior (then, of course, it is an extremum)
\end{definition}

This definition needs some clarification. In (\cite{b9}) A.~O.~Prishlyak
proved the following: Let $f:M\rightarrow\mathbb{R}$\ be a smooth function on
a closed surface $M$ with isolated critical points. Then in a neighbourhood of
a critical point of nonzero index, which is not a local extremum, $f$ is
(topologically) conjugated with $\operatorname{Re}(z^{k})$ for some
nonnegative integer $k$. It follows from this result that in a neighbourhood
of a critical point of nonzero index, the function $f$ is cojnugated to either
1) a typical local extremum: $f=\pm(x^{2}+y^{2})$, or 2) a typical saddle
point: $f=(y-x)(y-2x)\dots(y-nx)$, $n\geq2$ (for $n=2$ it is non degenerate).
A saddle has always an even number of separatrices going out of it. (In both
examples the critical point is $O$). Note that an economic extension cannot
have a critical point of index zero. The above arguments justify the
definition of economic extensions. Of course, we may speak about economic
\textit{functions}, instead of \textit{extensions} in \hyperref[d8]%
{Definition~\ref*{d8}}.

It may be shown that the economic extensions of any $a\in\mathcal{A}_{n}$ have
$\leq\frac{3n}{2}-1$ critical points (equality possible), that gives us the
estimate $\gamma\leq\frac{3n}{2}-1$, but we won't do that here, since the
sharper inequality $\gamma\leq\frac{n}{2}+1$ holds true. Note also that if we
allow ribbons with coinciding critical values, then a larger class of economic
extensions should be considered, where a level line may touch the boundary
multiple times.

Let $a=(\varphi,\nu)\in\mathcal{A}$ be a ribbon and $c$ be a non-critical
value of $\varphi$. Then $\varphi^{-1}(c)=\{x_{1},x_{2},\dots,x_{k}\}$, where
$k$ is even and $\varphi$ has opposite monotonicity at $x_{i}$ and $x_{j}$ for
$i-j$ odd. We shall call such a pair $(x_{i},x_{j})$ \textit{proper}.
Splitting of ribbons is possible only along proper pairs.

\begin{definition}
\label{d9}Let $c$ be a non-critical value of $\varphi$. A pairing in
$\varphi^{-1}(c)$ is an involution $\psi:\varphi^{-1}(c)\rightarrow
\varphi^{-1}(c)$, such that any two different pairs $(x,\psi(x))$,
$(y,\psi(y))$ are not linked in $\mathbb{S}^{1}$, i.e. the segments with the
corresponding ends are not intersecting in $\mathbb{B}^{2}$.
\end{definition}

It is easy to see that for a pairing $\psi$ any pair $(x,\psi(x))$ is proper.
At \hyperref[f24]{Fig.~\ref*{f24}} two different pairings are presented.

\begin{center}
\begin{figure}[ptb]
\includegraphics[width=85mm]{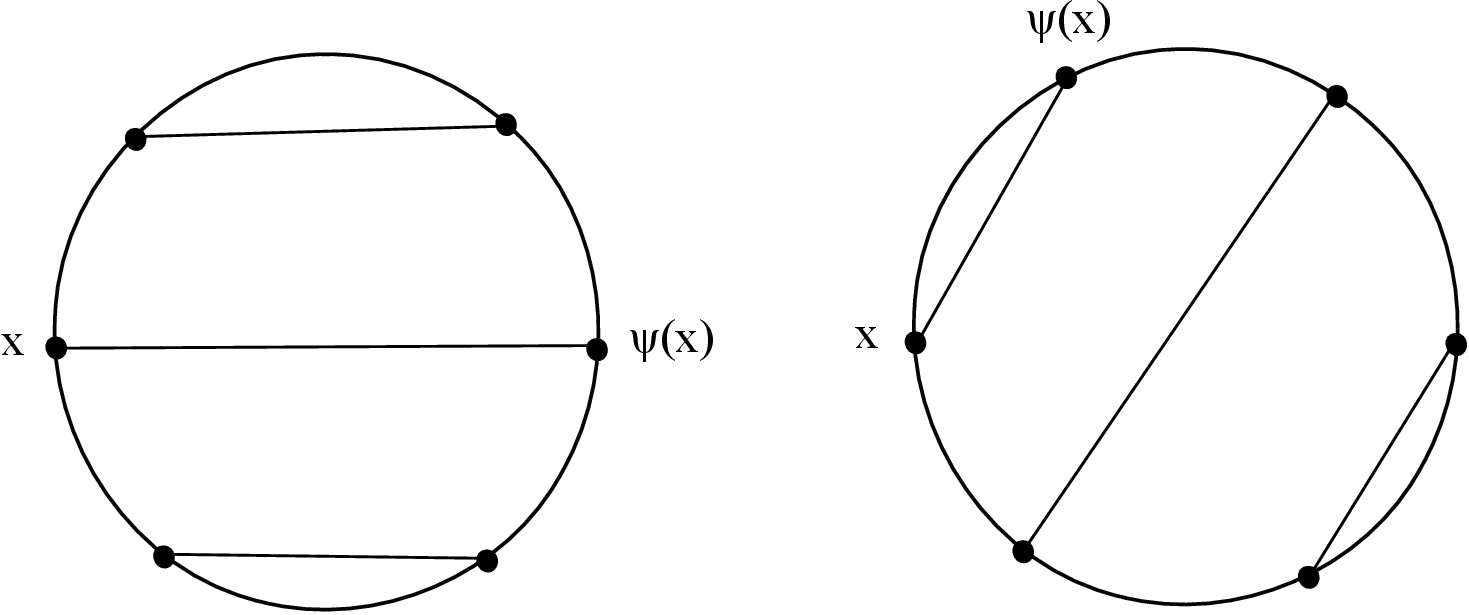}\caption{Two different pairings.}%
\label{f24}%
\end{figure}
\end{center}

\begin{definition}
\label{d10}A proper pair $(x_{i},x_{j})$ will be called essential, if after
splitting ribbon $a$ along $(x_{i},x_{j})$, $a=a_{1}\#a_{2}$, we have both
$a_{1}\prec a$ and $a_{2}\prec a$ in the lexicographic order (see
\hyperref[s4]{Section~\ref*{s4}}).

A non-critical value $c$ of $\varphi$ will be called essential, if for any
pairing $\psi:\varphi^{-1}(c)\rightarrow\varphi^{-1}(c)$ there is a pair
$(x,\psi(x))$, which is essential.
\end{definition}

Note that non-essential values exist; for example if the maximal value $M$ of
$\varphi$ is marked as ``positive'' and $c$ is near $M$, then $\varphi
^{-1}(c)=\{x_{1},x_{2}\}$, but it is easy to see that this pair is non-essential.

\begin{lemma}
\label{l2}Let $n\geq4$ and $a\in\mathcal{A}_{n}^{+}$ be a ribbon without
essential values. Then $a$ is an alternation.
\end{lemma}

\begin{proof}
Suppose first that the cluster number of $a$ is $\delta(a)\geq2$, so $a$ has
(at least) two clusters $C_{1}$ and $C_{2}$. (For the definition of cluster
number see \hyperref[s7]{Section~\ref*{s7}}.) Take a non-critical value $c$
between $C_{1}$ and $C_{2}$, then there are two groups of nodes $P_{1}$ and
$P_{2}$, such that $\varphi^{-1}(c)$ is a partition between them and
$|P_{1}|\geq3$, $|P_{2}|\geq3$. Take now some pairing $\psi:\varphi
^{-1}(c)\rightarrow\varphi^{-1}(c)$, then it is easily seen that some pair
$(x,\psi(x))$ is a partition between $P_{1}$ and $P_{2}$. But then splitting
$a$ along $(x,\psi(x))$, $a=a_{1}\#a_{2}$, one has that $a_{1}\prec a$ and
$a_{2}\prec a$ since $a_{1}$ and $a_{2}$ have smaller number of nodes than
$a$. Thus $c$ is an essential value. So we have $\delta(a)=1$, i.e. for the
node levels $l_{i}=\varphi(p_{i})$ we have $\cap_{i=1}^{n-1}[l_{i}%
,l_{i+1}]\neq\varnothing$. But since by assumption all nodes are positive, $a$
is an alternation. Note that some alternations may have essential values, but
this will not be important for us.
\end{proof}

Now we shall prove that the ribbon invariant is subadditive with respect to splittings.

\begin{lemma}
\label{l3}The following inequalities hold%
\[
\gamma(a_{1}\#a_{2})\leq\gamma(a_{1})+\gamma(a_{2})\text{, \ \ }\gamma
(a_{1}\ast a_{2}\ast a_{3})\leq\gamma(a_{1})+\gamma(a_{2})+\gamma
(a_{3})\text{,}%
\]

\end{lemma}

\begin{proof}
The proofs of both inequalities are short and very similar and we shall do
here only the first one. Let $a=a_{1}\#a_{2}$ and $f_{1}\in\mathcal{F}(a_{1}%
)$, $f_{2}\in\mathcal{F}(a_{2})$ are extensions realizing $\gamma(a_{1})$ and
$\gamma(a_{2})$, respectively. Now it is not difficult to construct
$f\in\mathcal{F}(a)$, so that $f=f_{1}\vee f_{2}$. But then $f$ has
$\gamma(a_{1})+\gamma(a_{2})$ critical points, therefore $\gamma(a_{1}%
\#a_{2})\leq\gamma(a_{1})+\gamma(a_{2})$ by the definition of $\gamma$.
\end{proof}

\begin{lemma}
\label{4}a) Suppose $f\in\mathcal{F}(a)$ is realizing $\gamma(a)$, $l$ is a
non-critical level line of $f$ and $f=f_{1}\vee f_{2}$ is a splitting along
$l$, inducing the splitting $a=a_{1}\#a_{2}$. Then%
\[
\gamma(a_{1}\#a_{2})=\gamma(a_{1})+\gamma(a_{2})\text{.}%
\]

b) Suppose $f\in\mathcal{F}(a)$ is realizing $\gamma(a)$, $l$ is a regular
level touching line and $f=f_{1}\circ f_{2}\circ f_{3}$ is a splitting along
$l$, inducing the splitting $a=a_{1}\ast a_{2}\ast a_{3}$. Then%
\[
\gamma(a_{1}\ast a_{2}\ast a_{3})=\gamma(a_{1})+\gamma(a_{2})+\gamma
(a_{3})\text{.}%
\]

\end{lemma}

\begin{proof}
a) In view of \hyperref[l3]{Lemma~\ref*{l3}}, suppose that $\gamma
(a_{1}\#a_{2})<\gamma(a_{1})+\gamma(a_{2})$. Then, since $f_{1}\in
\mathcal{F}(a_{1})$, $f_{2}\in\mathcal{F}(a_{2})$, it is clear that at least
one of $f_{1},f_{2}$ has less than $\gamma(a_{1})$ or $\gamma(a_{2})$ critical
points, a contradiction. b) is proved analogically.
\end{proof}

Before passing to the main result in the present section, let us make some
technical remark.

Suppose $a\in\mathcal{A}$ is a ribbon with a negative node $p$ and that the
extension $f\in\mathcal{F}(a)$ is realizing $\gamma(a)$. Then there is an
extension $f_{0}\in\mathcal{F}(a)$ which is realizing $\gamma(a)$ as well, and
the level set through $p$ is a regular touching line.

\begin{theorem}
\label{t1}For any ribbon $a\in\mathcal{A}$ there is an economic extension
$f_{0}\in\mathcal{F}(a)$ realizing $\gamma(a)$.
\end{theorem}

\begin{proof}
We shall proceed by induction on the lexicographic order in $\mathcal{A}$. For
$n\leq4$ this is very easily done ``by hand''. Another easy case is when $a$
is an alternation. Then $\gamma(a)=1$ and the economic extension with one
saddle (with $n$ separatrices) is the obvious solution. (Note that in this
case $\gamma(a)>0$, since otherwise one gets contradiction with the necessary
condition for $\gamma(a)=0$, see \hyperref[2]{Fact~\ref*{2}}.) Suppose that
$f\in\mathcal{F}(a)$ is some extension realizing $\gamma(a)$.\ Let now
$a\in\mathcal{A}_{n}$, where $n\geq6$ and assume that the theorem is true for
any ribbon $b$ such that $b\prec a$. Suppose first that $a$ has a negative
node $p$. As following from the above remark, we may suppose that the level
set of $f$ through $p$ is a regular touching line $l$. There are two cases a)
$l$ is closed. Then let $a_{0}$ be the ribbon identical with $a$, except in
node $p$, which is made ``positive''. It is clear that $\gamma(a)=\gamma
(a_{0})+1$ and $a_{0}\prec a$. Thus, there is an economic extension
$g\in\mathcal{F}(a_{0})$ realizing $\gamma(a_{0})$. Now, it is clear that we
may define an economic extension $f_{0}\in\mathcal{F}(a)$ with $\gamma
(a_{0})+1$ critical points, thus\ realizing $\gamma(a)$. b) $l$ is a
topological segment. Let $f=f_{1}\circ f_{2}\circ f_{3}$ be a splitting along
$l$, inducing the splitting $a=a_{1}\ast a_{2}\ast a_{3}$. It is easily seen
that $a_{i}\prec a$, $i=1,2,3$ since each $a_{i}$ has $<n$ nodes. Then by
assumption, there are economic extensions $g_{i}\in\mathcal{F}(a_{i})$
realizing $\gamma(a_{i})$. Thus we may define an economic extension $f_{0}%
\in\mathcal{F}(a)$ realizing $\gamma(a_{1})+\gamma(a_{2})+\gamma(a_{3})$, but
then \hyperref[4]{Lemma~\ref*{4}},\ b) implies that $f_{0}$ is realizing
$\gamma(a)$. Let now $a\in\mathcal{A}_{n}^{+}$. Then we may suppose that $a$
has an essential value $c$, since otherwise it is an alternation by
\hyperref[l2]{Lemma~\ref*{l2}}. Consider the pairing $\psi:\varphi
^{-1}(c)\rightarrow\varphi^{-1}(c)$ naturally induced by $f$ as follows:
$y=\psi(x)$ iff $x$ and $y$ are the ends of some non-critical level line of
$f$. Since $c$ is essential, there is an essential pair $(x,\psi(x))$. Let $l$
be the non-critical level line of $f$ connecting $x$ with $\psi(x)$. Then if
$f=f_{1}\vee f_{2}$ is a splitting along $l$, inducing the splitting
$a=a_{1}\#a_{2}$, we have $a_{i}\prec a$, $i=1,2$. By assumption, there are
economic extensions $g_{i}\in\mathcal{F}(a_{i})$ realizing $\gamma(a_{i})$.
Thus we may define an economic extension $f_{0}\in\mathcal{F}(a)$ realizing
$\gamma(a_{1})+\gamma(a_{2})$, but then \hyperref[4]{Lemma~\ref*{4}},\ a)
implies that $f_{0}$ is realizing $\gamma(a)$.
\end{proof}

\begin{remark}
In fact, we proved in the above theorem that any extension $f\in
\mathcal{F}(a)$ which is realizing $\gamma(a)$, is actually an economic one.
\end{remark}

We shall denote the set of economic extensions (up to topological equivalence)
of $a$ by $\mathcal{F}^{e}(a)$. The set of all economic extensions will be
denoted by $\mathcal{F}^{e}$ and the set of all economic extensions of ribbons
of order $n$ by $\mathcal{F}_{n}^{e}$.

\hyperref[t1]{Theorem~\ref*{t1}} allows us to somehow \textquotedblleft
finitize\textquotedblright\ the problem, as it is clear that there is only a
finite number of economic extensions of a given ribbon (up to, either
combinatorial, or topological equivalence). A \textquotedblleft brute
force\textquotedblright\ method for finding $\gamma$ would be to construct
\textit{all} economic extensions and then to look which ones have minimal
number of critical points. Of course, this is counterproductive (from point of
view of computation of $\gamma$), as the sets $\mathcal{A}_{n}$ and
$\mathcal{F}_{n}^{e}$ are very large and difficult to store and to search
within for big $n$. Moreover, we may consider the problem of finding the
(combinatorial) number $\phi_{r}(a)$ of economic extensions of a given ribbon
$a$ with exactly $r$ critical points. It is clear that%
\[
\gamma(a)=\min\{r|~\phi_{r}(a)\neq0\}\text{.}%
\]

Furthermore, it is not that difficult to see that%
\[
\phi_{r_{1}+r_{2}}(a\#b)\geq\phi_{r_{1}}(a)\phi_{r_{2}}(b),~~\phi_{r_{1}%
+r_{2}+r_{3}}(a\ast b\ast c)\geq\phi_{r_{1}}(a)\phi_{r_{2}}(b)\phi_{r_{3}%
}(c).
\]

Take now in the above inequalities $r_{i}=0$. Then%
\[
\phi_{0}(a\#b)\geq\phi_{0}(a)\phi_{0}(b),~~\phi_{0}(a\ast b\ast c)\geq\phi
_{0}(a)\phi_{0}(b)\phi_{0}(c)\text{,}%
\]

so, $\phi_{0}$ is a kind of \textit{super-multiplicative} invariant. Clearly,
$\phi_{0}$ makes sense only on the space of ribbons with $\gamma=0$, as it
counts the number of critical points free economic extensions. Note that every
critical points free extension \textit{is} economic in fact. We suppose that
the calculation of $\phi_{0}$ is a nontrivial task, although it may be
attacked algorithmically.

\textbf{Question 1.} For a given ribbon $a\in\mathcal{A}$,\ what is the
cardinality of $\mathcal{F}^{e}(a)$?

\textbf{Question 2.} What is the cardinality of $\mathcal{F}_{n}^{e}$?

\textbf{Question 3.} Are there any consistent estimates of the number
$\phi_{r}$ (especially for $\phi_{0}$)?

As we noticed above, it is not difficult to show by induction that for any
ribbon $a\in\mathcal{A}_{n}$, one has%
\[
\left\vert \mathcal{F}^{e}(a)\right\vert \leq\frac{3n}{2}-1\text{,}%
\]

where equality is possible for each $n$. Note also that there are ribbons with
unique economic extension ($\left\vert \mathcal{F}^{e}(a)\right\vert =1$).
Such are all the \textit{positive ladders} in addition with the 4 ribbons from
$\mathcal{A}_{2}$.

Another hard problem would be to find the number of \textit{topologically
}different economic extensions of a given ribbon $a$ with exactly $r$ critical
points (the number $\phi_{r}(a)$ is counting \textit{combinatorially}
different extensions). This seems to be the most difficult problem of this
kind, so we won't discuss it furthermore and shall concentrate our attention
mainly on the problems related to $\gamma$.

\section{\label{s7} The cluster number}

Here we shall define a simple and useful invariant of a given ribbon, the
so-called \textit{cluster number}, which is related to the ribbon invariant
$\gamma$\textit{.} It turns out that the cluster number is a \textit{ribbon
invariant} itself (see \hyperref[s22]{Section~\ref*{s22}}). The cluster number
provides us with a simply checkable estimate of $\gamma$ from below. However,
as a rule this estimate is quite far from reality, but may be useful in some
situations (for example when establishing $\gamma>0$.)

\begin{definition}
\label{d11}Let $\omega=\{X_{i}\}$ be a finite covering of set $X$. Its cluster
number $\delta(\omega)$ is defined as the minimal cardinality of finite
subsets $A$ of $X$, intersecting all elements of $\omega$: $A\cap X_{i}%
\neq\emptyset$ for any $i$.
\end{definition}

Sometimes the cluster number is called \textit{transversal number} (in the
context of hypergraphs).

Furthermore, we shall call the minimal non-empty intersections of elements of
$\omega$ \textit{clusters }of covering $\omega$. It is easily seen that the
cluster number $\delta(\omega)$ is less or equal to the number of clusters of
$\omega$. In general, the cluster number \textit{does not} equal the number of
clusters. Let us note that, from algorithmic point of view, the number of
clusters is much easier to be found than the cluster number$~\delta$.

\begin{definition}
\label{d12}Let $a=(\varphi,\nu)$ be a ribbon with nodes $p_{1},\dots,p_{n}$,
where $n\geq4$. Consider the node levels $l_{i}=\varphi(p_{i})$ and the system
$\omega=\{[l_{i},l_{i+1}]\}$, $i=1,\dots,n$ (assuming $l_{n+1}=l_{1}$). Then
the cluster number of $a$ is defined as the cluster number of $\omega$ and
will be denoted by $\delta(a)$. The cluster number of the ribbons with two
nodes\ is set to zero by convention. We shall call $\omega$ a level system of
ribbon$~a$.
\end{definition}

For example, the cluster number of an alternation equals one: $\delta(a)=1$,
since the elements of $\omega$ have non-empty intersection. On the other hand,
the cluster number of a ladder equals $\frac{n}{2}-1$, as it is not difficult
to be seen. Note also that clusters coincide with the segments of the form
$C=[\varphi(p_{i}),\varphi(p_{j})]$, where $p_{i}$ and $p_{j}$ have opposite
type (a minimum and a maximum) and there are no other critical levels of
$\varphi$ in $C$.

\begin{theorem}
\label{t2}Let $a=(\varphi,\nu)\in\mathcal{A}^{+}$ be a positive ribbon. Then
the inequality
\[
\gamma(a)\geq\delta(a)
\]
holds true.
\end{theorem}

\begin{proof}
We shall carry out induction on the lexicographic order of $a$. If $a$ is the
minimal ribbon, we have $\gamma(a)=\delta(a)=0$. For alternations
$a\in\mathcal{A}^{+}$ one has $\gamma(a)=\delta(a)=1$. So, the inequality is
true for irreducible ribbons in $\mathcal{A}^{+}$.\ Suppose the proposition is
true for ribbons $b$ such that $b\prec a$ and $a$ is reducible. Let
$f:\mathbb{B}^{2}\rightarrow\mathbb{R}$ be an economic extension of $a$:
$f\in\mathcal{F}^{e}(a)$ realizing $\gamma(a)$. Then $f$ has a regular level
$l$ such that for the corresponding splitting $a=a_{1}\#a_{2}$ we have
$a_{1}\prec a$, $a_{2}\prec a$. (Note that $f$ has no touching level lines,
since $a\in\mathcal{A}^{+}$ and has no negative nodes). Then by the induction
hypothesis $\gamma(a_{i})\geq\delta(a_{i})$, $i=1,2$. By \hyperref[4]%
{Lemma~\ref*{4}} we have $\gamma(a)=\gamma(a_{1}\#a_{2})=\gamma(a_{1}%
)+\gamma(a_{2})$, since $f$ is realizing $\gamma(a)$. It is clear also that
$\delta(a_{1})+\delta(a_{2})\geq\delta(a)$, as any cluster of $\omega$ is a
cluster either of $\omega_{1}$ or of $\omega_{2}$, where $\omega_{i}$ is the
level system of $a_{i}$.\ Finally%
\[
\gamma(a)=\gamma(a_{1})+\gamma(a_{2})\geq\delta(a_{1})+\delta(a_{2})\geq
\delta(a)\text{.}%
\]
The theorem is proved.
\end{proof}

This theorem will be obtained in a different way in \hyperref[s22]%
{Section~\ref*{s22}} where some maximal property of $\gamma$ among all
\textit{ribbon-type} invariants is proved. Note also that the cluster number
is actually estimating from below the number of different \textit{critical
values}, rather than the number of critical points.

It turns out that in class $\mathcal{A}^{+}$ the cluster number $\delta$\ is a
consistent estimate of $\gamma$ from below in many cases. Of course, for the
majority ribbons in $\mathcal{A}^{+}$ we have $\gamma\gg\delta$. For example
it is not difficult to construct ribbons with $\delta=2$ and arbitrarily large
$\gamma$, it suffices to make function $\varphi$\ oscillate between two
``cluster values'' many times. On the other hand, ribbons from $\mathcal{A}%
^{+}$ with $\gamma=\delta$ have some interesting properties. We shall call
them \textit{quasi-ladders}. For true ladders we have $\gamma=\delta=\frac
{n}{2}-1$. The elements of $\mathcal{A}_{n}^{+}$ with maximal ribbon invariant
$\gamma$ are true ladders.

It would be interesting to find the mathematical expectation of $\frac{\delta
}{n}$ in class $\mathcal{A}^{+}$. A much harder problem is to find the
expectation of $\frac{\gamma}{n}$ in class $\mathcal{A}$.

Let now $a\in\mathcal{A}$ be an arbitrary ribbon. Then the inequality
$\gamma(a)\geq\delta(a)$ fails.\ For example, take some function
$f:\mathbb{B}^{2}\rightarrow\mathbb{R}$\ without critical points and consider
the ribbon $a$ defined by $f$. Then $\gamma(a)=0$, while $\delta(a)$ may be
arbitrarily large. Anyway, it is possible to define a variant of the cluster
number that gives estimation from below of $\gamma(a)$ in this case.

\begin{definition}
\label{d13}Let $a=(\varphi,\nu)$ be a ribbon with nodes $p_{1},\dots,p_{n}$,
where $n\geq4$. Consider the level system $\omega$ of $a$ and the subsystem
$\omega_{0}\subset\omega$ with elements of the type $[\varphi(p_{i}%
),\varphi(p_{i+1})]$ where $p_{i},p_{i+1}$ are positive nodes and
$[\varphi(p_{i}),\varphi(p_{i+1})]$\ does not contain any $\varphi(p_{j})$,
where $p_{j}$ varies among all \textit{negative} nodes. Then the reduced
cluster number of $a$ is defined by $\delta_{0}(a)=\delta(\omega_{0})$.
Clearly, $\delta_{0}(a)=\delta(a)$ for $a\in\mathcal{A}^{+}$. It may be proved
that the reduced cluster number gives estimation from below of $\gamma$:
\end{definition}

\begin{theorem}
\label{t3}The inequality $\gamma\geq\delta_{0}$ holds in class $\mathcal{A}$
of arbitrary ribbons.
\end{theorem}

We shall omit the proof as it follows almost literally that of \hyperref[t2]%
{Theorem~\ref*{t2}}. So, it turns out that $\delta_{0}$ is an a priori
estimate for $\gamma$. It can be easily shown that $\delta_{0}$ is a ribbon
invariant; however, it is not a quite consistent estimate for $\gamma$ (in
general, $\gamma$ is much greater than $\delta_{0}$). Anyway, it may be useful
in some situations, especially for proving that $\gamma>0$ for a given ribbon.

In fact, the above inequality may be specified as follows:%
\[
\gamma\geq\beta\geq\delta_{0}\text{,}%
\]

where $\beta$ is the minimal number of critical values of all extensions of
the ribbon (see \hyperref[s2]{Section~\ref*{s2}}). It turns also that in class
$\mathcal{A}^{+}$ the invariant $\beta$ equals exactly the cluster number
$\delta$:

\begin{theorem}
\label{t4}For any ribbon $a\in\mathcal{A}^{+}$ we have $\beta(a)=\delta(a)$.
\end{theorem}

\section{\label{s8} General estimates of $\gamma$}

We shall establish here some general inequalities involving the ribbon invariant.

Let the ribbon $a\in\mathcal{A}_{n}$ have $s_{+}$ positive and $s_{-}$
negative nodes (so, $n=s_{+}+s_{-}$). The signature $\sigma(a)$ is defined by%
\[
\sigma(a)=s_{+}-s_{-}\text{.}%
\]

Note that $\sigma$ is an even number. It turns out that the signature is an
useful invariant, as it is involved in many results about $\gamma$. One of the
reasons for this is the following fact:

\begin{proposition}
\label{3}Let $a\in\mathcal{A}$ and $f\in\mathcal{F}(a)$. Then%
\[
\deg(\nabla f|_{\mathbb{S}^{1}})=1-\frac{\sigma}{2}\text{,}%
\]

where $\sigma=\sigma(a)$ is the signature of $a$.
\end{proposition}

For example, if $a$ is the minimal ribbon, we have $\sigma=2$, so the above
degree is~0, which agrees with the fact that there is an extension without
critical points. We shall omit the proof of the proposition, as this is a
common fact that becomes evident from a simple calculation of the degree. Let
us note some immediate corollaries from this proposition.

\begin{fact}
Let $f\in\mathcal{F}(a)$ have a finite number of critical points. Then the
algebraic sum of their indices equals $1-\frac{\sigma}{2}$, where
$\sigma=\sigma(a)$.
\end{fact}

This is an immediate corollary from Hopf's Theorem and \hyperref[3]%
{Proposition~\ref*{3}}.

\begin{fact}
\label{2}If $\gamma(a)=0$, then $\sigma(a)=2$.
\end{fact}

This is equally clear, since if there is a critical points free $f\in
\mathcal{F}(a)$, then $\deg(\nabla f|_{\mathbb{S}^{1}})=0$ and \hyperref[3]%
{Proposition~\ref*{3}} implies $\sigma=2$.

Note that the converse is \textit{not} true and this is one of the motivating
reasons for the appearance of this article. It suffices to look at the ribbon
from \hyperref[f17]{Fig.~\ref*{f17}} where $\sigma=2$, but $\gamma>0$, as it
was shown previously in the text. Another almost trivial situation is a ribbon
with $\sigma=2$ which has a negative minimal or maximal node. Then it is
obvious that any extension should have a critical point, so $\gamma>0$.

For a given ribbon $a$, it is convenient to denote%
\[
i(a)=1-\frac{\sigma(a)}{2}\text{.}%
\]

We shall call $i(a)$ \textit{index} of $a$. Note that the signature $\sigma$
is not additive under splittings, while the index is.

\begin{lemma}
\label{l5}a) $i(a_{1}\#a_{2})=i(a_{1})+i(a_{2})$, \ \ b) $i(a_{1}\ast
a_{2}\ast a_{3})=i(a_{1})+i(a_{2})+i(a_{3})$.
\end{lemma}

As this is a common property of the degree, we leave the proof of
\hyperref[l5]{Lemma~\ref*{l5}} to the reader.

It is a simple observation that the signature has some nontrivial topological
meaning concerning the ribbon space $\mathcal{A}$. In fact, under some natural
convention about the admissible moves in $\mathcal{A}$, it turns out that

\begin{center}
\textit{the subspaces of }$\mathcal{A}$\textit{ of fixed signature }$\sigma
$\textit{ are in fact the components of space }$\mathcal{A}$\textit{.}
\end{center}

Here we explain in brief the situation. Let us allow, for the moment, ribbons
with coinciding critical values as well as \textquotedblleft
birth-death\textquotedblright\ of a pair of consecutive nodes of type
$(p_{i}^{+}~p_{i+1}^{-})$ or $(p_{i}^{-}~p_{i+1}^{+})$, which operation does
not affect any of the ribbon invariants. Then it is not difficult to see that
any two ribbons $a,b$ of the same signature $\sigma(a)=\sigma(b)$ may be
connected by a path in $\mathcal{A}$, and of course, if $\sigma(a)\neq
\sigma(b)$ this cannot be done. The topology of the ribbon space $\mathcal{A}$
will be considered in more detail in Part II.

The following is the main general inequality for the ribbon invariant.

\begin{theorem}
\label{t5}For any ribbon the following inequality holds%
\begin{equation}
1-\frac{\sigma}{2}\leq\gamma\leq n-1-\frac{\sigma}{2}. \label{12}%
\end{equation}

\end{theorem}

\begin{proof}
We proceed by induction on the ordering of $\mathcal{A}$. Consider the first
part $1-\frac{\sigma}{2}\leq\gamma$, that is in fact $i\leq\gamma$. For
irreducible ribbons it is obvious. Let $a\in\mathcal{A}$ be a reducible one
and suppose the inequality is true for any $b\in\mathcal{A}$ with $b\prec a$.
Let $f\in\mathcal{F}(a)$ is realizing $\gamma(a)$. Then there is a
non-critical level line $l$ (either regular, or touching) such that after
splitting $a$ along $l$, $a=a_{1}\#a_{2}$ or $a=a_{1}\ast a_{2}\ast a_{3}$ we
have $a_{k}\prec a$, $\forall k$. By induction hypothesis $i(a_{k})\leq
\gamma(a_{k})$, $k=1,2,3$, so%
\[
i(a)=i(a_{1}\#a_{2})=i(a_{1})+i(a_{2})\leq\gamma(a_{1})+\gamma(a_{2}%
)=\gamma(a)\text{, or}%
\]%
\[
i(a)=i(a_{1}\ast a_{2}\ast a_{3})=i(a_{1})+i(a_{2})+i(a_{3})\leq\gamma
(a_{1})+\gamma(a_{2})+\gamma(a_{3})=\gamma(a)\text{,}%
\]
which proves the first part. Note that for the last equalities we make use of
\hyperref[4]{Lemma~\ref*{4}}. Consider now the second part $\gamma\leq
n-1-\frac{\sigma}{2}$, which is equivalent to $\gamma\leq n-2+i$. As above, it
is clear for irreducible ribbons. Under the above assumptions, take a
reducible $a\in\mathcal{A}$ and consider first a binary splitting
$a=a_{1}\#a_{2}$. Let $n_{1}$, $n_{2}$ be the number of nodes of $a$, lying
respectively in $a_{1}$, $a_{2}$, then the number of nodes of $a_{1}$, $a_{2}$
equals $n_{1}+1$, $n_{2}+1$, respectively. Now%
\[
\gamma(a)\leq\gamma(a_{1})+\gamma(a_{2})\leq n_{1}+1-2+i(a_{1})+n_{2}%
+1-2+i(a_{2})=n-2+i(a)\text{.}%
\]
Consider now a ternary splitting $a=a_{1}\ast a_{2}\ast a_{3}$. If as above
$n_{i}$ is the number of nodes of $a$, lying in $a_{k}$, then the number of
nodes of $a_{k}$ equals $n_{k}+1$. Note that $n=n_{1}+n_{2}+n_{3}+1$. Then we
have $\gamma(a_{k})\leq n_{k}+1-2+i_{k}=n_{k}-1+i_{k}$ by induction
hypothesis. Now one has%
\[
\gamma(a)\leq\gamma(a_{1})+\gamma(a_{2})+\gamma(a_{3})\leq n_{1}+n_{2}%
+n_{3}-3+i(a)=n-4+i(a)<n-2+i(a)\text{.}%
\]
The theorem is proved.
\end{proof}

Let us make some remarks about this inequality and its proof.

We shall show later, that in fact any $f\in\mathcal{F}(a)$ has $\geq
1-\frac{\sigma}{2}$ local extrema.

Note that in the second part we may take $f\in\mathcal{F}(a)$ simply being an
economic extension, not supposing that it is realizing $\gamma(a)$. In such a
way, we prove that for any economic $f\in\mathcal{F}(a)$ with $\Gamma$
critical points we have%
\[
1-\frac{\sigma}{2}\leq\Gamma\leq n-1-\frac{\sigma}{2}\text{,}%
\]

as $\gamma(a)\leq\Gamma$. Note also that in the final part of the proof there
is a ``gap'' of 2 units, that will be important for us in order to improve the
second part of the inequality by reducing the upper limit. In some cases the
quantity $n-1-\frac{\sigma}{2}$ is too big and gives an inconsistent estimate
of $\gamma$ from above. For example if $a\in\mathcal{A}^{-}$, we have
$\sigma=-n$ and we get $\frac{3n}{2}-1$ as an upper limit for $\gamma(a)$,
which is very far from reality, since later we show that $\gamma=\frac{n}%
{2}+1$ in class $\mathcal{A}^{-}$ (and in general $\gamma\leq\frac{n}{2}+1$).

The first part $\gamma\geq1-\frac{\sigma}{2}$ gives a simple ``a priory''
estimate of $\gamma$ and makes sense only for $\sigma\leq0$. It does not
depend on the $C^{0}$-data of the ribbon. Later we shall explore it to
estimate the number of critical points of a function on the 2-sphere. Note
also that not all non-negative integers in $[1-\frac{\sigma}{2},n-1-\frac
{\sigma}{2}]$ arise as a ribbon invariant of a ribbon with corresponding data
$n,\sigma$. It even happens that not all such integers arise as the number of
critical points of an economic extension of a ribbon with the corresponding
data. For example, there are no economic extensions with $\Gamma=1$ critical
points of a ribbon with $\sigma=2$ although $1\in\lbrack0,n-2]$. It implies
that $\gamma=1$ is not a realizable value among all ribbons with $\sigma=2$.
This is easily seen from the fact that $i(a)=1-\frac{\sigma}{2}=0$, so
supposing that there is an economic extension with unique critical point $p$,
then the index of $p$ should equal zero. But this is impossible in the
economic class.

In order to improve the second part of the general inequality (\ref{12}) we
shall define another invariant of a ribbon $a$, counting the maximal possible
number of touching lines among all extensions of $a$.

\begin{definition}
\label{d14}Let $a\in\mathcal{A}$ be a ribbon, then by $t(a)$ we shall denote
the maximal possible number of regular touching lines of $f$, when
$f\in\mathcal{F}^{e}(a)$ varies among all economic extensions of $a$. We shall
call $t(a)$ ``touching number'' of ribbon $a$.
\end{definition}

For example, for the ribbon from \hyperref[f17]{Fig.~\ref*{f17}} we have
$t=1$, since $t=2$ is contradictory, as shown at \hyperref[f18]%
{Fig.~\ref*{f18}}.

\begin{remark}
As in the case of $\gamma$, it may be shown that $t(a)$ has the same value
when varying $f$ among all extensions $f\in\mathcal{F}(a)$, not only the
economic ones. Anyway, this won't be of crucial importance for us, so we shall
not prove it here.
\end{remark}

\begin{proposition}
\label{p-1}For any ribbon we have%
\begin{equation}
\gamma\leq n-1-\frac{\sigma}{2}-2t\text{,} \label{13}%
\end{equation}

where $t$ is the touching number of the ribbon.
\end{proposition}

\begin{proof}
To get this improved variant of (\ref{12}), one has only to notice that in the
proof of the last inequality there is a ``gap'' of 2 units (in our favor) and
this is due to the fact that we are splitting along a regular touching line.
Now, if we take $f\in\mathcal{F}^{e}(a)$ with $t=t(a)$ such lines, we get a
``gap'' of $2t$ units, which gives the improved inequality (\ref{13}).
\end{proof}

To demonstrate the consistency of (\ref{13}), take an arbitrary $a\in
\mathcal{A}^{-}$. It may be shown that in this case $t=\frac{n}{2}-1$, so
(\ref{13}) gives $\gamma\leq n-1+\frac{n}{2}-n+2=\frac{n}{2}+1$, which is in
fact the right value of $\gamma$, in contrast with the non improved version,
which gave $\frac{3n}{2}-1$ as an upper limit for $\gamma$.

It seems that the ``simple'' inequality $\gamma\leq\frac{n}{2}+1$ for
arbitrary ribbons is the hardest one to prove, as it is not affordable by
induction. It will be proved later (\hyperref[t10]{Theorem~\ref*{t10}}), but
we shall use it in advance. By the way, a simple calculation based on
(\ref{13}) shows that $t\geq\frac{s_{-}}{2}-1$ implies $\gamma\leq\frac{n}%
{2}+1$, but this is equally hard.

\begin{proposition}
\label{p1}a) $\gamma\leq\frac{n}{2}-1$ for $a\in\mathcal{A}^{+}$ \ \ \ b)
$\gamma=\frac{n}{2}+1$ for $a\in\mathcal{A}^{-}$
\end{proposition}

\begin{proof}
a) In this case $\sigma=n$ and by the general inequality (\ref{12}) we have
$\gamma\leq\frac{n}{2}-1$. b) Now $\sigma=-n$ and the first part of (\ref{12})
gives $\gamma\geq\frac{n}{2}+1$, thus $\gamma=\frac{n}{2}+1$. We show later
that any $a\in\mathcal{A}^{-}$ has in fact at least $\frac{n}{2}+1$ local extrema.
\end{proof}

This proposition means that the class $\mathcal{A}^{+}$ is, in some sense,
more interesting than class $\mathcal{A}^{-}$, since for the latter $\gamma$
is computed and does not depend on the $C^{0}$- part of the boundary data.

As a consequence of a) we get also that $\gamma=\frac{n}{2}-1$ for positive
ladders. Indeed, we have $\gamma\geq\delta=\frac{n}{2}-1\,$, which combined
with a) gives $\gamma=\frac{n}{2}-1$.

\textbf{The involution} $a\rightarrow\overline{a}$.

Let us define an useful involution in $\mathcal{A}$.

\begin{definition}
\label{d15}Let $a=(\varphi,\nu)\in\mathcal{A}$, then set $\overline
{a}=(\varphi,\overline{\nu})$, where $\overline{\nu}=-\nu$.
\end{definition}

In such a way $a\rightarrow\overline{a}$ is changing the marking of any node
to the opposite one.

\begin{proposition}
\label{p2}Let $a\in\mathcal{A}$ be a ribbon with signature $\sigma\neq\pm2$.
Then we have
\[
\gamma(a)+\gamma(\overline{a})\geq2+\frac{|\sigma|}{2}\text{.}%
\]

\end{proposition}

\begin{proof}
Note that $\sigma(\overline{a})=-\sigma(a)$. We may assume that $\sigma<0$,
then by \hyperref[t5]{Theorem~\ref*{t5}} $\gamma(a)\geq1-\frac{\sigma}%
{2}=1+\frac{|\sigma|}{2}$. But $\gamma(\overline{a})\geq1$, since $\sigma
\neq\pm2$, whence $\gamma(a)+\gamma(\overline{a})\geq2+\frac{|\sigma|}{2}$.
\end{proof}

This simple inequality will be useful for the estimation of the number of
critical points on the 2-sphere in \hyperref[s15]{Section~\ref*{s15}}.

Let us finally ask some curious question of geometric nature.

Let us call a ribbon $a=(\varphi,\nu)\in\mathcal{A}$ \textit{harmonic} if the
(unique) solution $f$ of the Dirichlet problem $\Delta f=0,~$ $f|_{\mathbb{S}%
^{1}}=\varphi$ is inducing ribbon $a$ itself on $\mathbb{S}^{1}$. Clearly, not
any ribbon is harmonic, as simple examples show. So,

\begin{center}
\textit{Describe the class of harmonic ribbons.}
\end{center}

As the answer may sensibly depend on the $C^{0}$-part of the ribbon, here is a
discrete version of the problem:

\begin{center}
\textit{Describe the class of discrete ribbons which are similar to a harmonic
one.}
\end{center}

\section{\label{s9} The ribbon invariants $\gamma_{0}$, $\gamma
_{\mathtt{\operatorname{ext}}}$, $\gamma_{\mathtt{\operatorname{sad}}}$}

Let $a\in\mathcal{A}$ be a ribbon. One may be interested in extensions
$f\in\mathcal{F}(a)$, which are Morse functions and to look for the minimal
possible number of critical points in this class.

\begin{definition}
\label{d16}Let $a=(\varphi,\nu)\in\mathcal{A}$ be a ribbon. Then we shall
denote by $\gamma_{0}(a)$ the minimal number of critical points of $f$, where
$f:\mathbb{B}^{2}\rightarrow\mathbb{R}$ varies among all Morse\ functions
$f\in\mathcal{F}(a)$.
\end{definition}

The class of Morse\ functions $f\in\mathcal{F}(a)$ will be be denoted by
$\mathcal{F}_{0}(a)$. Of course, here $f$'s are Morse functions in the context
of manifolds with boundary. These have only non degenerate saddles and local
extrema as critical points.

It is clear that $\gamma_{0}$ is finite and $\gamma_{0}\geq\gamma$. It turns
out that $\gamma_{0}$ is an invariant, which is interesting for itself, as the
computation of $\gamma$ does\textit{ not }imply automatic computation of
$\gamma_{0}$, and conversely. Anyway, the algorithm for $\gamma$ may be
adapted to $\gamma_{0}$ with minor changes. Of course, it is possible that
$\gamma_{0}=\gamma$, but in general $\gamma_{0}>\gamma$.

\begin{fact}
$\gamma=0$ is equivalent to $\gamma_{0}=0$.
\end{fact}

\begin{fact}
For $a\in\mathcal{A}^{+}$, $\gamma_{0}=\frac{n}{2}-1$.
\end{fact}

\begin{fact}
For $a\in\mathcal{A}^{-}$, $\gamma_{0}=\gamma=\frac{n}{2}+1$.
\end{fact}

\begin{fact}
For positive ladders $\gamma_{0}=\gamma=\frac{n}{2}-1$.
\end{fact}

\begin{fact}
For alternations $\gamma_{0}=\frac{n}{2}-1\geq1=\gamma$.
\end{fact}

These facts are almost clear, the first one is obvious, the second follows
from the fact that the minimizing Morse extension $f\in\mathcal{F}_{0}(a)$ has
only non-degenerate saddles and no extrema, the third one follows from
\hyperref[p1]{Proposition~\ref*{p1}} and the observation that the extension
realizing $\gamma$ is Morse and the other two follow easily from our further investigations.

\begin{proposition}
\label{p3}The next properties of $\gamma_{0}$ hold%
\[
\text{a) }\gamma_{0}\geq\left\vert 1-\frac{\sigma}{2}\right\vert \text{,
\ }\ \text{\ b) }\gamma_{0}\equiv\left(  1-\frac{\sigma}{2}\right)
\operatorname{mod}2
\]

\[
\text{c) }\gamma_{0}\leq n-1-\frac{\sigma}{2}-2t\text{, \ where }t\text{ is
the touching number.}%
\]

\end{proposition}

\begin{proof}
Let $a\in\mathcal{A}$ be a ribbon and $f\in\mathcal{F}_{0}(a)$ be a Morse
extension realizing $\gamma_{0}(a)$. Let $f$ have $m$ local extrema and $k$
saddles. Then $\gamma_{0}=m+k$, the index $i(a)=m-k=1-\frac{\sigma}{2}$, thus%
\[
\gamma_{0}=m+k\geq|m-k|=|1-\frac{\sigma}{2}|\text{,}%
\]
which proves a) and%
\[
\gamma_{0}=m+k\equiv(m-k)\operatorname{mod}2=\left(  1-\frac{\sigma}%
{2}\right)  \operatorname{mod}2\text{,}%
\]
which proves b). To show c), it is enough to prove it for irreducible ribbons
and then to follow the proof of \hyperref[p-1]{Proposition~\ref*{p-1}}. For
the 3 small irreducible ribbons it is obvious. Take now an alternation
$a\in\mathcal{A}_{n}^{+}$, then as following from Fact $\ $, $\gamma
_{0}(a)=\frac{n}{2}-1$, and since $\sigma=n$, $t=0$, c) is fulfilled.
\end{proof}

Note that neither a), nor b) is true for $\gamma$. Indeed, a) fails for any
alternation with $\geq6$ nodes and b) fails for an alternation with 6 nodes
(or with $4k+2$ nodes). We shall show later that $\gamma_{0}\leq\frac{n}{2}+1$.

Note also that the gap between $\gamma_{0}$ and $\left\vert 1-\frac{\sigma}%
{2}\right\vert $ may be done large ($\sim\frac{n}{2}$). This is equally true
for $\gamma$. In \hyperref[s18]{Section~\ref*{s18}} we construct a ribbon with
$\gamma=\gamma_{0}\sim\frac{n}{2}$ and $\sigma=2$.

It is clear, that $\gamma_{0}$ is very similar to $\gamma$ and many of the
results for $\gamma$ may be automatically transferred to $\gamma_{0}$.

\begin{proposition}
\label{p4}For any ribbon $a\in\mathcal{A}$ there is an economic extension
$f\in\mathcal{F}_{0}(a)$ realizing $\gamma_{0}(a)$. The invariant $\gamma_{0}$
is subadditive under splittings%
\[
\gamma_{0}(a_{1}\#a_{2})\leq\gamma_{0}(a_{1})+\gamma_{0}(a_{2})\text{,
\ \ }\gamma_{0}(a_{1}\ast a_{2}\ast a_{3})\leq\gamma_{0}(a_{1})+\gamma
_{0}(a_{2})+\gamma_{0}(a_{3})\text{.}%
\]

If the splitting is along a level line of an extension $f\in\mathcal{F}%
_{0}(a)$ realizing $\gamma_{0}$, then equality occurs in the above formulas.
\end{proposition}

The proof follows step by step the proof of the corresponding proposition
\newline about~$\gamma$.

\begin{proposition}
\label{p5}For any ribbon $a\in\mathcal{A}$, \
\[
\gamma_{0}(a)+\gamma_{0}(\overline{a})\geq|\sigma|.
\]

\end{proposition}

\begin{proof}
By \hyperref[p3]{Proposition~\ref*{p3}},\ a) $\gamma_{0}(a)+\gamma
_{0}(\overline{a})\geq|1-\frac{\sigma}{2}|+|1+\frac{\sigma}{2}|=2\frac
{\left\vert \sigma\right\vert }{2}=\left\vert \sigma\right\vert $. The
equality holds, since $\frac{\sigma}{2}$ is integer.
\end{proof}

Note that in case $\sigma=0$, still $\gamma_{0}(a)+\gamma_{0}(\overline
{a})\geq2$, as $\gamma_{0}(a)+\gamma_{0}(\overline{a})\geq\gamma
(a)+\gamma(\overline{a})\geq2$ (\hyperref[p2]{Proposition~\ref*{p2}}).

Now we introduce two other invariants, estimating from below the number of
local extrema and saddle points of a given ribbon's extension.

\begin{definition}
\label{d17}Let $a=(\varphi,\nu)\in\mathcal{A}$ be a ribbon. Then
$\gamma_{\mathtt{\operatorname{ext}}}(a)$ will denote the minimal number of
local extrema of $f$, when $f$ varies among all\ extensions $f\in
\mathcal{F}(a)$. Similarly, by $\gamma_{\mathtt{\operatorname{sad}}}(a)$ we
shall denote the minimal number of saddle points of function $f$, when $f$
varies among all economic\ extensions $f\in\mathcal{F}^{e}(a)$.
\end{definition}

\begin{remark}
For the definition of $\gamma_{\mathtt{\operatorname{sad}}}$ one has to
consider economic extensions, as we may always \textquotedblleft
destroy\textquotedblright\ saddle points at the price of the appearance of
infinite number of critical points, while $\gamma_{\mathtt{\operatorname{ext}%
}}$ is stable under such perturbations.
\end{remark}

Evidently, $\gamma_{\mathtt{\operatorname{ext}}}=0$ in $\mathcal{A}^{+}$ and
$\gamma_{\mathtt{\operatorname{sad}}}=0$ in $\mathcal{A}^{-}$.

Note also that similarly to $\gamma$ and $\gamma_{0}$, the invariants
$\gamma_{\mathtt{\operatorname{ext}}}$ and $\gamma_{\mathtt{\operatorname{sad}%
}}$ satisfy the inequalities in \hyperref[p4]{Proposition~\ref*{p4}}, i.e.
these are subadditive under splittings.

It turns out that the general estimate from below for $\gamma$ is in fact an
estimate for the number of local extrema. Also, there is an estimate from
below of $\gamma_{\mathtt{\operatorname{sad}}}$ involving the cluster number
$\delta$.

Note that even in case $\sigma=2$, it may happen $\gamma
_{\mathtt{\operatorname{ext}}}>0$. For example for $a=(1^{+},6^{+},2^{-}%
,4^{+},3^{+},5^{-})$ it holds that $\gamma_{\mathtt{\operatorname{ext}}%
}=\gamma_{\mathtt{\operatorname{sad}}}=1$, $\gamma=\gamma_{0}=2$, although
$\sigma=2$.

\begin{proposition}
\label{p6}For the touching number we have%
\[
\text{a) }t=\frac{1}{2}(n-\sigma)-\gamma_{\mathtt{\operatorname{ext}}}\text{
\ b) }t\leq\frac{n}{2}-1.
\]

\end{proposition}

\begin{proof}
To prove a), it suffices to note that $t+\gamma_{\mathtt{\operatorname{ext}}%
}=s_{-}=\frac{1}{2}(n-\sigma)$, since maximizing $t$ means minimizing
$\gamma_{\mathtt{\operatorname{ext}}}$. Now \hyperref[p7]%
{Proposition~\ref*{p7}} and a)\ immediately imply $t\leq\frac{n}{2}-1$.
\end{proof}

\begin{proposition}
\label{p7}For any ribbon $a\in\mathcal{A}$,%
\begin{equation}
\text{1) }\gamma_{\mathtt{\operatorname{ext}}}(a)\geq1-\frac{\sigma}{2}\text{
\ \ \ \ 2) }\gamma_{\mathtt{\operatorname{sad}}}(a)\geq\delta_{0}(a)+\frac
{1}{2}(\sigma-n)\text{.} \label{14}%
\end{equation}

\end{proposition}

\begin{proof}
1) For irreducible ribbons it is obvious. Now we have only to follow literally
the proof of the first part of \hyperref[t5]{Theorem~\ref*{t5}}. 2) Recall
that $\delta_{0}(a)$ denotes the \textit{reduced cluster number} (see
\hyperref[s7]{Section~\ref*{s7}}). As we noted above, $s_{-}=\frac{1}%
{2}(n-\sigma)$, so the wanted inequality is equivalent to $\gamma
_{\mathtt{\operatorname{sad}}}\geq\delta_{0}-s_{-}$. Now, this is almost
obvious: indeed, for a given extension $f$, to any principal cluster should be
\textit{attached} at least one critical point (of saddle type), but the number
of such clusters is clearly $\geq\delta_{0}-s_{-}$.
\end{proof}

\begin{remark}
The maximal possible number of nondegenerate saddles of extensions
$f\in\mathcal{F}^{e}(a)$ equals $\gamma_{\mathtt{\operatorname{ext}}}%
(a)+\frac{\sigma}{2}-1$. Indeed, for any $f\in\mathcal{F}^{e}(a)$ we may
perform a morsification of saddles and get some $f^{\prime}\in\mathcal{F}%
^{e}(a)$ with $m$ extrema and $k$ saddles. Then by Hopf Theorem,
$m-k=i=1-\frac{\sigma}{2}$. The minimal possible value of $m$ is
$\gamma_{\mathtt{\operatorname{ext}}}(a)$ and, clearly, minimizing $m$ means
maximizing $k$, hence the maximal value of $k$ equals $\gamma
_{\mathtt{\operatorname{ext}}}(a)+\frac{\sigma}{2}-1$. Note that according to
\hyperref[p7]{Proposition~\ref*{p7}} 1), this number is nonnegative. Note also
that we get the same number if $f$ varies in the class $\mathcal{F}(a)$ of all
extensions, not only the economic ones.
\end{remark}

In order to measure the gap between $\gamma_{0}$ and $\gamma$ it is convenient
to introduce the \textit{compression }$\varkappa$ of a ribbon, defined as
follows. Let $a$ be a ribbon and $f\in\mathcal{F}^{e}(a)$. For any saddle $P$
of $f$ with $k$ separatrices, define $\varkappa(P)=\frac{k}{2}-2$. Note that
$\varkappa(P)$ is the number of new born\ critical points through a
``morsification'' of saddle $P$. Now let $\varkappa(f)$ be the sum
$\sum\varkappa(P_{i})$, where the sum runs over all saddles of $f$. Set
finally%
\[
\varkappa(a)=\max\{\varkappa(f)|~f\in\mathcal{F}^{e}(a)\}\text{.}%
\]

It is clear that $\varkappa(a)$ is a finite number, since the number of
critical points of economic extensions is bounded from above by $\frac{3n}%
{2}-1$.\ We call $\varkappa(a)$ \textit{compression }of ribbon $a$. For
example, the compression of a ladder is 0, while the compression of an
alternation is $\frac{n}{2}-2$.

\begin{proposition}
\label{p8}1) $\gamma\geq\gamma_{\mathtt{\operatorname{ext}}}+\gamma
_{\mathtt{\operatorname{sad}}}$ \ \ \ 2) $\gamma_{0}-\gamma\leq\varkappa$.
\end{proposition}

\begin{proof}
1) This is almost obvious, just consider some $f\in\mathcal{F}^{e}(a)$
realizing $\gamma(a)$, then the number of extrema and saddles of $f$ is
greater or equal to $\gamma_{\mathtt{\operatorname{ext}}}$ and $\gamma
_{\mathtt{\operatorname{sad}}}$, respectively. 2) Let $f\in\mathcal{F}^{e}(a)$
be realizing $\gamma(a)$ and has $m$ extrema and $k$ saddles, so $\gamma=m+k$.
Perform a ``morsification'' of $f$, then it is easy to see that the
``increment'' of critical points is exactly $\varkappa(a)$. But then one has
$m+k+\varkappa\geq\gamma_{0}$, whereby $\varkappa\geq\gamma_{0}-\gamma$.
\end{proof}

Note that it may happen for some ribbons that $\gamma>\gamma
_{\mathtt{\operatorname{ext}}}+\gamma_{\mathtt{\operatorname{sad}}}$. Here is
a corresponding example.

\begin{example}
\label{e1}Let $a=(\varphi,\nu)\in\mathcal{A}_{8}$ be a general alternation%

\[
a=(p_{1}^{-},p_{2}^{+},p_{3}^{+},p_{4}^{+},p_{5}^{-},p_{6}^{+},p_{7}^{+}%
,p_{8}^{+})
\]
such that $\varphi(p_{5})$ is a minimum, which is the closest one to
$\varphi(p_{1})=\min\varphi$ (see \hyperref[f26]{Fig.~\ref*{f26}}). Then it is
not difficult to see that $\gamma_{\mathtt{\operatorname{ext}}}=\gamma
_{\mathtt{\operatorname{sad}}}=1$, while $\gamma=3$. This is evident from
\hyperref[f27]{Fig.~\ref*{f27}} where the only 2 minimal economic extensions
are depicted. The second one shows that $\gamma_{0}=\gamma=3$.
\end{example}

\begin{center}
\begin{figure}[ptb]
\includegraphics[width=70mm]{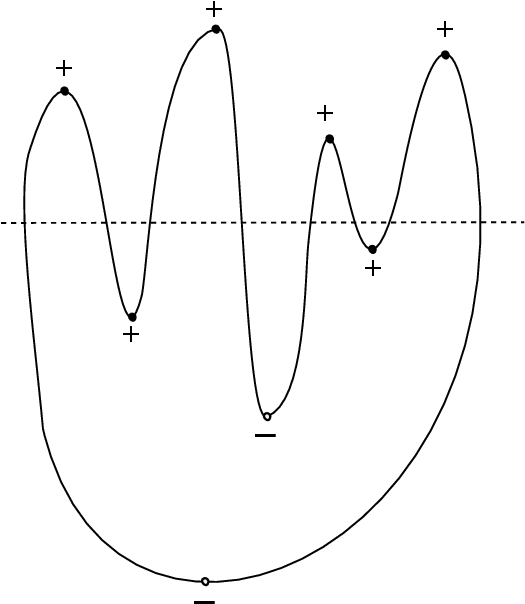}\caption{A ribbon for which
$\gamma>\gamma_{\mathtt{\operatorname{ext}}}+\gamma
_{\mathtt{\operatorname{sad}}}$.}%
\label{f26}%
\end{figure}

\begin{figure}[ptb]
\includegraphics[width=100mm]{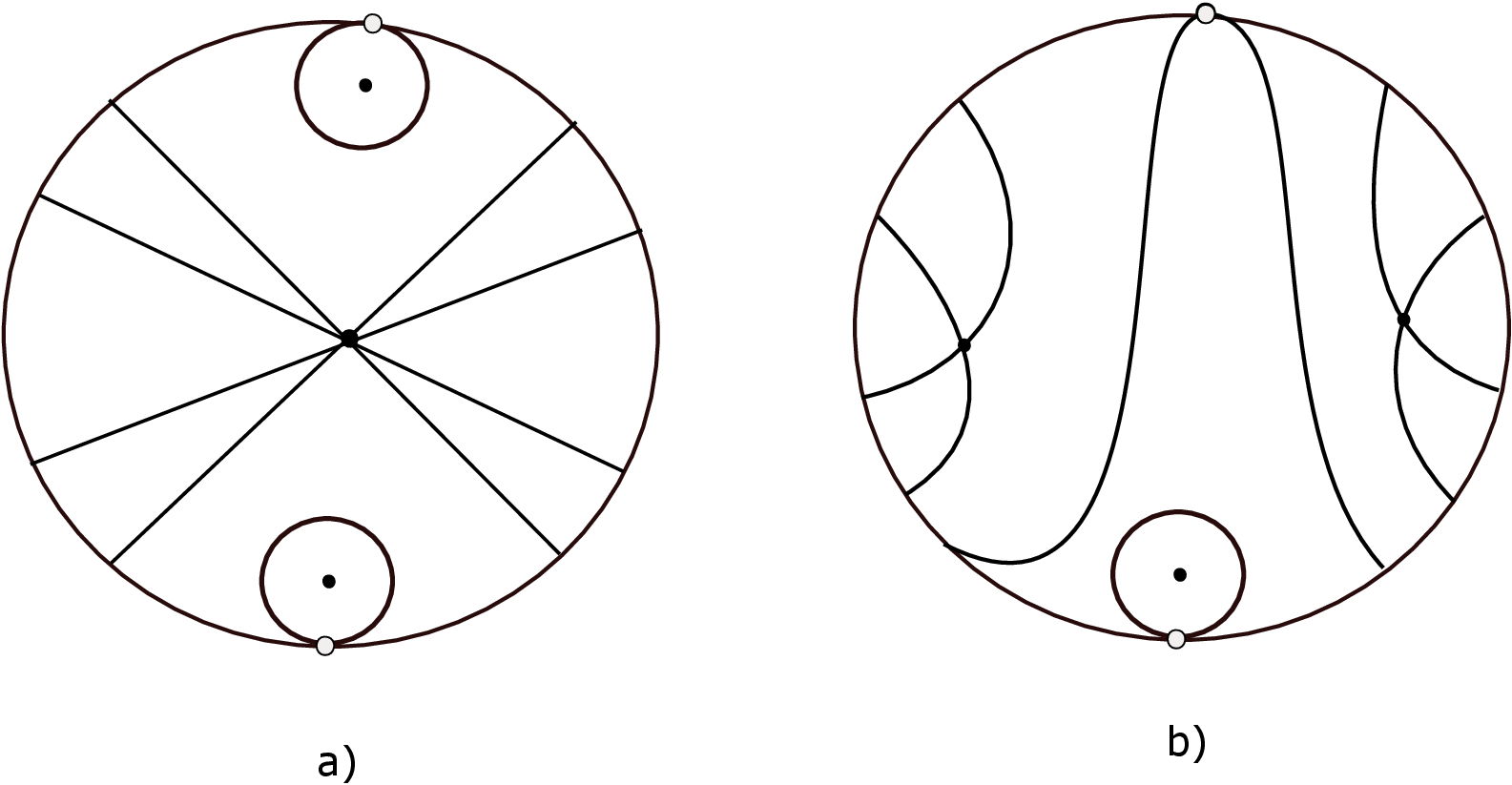}\caption{a) $\gamma
_{\mathtt{\operatorname{sad}}}=1$, b) $\gamma_{\mathtt{\operatorname{ext}}}%
=1$, but $\gamma=3$.}%
\label{f27}%
\end{figure}
\end{center}

This example shows that the knowledge of a solution for $\gamma$ does not
automatically solve the problem with either $\gamma_{0}$, $\gamma
_{\mathtt{\operatorname{ext}}}$ or $\gamma_{\mathtt{\operatorname{sad}}}$. In
some sense, we have four, more or less, independent problems and the
calculation of some of the ribbon invariants does not imply anything about the
other three, except the obvious inequalities.

\begin{remark}
We cannot estimate from below the number of either local maxima or local
minima of an extension separately, but only their sum: $\gamma
_{\mathtt{\operatorname{ext}}}\geq1-\frac{\sigma}{2}$ . In fact, for any
ribbon with positive maximal node we may find an extension without local
maxima at all. Surely, the same is equally true for local minima. This is
easily seen for ladders and then for an arbitrary ribbon $a$ it follows from
the fact that we may obtain it from a ladder $b$ by performing only
\textquotedblleft meetings\textquotedblright. Then any solution for $b$
remains valid for $a$ as well, since the old \textquotedblleft contact
zones\textquotedblright\ of the ribbon remain untouched along such a move.
\end{remark}

Of course, one may formally define ``invariants'' $\gamma_{\max}$,
$\gamma_{\min}$ estimating from below the number of maxima (minima) of a
ribbon's extensions, but then these become almost trivial in view of the above
remark and depend only on whether the maximal (negative) node of the ribbon is
positive or negative (and take value 0 or 1). On the other hand, albeit
elementary, $\gamma_{\max}$ and $\gamma_{\min}$ are examples of
\textit{algebraic} ribbon invariants (\hyperref[s22]{Section~\ref*{s22}}) with
the corresponding normalization rules on elementary ribbons. For example, for
$\gamma_{\min}$ the normalization should be $\gamma_{\min}(\alpha_{0}%
)=\gamma_{\min}(\alpha_{1})=0$, $\gamma_{\min}(\alpha_{2})=1,\gamma_{\min
}(\beta_{n})=0$.

Now we shall give a general estimate of $\gamma$, which is an improvement of
the basic estimate (\ref{12}) and involves $\gamma_{\mathtt{\operatorname{ext}%
}}$ and $\gamma_{\mathtt{\operatorname{sad}}}$.

\begin{theorem}
\label{t6}It holds that%
\[
1-\frac{\sigma}{2}+\gamma_{\mathtt{\operatorname{sad}}}\leq\gamma\leq
-1+\frac{\sigma}{2}+2\gamma_{\mathtt{\operatorname{ext}}}\text{.}%
\]

\end{theorem}

\begin{proof}
The first part is easy, $\gamma\geq\gamma_{\mathtt{\operatorname{ext}}}%
+\gamma_{\mathtt{\operatorname{sad}}}\geq1-\frac{\sigma}{2}+\gamma
_{\mathtt{\operatorname{sad}}}$ (\hyperref[p7]{Proposition~\ref*{p7}}). For
the second part, combine $\gamma\leq n-1-\frac{\sigma}{2}-2t$ with $t=\frac
{1}{2}(n-\sigma)-\gamma_{\mathtt{\operatorname{ext}}}$ (\hyperref[p-1]%
{Propositions~\ref*{p-1}}, \ref{p6}). It is easy to see that this result is
stronger than the basic inequality (\hyperref[t5]{Theorem~\ref*{t5}}).
\end{proof}

The invariants $\gamma$, $\gamma_{0}$, $\gamma_{\mathtt{\operatorname{ext}}}$,
$\gamma_{\mathtt{\operatorname{sad}}}$ are almost independent from each other,
except for some general inequalities, but there is a common algorithm for
their calculation, we have only to assign different weights to the irreducible
ribbons and then to proceed by induction on the lexicographic order. This will
be explained in \hyperref[s11]{Section~\ref*{s11}}.

\textbf{The problem with the realization of a pair }$(\sigma,\gamma
)$\textbf{.}\label{realiz} Consider the general estimate (\ref{12}) for
$\gamma$:%
\[
1-\frac{\sigma}{2}\leq\gamma\leq n-1-\frac{\sigma}{2}.
\]

It turns out that for a given $n$, not every pair $(\sigma,\gamma)$ satisfying
the above inequality may be realized by some ribbon. Two simple examples are
given below:

1. $(\sigma\neq2$, $\gamma=0)$. Such pairs are impossible by the basic
property of the ribbon invariant $\gamma=0\Longrightarrow i=1-\frac{\sigma}%
{2}=0\Longrightarrow\sigma=2$.

2. $(\sigma=2$, $\gamma=1)$. Suppose that there is such a minimal economic
extension with unique critical point $P$. But then $i=0$ and since the sum of
the indices of all critical points equals $i$, it follows that the index of
$P$ is zero. However, economic extensions don't have such elements, a
contradiction. Note that we prove in such a way that there are no economic
extensions with $\sigma=2$ and unique critical point. So, this situation is,
in some sense, \textit{strongly} non-realizable.

Now we shall find an infinite series of non-realizable pairs $(\sigma,\gamma)$
satisfying the general estimate. It is based on case 2.

\begin{proposition}
\label{p9}There are no ribbons with $\sigma<0$ and $\gamma=2-\frac{\sigma}{2}$.

So, the pairs $(\sigma=-2,\gamma=3)$, $(\sigma=-4,\gamma=4)$, $(\sigma
=-6,\gamma=5)$,\dots are all non-realizable, although they are satisfying the
general inequality (\ref{12}).
\end{proposition}

\begin{proof}
Suppose that there is a ribbon $a$ with $\sigma<0$ and $\gamma=2-\frac{\sigma
}{2}$. Take $f\in\mathcal{F}^{e}(a)$ with $2-\frac{\sigma}{2}$ critical
points.\ By (\ref{14}) we have $\gamma_{\mathtt{\operatorname{ext}}}%
(a)\geq1-\frac{\sigma}{2}$, so $f$ has (at least) $1-\frac{\sigma}{2}$ extrema
$P_{i}$. Consider the corresponding circles touching the boundary at negative
nodes $p_{i}$. Now, make nodes $p_{i}$ \textit{positive}, obtaining in such a
way some ribbon $b$ with extension $g\in\mathcal{F}^{e}(b)$ with 1 critical
point. It is easy to see that the signature of $b$ is $\sigma(b)=\sigma
+2(1-\frac{\sigma}{2})=2$. But this is a contradiction with case 2 considered
above. Note that we proved in fact that the pairs $(\sigma<0,\gamma
=2-\frac{\sigma}{2})$ are \textit{strongly} non-realizable, in other words,
non realizable by an economic extension with $\Gamma=2-\frac{\sigma}{2}$
critical points.
\end{proof}

It is clear that some general question may be asked about the possible pairs
$(\sigma,\gamma)$:

\textbf{Question.} For which pairs $(\sigma,\gamma)$ do a ribbon exist with
signature $\sigma$ and ribbon invariant $\gamma$?

Another more general problem is to study the topology (the homology) of
\textit{rigid }ribbons from class $\mathcal{A}(\sigma,\gamma)$ containing
ribbons with signature $\sigma$ and ribbon invariant $\gamma$. This will be
discussed later in Part II. For example, we show that the class of ribbons
with a critical points free extension, or equivalently, the ribbons with
$\gamma=0$, is a connected subspace of $\mathcal{A}$ (in the appropriate
topology). Equivalently, we may say that the space $\mathcal{A}(2,0)$ is
connected.\ Let us ask another natural

\textbf{Question.} Is it true that $\mathcal{A}(\sigma,\gamma)$ has finite
number of components?

Note that different questions about realization of the ribbon invariant may be
asked, for example:

\textbf{Question.} Let $\sigma$ be fixed and $t$ be a zig-zag permutation with
$n$ nodes. Then for which numbers $k\in\left[  1-\frac{\sigma}{2}%
,n-1-\frac{\sigma}{2}\right]  $ does a marking $\nu$ of the nodes of $t$ exist
with signature $\sigma(\nu)=\sigma$, such that for the corresponding ribbon
$a=(t,\nu)$ we have $\gamma(a)=k$?

To see that this question makes sense and the answer depends essentially on
$t$, consider the following two examples:

1) Let $\sigma=2$ and $t$ be the alternation with 6 nodes from \hyperref[f17]%
{Fig.~\ref*{f17}}. Then the \textquotedblleft suspected\textquotedblright%
\ values are $k\in\{0,1,2,3,4\}$, but only $\{0,2,3\}$ are realizable by a ribbon.

2) Let again $\sigma=2$ and t be a ladder with 6 nodes. Then, as above,
$k\in\{0,1,2,3,4\}$, but in this case only $\{0,2,4\}$ are the realizable
values of $\gamma$.

Note that in both cases $k=1$ is a generally forbidden value, as we showed
earlier that $(\sigma=2$, $\gamma=1)$ is an impossible case.

Of course, similar \textquotedblleft realization\textquotedblright\ questions
may be asked about the other ribbon invariants.

\section{\label{s9.5}Zeroes of gradient fields on $\mathbb{B}^{2}$}

The main problem of the present article may be reformulated in terms of some
class of vector fields on $\mathbb{S}^{1}$ and their gradient extensions on
the 2-dimensional disk $\mathbb{B}^{2}$. This class contains the fields which
are restriction of gradient ones on the disk. Note that similar problems, in
terms of gradient homotopy have been previously investigated by A.
Parusi\'{n}ski \cite{b10}, but therein results seem not to intersect with
ours. In fact, the main result in \ is that any two gradient fields in
$\mathbb{B}^{n}$ non vanishing on $\mathbb{S}^{n-1}$\ are homotopic by a
gradient homotopy non vanishing on $\mathbb{S}^{n-1}\times\lbrack0,1]$, but
this doesn't tell us anything about the zero set of the gradient fields under
consideration. We shall formulate in the present section sufficient conditions
for existence of zeroes of a gradient field $v$ on the disk $\mathbb{B}^{2}$
that explore only information from the restriction $v|_{\mathbb{S}^{1}}$ and
do not correlate in general with the degree $\deg(v|_{\mathbb{S}^{1}})$. It
turns out that even in case $\deg(v|_{\mathbb{S}^{1}})=0$, it may happen that
$v$ has some minimal (positive) number of zeroes inside $\mathbb{B}^{2}$ that
cannot be \textquotedblleft killed\textquotedblright\ by a gradient homotopy
not disturbing $v|_{\mathbb{S}^{1}}$. Let us note that in the introductory
lines in \cite{b10}\ the same problem is formulated, but in our opinion, the
author is solving a different one.

The method consists in constructing some ribbon $a_{v}$ associated with vector
field $v$ and then computing the ribbon invariant $\gamma(a_{v})$ which
guarantees at least $\gamma(a_{v})$ geometrically different zeroes of field
$v$ inside $\mathbb{B}^{2}$. The $C^{0}$-part of ribbon $a_{v}$ depends only
on the tangential component $v_{\tau}=\left\langle v,\tau\right\rangle $ of
$v|_{\mathbb{S}^{1}}$, where $\tau$ is the unit tangent vector field on
$\mathbb{S}^{1}$, while the $C^{1}$-part depends on the direction of vector
$v$ at the zeroes of $v_{\tau}$ on $\mathbb{S}^{1}$.

First we shall describe a simple procedure which associates a permutation to
any finite string of different real numbers.

Let $t=\left(  \alpha_{1},\dots,\alpha_{n}\right)  $ be a string of different
reals $\alpha_{i}\in\mathbb{R}$. Define%
\[
k_{i}=\#\left\{  s|~\alpha_{s}\leq\alpha_{i}\right\}
\]

and then set $t^{\ast}=(k_{1},\dots,k_{n})$. It is easy to see that $t^{\ast}$
is a permutation of $\{1,2,\dots,n\}$ which preserves the order in $t$. For
example, if $t=(7.1,-0.5,6.9)$ then $t^{\ast}=(3,1,2)$.

Consider now some generic gradient field $v:\mathbb{B}^{2}\rightarrow
\mathbb{R}^{2}$ which is nonzero at $\mathbb{S}^{1}$. Let $p_{1},\dots,p_{n}$
be the zeroes of the scalar product $\left\langle v|_{\mathbb{S}^{1}}%
,\tau\right\rangle $, where $\tau$ is the unit tangent field on $\mathbb{S}%
^{1}$, we shall refer to them as \textquotedblleft nodes\textquotedblright. We
suppose that the nodes are cyclically ordered. Note that $n$ is an even
number. Now we shall define some marking $\nu$ of the nodes following a simple rule.

Let $p_{i}$ be a node. Set first $\mathrm{sign}(p_{i})=+1$, if $\left\langle
v|_{\mathbb{S}^{1}},\tau\right\rangle $ is changing its sign from $+$ to $-$
at $p_{i}$, and $\mathrm{sign}(p_{i})=-1$ otherwise (here the natural
parameter is running counter clockwise). Then define%
\begin{equation}
\nu(p_{i})=\mathrm{sign}(p_{i}).\mathrm{sgn}\left\langle v|_{\mathbb{S}^{1}%
},n\right\rangle |_{p_{i}}\text{,} \label{sign}%
\end{equation}

where $n$ is the normal vector field on $\mathbb{S}^{1}$. So we get some
marking of the nodes which divides them into \textquotedblleft
positive\textquotedblright\ and \textquotedblleft negative\textquotedblright%
\ ones. The 4 general cases for the behaviour of field $v|_{\mathbb{S}^{1}}$
at a node $p_{i}$ are depicted at \hyperref[ffield]{Fig.~\ref*{ffield}}. This defines the $C^{1}$-part of ribbon
$a_{v}$. Now we define its $C^{0}$-part \ as follows.

\begin{center}
\begin{figure}[ptb]
\includegraphics[width=75mm]{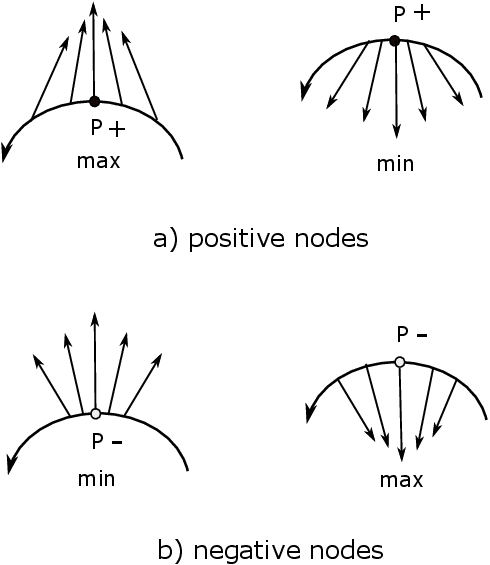}\caption{Behaviour of field $v$ at the nodes.}%
\label{ffield}%
\end{figure}
\end{center}

Let $L_{i}=\left[  p_{1},p_{i}\right]  $ be the arc of $\mathbb{S}^{1}$
between nodes $p_{1}$ and $p_{i}$ , $i=2,...,n+1$ (assuming $p_{n+1}\equiv
p_{1}$). Consider the quantities%
\[
\alpha_{i}=\int_{L_{i}}\left\langle v|_{\mathbb{S}^{1}},\tau\right\rangle ds
\]

which are all different as $v$ is assumed generic. Note that $\alpha
_{n+1}=\int_{\mathbb{S}^{1}}\left\langle v|_{\mathbb{S}^{1}},\tau\right\rangle
ds=0$, since $v$ is a gradient field. Consider now the string $t=\left(
\alpha_{2},\dots,\alpha_{n+1}\right)  $ and its reduction $t^{\ast}$ which is
a permutation of $\{1,2,\dots,n\}$. Observe that $t^{\ast}$ is an alternating
permutation, as $\left\langle v|_{\mathbb{S}^{1}},\tau\right\rangle $ is
changing sign at any node $p_{i}$. We may assume that $t^{\ast}$ is a cyclic
alternating permutation.

\begin{theorem}
Let $a_{v}\in\mathcal{A}_{n}$ be the ribbon $a_{v}=(t^{\ast},\nu)$ where the
marking $\nu$ is defined by (\ref{sign}). Then field $v$ has at least
$\gamma(a_{v})$ geometrically distinct zeroes inside $\mathbb{B}^{2}$. (Here
$\gamma$ is the main ribbon invariant defined in \hyperref[s2]%
{Section~\ref*{s2}}.)
\end{theorem}

\begin{proof}
The proof is a straightforward application of the previous results; it
suffices to take a function $F:\mathbb{B}^{2}\rightarrow\mathbb{R}$ such that
$\nabla F=v$ and to note that the induced ribbon coincides with $a_{v}$ by the
construction procedure of the latter.
\end{proof}

\begin{remark}
Similar results hold true for the other ribbon invariants $\gamma_{0}%
,\gamma_{\operatorname{ext}},\gamma_{\operatorname{sad}}$.
\end{remark}

\section{\label{s10}Geometric restatement of the problem}

There is a geometric way to state the problem of finding $\gamma$ and the
other ribbon invariants. Although it does not give much for the calculation of
$\gamma$, this method enlightens the economic extensions from geometric point
of view and is suitable for finding solution for small $n$ \textquotedblleft
by hand\textquotedblright.

Let $a=(\varphi,\nu)\in\mathcal{A}$ be a ribbon, so $\varphi:\mathbb{S}%
^{1}\rightarrow\mathbb{R}$ is a Morse function and $\nu$ is a marking of the
nodes. Let $Q$ be a polygon inscribed in $\mathbb{S}^{1}$. A vertex of $Q$
will be called \textit{regular}, if it is not a node. We shall say that $Q$ is
an \textit{iso-gon}, if $\varphi$ is constant on the set of its vertices. A
0-\textit{gon} is a disk touching $\mathbb{S}^{1}$ from inside. We distinguish
3 types of such objects, which will play the role of critical points:

1) iso-gons with even number of vertices all being regular

2) iso-triangles with one vertex being a negative node and two regular vertices

3) 0-gons touching $\mathbb{S}^{1}$ at a negative node

We shall call such figures \textit{critical elements} of first, second and
third kind, respectively. Furthermore, we shall consider \textit{only
}critical elements $Q$ satisfying the following condition:

\begin{center}
\textit{each component of }$S^{1}\backslash Q$\textit{ contains an odd number
of nodes.}
\end{center}

Let $q=\left\{  Q_{i}\right\}  $ be a system of critical elements inside
$\mathbb{S}^{1}$.

\begin{definition}
\label{d18}We shall call $q$ packing of $a$, if

a) $Q_{i}$ are disjoint

b) any negative node belongs to either an iso-triangle, or to a 0-gon from $q$

c) each component of $\mathbb{S}^{1}\backslash\cup Q_{i}$ contains no more
than one positive node

d) each component of $\mathbb{S}^{1}\backslash\cup Q_{i}$ not containing a
positive node touches exactly two elements of $q$.
\end{definition}

See \hyperref[f28]{Fig.~\ref*{f28}} and \hyperref[f29]{Fig.~\ref*{f29}} for
examples of packings and no-packings. Of course, the geometric size of the
critical elements does not matter to us, so we are treating $q$ as a purely
combinatorial object. Let us note that there is a different way to say that
$q$ is a packing: $q$ is a \textit{disjoint maximal} (with respect to
inclusion) collection of critical elements inscribed in $\mathbb{S}^{1}$.

\begin{center}
\begin{figure}[ptb]
\includegraphics[width=120mm]{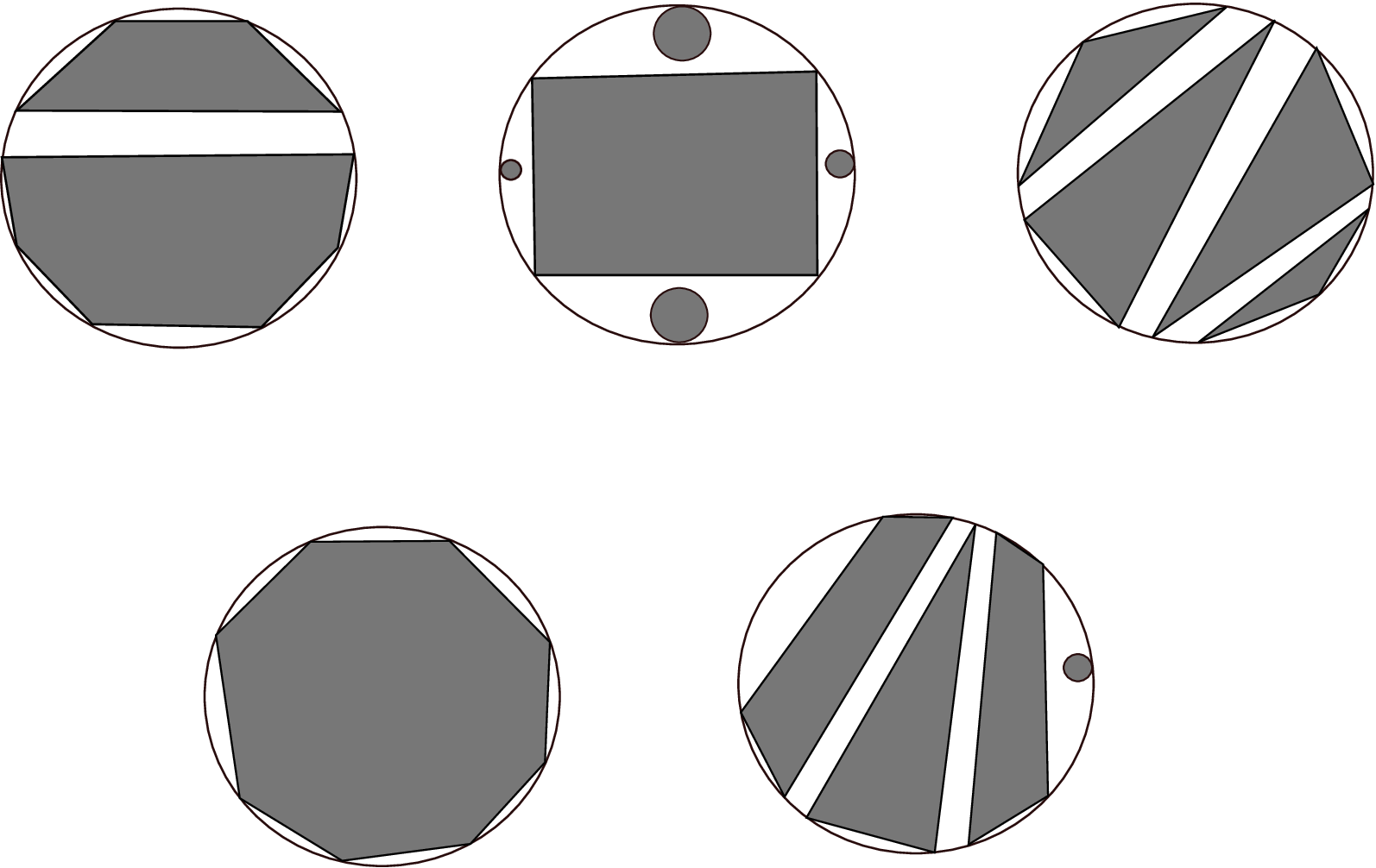}\caption{Examples of packings.}%
\label{f28}%
\end{figure}

\begin{figure}[ptb]
\includegraphics[width=120mm]{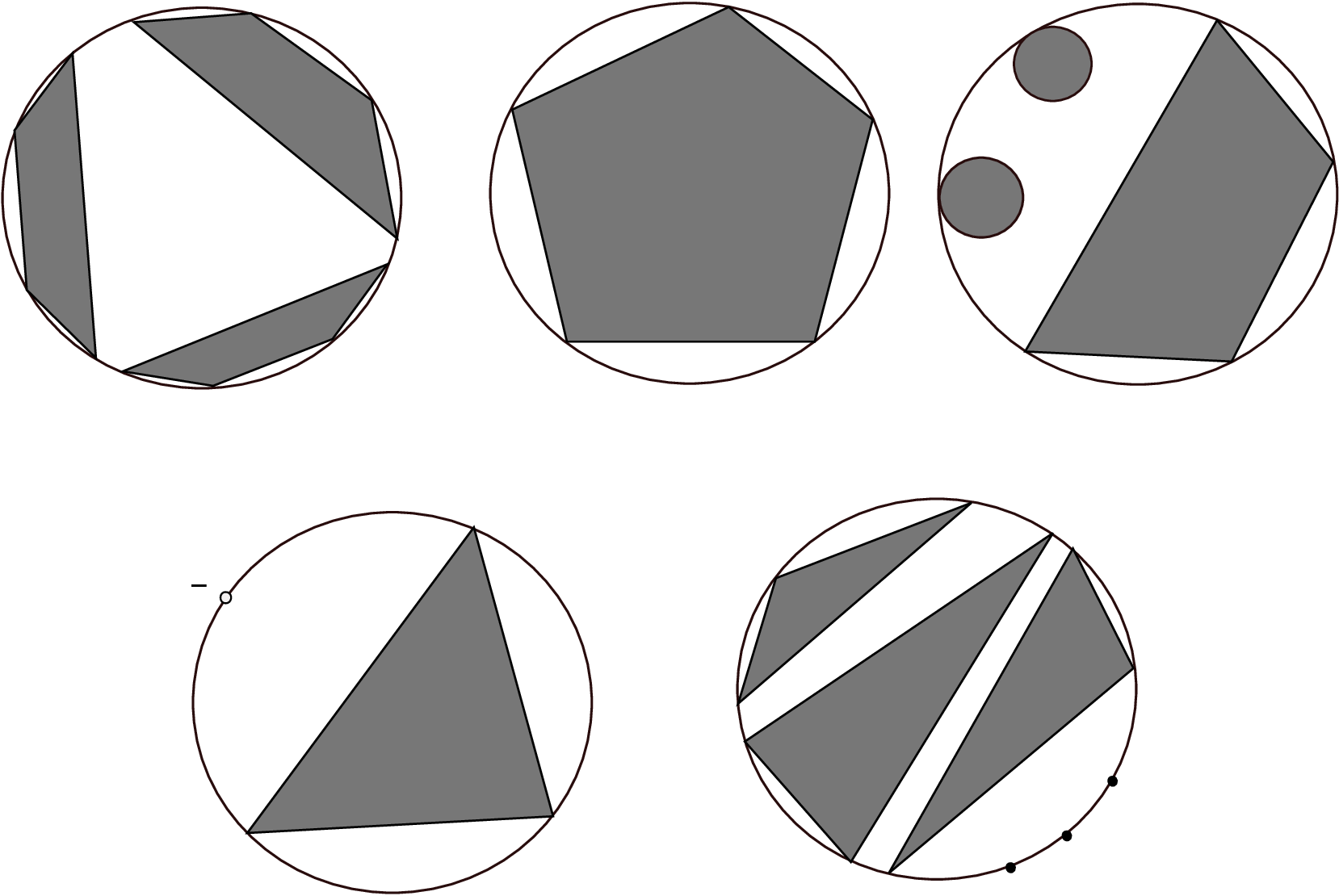}\caption{No-packings.}%
\label{f29}%
\end{figure}
\end{center}

Let us assign to any critical element $Q$ its \textit{weight} as follows:

1) $w(Q)=1$ if $Q$ is of first or third kind

2) $w(Q)=0$ if $Q$ is of second kind

and let%
\begin{equation}
w(q)=\sum w(Q_{i}) \label{15}%
\end{equation}

be the \textit{weight} of the packing $q$. It is more or less geometrically
evident, that packings correspond to economic extensions and finding $\gamma$
means minimizing $w(q)$.

\begin{theorem}
\label{t7}Let $a=(\varphi,\nu)\in\mathcal{A}$ be a ribbon, then there is a
one-to-one correspondence between the economic extensions $f\in\mathcal{F}%
^{e}(a)$ of $a$ (up to combinatorial equivalence), and the packings of $a$.
The ribbon invariant $\gamma(a)$ equals then the minimal value of $w(q)$, when
$q$ varies among all packings of $a$.
\end{theorem}

\begin{proof}
Let first $f\in\mathcal{F}^{e}(a)$ be an economic extension of $a$. One may
define the corresponding packing as follows: If $x$ is a saddle of $f$ whose
separatrices are ending at points $x_{1},\dots,x_{2k}\in\mathbb{S}^{1}$, then
the polygon with vertices $x_{1},\dots,x_{2k}$ is a critical element of first
kind corresponding to $x$. If $x$ is a local extremum of $f$, there is a
closed level line $l$ surrounding $x$ and touching $\mathbb{S}^{1}$ at a
negative node $p$. Then take a little disk touching $\mathbb{S}^{1}$ at $p$
from inside. This is a critical element of third kind corresponding to $x$ in
this case. Finally, let $l$ be a touching line of $f$ which touches the
boundary at a negative node $p$ and has two ends $x_{1},x_{2}\in\mathbb{S}%
^{1}$. Then we associate with $l$ an iso-triangle with vertices $p,x_{1}%
,x_{2}$, which is a critical element of second kind. It is not difficult to
see that we obtain, in such a way, some packing of $a$. The inverse
correspondence is also geometrically evident (see \hyperref[f30]%
{Fig.~\ref*{f30}}). The only \textquotedblleft delicate\textquotedblright%
\ moment is the attachments of noncritical bands between the critical
elements. But here comes to help condition d) from the definition of a
packing, it allows us to do it correctly in a monotonic way. Now, the
statement about $\gamma(a)$ is straightforward, since the critical elements of
first and third kind correspond to saddles and extrema, while an element of
second kind corresponds to a touching line.
\end{proof}

\begin{center}
\begin{figure}[ptb]
\includegraphics[width=120mm]{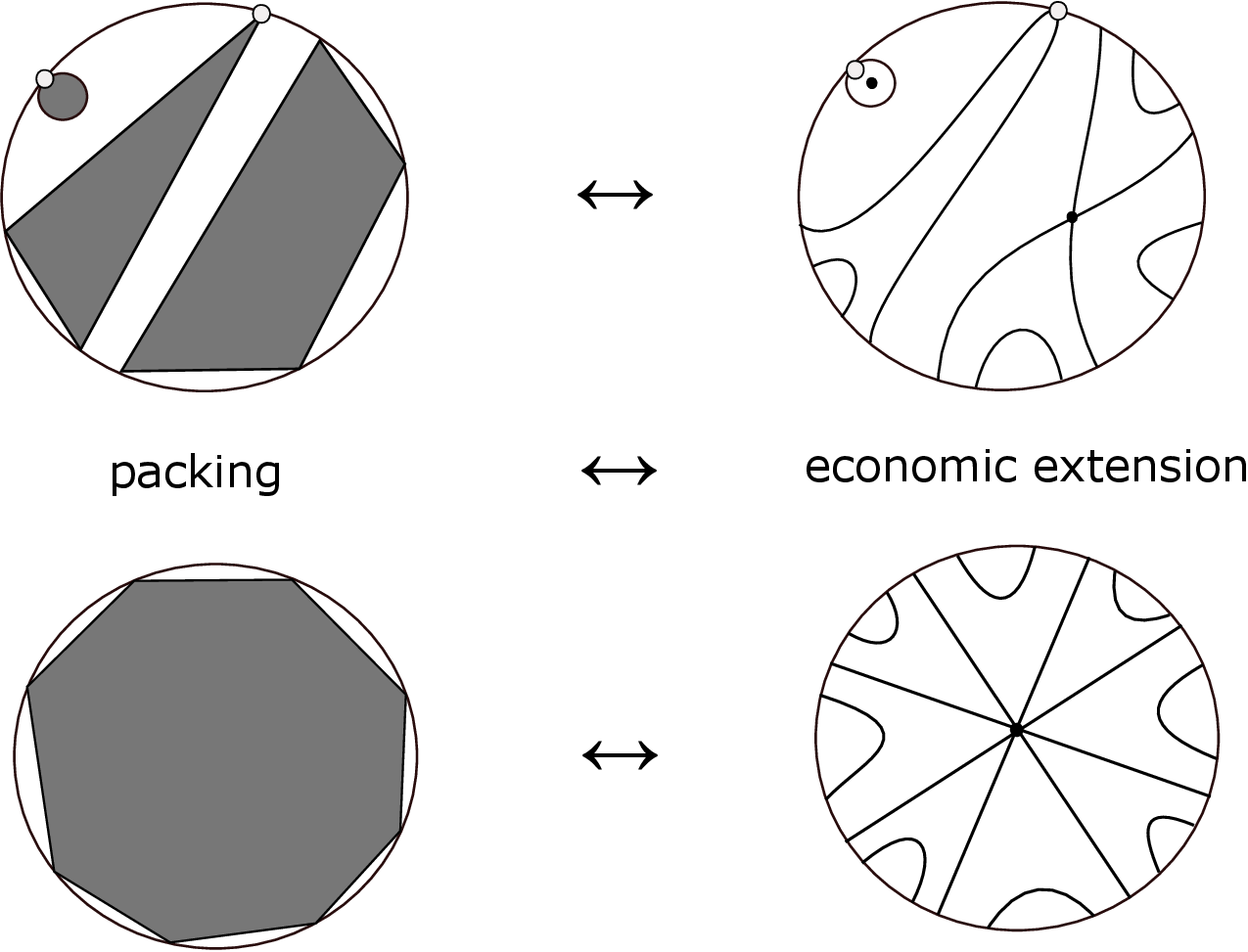}\caption{Correspondence between
packings and economic extensions.}%
\label{f30}%
\end{figure}
\end{center}

Note that it is almost a discretization of the problem of finding $\gamma$, as
for a given ribbon there are only a finite number of possibilities for
arranging a packing. We may also describe via packings all the economic
extensions $f\in\mathcal{F}^{e}(a)$.

Let us look now at the other ribbon invariants. It turns out that the problem
here may be resolved just in the same way, we have only to assign different
weights to the critical elements.

a) For $\gamma_{0}$ we have to take the following weight system:

1) If $Q$ is of first kind and is a $2k$-gon, set $w_{0}(Q)=k-1$. This is due
to the fact that the morsification of a saddle with $2k$ separatrices contains
$k-1$ non-degenerate saddles. (This morsification may be performed by a small
perturbation in each saddle with $>4$ separatrices.) An equivalent way is to
consider from the beginning packings with only quadrilateral iso-gons of first kind.

2) $w_{0}(Q)=0$ if $Q$ is of second kind

3) $w_{0}(Q)=1$ if $Q$ is of third kind

\medskip

b) for $\gamma_{\mathtt{\operatorname{ext}}}$ only extrema matter, so take

1) $w_{\operatorname{ext}}(Q)=0$ if $Q$ is of first or second kind

2) $w_{\operatorname{ext}}(Q)=1$ if $Q$ is of third kind

\medskip

c) for $\gamma_{\mathtt{\operatorname{sad}}}$ only saddles matter, so take

1) $w_{\operatorname{sad}}(Q)=1$ if $Q$ is of first kind

2) $w_{\operatorname{sad}}(Q)=0$ if $Q$ is of second or third kind.

\medskip Now we have to look for packings minimizing the corresponding weight
function (\ref{15}), to determine the value of the corresponding ribbon invariant.

\begin{center}
\begin{figure}[ptb]
\includegraphics[width=80mm]{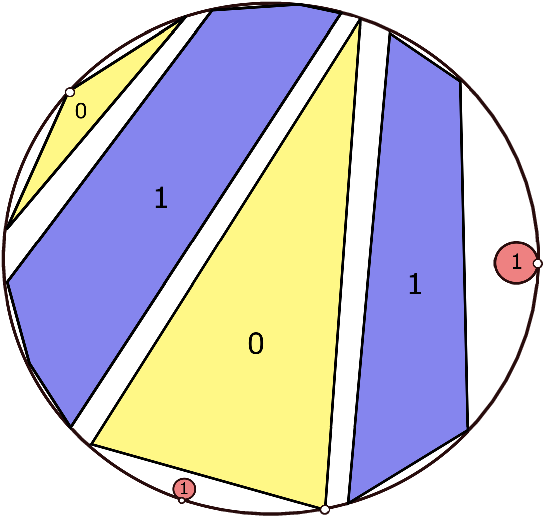}\caption{A packing with $w=4$,
$w_{0}=5$, $w_{\operatorname{ext}}=2$, $w_{\operatorname{sad}}=2$.}%
\label{big-ex}%
\end{figure}
\end{center}

At \hyperref[big-ex]{Fig.~\ref*{big-ex}} we present a packing $q$ where
different types of elements are differently coloured: first kind - blue,
second kind - yellow, third kind - red. The $w$-weights are also given at the
picture. Note that $w(q)=4$, $w_{0}(q)=5$, $w_{\operatorname{ext}}(q)=2$,
$w_{\operatorname{sad}}(q)=2$.

This geometrical description of the ribbon invariants is a base for another
combinatorial description via graphs given in Part~II. Therein, for a given
ribbon, a particular graph of \textit{virtual critical elements} is
constructed and its \textit{weighted independent domination number} turns out
to be equal to some of the four ribbon invariants, depending on the weight
system selected.

\section{\label{s11}A general algorithm for computing $\gamma$}

We describe in this section some general \textquotedblleft brute
force\textquotedblright\ algorithm for parallel computation of all the ribbon
invariants. It is based on splittings and induction on the lexicographic
ordering in the space $\mathcal{B}$ of discrete ribbons. In fact, we already
used before such a procedure for proving things by induction.

Recall first the ordering of $\mathcal{A}$. Let $a=(\varphi,\nu)\in
\mathcal{A}$ be a ribbon, then we may consider its \textquotedblleft
discretization\textquotedblright\ $i(a)\in\mathcal{B}$, where $\mathcal{B}$ is
the set of pairs $b=(t,\nu)$, such that $t=(c_{1},c_{2},\dots,c_{n})$ is a
cyclic zig-zag permutation of $\{1,2,\dots,n\}$ with $n$ even, and $\nu
(c_{i})=\pm1$ is some mark function. The discretization is defined in
\hyperref[s4]{Section~\ref*{s4}}. Set $\nu(t)=(\nu(c_{1}),\dots,\nu(c_{n}))$.
Consider the triple $(n,t,\nu(t))$. It is clear now that we may introduce a
lexicographic ordering in these triples, assuming that $+1\prec-1$. In such a
way $\mathcal{B}$ is linearly ordered and $\mathcal{A}$ inherits a linear
order as well.

Before exposing the algorithm, we shall define some very simple operation that
we call \textquotedblleft short cancellation\textquotedblright\ and which
should be performed at the beginning. It does not affect $\gamma$ (and the
other ribbon invariants defined before).

Let the ribbon $a=(\varphi,\nu)$ have two consecutive nodes of type
$(p_{i}^{+}~p_{i+1}^{-})$ or $(p_{i}^{-}~p_{i+1}^{+})$ such that there are no
node levels $l_{j}$ between $l_{i}$ and $l_{i+1}$, or equivalently for
discrete ribbons, $\left\vert l_{i}-l_{i+1}\right\vert =1$. Then this pair is
cancellable by a \textit{short cancellation} which consists in removing its
nodes from $a$ (see \hyperref[fxx]{Fig.~\ref*{fxx}}). It is clear that the new
ribbon $a^{\prime}$ has the same signature as $a$, moreover, it is not
difficult to see that $\gamma(a)=\gamma(a^{\prime})$. The latter holds true
for the other ribbon invariants $\gamma_{\ast}$ as well.

In brief, the algorithm (before discretization) consists in the next steps:

0) perform all the possible \textit{short cancellations} in $a$ (if any).

1) if $a=(\varphi,\nu)\notin\mathcal{A}^{+}$, i.e. there is a negative node
$p$, then there are 2 cases

1a) $\varphi(p)$ equals $\min\varphi$ or $\max\varphi$. Then we take
$a^{\prime}=(\varphi,\nu^{\prime})$, where $\nu^{\prime}$ is identical to
$\nu$, except for node $p$, which is marked as ``positive''. Consequently
$a^{\prime}\prec a$ and we have $\gamma(a)=\gamma(a^{\prime})+1$.

1b) $\varphi(p)$ does not equal $\min\varphi$ or $\max\varphi$. Let
$\varphi^{-1}(\varphi(p))=\{p,x_{1},\dots,x_{k}\}$ consider the
\textit{proper} pairs $(x_{i},x_{j})$ i.e. for which $i-j$ is odd. For any
such pair perform the ternary splitting of $a$ along $\{p,x_{i},x_{j}\}$%
\[
a=a_{1}\ast a_{2}\ast a_{3}\text{.}%
\]

Then $a_{i}\prec a$, $i=1,2,3$, since each $a_{i}$ has less nodes than $a$.
So, we may compute the number $\gamma_{i,j}=\gamma(a_{1})+\gamma(a_{2}%
)+\gamma(a_{3})$ and store it in some massif $Z$. Now take as in 1a)
$a^{\prime}=(\varphi,\nu^{\prime})$, where $\nu^{\prime}$ is identical to
$\nu$, except for node $p$, which is marked as ``positive''. Consequently
$a^{\prime}\prec a$ and we may compute $\gamma(a^{\prime})$. Then add the
value $\gamma(a^{\prime})+1$ to massif $Z$. This corresponds to the
possibility the extension $f\in\mathcal{F}^{e}(a)$ which is realizing
$\gamma(a)$ to have a touching circle at $p$.\ After the end of this procedure
it turns out that%
\[
\gamma(a)=\min Z\text{.}%
\]

This is due to the fact that any extension $f\in\mathcal{F}^{e}(a)$ which is
realizing $\gamma(a)$ has either a touching level line $l$ which is passing
through some triple $\{p,x_{i},x_{j}\}$, or a touching circle at $p$. Finally,
we refer to \hyperref[4]{Lemma~\ref*{4}}.

2) Let $a=(\varphi,\nu)\in\mathcal{A}^{+}$, then it is easy to find the
clusters $C_{i}$. There are two cases

2a) There is only one cluster. Then $\delta(a)=1$, $a$ is an alternation and
one has $\gamma(a)=1$.

2b) There are at least two clusters, so $\delta(a)\geq2$. Then take some
non-critical value $c$ of $\varphi$ which is situated between two adjacent
clusters $C_{1}$ and $C_{2}$ and let $\varphi^{-1}(c)=\{x_{1},\dots,x_{k}\}$.
Consider now the \textit{proper pairs} of type $(x_{i},x_{j})$ for which $i-j$
is odd. For any such pair perform the binary splitting of $a$ along
$\{x_{i},x_{j}\}$%
\[
a=a_{1}\#a_{2}\text{.}%
\]

Then we take into account only such binary splittings, for which $a_{i}\prec
a$, $i=1,2$, since the value $c$ is essential and hence for any pairing in
$\varphi^{-1}(c)$ there exists some essential pair $(x_{i},x_{j})$, which is
\textit{paired }(\hyperref[s5]{Section~\ref*{s5}}). So, we may compute the
number $\gamma_{i,j}=\gamma(a_{1})+\gamma(a_{2})$ and store it in some massif
$Z$. After the end of this procedure, one has $\gamma(a)=\min Z$ again. This
follows from the fact that any extension $f\in\mathcal{F}^{e}(a)$ which is
realizing $\gamma(a)$ has a regular level line $l$ which is passing through
some pair $\{x_{i},x_{j}\}$. Now, refer to \hyperref[4]{Lemma~\ref*{4}} again.

Note that the result surely does not depend on the particular choice of the
negative node $p$ in case 1), or the choice of the non-critical value $c$ in
case 2). It gives us the freedom to choose an appropriate splitting level
where the ribbon is as ``thin'' as possible. This might reduce calculations.

Note also that in fact we do not use effectively the lexicographic order of
the zig-zag permutations $t$. The only things that matter are the rank (the
number of nodes) and the lexicographic ordering of $\nu(t)$.

This procedure is common for all invariants $\gamma$, $\gamma_{0}$,
$\gamma_{\mathtt{\operatorname{ext}}}$, $\gamma_{\mathtt{\operatorname{sad}}}%
$. The only difference is that we have to assign different values on the
irreducible ribbons as follows. Let $\alpha_{0}=(1^{+},2^{+})$, $\alpha
_{1}=(1^{+},2^{-})$, $\beta_{n}\in\mathcal{A}_{n}^{+}$ be an arbitrary
alternation. Then we set the normalization rules:

1) $\gamma(\alpha_{0})=\gamma_{0}(\alpha_{0})=\gamma
_{\mathtt{\operatorname{ext}}}(\alpha_{0})=\gamma_{\mathtt{\operatorname{sad}%
}}(\alpha_{0})=0$

2) $\gamma(\alpha_{1})=\gamma_{0}(\alpha_{1})=\gamma
_{\mathtt{\operatorname{ext}}}(\alpha_{1})=1$, $\gamma
_{\mathtt{\operatorname{sad}}}(\alpha_{1})=0$

3) $\gamma(\beta_{n})=1$, $\gamma_{0}(\beta_{n})=\frac{n}{2}-1$,
$\gamma_{\mathtt{\operatorname{ext}}}(\beta_{n})=0$, $\gamma
_{\mathtt{\operatorname{sad}}}(\beta_{n})=1$.

Note that furthermore we shall somehow identify all the positive alternations
$\beta_{n}\in\mathcal{A}_{n}^{+}$ and this may be justified by the fact that
all geometric ribbon invariants take the same values on these ribbons. Another
argument for this to be done is the observation that any two alternations
$\beta_{n},\beta_{n}^{\prime}\in\mathcal{A}_{n}^{+}$ may be transformed into
each other only by positive \textit{bypasses}, which are self-inverse
elementary moves (see \hyperref[s13]{Section~\ref*{s13}}).

Of course, if one wants to write a program or at least a flowchart for
realizing this algorithm, it certainly should be done for discrete ribbons
from class $\mathcal{B}$. Anyhow, we won't do that here. Note only that the
presented algorithm is in essence inductively ramifying and thus should be
heavy and slow for big $n$. Another serious obstacle is the ``immensity'' of
the ribbon space $\mathcal{B}_{n}$ for big $n$ and the fact that the algorithm
presumes that we have to keep a previously built up library of \textit{all}
smaller ribbons (together with their invariants).

\begin{remark}
From aesthetical point of view it would be fine to find an algorithm which
avoids ternary splittings. However, this is not possible for the presented
one. Indeed, there exist general alternations $\notin\mathcal{A}^{+}$ without
essential values. Then the only thing that remains is to do a ternary
splitting at a negative node. Anyway, in some cases an essential value can be
found even in this setting. For example, let $a$ be a general alternation with
two negative nodes $p_{i}$ and $p_{j}$, such that $i-j$ is even. Then taking
$c$ inside the unique cluster, it is easy to see that $c$ is an essential
value and we may keep on calculations without ternary splitting. On the other
hand, ternary splitting may be useful, if performed at a suitable negative node.
\end{remark}

Let us now give two examples and some general calculation based on the algorithm.

\begin{example}
Let $a=(1^{+},3^{+},2^{+},5^{+},4^{+},7^{+},6^{-},9^{+},8^{+},10^{+})$. We
shall find $\gamma(a)$. Schematically, the algorithm works as follows: choose
the negative node $6^{-}$; it is easy to see that: making $6^{-}$ positive
$\rightarrow$ $1+4=5$, ternary splitting along $6^{-}$ $\rightarrow$
$0+1+2=3$. Hence $Z=\{5,3\}$, thus $\gamma(a)=3$.
\end{example}

\begin{example}
Let $a=(1^{+},6^{+},2^{-},4^{+},3^{+},5^{-})$ be the ribbon from
\hyperref[f17]{Fig.~\ref*{f17}}. Work at $2^{-}$; making $2^{-}$ positive
$\rightarrow$ $1+1=2$, ternary splitting along $2^{-}$ $\rightarrow$
$0+0+2=2$. Thus $Z=\{2,2\}$ and $\gamma(a)=2$.

Look now for $\gamma_{\mathtt{\operatorname{sad}}}$. The same scheme gives
$\rightarrow$ $0+1=1$, $\rightarrow$ $0+0+1=1$. Thus $Z=\{1,1\}$ and
$\gamma_{\mathtt{\operatorname{sad}}}(a)=1$.
\end{example}

Note that $a$ from the first example is a ladder (see \hyperref[s3]%
{Section~\ref*{s3}}). The next proposition finds $\gamma$ for all such ribbons.

\begin{proposition}
\label{p10}Let $a=(\varphi,\nu)\in\mathcal{A}$ be a ladder with nodes
$p_{1},\dots,p_{n}$. Suppose that the minimal and the maximal nodes of $a$ are
positive. Then%
\[
\gamma(a)=\#\left\{  k|\ \nu(p_{2k})\nu(p_{2k+1})>0,\ k=1,\dots,\frac{n}%
{2}-1\right\}  \text{.}%
\]

\end{proposition}

\begin{proof}
It suffices to make binary splittings at the non-critical levels $c_{k}%
=\frac{4k-1}{2}$, $k=1,\dots,\frac{n}{2}-1$ decomposing in such a way $a$ into
$\frac{n}{2}-1$ elementary ribbons $a_{k}\in\mathcal{A}_{4}$ with positive
minimal and maximal nodes. Now observe that $\gamma(a_{k})=1$ if the other two
nodes of $a_{k}$ have the same marking, and $\gamma(a_{k})=0$ otherwise. It
may also be seen that $\gamma_{0}=\gamma$ for such ladders.
\end{proof}

\textbf{Question.} Is it true that $\gamma_{0}(a)=\gamma(a)$ implies that the
ribbon $a$ is a ladder?

Note that the convention about the minimal and the maximal nodes, say $p$ and
$q$, is not restrictive, since if it is not assumed, anyway one easily finds
$\gamma(a)=\gamma(a^{\prime})+\varepsilon$, where $a^{\prime}$ is obtained
from $a$ by making $p$ and $q$ positive, and $\varepsilon$ is the number of
negative nodes among $p,q$. It is also clear that the other ribbon invariants
may be found in a similar way, relying on the corresponding splitting
equality. More precisely, we have the following

\begin{proposition}
\label{p11}Let $a$ be a ladder as in \hyperref[p10]{Proposition~\ref*{p10}}. Then

1) $\gamma_{0}(a)=\gamma(a)=\#\left\{  k|\ \nu(p_{2k})\nu(p_{2k+1}%
)>0,\ k=1,\dots,\frac{n}{2}-1\right\}  $

2) $\gamma_{\mathtt{\operatorname{ext}}}(a)=\#\left\{  k|\ \nu(p_{2k}%
)\nu(p_{2k+1})>0\text{ and }\nu(p_{2k})<0,\ k=1,\dots,\frac{n}{2}-1\right\}  $

3) $\gamma_{\mathtt{\operatorname{sad}}}(a)=\#\left\{  k|\ \nu(p_{2k}%
)\nu(p_{2k+1})>0\text{ and }\nu(p_{2k})>0,\ k=1,\dots,\frac{n}{2}-1\right\}  $.
\end{proposition}

These formulas about ladders will be important for the general ``fast''
algorithm for parallel computation of the ribbon invariants presented in
Part~II. Another interesting property of ladders is that the class of economic
extension $f\in\mathcal{F}^{e}(a)$ realizing the ribbon invariants may be
quite easily described (up to topological similarity).

\section{\label{s12}When does $\gamma=0$ ?}

It is natural to be interested in the case $\gamma=0$, as this would answer
the question under what boundary conditions there is a critical points free
extension. Note that for a general position function $\varphi:\mathbb{S}%
^{1}\rightarrow\mathbb{R}$ critical points free extensions always exist
(\hyperref[s16]{Section~\ref*{s16}}).\ In the setting of ribbons, an obvious
necessary condition for this is, of course, $\sigma=2$, which means that for
any extension $f$ of the corresponding ribbon we have $\deg(\nabla
f|_{\mathbb{S}^{1}})=0$. However, as we pointed out on several occasions,
$\sigma=2$ is not sufficient for $\gamma=0$. We shall try, in this section, to
describe all such ribbons. Again, it turns out that there is no easy or
immediate answer to the problem. Of course, one may apply the general
algorithm for $\gamma$, but, as we pointed above, it is very heavy and slow.
Here we present another algorithm for establishing $\gamma=0$ based on ``cancellations''.

From the geometrical point of view adopted in \hyperref[s10]%
{Section~\ref*{s10}}, we have $\gamma=0$ if and only if there is a packing
consisting only of triangles (\hyperref[f31]{Fig.~\ref*{f31}}). Let us call a
face of a triangle \textit{boundary}, if it is adjacent to a positive node.
The number of such faces equals $\frac{n}{2}+1$, the number of positive nodes.
Now, it is evident that there is a boundary face $l$, having an end which is a
negative node $p$. Let $q$ be the positive node adjacent to $l$. Roughly
speaking, we shall \textquotedblleft cancel\textquotedblright\ nodes $p$ and
$q$, obtaining in such a way a new ribbon with $n-2$ nodes. It turns out that
the new ribbon has $\gamma=0$ again and we may proceed by induction. Then we
should finally finish with the minimal ribbon $\alpha_{0}=(1^{+},2^{+})$.

\begin{center}
\begin{figure}[ptb]
\includegraphics[width=83mm]{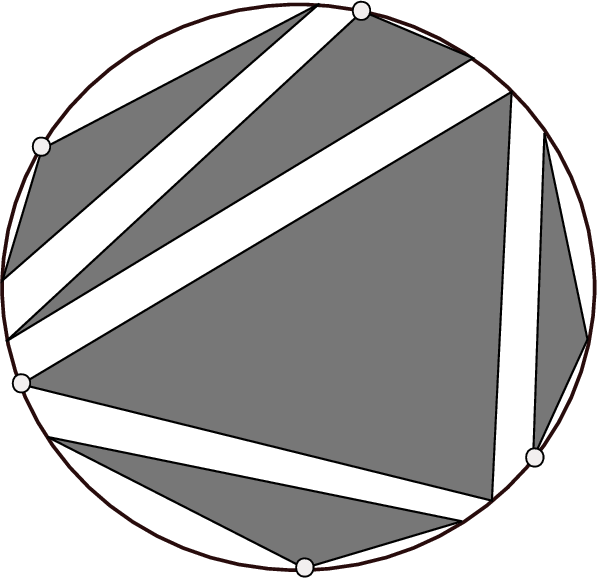}\caption{A packing of triangles yields
$\gamma=0$.}%
\label{f31}%
\end{figure}
\end{center}

Of course, this scheme works only if we were told in advance that $\gamma$
equals zero and a triangle packing were presented, but this is actually the
problem! In fact, supposing the problem is resolved for $n-2$, we have to
perform all \textit{admissible} cancellations and to look whether they lead to
some ribbon with $\gamma=0$ or not. This is the inductive variant of the
algorithm which implies a previously built library of ribbons with $\gamma=0$
and $\leq n-2$ nodes. The other way is to present a ramifying algorithm which
ends when reaching $\alpha_{0}=(1^{+},2^{+})$. Anyway, both algorithms are
expensive, but not so bad as the general algorithm for calculating $\gamma$.

Let $a=(\varphi,\nu)\in\mathcal{A}$ be a ribbon, $p$, $q$ and $r$ be
consecutive nodes (in either orientation of the circle) such that $p$ is
negative, $q$ is positive and $|\varphi(p)-\varphi(q)|<|\varphi(q)-\varphi
(r)|$. We shall say that the pair $(p,q)$ is \textit{cancellable, }since we
may consider the ribbon\textit{ }$a^{\prime}$ obtained from $a$ by removing
nodes $p$ and $q$ and keeping unchanged the other information of $a$. The only
difference between $a$ and $a^{\prime}$ is that we replace the
\textquotedblleft twist\textquotedblright\ $(p,q)$ by a monotonic piece in the
graph of $\varphi$. (See \hyperref[fxx]{Fig.~\ref*{fxx}})

\begin{center}
\begin{figure}[ptb]
\includegraphics[width=100mm]{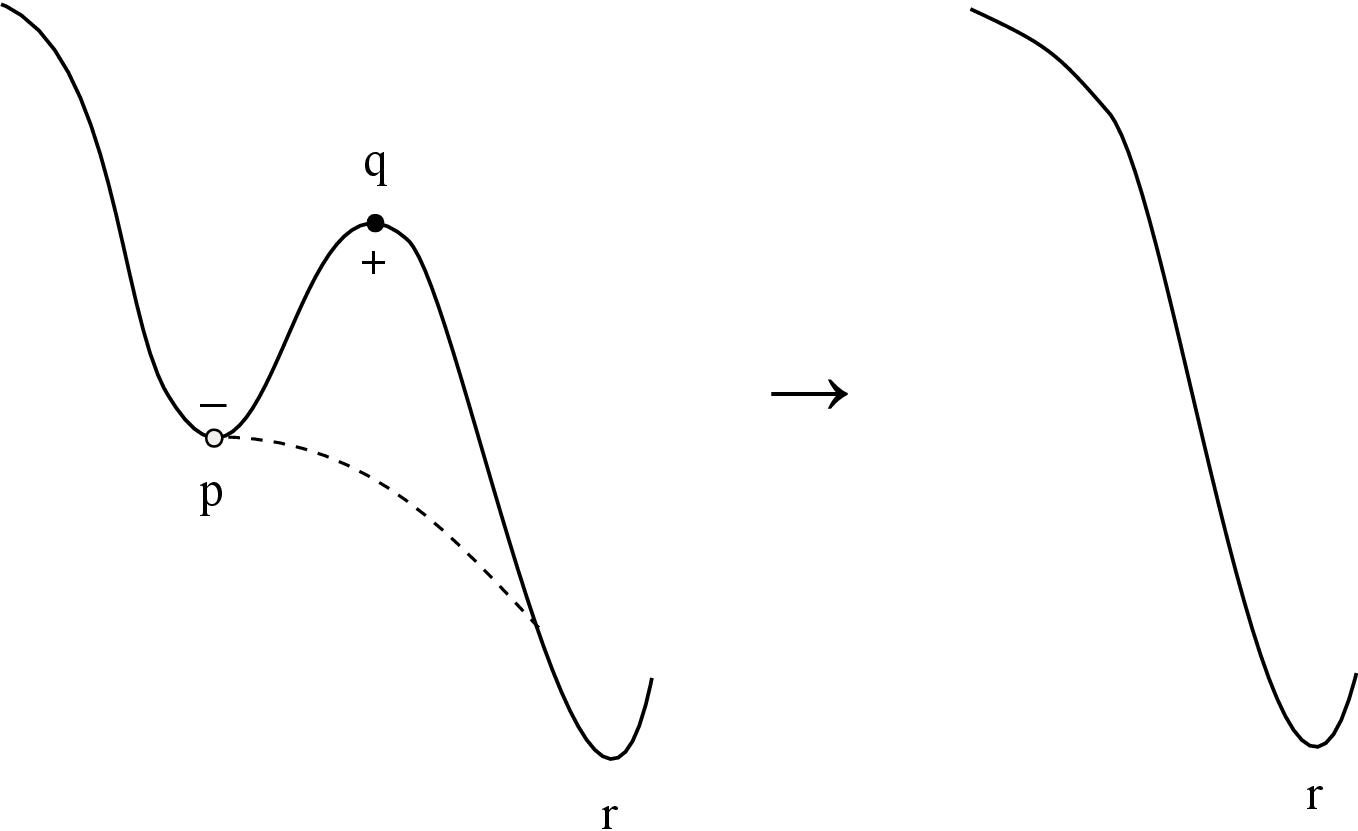}\caption{Cancellation of the pair
$(p^{-},q^{+})$.}%
\label{fxx}%
\end{figure}
\end{center}

\begin{definition}
\label{d19}We shall call the described above operation $a\rightarrow
a^{\prime}$ ``cancellation'' of nodes $p$ and $q$. The inverse operation will
be called ``extension'' of the ribbon. Then two new born nodes $p$ and $q$ appear.
\end{definition}

Note that cancellation and extension do not change the signature $\sigma$. It
should be noted also that these are ``long'' operations and may change
$\gamma$, in contrast with ``death'' and ``birth'' of a couple $(p,q)$ which
are ``short'' operations and do not affect $\gamma$.

For discrete ribbons cancellation is defined analogically:

Let $a=(t,\nu)\in\mathcal{B}$ be a discrete ribbon, where $t=(1,c_{2}%
,\dots,c_{n})$ is a cyclic zig-zag permutation and $\nu(c_{i})=\pm1$ is some
marking. Suppose for some $i$ that $\nu(c_{i})=-1$, $\nu(c_{i+1})=+1$ and
$|c_{i}-c_{i+1}|<|c_{i+1}-c_{i+2}|$. Then removing $c_{i}$ and $c_{i+1}$ and
rearranging the rest nodes, we get the canceled ribbon $a^{\prime}$.

\begin{lemma}
\label{l6}Let $a$ be a ribbon with $\gamma(a)=0$. Then it has a cancellable
pair of nodes $(p,q)$.
\end{lemma}

\begin{proof}
Let $f\in\mathcal{F}^{e}(a)$ be some critical points free extension. Take an
arbitrary positive node $p$ and consider the closest to it touching line $l$.
Let $l$ is touching the boundary at some negative node $q$. Then it is evident
that $(p,q)$ is a cancellable pair.
\end{proof}

\begin{lemma}
\label{l7}Let $a^{\prime}$ be obtained from $a$ by cancellation. Then
$\gamma(a^{\prime})\geq\gamma(a)$.
\end{lemma}

\begin{proof}
Let $f\in\mathcal{F}^{e}(a^{\prime})$ be an extension of $a^{\prime}$ with
$\gamma(a^{\prime})$ critical points. Since $a$ is obtained from $a^{\prime}$
by extension, it is evident from picture that we may attach a noncritical
\textquotedblleft shoulder\textquotedblright\ to $f$, obtaining in such a way
an extension of $a$ with $\gamma(a^{\prime})$ critical points again.
\end{proof}

This inequality sounds a little bit paradoxically, as \textquotedblleft
simplifying\textquotedblright\ the ribbon, its ribbon invariant is increasing
as a result! On the other hand, it is clear that such jumps of $\gamma$ are
possible: Imagine, for example, that $\gamma(a)=0$, but after cancellation
some negative node of $a^{\prime}$ becomes a maximal one. Then, surely,
$\gamma(a^{\prime})>0$. Other examples show that the jump of $\gamma$ after a
single cancellation may be large enough. This is due to the fact that such a
cancellation may be done continuously and the examination of the jump of
$\gamma$ at the possible elementary moves shows that $\gamma$ may only
increase (see \hyperref[s13]{Section~\ref*{s13}}).

Now we may summarize the above considerations in the following proposition:

\begin{theorem}
\label{t8}The equality $\gamma(a)=0$ holds true if and only if ribbon $a$ may
be connected with the trivial ribbon $\alpha_{0}=(1^{+},2^{+})$ by a chain of cancellations.
\end{theorem}

Note that it follows from this theorem, that $\sigma(a)=2$ is a necessary
condition for $\gamma(a)=0$, which was our starting observation.

The following proposition is not crucial for this section, but enlightens the
behaviour of the ribbon invariant through extensions. It turns out that the
ribbon invariant of any ribbon may be reduced by extensions to the minimal
possible value for ribbons with the same signature $\sigma$. It is easy to see
that this is in fact the number $i_{0}(\sigma)$, defined as follows:

\begin{center}
$i_{0}(\sigma)=1-\frac{\sigma}{2}$, if $\sigma\leq2$ and $i_{0}(\sigma)=1$, if
$\sigma>2$.
\end{center}

\begin{proposition}
\label{p12}Any ribbon $a$ with signature $\sigma$ may be upgraded by
extensions to a ribbon $a^{\prime}$ with $\gamma(a^{\prime})=i_{0}(\sigma)$.
\end{proposition}

For example, any ribbon with $\sigma=2$\ may extended to a ribbon with
$\gamma=0$. The proof is by induction on the number of nodes. Clearly, it may
be done by $\leq\gamma(a)-i_{0}(\sigma)$ extensions. Moreover, by the same
method it may be shown that any integer from $\left[  i_{0}(\sigma
),\gamma(a)\right]  $ may be realized by extensions of ribbon $a$. The proof
of \hyperref[p12]{Proposition~\ref*{p12}} is easily done by induction and the
splitting technique.

Now we present a simple (but not fast) algorithm for resolving the problem
whether $\gamma=0$ for a given ribbon.

\textbf{The algorithm. }Of course, it is convenient to work with discrete
ribbons. Now, in view of \hyperref[l6]{Lemmas~\ref*{l6}} and \ref{l7}, the
procedure of checking whether $\gamma=0$ is quite clear:

0) Check $\sigma=2$, if ``yes'': continue, if ``no'': stop: $\gamma>0$.

Perform all possible cancellations in $a$, obtaining in such a way a set
$R_{n-2}$ of ribbons if order $n-2$, then perform all cancellations of the
elements of $R_{n-2}$, obtaining the set $R_{n-4}$, etc. After the process
stops, there are two cases:

1) $R_{2}\neq\emptyset$. Then $\gamma(a)=0$, since we connected $a$ with the
minimal ribbon $(1^{+},2^{+})\in R_{2}$ through cancellations, which, in view
of \hyperref[l7]{Lemma~\ref*{l7}} are raising $\gamma$.

2) $R_{2}=\emptyset$. Then $\gamma(a)\neq0$, as if $\gamma(a)=0$, there should
be a path from $a$ to $(1^{+},2^{+})$ through cancellations, as we pointed out
at the beginning of the section. This case is available if the cancellation
process always stops at some ribbon with $\geq4$ nodes.

Surely, this algorithm, being inductively ramifying, should be expensive in time.

There are some situations, when the algorithm may be shortened.

a) It is not necessary to find all paths leading to $R_{2}$, we may stop at
the first moment when $R_{2}\neq\emptyset$ is established. Anyway, if one
wants to find \textit{all} critical points free extensions (in case they
exist), it is equivalent to finding \textit{all} paths leading to $R_{2}$.

b) We don't need to follow a path which leads to a ribbon with a negative
minimal or maximal node, since then surely $\gamma>0$ for that ribbon and
$\gamma$ cannot decrease through cancellations. If, by chance, the initial
ribbon is so, we should stop immediately.

At \hyperref[f32]{Fig.~\ref*{f32}} and \hyperref[f33]{Fig.~\ref*{f33}} we
present graphical example of application of the algorithm for two almost
identical ribbons with $\sigma=2$, that lead to different results. Note that
in the second example we don't need to follow both paths, only one of them
suffices to establish $\gamma=0$.

\begin{center}
\begin{figure}[ptb]
\includegraphics[width=110mm]{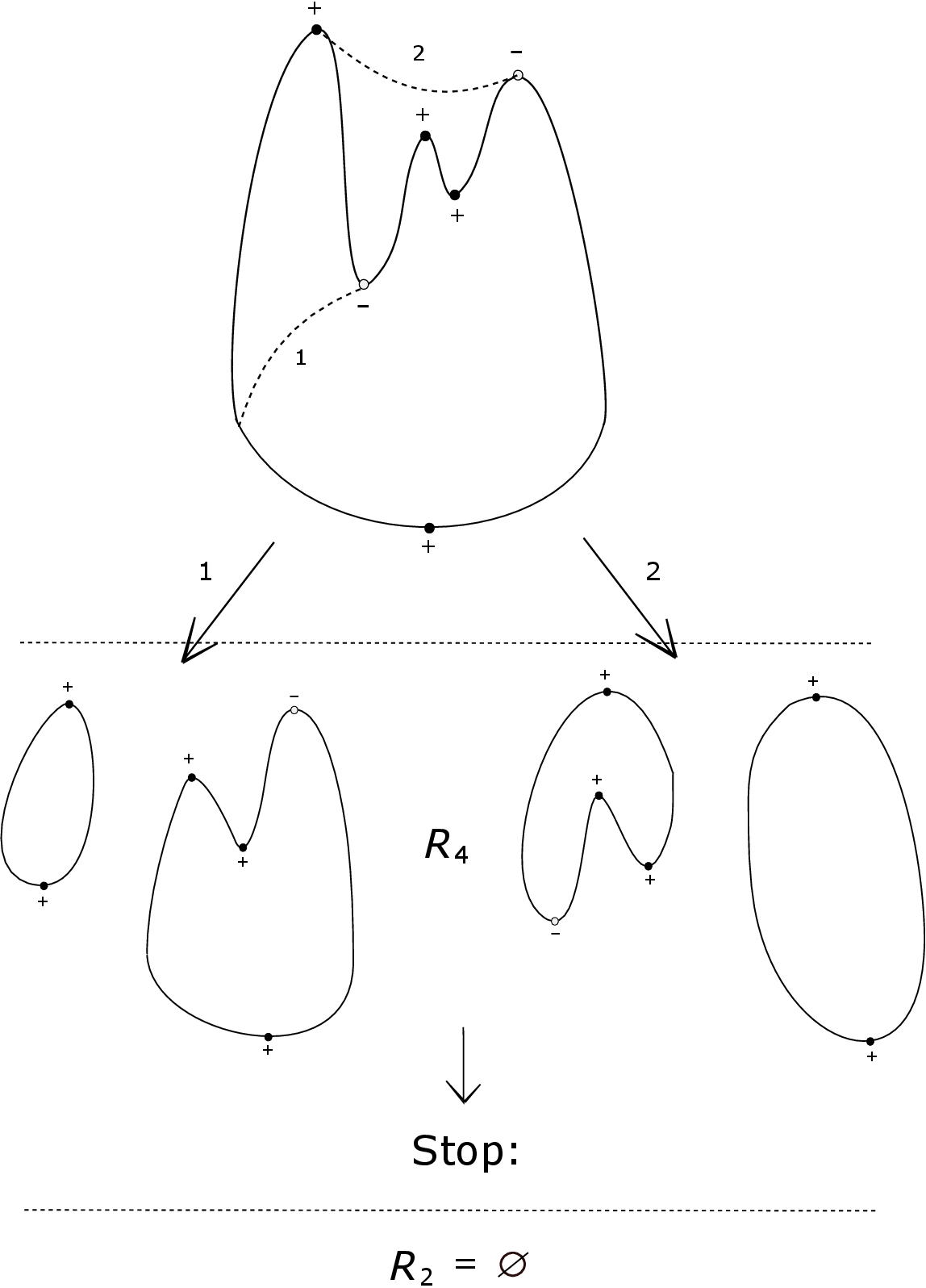}\caption{$R_{2}=\varnothing
\Rightarrow\gamma\neq0$.}%
\label{f32}%
\end{figure}
\end{center}

As we pointed out, another inductive variant of the algorithm is possible,
based on a previously built library of ribbons of rang $\leq n-2$ with
$\gamma=0$. Then the problem for a ribbon of rank $n$ is resolved by
$\leq\frac{n}{2}$ cancellations. The problem here is that this algorithm would
be expensive in storage.

\begin{center}
\begin{figure}[ptb]
\includegraphics[width=110mm]{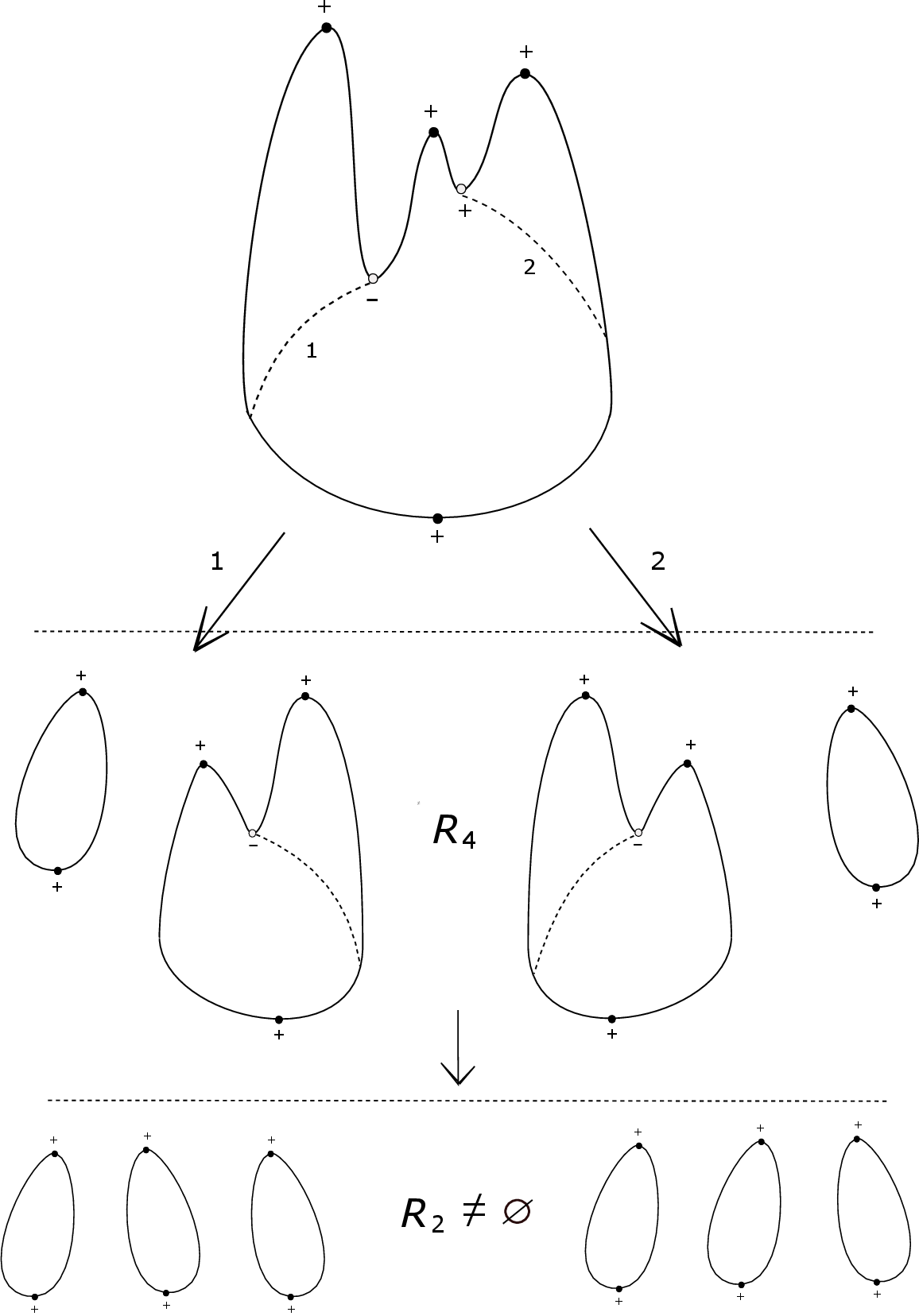}\caption{$R_{2}\neq\varnothing
\Rightarrow\gamma=0$.}%
\label{f33}%
\end{figure}
\end{center}

It would be interesting to find/estimate the number of ribbons from
$\mathcal{A}_{n}$ with $\gamma=0$, but this seems to be a hard task. Anyway,
it is easy to find a relation between the ribbons with $\gamma=0$ and those
with $\gamma_{\operatorname{sad}}=0$.

\begin{proposition}
Let $N^{0}$ be the number of ribbons from $\mathcal{A}_{n}$ with $\gamma=0$
and $N_{\operatorname{sad}}^{0}$ be the number of ribbons from $\mathcal{A}%
_{n}$ with $\gamma_{\operatorname{sad}}=0$. Then%
\[
N_{\operatorname{sad}}^{0}=2^{\frac{n}{2}+1}N^{0}.
\]

\end{proposition}

\begin{proof}
Let $a$ be a ribbon with $n$ nodes with $\gamma=0$. It is evident that making
some positive node $p$ of $a$ negative, we get a ribbon $b$ with
$\gamma_{\operatorname{sad}}(b)=0$, since we may modify the solution for $a$
by attaching a local extremum at $p$. On the other hand, it is easy to see
that any ribbon with $\gamma_{\operatorname{sad}}=0$ may be obtained from some
ribbon with $\gamma=0$ by multiple application of the above operation. But the
number of positive nodes of $a$ equals $\frac{n}{2}+1$ (since $\sigma(a)=2$),
whence the above formula.
\end{proof}

\textbf{A \textquotedblleft ribbon game\textquotedblright.}\label{r-game} Here
we describe some game based on the ribbon invariant.

At the beginning, a cyclic zig-zag permutation $t$ of even order $n$ is
automatically generated. Two players, say A and B, are playing consecutively.
Player A has $\frac{n}{2}-1$ \textquotedblleft negative\textquotedblright%
\ pools, while player B has $\frac{n}{2}-1$ \textquotedblleft
positive\textquotedblright\ pools. At each move, the corresponding player is
putting a pool on some element of the permutation $t$, marking it in such a
way either positive, or negative. The minimal and the maximal elements of $t$
are marked as \textquotedblleft positive\textquotedblright\ automatically from
the beginning. Player A starts first. At the end, we get some discrete ribbon
$a=(t,\nu)\in\mathcal{B}$. The score is calculated according to the rule:

If $\gamma(a)\neq0$, then player A wins, if $\gamma(a)=0$, then player B wins.

In such a way, player A is trying to \textquotedblleft
destroy\textquotedblright\ the ribbon making $\gamma(a)$ as big as possible,
while player B is trying to minimize $\gamma(a)$, reducing it finally to zero.
Note that the final ribbon has signature $\sigma=2$, which is necessary for
$\gamma=0$. From practical point of view, it seems reasonable to play with
permutations of order $\leq8$, otherwise it may happen that the players cannot
determine who is the winner. Another variant is to generate ribbons with some
symmetry, that allows the players to follow a particular strategy.

We conjecture here the following:

\textbf{There is a winning strategy for player B.}

In other words, player B may always minimize the ribbon invariant to 0, by
putting positive pools. Here is an example of such a strategy for some basic
class of zig-zag permutations.

Let $t=(1,3,2,5,4,7,\dots,(n-2),n)$ be a ``ladder''. Then the winning strategy
for player B is the following one:

1) If A selects $2k$, then B selects $2k+1$

1) If A selects $2k+1$, then B selects $2k$.

So we get at the end some marked ladder, then, as follows from \hyperref[p10]%
{Proposition~\ref*{p10}}, this ribbon surely has $\gamma=0$, thus player B wins.

Yet another example. Consider the zig-zag permutation from \hyperref[f17]%
{Fig.~\ref*{f17}} (it is a $C^{0}$-alternation) and let's play the ribbon
game. Then the winning strategy for player B is quite simple: if A selects
node $\{3\}$, then B selects node $\{6\}$, and vice versa, if A selects node
$\{6\}$, then B selects node $\{3\}$. Similarly, if A selects $\{5\}$, then B
selects $\{4\}$, if A selects $\{4\}$, then B selects $\{5\}$. Then it is easy
to see that player B wins. The crucial moment here is not to allow nodes
$\{3,6\}$ to be marked simultaneously as \textquotedblleft
negative\textquotedblright(then the situation from \hyperref[f18]%
{Fig.~\ref*{f18}} occurs), neither nodes $\{4,5\}$ to be marked simultaneously
as \textquotedblleft positive\textquotedblright. These two examples suggest
that there might be a general algorithm based on some \textit{pairing} of the nodes.

There is a stronger variant of the game, when the result is calculated
according to whether the final ribbon is \textit{Jordan}, or not (see
p.~\pageref{jordan}). If a ribbon $a$ is Jordan, then surely $\gamma(a)=0$.
Note that this variant is more geometric than the original one. We suppose
again that there is a winning strategy for player B in this \textit{Jordan
game}. Of course, this yields a winning strategy for player B in the original game.

A more general variant of the game for arbitrary value of $\sigma$ is possible
to consider, where, roughly speaking, player A is trying to maximize $\gamma$,
while player B is trying to minimize it. The difference with the above variant
is that at the beginning some quantity of random pools is automatically
distributed between players. Anyhow, we won't go into details about that
possible variant of the game.

\section{\label{s13}Changes of $\gamma$ during elementary moves}

Suppose now the ribbon $a=(\varphi,\nu)\in\mathcal{A}_{n}$ is moving in a
generic way in class $\mathcal{A}_{n}$. This means that we consider the path
$a_{t}=(\varphi_{t},\nu)\in\mathcal{A}_{n}$, where $\varphi_{t}$ is a generic
homotopy. In general, there are two types of elementary moves - a
\textit{meeting} (and its inverse - \textit{separation}) and a \textit{bypass}%
. These are depicted at \hyperref[f25]{Fig.~\ref*{f25}}. The inverse of a
bypass is a bypass itself. It is natural to look at the possible jumps of the
ribbon invariant when the ribbon passes through a non general position state.
We shall list in this section all possible changes of $\gamma$ (and the other
invariants) both for noncritical and critical elementary moves. Roughly
speaking, any such a move changes the invariant by $0,\pm1,\pm2$ (some jumps
are impossible). Of course, nobody tells us what exactly the jump should be
(if any) at a given moment. In this sense, the ribbon invariant $\gamma$ is a
\textit{global} invariant, as it behaviour depends on the whole situation and
is not defined by any local rules, in contrast with the cluster number
$\delta$ for example (\hyperref[s7]{Section~\ref*{s7}}).

What is an elementary move? When a soft ribbon $a=(\varphi,\nu)\in\mathcal{A}$
is moving in a generic way, there are moments when he is changing its type.
The process of passing through such a ``degenerate'' position is called
\textit{elementary move}. In general, there are two types of such moves:

1. Noncritical moves. These are moves during which the band of the ribbon
remains critical points free. At such moves the signature $\sigma$ remains unchanged.

2. Critical moves. These moves allow the appearance of a single critical point
at the boundary. Strictly speaking, at the critical moment the ribbon is not a
\textit{true} ribbon anymore, but we shall accept them as true ribbons in the
present section. During a critical move, the signature changes by $\pm2$.

\begin{center}
\begin{figure}[ptb]
\includegraphics[width=90mm]{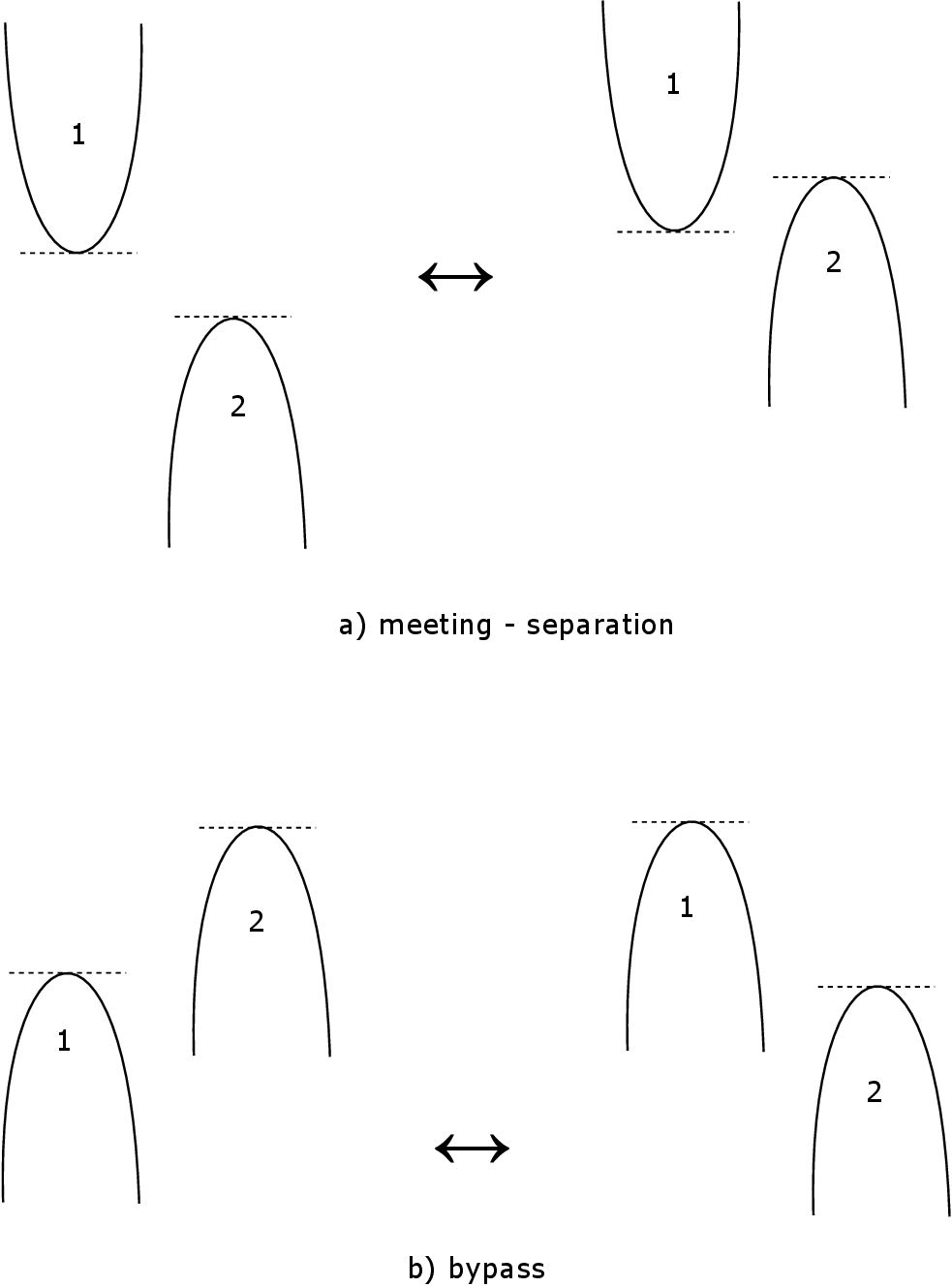}\caption{Elementary moves.}%
\label{f25}%
\end{figure}
\end{center}

In general, the noncritical moves are: meeting, separation, bypass, birth and
death. Meeting, separation and bypass (\hyperref[f25]{Fig.~\ref*{f25}}) are
actually interchanging two consecutive critical levels; when they are of
opposite type it is a meeting or separation, otherwise - a bypass. Meeting and
separation are mutually inverse operations, bypass is self inverse. Birth is
the appearance of a couple of close nodes with opposite marking $(+-)$ or
$(-+)$, death is the opposite action. As we pointed above, noncritical moves
do not change the signature of the ribbon. Birth and death change the number
of nodes by $\pm2$.

There is only one type of critical moves - change of the marking of a single
node to the opposite one: $(+)\rightarrow(-)$, or $(-)\rightarrow(+)$. From
geometric point of view, it corresponds to the tilt of the ribbon at a node,
making it for a moment horizontal and then changing the marking of the node to
the opposite one. These moves will be important for us in the next section
while proving the basic inequality $\gamma\leq\frac{n}{2}+1$.

Now we shall specify the possible jumps of $\gamma$. Most of them are
geometrically evident.

\begin{theorem}
\label{t9}Let $a$ be a ribbon and $a^{\prime}$ be obtained from $a$ by an
elementary move. Let $\gamma(a^{\prime})=\gamma(a)+\varepsilon$. Then the
possible values of the jump $\varepsilon$ are: \begin{tabbing}
a) meeting of two positive nodes, \hspace{3cm} \=$\varepsilon=0,-1$\\
b) separation of two positive nodes, \>$\varepsilon=0,+1$\\
c) meeting or separation of two negative nodes, \>$\varepsilon=0$\\
d) meeting of a positive and a negative node, \>$\varepsilon=0,-1,-2$\\
e) separation of a positive and a negative node, \>$\varepsilon=0,+1,+2$\\
f) bypass of tho positive or two negative nodes, \>$\varepsilon=0$\\
g) a positive node is bypassing a negative one, \>$\varepsilon=0,-1,-2$\\
h) a negative node is bypassing a positive one, \>$\varepsilon=0,+1,+2$\\
i) birth/death of a couple of nodes, \>$\varepsilon=0$\\
j) for a critical move $(+)\rightarrow(-)$ or $(-)\rightarrow(+)$,
\>$\varepsilon=0,\pm1$.\\
\end{tabbing}All listed values of the jump $\varepsilon$ are attained in some situations.
\end{theorem}

Detailed proof with the corresponding bifurcation diagrams will be exposed in
Part II of the article. It should be noted that although $\varepsilon=0$ in
some cases, the set of solutions may change at the corresponding elementary
move. Note also that a similar table for the jumps $\varepsilon_{\ast}$ of the
other ribbon invariants $\gamma_{\ast}$ are available, but we shall deal with
it in Part II of the article. For example, for move g) with $\varepsilon
=0,-1,-2$ we have $\varepsilon_{0}=0,-2$, $\varepsilon_{\operatorname{ext}%
}=0,-1$, $\varepsilon_{\operatorname{sad}}=0,-1$.

It turns out that elementary moves may be counted by suitable
\textquotedblleft invariants\textquotedblright, which are, in some sense,
similar to the Arnold's $J^{\pm}$ invariants, the latter being introduced for
immersed curves in the plane \cite{b4}. Like Arnold's invariants, these are
changing in a deterministic way at elementary moves. On the other hand, they
may be useful for the estimation of $\gamma$, in some situations. We describe
below one of these \textquotedblleft invariants\textquotedblright.

1) The number $cl^{++}$. Let $p$ and $q$ be positive nodes of ribbon $a$. Then
$cl^{++}$ changes in a deterministic way at their meeting or separation by
$0,\pm1$. It may be defined as follows.

Let $p$ and $q$ be two positive nodes of $a$ subject to meeting or separation.
Clearly, $p$ and $q$ are different from the minimal and the maximal node and
are not adjacent in the corresponding zig-zag permutation $\tau$. Then they
are of different type (minimum or maximum) and have consecutive values, say
$k$ and $k+1$ in $\tau$. Now we define $cl^{++}$ as the number of all such
pairs $k$,$~k+1$ in $\tau$ where $k$ is of maximal type (i.e. it is evenly
placed in $\tau$). Then, of course, $k+1$ is of minimal type and is oddly
placed. It is clear that $cl^{++}$ is the number of \textit{positive} clusters
in the level system of the ribbon. In such a way, it is an \textquotedblleft
invariant\textquotedblright\ in some trivial way.

\textit{Examples.} 1) If $a\in\mathcal{A}_{n}^{+}$ is a positive ladder, then
it is easy to see that $cl^{++}(a)=\frac{n}{2}-2$. For example, if
$a=(1^{+},3^{+},2^{+},5^{+},4^{+},7^{+},6^{+},8^{+})$, then $(3,4)$ and
$(5,6)$ are the only countable pairs, so $cl^{++}(a)=2$.

2) If $a\in\mathcal{A}_{n}^{+}$ is a positive alternation, then $cl^{++}%
(a)=1$. It suffices to notice that the only countable pair is $(k,~k+1)$,
where $k$ is the maximal minimum and $k+1$ is the minimal maximum.

3) Let $a\in\mathcal{A}_{n}$ be a general ladder, then it is not difficult to
see that $cl^{++}(a)$ equals the number of \textit{positive} pairs of type
$(2k+1,2k)$ in $a$. In case $a\in\mathcal{A}_{n}^{+}$ all such pairs are
positive, and since their number is $\frac{n}{2}-2$, we get ex. 1).

It may be of some interest to describe dynamically the change of $cl^{++}$
under elementary moves.

First of all, we shall divide the nodes of a ribbon $a\in\mathcal{A}$ into two
types - \textit{primary} and \textit{secondary} ones. A positive node $p$ is
\textit{primary}, if it is the end of a positive cluster. Otherwise it is
defined as \textit{secondary} (so, all the negative nodes are secondary by
definition). Now we consider the following indicator function on the nodes of
a given ribbon $a$:

\begin{center}
$\xi(p)=1$, if $p$ is a primary node, and $\xi(p)=0$, if $p$ is a secondary one.
\end{center}

\begin{proposition}
Let $a\in\mathcal{A}$ be a ribbon and we perform a \textquotedblleft
meeting\textquotedblright\ of two of its nodes $p$ and $q$, obtaining in such
a way the ribbon $a^{\prime}$. Then for the \textquotedblleft
jump\textquotedblright\ $\varepsilon=cl^{++}(a^{\prime})-cl^{++}(a)$ of
$cl^{++}$\ we have%
\[
\varepsilon=1-\xi(p)-\xi(q).
\]

\end{proposition}

\begin{proof}
It suffices to check this equality for all possible values of the pair
$\left(  \xi(p),\xi(q)\right)  $, namely that for $\left(  \xi(p),\xi
(q)\right)  \longrightarrow\varepsilon$ we have%
\begin{align*}
\left(  1,1\right)   &  \longrightarrow\varepsilon=-1\\
\left(  1,0\right)   &  \longrightarrow\varepsilon=0\\
\left(  0,1\right)   &  \longrightarrow\varepsilon=0\\
\left(  0,0\right)   &  \longrightarrow\varepsilon=1\text{,}%
\end{align*}

but this is quite easy to be verified. For example, $\left(  1,1\right)
\longrightarrow\varepsilon=-1$ follows from the observation that after the
meeting of two primary nodes one new positive cluster is \textquotedblleft
born\textquotedblright, but two old ones \textquotedblleft
die\textquotedblright.
\end{proof}

Of course, at separations the jump should be $\varepsilon=\xi(p)+\xi(q)-1$,
where $p$ and $q$ are the nodes after separation.

Denote by $\varkappa$ the difference between the new and the old value of
$cl^{++}$. Now, if we have a series of elementary moves bringing $a$ to $b$,
set $d(a,b)=\sum_{m}\varkappa_{m}$, where the sum is taken over all elementary
moves. It is clear that we have%
\[
d(a,b)=cl^{++}(b)-cl^{++}(a)\text{,}%
\]

since $cl^{++}+\varkappa$ is the new value of $cl^{++}$ by definition after
the corresponding move. In such a way, our invariant may be defined in class
$\mathcal{A}$ by the above rules and the normalization rule $cl^{++}(a)=1$ for
any alternation $a$. It follows from the above, that if $a$ is sent to $b$ by
a generic homotopy, the result about $cl^{++}$ does not depend on the path selected.

From this point of view, there is some analogy between the number $cl^{++}$
and Arnold's $J^{\pm}$ invariants \cite{b4}. The latter are defined for
immersed curves in the plane and are changing in a prescribed manner at the
moments of self-crossing of a given type. Another unifying property between
all these invariants is that all are of a \textit{local} type, i.e., they are
changing in a deterministic way at elementary crossings. We shall see later,
that this is not the case with the ribbon invariant $\gamma$, which changes in
an \textit{irrational} (unpredictable, at least to us) way at crossings. This
gives us the ground to say that $\gamma$ is a \textit{global} type invariant,
as its behaviour depends on the whole ribbon's information.

Note also that $cl^{++}$ is \textit{not} changing by $-1$ at meeting and by
$+1$ at separation, but is subject to more complicated rules. This raises the following

\textbf{Question.} Consider in class $\mathcal{A}^{+}$ the number $\rho$ which
is changing by $-1$ at meeting and by $+1$ at separation (and has the initial
value $\frac{n}{2}-2$ on ladders with $n$ nodes). Then is this number $\rho$
an invariant, i.e. is it independent of the path selected in $\mathcal{A}^{+}%
$? If so, does it have some intrinsic definition in terms of the corresponding
zig-zag permutation?

\begin{remark}
If $a=(\varphi,\nu)\in\mathcal{A}^{+}$ is moving in a generic way in class
$\mathcal{A}^{+}$ (so, the number of nodes may change), we have to add two new rules:

1) \textquotedblleft birth\textquotedblright\ of a couple of nodes
\ \ $cl^{++}\rightarrow cl^{++}+1$

2) \textquotedblleft death\textquotedblright\ of a couple of nodes
\ \ $cl^{++}\rightarrow cl^{++}-1$.
\end{remark}

This is caused by the fact that a little cluster is born or, respectively,
dies. In such a way we may control $cl^{++}$ during an arbitrary generic
homotopy in class $\mathcal{A}^{+}$.

Note finally that we may define and investigate in a similar way several other
\textquotedblleft invariants\textquotedblright\ (say $cl^{--}$, $cl^{+-}$,
etc.), which are changing during meeting/separation of negative nodes, of a
negative and a positive node and finally, a bypass of a negative and a
positive node. Note that the bypass of two positive or two negative nodes does
not produce a nontrivial invariant, since these are self-inverse moves.

\section{\label{s14}Proof of $\gamma\leq\frac{n}{2}+1$}

As we mentioned in the beginning, the most ``simple'' inequality $\gamma
\leq\frac{n}{2}+1$ turns out to be quite nonelementary, as it is not
attackable by induction via splittings. We shall give, in this section, a
proof based on critical moves.

\begin{proposition}
\label{p13}Let $\varphi:\mathbb{S}^{1}\rightarrow\mathbb{R}$ be a general
position smooth function. Then there is a marking $\nu:P\rightarrow\{-1,+1\}$
of its critical set $P$, such that for the corresponding ribbon $a=(\varphi
,\nu)$ we have $\gamma(a)=0$.
\end{proposition}

\begin{proof}
We shall define the marking $\nu$ as follows: $\nu(p)=+1$ if either $\varphi$
has a maximum in $p$, or an absolute minimum in $p$, and $\nu(p)=-1$
otherwise. We shall show that for the ribbon $a=(\varphi,\nu)$ we have
$\gamma(a)=0$. (Note that $a$ has signature $\sigma=2$, which is necessary for
$\gamma=0$.) It is easy to see that one can find smooth function
$\psi:\mathbb{S}^{1}\rightarrow\mathbb{R}$ such that for the map
$F=(\varphi,\psi):\mathbb{S}^{1}\rightarrow\mathbb{R}^{2}$ the image
$F(\mathbb{S}^{1})$ is a Jordan curve in $\mathbb{R}^{2}$. Indeed, suppose
that $\varphi(p_{0})=\min\varphi$ and take a small arc $(x,y)\subset
\mathbb{S}^{1}$ around $p_{0}$. Now, it suffices to define $\psi$ as strictly
increasing in the ``big'' arc $(y,x)$ and strictly decreasing in the small arc
$(x,y)$; then it is evident that $F(\mathbb{S}^{1})$ is a simple closed curve
(\hyperref[f34]{Fig.~\ref*{f34}}). By Jordan-Schoenflies Theorem, there is an
extension of $F$, $\tilde{F}:\mathbb{B}^{2}\rightarrow\mathbb{R}^{2}$
($\tilde{F}|_{\mathbb{S}^{1}}=F$), which is a homeomorphism. Now let
$p_{1}:\mathbb{R}^{2}\rightarrow\mathbb{R}$ be the projection onto the first
factor. Clearly, $p_{1}\tilde{F}\in\mathcal{F}(a)$ is a critical points free
extension of $a$, thus $\gamma(a)=0$.
\end{proof}

Note that the ribbon $a$ satisfies some stronger condition except simply
$\gamma(a)=0$: it has a \textit{representation} by a Jordan curve. It turns
out that not any ribbon with $\gamma=0$ has such a representation. In this
sense, ribbons like $a$ may be regarded as the most simple ones. We shall call
these \textit{Jordan ribbons}.\label{jordan} It is not obvious at all when a
ribbon with $\gamma=0$ is a Jordan one. There exist some, more or less obvious
obstacles for a ribbon with $\gamma=0$ (thus $\sigma=2$) to be a Jordan one.
We shall discuss these in Part II in more detail.

\begin{center}
\begin{figure}[ptb]
\includegraphics[width=60mm]{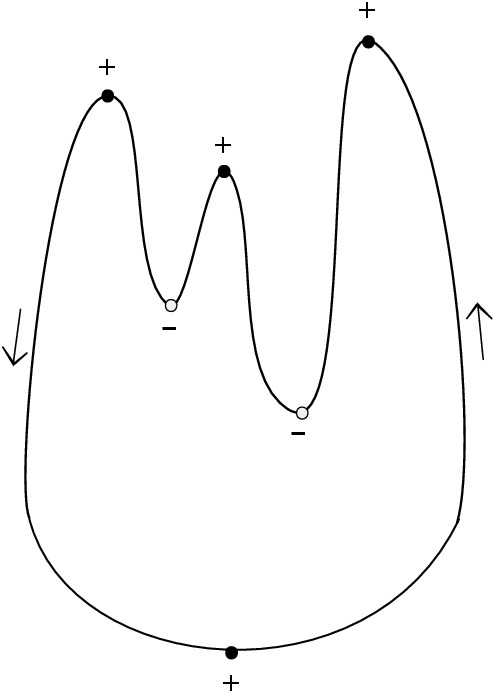}\caption{The ribbon $a=(1^{+}%
6^{+}2^{-}4^{+}3^{-}5^{+})$ coded by a Jordan curve.}%
\label{f34}%
\end{figure}
\end{center}

\textbf{Question.} Describe the class of Jordan ribbons.

There is a relationship between the ribbon invariant and the theory of
immersed curves in the plane, this will be explained in Part~II of the
article. The construction in the above proposition is a quite simple example
of this connection.

\begin{corollary}
Any general position smooth function $\varphi:\mathbb{S}^{1}\rightarrow
\mathbb{R}$ has a critical points free extension $\Phi:\mathbb{B}%
^{2}\rightarrow\mathbb{R}$.
\end{corollary}

Later we shall investigate the case of non general position (\hyperref[s16]%
{Section~\ref*{s16}}).

Now we shall improve \hyperref[p13]{Proposition~\ref*{p13}} in the following way.

\begin{lemma}
\label{l8}Any ribbon $a$ may be transformed into a ribbon $a^{\prime}$ with
$\gamma(a^{\prime})=~0$ by $\leq\frac{n}{2}+1$ critical moves
$(+)\leftrightarrow(-)$.
\end{lemma}

\begin{proof}
We shall proceed by induction on the number of nodes $n$. For $n=2$ it is
obvious: any ribbon transforms into $(1^{+},2^{+})$ by $\leq2$ critical moves.
Suppose the proposition is true for $n-2$ and $a$ has $n$ nodes. It is not
difficult to see that by $\leq1$ critical move we may provide a
\textit{cancellable} pair of nodes. Recall that a pair of a negative node $p$
and a positive node $q$ is cancellable, if $p$, $q$ are consecutive nodes and
taking the other neighbour $r$ of $q$ (different from $p$) we have
$|\varphi(p)-\varphi(q)|<|\varphi(q)-\varphi(r)|$. Indeed, for $n\geq4$ a
tetrad of consecutive nodes $s,p,q,r$ always exists such that $|\varphi
(s)-\varphi(p)|>|\varphi(p)-\varphi(q)|<|\varphi(q)-\varphi(r)|$. Now, if $p$
and $q$ have opposite marking, $(p,q)$ is clearly a cancellable pair. In case
$p$ and $q$ have the same marking, then changing the marking of, say $p$, to
the opposite one by a critical move, we get a cancellable pair again. Let
$a_{0}$ be the ribbon obtained from $a$ by cancellation. Then, according to
the induction hypothesis, $a_{0}$ may be transformed into a ribbon
$a_{0}^{\prime}$ with $\gamma(a_{0}^{\prime})=~0$ by $\leq\frac{n-2}%
{2}+1=\frac{n}{2}$ critical moves $(+)\leftrightarrow(-)$. Now let $a^{\prime
}$ be obtained from $a_{0}^{\prime}$ by the \textit{inverse cancellation},
actually an \textit{extension}. By \hyperref[l7]{Lemma~\ref*{l7}}, we have%
\[
\gamma(a^{\prime})\leq\gamma(a_{0}^{\prime})=~0\text{,}%
\]
hence $\gamma(a^{\prime})=0$. In such a way, $a^{\prime}$ is obtained from $a$
by $\leq\frac{n}{2}+1$ critical moves and the lemma is proved.
\end{proof}

Now we are a step away from the proof of the desired inequality.

\begin{theorem}
\label{t10}For any ribbon $a\in\mathcal{A}$ we have $\gamma(a)\leq\frac{n}%
{2}+1$.
\end{theorem}

\begin{proof}
This follows immediately from \hyperref[l8]{Lemma~\ref*{l8}} and the fact that
$\gamma$ makes a jump of $\varepsilon=0,\pm1$ during a critical move
(\hyperref[t9]{Theorem~\ref*{t9}}, i)).
\end{proof}

As we have the inequality $0\leq\gamma\leq\frac{n}{2}+1$, it is natural to ask
which values of $\gamma$ in the range $[0,\frac{n}{2}+1]$ are realizable by a
ribbon. The affirmative answer, in some strong sense, is given below.

\begin{proposition}
\label{p14}Let $\varphi:\mathbb{S}^{1}\rightarrow\mathbb{R}$ be a general
position smooth function. Then for any $i\in\lbrack0,\frac{n}{2}+1]$ there is
a marking $\nu$ of the critical set, such that for the ribbon $a=(\varphi
,\nu)$ we have $\gamma(a)=i$.
\end{proposition}

\begin{proof}
By \hyperref[p13]{Proposition~\ref*{p13}} there is a marking $\nu
:P\rightarrow\{-1,+1\}$ of the critical set $P$, such that for the ribbon
$a_{0}=(\varphi,\nu)$ we have $\gamma(a_{0})=0$. Since $\sigma(a_{0})=2$, the
ribbon $a_{0}$ has $s_{-}=\frac{n}{2}-1$ negative nodes. Let us perform now
consecutively critical moves of type $(+)\rightarrow(-)$, obtaining in such a
way a series of ribbons $a_{0},a_{1},\dots,a_{n/2}$. It is clear that
$a_{n/2}\in\mathcal{A}_{-}$, hence $\gamma(a_{n/2})=\frac{n}{2}+1$. But since
the number of the ribbons $a_{i}$ is $\frac{n}{2}+1$ and the jump of $\gamma$
at a critical move is $\varepsilon=0,\pm1$, it follows that actually all jumps
equal $\varepsilon=+1$ and therefore $\gamma(a_{i})=i$.
\end{proof}

Note that this result, combined with the construction from \hyperref[p13]%
{Proposition~\ref*{p13}}, provides us a method for effective geometric
construction of ribbons with arbitrary ribbon invariant. All these ribbons
have negative nodes. It is not difficult to construct such ribbons from class
$\mathcal{A}^{+}$. Another straightforward corollary from \hyperref[p14]%
{Proposition~\ref*{p14}} is the fact that the quantities $\frac{\gamma
(a)}{n(a)}$ cover all the rationals in $\mathbb{Q\cap\lbrack}0,1/2\mathbb{]}$,
when $a\in\mathcal{A}$ ($n(a)$ stands for the number of nodes of ribbon $a$).
In fact, one has

\begin{proposition}
\label{p15}$\left\{  \frac{\gamma(a)}{n(a)}|~a\in\mathcal{A}\right\}  =\left(
\mathbb{Q\cap\lbrack}0,1/2\mathbb{]}\right)  \cup\left\{  \frac{1}{2}+\frac
{1}{n},\text{ }n\text{ even}\right\}  $.
\end{proposition}

The numbers $\frac{1}{2}+\frac{1}{n}$ come from the extremal ribbons with
$\gamma=\frac{n}{2}+1$. Such are for example all ribbons from class
$\mathcal{A}^{-}$, although there are some others with this property. This
raises the question about the distribution of the quantity $\frac{\gamma}{n}$
in $\mathbb{[}0,1/2\mathbb{]}$. More precisely, let $\mathcal{A=\{}%
a_{i}\mathcal{\}}$ be ordered by the lexicographic order (\hyperref[s4]%
{Section~\ref*{s4}}).

\textbf{Question.} What is the distribution of the sequence $\left\{
\frac{\gamma(a_{i})}{n(a_{i})}\right\}  $ in $\mathbb{[}0,1/2\mathbb{]}$? For
example, is it uniformly distributed or it has some other peculiar behaviour?

Suppose now that the signature is fixed: $\sigma=\sigma_{0}$, then one may ask
again about the distribution of $\frac{\gamma}{n}$.

\begin{proposition}
\label{p16}Consider the class $\mathcal{A(\sigma}_{0}\mathcal{)}$ of ribbons
with signature $\sigma_{0}$. Then we have%
\[
\overline{\lim}\left\{  \frac{\gamma(a)}{n(a)}|~a\in\mathcal{A(\sigma}%
_{0}\mathcal{)}\right\}  =\frac{1}{2}.
\]

\end{proposition}

\begin{proof}
We construct later in \hyperref[s18]{Section~\ref*{s18}} a ribbon
$a\in\mathcal{A}_{n}$ with $\sigma(a)=2$ and $\gamma(a)=\frac{n}{2}$
(\hyperref[e2]{Example~\ref*{e2}}). Then by $\left\vert \sigma_{0}\right\vert
\pm2$ critical elementary moves of type $(+)\leftrightarrow(-)$ we may
transform $a$ into some ribbon $b$ with $\sigma(b)=\sigma_{0}$. Now, one has
$\gamma(b)\geq\gamma(a)-\left\vert \sigma_{0}\right\vert -2=\frac{n}%
{2}-\left\vert \sigma_{0}\right\vert -2$, thus $\frac{\gamma(b)}{n}\geq
\frac{1}{2}-\frac{\left\vert \sigma_{0}\right\vert }{n}-\frac{2}{n}$, which
implies that the above limit superior equals $\frac{1}{2}$.
\end{proof}

Note that the set $\left\{  \frac{\gamma(a)}{n(a)}|~a\in\mathcal{A(\sigma}%
_{0}\mathcal{)}\right\}  $ is surely not dense in $\mathbb{[}0,1/2\mathbb{]}$
for $\sigma_{0}<0$, in view of the basic inequality $\gamma\geq1-\frac{\sigma
}{2}$.

Another issue should be to consider the distribution of $\frac{\gamma}{n}$ in
class $\mathcal{A}^{+}$. It seems likely that every rational number from
$[0,1/2)$ is a value of $\frac{\gamma}{n}$ for some positive ribbon from
$\mathcal{A}^{+}$.

Of course, many questions of that kind may be asked, for example: What is the
probability while picking a ribbon with $\sigma=2$ to get a ribbon with
$\gamma=0$, or about the probability while picking a ribbon with $\gamma=0$ to
get a \textit{Jordan} one, etc.

Let us note finally that the \textit{dual} statement to \hyperref[p14]%
{Proposition~\ref*{p14}} does not hold true in general. This situation seems
more intriguing and complicated.

\textbf{Question.} Let $P\subset\mathbb{S}^{1}$ be a set of $n$ points
(\textit{nodes}), $n$ is even, and $\nu:P\rightarrow\left\{  +,-\right\}  $ be
a \textit{marking} with signature $\sigma$. Then, under what conditions is it
true that for a given number $k$, satisfying the general inequality
$1-\frac{\sigma}{2}\leq k\leq n-1-\frac{\sigma}{2}$, there is a smooth
function $\varphi:\mathbb{S}^{1}\rightarrow\mathbb{R}$ with node set $P$, such
that for the ribbon $a=(\varphi,\nu)$ we have $\gamma(a)=k$?

Note that the answer in general is \textquotedblleft no\textquotedblright, as
we pointed out in \hyperref[p9]{Proposition~\ref*{p9}} that some
\textit{admissible} pairs $(\sigma,k)$ are not realized by any ribbon $a$.
Note also that this is a more detailed variant of this question in
\hyperref[s9]{Section~\ref*{s9}}.

The general solution is not known to us, anyway, in some particular cases we
have a positive answer to the above question. For example, if $\sigma=2$ and
$k=0$, this may be done as follows. First, the equality $\sigma=2$ allows one
to find a \textit{topological} solution to the problem, i.e. a level lines
portrait without critical points, which agrees with the marking $\nu$. Then it
is easy to find some $\varphi:\mathbb{S}^{1}\rightarrow\mathbb{R}$ with this
level lines portrait, therefore for the ribbon $a=(\varphi,\nu)$ we have
$\gamma(a)=0$.

\section{\label{s15}Truncated $C^{1}$-ribbons}

Let $a=(\varphi,\nu)\in\mathcal{A}$ be a ribbon. Then the function $\varphi$
may be treated as the $C^{0}$-part of the ribbon, while the mark function
$\nu$ should be its $C^{1}$-part. It is natural to ask what happens if we
neglect one of the two parts and look what can be said about the critical
points set in the corresponding case. It turns out that separating $C^{0}$
from $C^{1}$ data is not a very good idea, as the problem is losing geometric
flavor and difficulty. Anyway, some results may be obtained in these cases
that seem to have interesting applications. So,\ we deal with such problems in
the present and the next sections, resolving them, more or less, completely.
At least, no algorithm is needed for solution of the problem, unlike the
general case.

The $C^{1}$-part of a ribbon $a$ is simply a cyclic permutation of two symbols
- $(+)$ and $(-)$. In fact, all the smooth information is carried by the
signature $\sigma=s_{+}-s_{-}$. Curiously, in some cases a satisfactory
estimate of the critical set of any extension of the ribbon $a$ exists.
Actually, we have already obtained such an estimate in \hyperref[p7]%
{Proposition~\ref*{p7}}, where the inequality
\[
\gamma\geq\gamma_{\mathtt{\operatorname{ext}}}\geq1-\frac{\sigma}{2}%
\]

for any ribbon is proved. Therefore, we may take for granted the existence of
$\geq1-\frac{\sigma}{2}$ different local extrema of any extension of $a$.\ Of
course, this makes sense only for $\sigma\leq0$. Now we shall use this simple
inequality to obtain an estimate from below of the number of critical points
of a function on the 2-sphere.

\begin{theorem}
\label{t11}Let $f:\mathbb{S}^{2}\rightarrow\mathbb{R}$ be a smooth function
and $\lambda\subset\mathbb{S}^{2}$ be a general position smooth simple closed
curve. Suppose that $\nabla f|_{\lambda}\neq0$ and let the level lines of $f$
are touching inwards $\lambda$ in $s_{+}$ points, and are respectively
touching outwards $\lambda$ in $s_{-}$ points (no matter the orientation of
$\lambda$). Set $\sigma=s_{+}-s_{-}$ and suppose $\sigma\neq\pm2$. Then for
the number $\Gamma_{f}$ of critical points of $f$, the following inequality
holds%
\[
\Gamma_{f}\geq2+\frac{|\sigma|}{2}\text{.}%
\]

Moreover, $f$ has at least $1+\frac{|\sigma|}{2}$ local extrema, situated from
one and the same side of $\lambda$.

If in addition $f$ is supposed to be a Morse function, then we have a sharper
estimate:%
\[
\Gamma_{f}\geq|\sigma|\text{.}%
\]

\end{theorem}

\begin{proof}
Note first that the inequalities do not depend on the orientation of $\lambda
$, as we have $|\sigma|$\ in the right-hand side. Considering now a small
nieghbourhood of $\lambda$ we get some ribbon $a_{\lambda}=(\varphi,\nu)$,
where $\varphi=f|_{\lambda}$ and the marking $\nu$ is assigning $+1$ at a
point of inward touching between $\lambda$ and a level line of $f$, and $-1$
elsewhere. Note that $\sigma(a_{\lambda})=\pm\sigma$, depending on the chosen
orientation of $\lambda$. By Jordan's Theorem $\lambda$ is separating
$\mathbb{S}^{2}$ into two regions $S_{1}$ and $S_{2}$ and clearly, $f|_{S_{1}%
}$ is an extension of, say $a_{\lambda}$, and then $f|_{S_{2}}$ is an
extension of the ribbon $\overline{a_{\lambda}}$. Now, taking the sum of
critical points of both extensions and referring to \hyperref[p2]%
{Proposition~\ref*{p2}}, we have%
\[
\Gamma_{f}\geq\gamma(a_{\lambda})+\gamma(\overline{a_{\lambda}})\geq
2+\frac{|\sigma|}{2}\text{.}%
\]
As following from \hyperref[p7]{Proposition~\ref*{p7}}, $\gamma
_{\mathtt{\operatorname{ext}}}(a_{\lambda})\geq1+\frac{|\sigma|}{2}$, or
$\gamma_{\mathtt{\operatorname{ext}}}(\overline{a_{\lambda}})\geq
1+\frac{|\sigma|}{2}$ is fulfilled, whence the remark about the local extrema.
In case $f$ is a Morse function, we have to refer to \hyperref[p5]%
{Proposition~\ref*{p5}}:%
\[
\Gamma_{f}\geq\gamma_{0}(a_{\lambda})+\gamma_{0}(\overline{a_{\lambda}}%
)\geq|\sigma|.
\]
It would be interesting to find similar estimates for the other closed 2-manifolds.
\end{proof}

For $\sigma=0,\pm2$ some minor specifications can be made.

a) If $\sigma=\pm2$, then $f$ has two local extrema of different type
($\mathrm{min}$ and $\mathrm{max}$), situated from one and the same side of
$\lambda$.

b) If $\sigma=0$, then $f$ has two local extrema of different type, situated
from different sides of $\lambda$.

(See \hyperref[tr1]{Fig.~\ref*{tr1}} for visualization.)

Clearly, the possible applications of the above theorem depend on the suitable
choice of the curve $\lambda$. For example, if $\lambda\subset\mathbb{S}^{2}$
is a \textquotedblleft small\textquotedblright\ generic curve, then
$\sigma=\pm2$ and we get only the trivial estimate $\Gamma_{f}\geq2$. So, it
is a good idea to look for some more \textquotedblleft
globally\textquotedblright\ situated curve $\lambda$ providing us with a
bigger number of critical points.

At \hyperref[tr1]{Fig.~\ref*{tr1}} and \hyperref[tr2]{Fig.~\ref*{tr2}} we
illustrate \hyperref[t11]{Theorem~\ref*{t11}} for the case of a Morse, and non
Morse function $f$, correspondingly, where $f$ is a projection of an embedded
copy of $\mathbb{S}^{2}$ on a line. Furthermore, at \hyperref[tr12]%
{Fig.~\ref*{tr12}} and \hyperref[tr22]{Fig.~\ref*{tr22}} we give the
corresponding level portraits. The touching lines of curve $\lambda$ with the
level lines, counting $\sigma$, are red colored. Note that there is yet
another extremum situated \textquotedblleft from behind\textquotedblright,
corresponding to the global minimum of $f$. Observe that in the second example
(\hyperref[tr2]{Fig.~\ref*{tr2}}, \hyperref[tr22]{Fig.~\ref*{tr22}}) the curve
$\lambda$ does not provide us with the actual number of critical points
$\Gamma_{f}=7$.

\begin{center}
\begin{figure}[ptb]
\includegraphics[width=110mm]{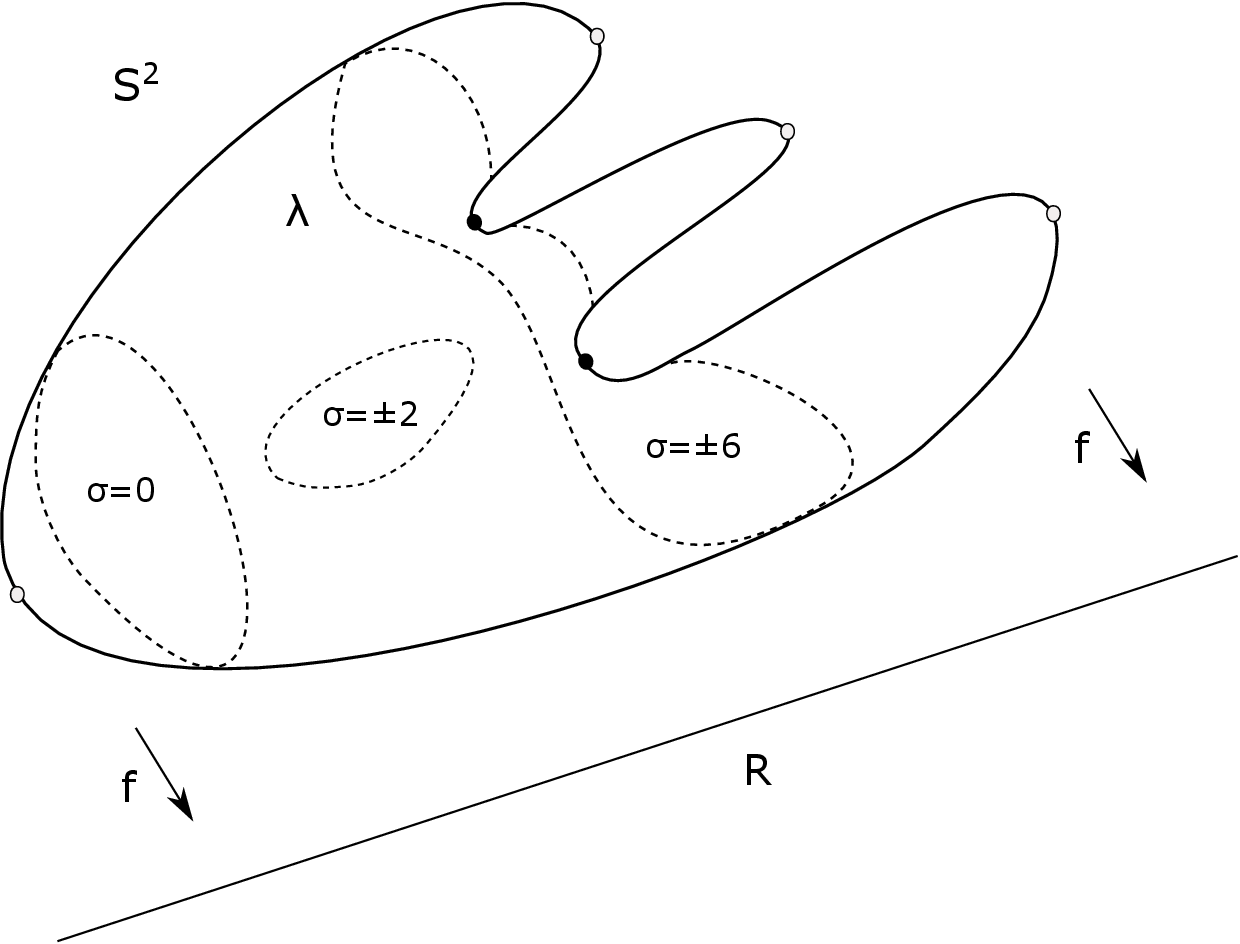}\caption{$\sigma=\pm6$, $\Gamma
_{f}\geq|\sigma|=6$, Morse $f$. }%
\label{tr1}%
\end{figure}

\begin{figure}[ptb]
\includegraphics[width=100mm]{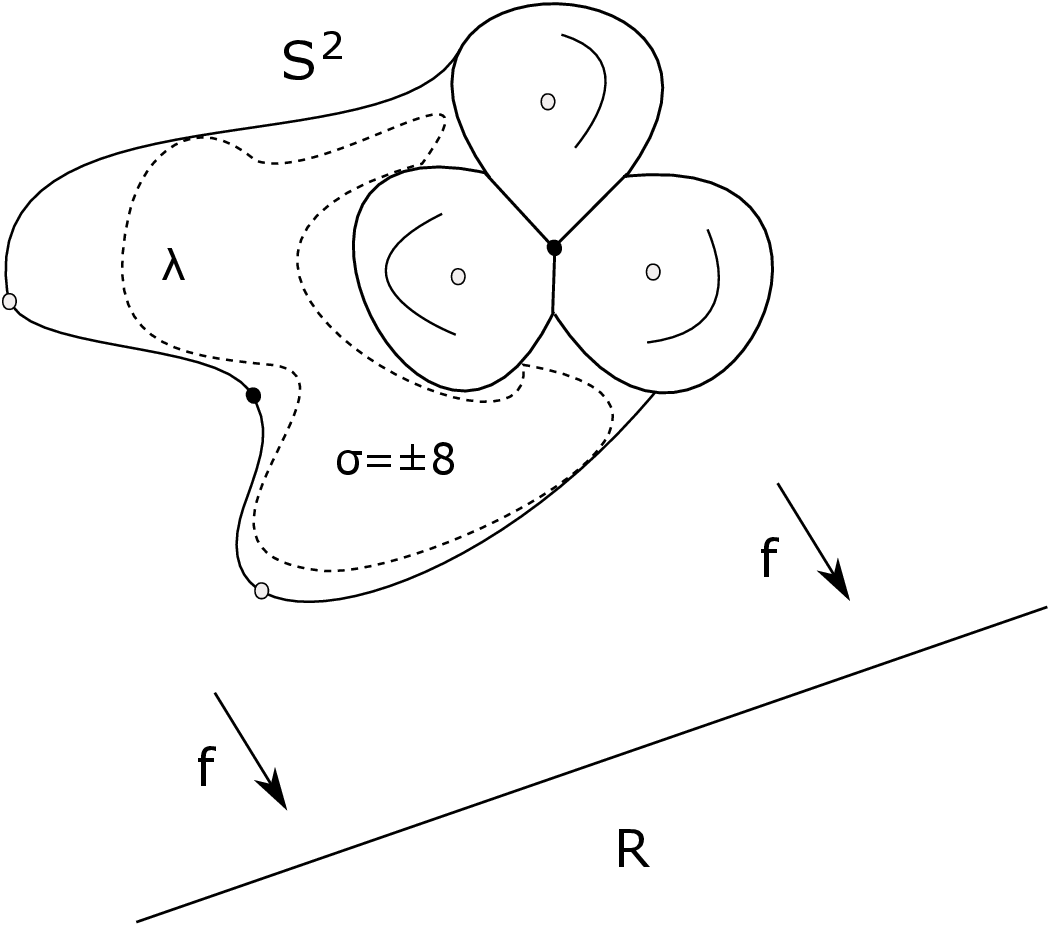}\caption{$\sigma=\pm8$, $\Gamma
_{f}\geq2+\frac{|\sigma|}{2}=6$, but $\Gamma_{f}=7$, non Morse $f$.}%
\label{tr2}%
\end{figure}

\begin{figure}[ptb]
\includegraphics[width=90mm]{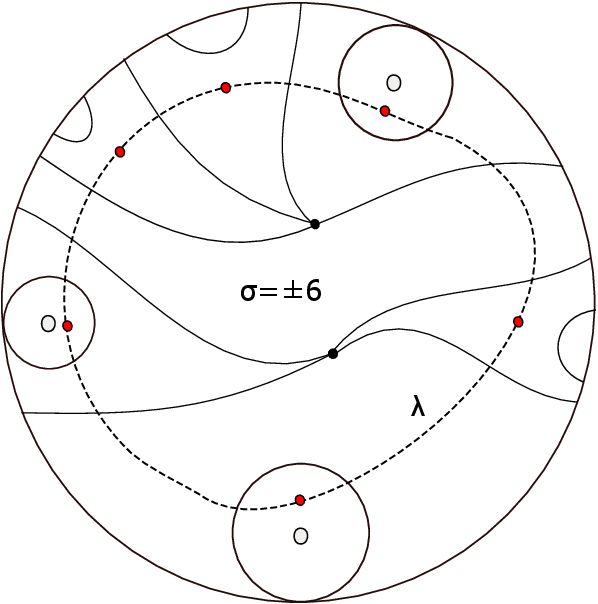}\caption{A level picture of
\hyperref[tr1]{Fig.~\ref*{tr1}}, Morse $f$.}%
\label{tr12}%
\end{figure}

\begin{figure}[ptb]
\includegraphics[width=90mm]{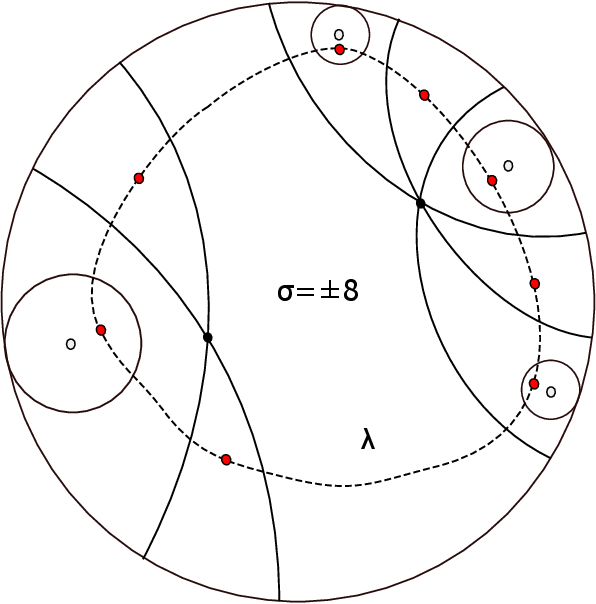}\caption{A level picture of
\hyperref[tr2]{Fig.~\ref*{tr2}}, non Morse $f$.}%
\label{tr22}%
\end{figure}
\end{center}

Note that in the proof of \hyperref[t11]{Theorem~\ref*{t11}} we don't refer to
the $C^{0}$-part of ribbon $a$, as we get some estimate of the critical set
only from the signature $\sigma$. Yet from the proof it follows that we may
get a stronger estimate of $\Gamma_{f}$, since $\gamma(a_{\lambda}%
)+\gamma(\overline{a_{\lambda}})$ may be quite bigger than $2+\frac{|\sigma
|}{2}$, due to the $C^{0}$-part of ribbons $a$ and $\overline{a}$. It is not
difficult to find $f$ and $\lambda$ such that $\sigma=0$, while $\gamma
(a_{\lambda})+\gamma(\overline{a_{\lambda}})$ is arbitrarily large. So, it
seems appropriate to formulate the corresponding estimates in full strength.

\begin{definition}
\label{d20}Let $f:\mathbb{S}^{2}\rightarrow\mathbb{R}$ be a smooth function
and $\lambda\subset\mathbb{S}^{2}$ be a general position smooth simple closed
curve such that $\nabla f|_{\lambda}\neq0$. Let $a_{\lambda}$ denote the
ribbon associated with $\lambda$ defined (up to orientation) in \hyperref[t11]%
{Theorem~\ref*{t11}}. Consider the following two numbers

1) the critical index:%
\[
\mu(f)=\max\{\gamma(a_{\lambda})+\gamma(\overline{a_{\lambda}})|~\lambda
\subset\mathbb{S}^{2}\},
\]
where $\lambda$ runs over all curves in $\mathbb{S}^{2}$ with the above property

2) the Morse critical index:%
\[
\mu_{0}(f)=\max\{\gamma_{0}(a_{\lambda})+\gamma_{0}(\overline{a_{\lambda}%
})|~\lambda\subset\mathbb{S}^{2}\}
\]
where $\lambda$ runs in the same class.
\end{definition}

Then, in \hyperref[t11]{Theorem~\ref*{t11}} we have proved in fact that for
arbitrary $f$%
\[
\Gamma_{f}\geq\mu(f),
\]

while for a Morse function $f$%
\[
\Gamma_{f}\geq\mu_{0}(f).
\]

Note that for arbitrary $f$ it may happen $\Gamma_{f}>\mu(f)$. For example, if
$f$ has 3 critical points, one of which has index 0 (removable singularity),
then $\Gamma_{f}=3$, while it is easy to see that $\mu(f)=2$.

On the other hand, for a Morse function $f$ such examples with a removable
critical point do not exist. So, let us ask the following

\textbf{Question.} If $f:\mathbb{S}^{2}\rightarrow\mathbb{R}$ is a Morse
function, is it true that $\Gamma_{f}=\mu_{0}(f)$?

From point of view of multidimensional ribbons, that we don't investigate
here, but anyway mentioned in the introduction, it is clear that we may look
for some ribbon-based method strengthening the Lusternik-Schnirelmann and
Morse inequalities. In the case of $\mathbb{S}^{2}$ these inequalities give
the trivial estimate 2 for the number of critical points. It is clear that for
an arbitrary ``random'' function $f$ the number of critical points $\Gamma
_{f}$ is surely greater than 2 and the critical index $\mu(f)$ is also surely
greater than 2. So, there is a reason to look for $\mu(f)$ in order to obtain
a valuable estimate for $\Gamma_{f}$.

Let us formulate, anyway, some result for the multidimensional case, as it is
straightforward and needs only the right ribbon definitions.

Let $M$ be a closed $n$-manifold and $f:M\rightarrow\mathbb{R}$ be a smooth
function. Denote by $b(M)$ the sum of the Betti numbers of $M$, and by
$\mathrm{cat}(M)$ the Lusternik-Schnirelmann category of $M$. Then for the
number $\Gamma_{f}$ of critical points of $f$, the following inequalities hold:%

\[
\Gamma_{f}\geq\mu(f)\geq\mathrm{cat}(M),
\]

and in the case of a Morse function $f$%
\[
\Gamma_{f}\geq\mu_{0}(f)\geq b(M).
\]

The left inequalities are proved above, while $\mu(f)\geq\mathrm{cat}(M)$ and
$\mu_{0}(f)\geq b(M)$ follow from the fact that both quantities $\gamma
(a_{\lambda})+\gamma(\overline{a_{\lambda}})$ and $\gamma_{0}(a_{\lambda
})+\gamma_{0}(\overline{a_{\lambda}})$ are realizable as the number of
critical points of some $f$, which in the second case is Morse, and of course,
from the fundamental inequalities in Lusternik-Schnirelmann and Morse
theories. The multidimensional generalizations of the ribbon invariant will be
discussed later, although almost nothing is done in the present article in
this direction. Note at this stage, that we have to be careful with the
definition of a multidimensional ribbon and $\gamma$, as these have to be
elaborated for the case of a manifold with boundary $(M,\partial M)$, instead
of $(\mathbb{B}^{n},\mathbb{S}^{n-1})$.

\section{\label{s16}Truncated $C^{0}$-ribbons}

The most simple situation of a Rolle type problem in dimension 2 arises when
considering a smooth function $\varphi:\mathbb{S}^{1}\rightarrow\mathbb{R}$
and asking whether $\varphi$ has a critical points free extension on the
2-disk, or not. We may say that $\varphi$ is a truncated $C^{0}$-ribbon, as we
are ``forgetting'' the $C^{1}$-information of some ribbon $a=(\varphi,\nu)$.
Note that for a general position function $\varphi$ (with different critical
levels), a critical points free extension always exists, according to
\hyperref[p13]{Proposition~\ref*{p13}}. So, this question is not very
interesting for a generic $\varphi$.

In this section, we shall anyway prove some result about non generic functions
$\varphi$, where coincidence of critical levels is allowed. Roughly speaking,
this result says that if $\varphi$ has sufficiently many absolute extrema
(absolute majority), then any smooth extension of $\varphi$ has a critical point.

Let us say that $\varphi:\mathbb{S}^{1}\rightarrow\mathbb{R}$ is a
\textit{Rolle function}, if any its extension $f:\mathbb{B}^{2}\rightarrow
\mathbb{R}$ has a critical point. In the ribbon terminology, $\varphi$ is a
Rolle function, if for any marking $\nu$ of the node set the ribbon
$a=(\varphi,\nu)$ has nonzero ribbon invariant: $\gamma(a)>0$. Here we shall
allow non general position ribbons with coinciding critical levels.

\begin{theorem}
\label{t12}Let $\varphi:\mathbb{S}^{1}\rightarrow\mathbb{R}$ be a smooth
function with $n$ local and $s$ absolute extrema. Then $\varphi$ is a Rolle
function if and only if%
\[
s>\frac{n}{2}+1\text{.}%
\]

\end{theorem}

\begin{proof}
Let $s>\frac{n}{2}+1$ for $\varphi:\mathbb{S}^{1}\rightarrow\mathbb{R}$.
Suppose that $\varphi$ is not a Rolle function, i.e. it has an extension
$f:\mathbb{B}^{2}\rightarrow\mathbb{R}$ without critical points. Then at each
absolute extremum of $\varphi$ the marking of the corresponding ribbon is
positive, since otherwise $f$ should have an absolute extremum in an internal
point of $\mathbb{B}^{2}$, and thus a critical point. Therefore for the ribbon
corresponding to $f$ we have $s_{+}>\frac{n}{2}+1$, hence $\sigma=s_{+}%
-s_{-}\geq4$. But we know that the equality $\sigma=2$ is necessary for the
existence of a critical points free extension (even for non general position
ribbons), which is a contradiction. Let now $\varphi:\mathbb{S}^{1}%
\rightarrow\mathbb{R}$ be a function satisfying $s\leq\frac{n}{2}+1$. We have
to show that $\varphi$ has an extension $f:\mathbb{B}^{2}\rightarrow
\mathbb{R}$ without critical points. But this is not that difficult to be done
by induction on $n$ by cancellation of a pair of nodes. For $n=2$ the
statement is trivially true. Suppose it is true for $n-2$. Take 3 consecutive
nodes $p$, $q$ and $r$ such that $\varphi$ has a non absolute extremum in $p$
and an absolute one in $q$. There are 2 cases (see \hyperref[f35]%
{Fig.~\ref*{f35}}): 1) $|\varphi(p)-\varphi(q)|\leq|\varphi(q)-\varphi(r)|$.
Then $p$ is not a node of absolute extremum and we may do cancellation of $p$
and $q$ (see \hyperref[d19]{Definition~\ref*{d19}}), obtaining in such a way a
new function $\varphi^{\prime}$. 2) $|\varphi(p)-\varphi(q)|\geq
|\varphi(q)-\varphi(r)|$. Then $r$ is not a node of absolute extremum and we
may perform cancellation of $q$ and $r$, obtaining again some $\varphi
^{\prime}$. Now, in both cases $\varphi^{\prime}$ is satisfying the assumed
inequality. Indeed, $\varphi^{\prime}$ has $n-2$ nodes and $s-1$ absolute
extrema, hence $s-1\leq\frac{n-2}{2}+1$, as the latter is equivalent to
$s\leq\frac{n}{2}+1$. Therefore, $\varphi^{\prime}$ has a critical points free
extension $f^{\prime}:\mathbb{B}^{2}\rightarrow\mathbb{R}$. Let $a^{\prime
}=(\varphi^{\prime},\nu^{\prime})$ be the ribbon defined by $f^{\prime}$. Now,
define the ribbon $a=(\varphi,\nu)$ so that the marking $\nu$ coincides with
$\nu^{\prime}$ in the common nodes, the canceled absolute extremum is made
``positive'' and the other canceled node - ``negative''.\ Since $a^{\prime}$
is obtained from $a$ by cancellation, we have $\gamma(a)\leq\gamma(a^{\prime
})$ (\hyperref[l7]{Lemma~\ref*{l7}}). But as $\gamma(a^{\prime})=0$, it
follows that $\gamma(a)=0$, i.e. $\varphi$ has a critical point free extension
on $\mathbb{B}^{2}$.
\end{proof}

\begin{center}
\begin{figure}[ptb]
\includegraphics[width=100mm]{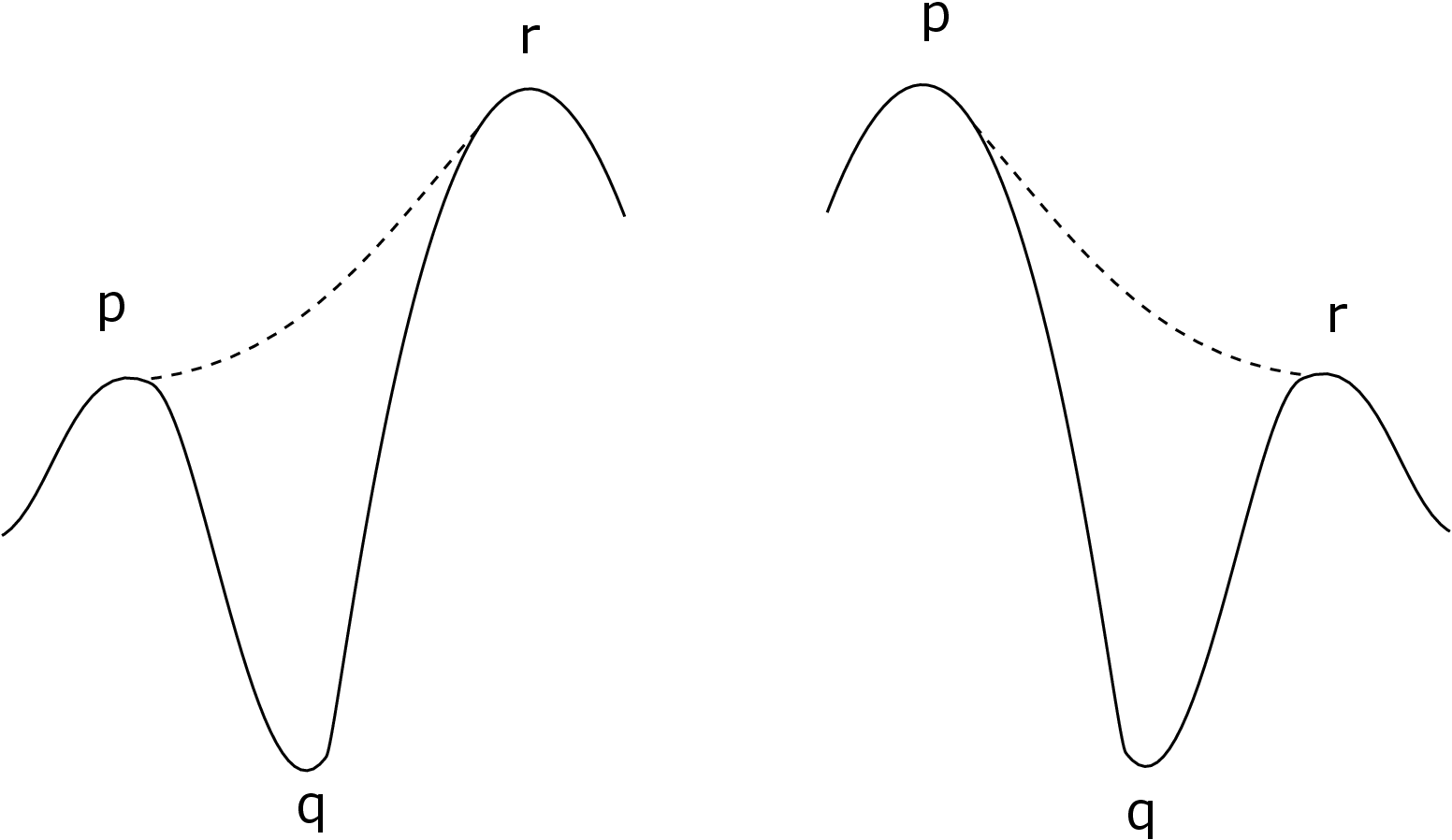}\caption{Two types of
cancellation.}%
\label{f35}%
\end{figure}
\end{center}

It would be interesting to describe the (combinatorially/topologically)
different critical points free extensions of a function satisfying $s\leq
\frac{n}{2}+1$. Even in the case of a general position function (then $s=2$),
this may be of some interest. In particular, one may ask about some estimates
of the number of critical points free extensions of such functions. For
example, in the case of a ladder with $n$ nodes, this number equals exactly
$2^{\frac{n}{2}-1}$, as follows easily from \hyperref[p11]%
{Proposition~\ref*{p11}}.

There is a geometrical application of \hyperref[t12]{Theorem~\ref*{t12}} about
the realization of a given direction by a membrane in $\mathbb{R}^{3}$ spanned
by a fixed frame. The setting of the problem is the following.

Let $l$ be an immersed smooth closed curve in $\mathbb{R}^{3}$ and $\xi$ be
some fixed direction. Then the problem is:

Under what conditions is it true that any smooth immersed 2-disk $M$ with
boundary $\partial M=l$ is realizing direction $\xi$, i.e., some normal to $M$
is parallel to $\xi$?

Note that the answer to this question may be viewed as a 2-dimensional variant
of Lagrange's Theorem in $\mathbb{R}$. In some more geometric terminology, we
may say that $l$ is a frame and $M$ is a membrane spanned by $l$. It is now
clear, that \hyperref[t12]{Theorem~\ref*{t12}} provides us with an easy
sufficient condition for realization of direction $\xi$.

\begin{proposition}
\label{p17}Let $l$ be an immersed smooth closed curve in $\mathbb{R}^{3}$
which is the image of a smooth map $h:\mathbb{S}^{1}\mathbb{\rightarrow R}%
^{3}$\ and $\xi$ be some fixed direction. Take some line $L$ with direction
$\xi$ and let $\pi:\mathbb{R}^{3}\rightarrow L$ be the projection. Suppose
that the composition $\pi h:\mathbb{S}^{1}\mathbb{\rightarrow}L$ has $n$ local
and $s$ absolute extrema, so that $s>\frac{n}{2}+1$. Then for any smooth
immersed 2-disk $M$ with boundary $\partial M=l$ there is a point $x\in M$
where the normal vector is parallel to $\xi$.
\end{proposition}

\begin{proof}
This is a straightforward corollary from \hyperref[t12]{Theorem~\ref*{t12}},
as the set of points $x\in M$ with normal vector parallel to $\xi$ coincides
with the set of critical points of $\pi h$.
\end{proof}

It would be interesting to prove the converse proposition. This could be done
by following the next plan: since $s\leq\frac{n}{2}+1$, by \hyperref[t12]%
{Theorem~\ref*{t12}}\ there is a critical points free extension $f:\mathbb{B}%
^{2}\rightarrow\mathbb{R}$ of $\pi h$. Now, if we can find an immersion
$H:\mathbb{B}^{2}\rightarrow\mathbb{R}^{3}$ such that $H\pi=f$, then
$M=H(\mathbb{B}^{2})$ would be a membrane spanned by $l$, not realizing $\xi$
as a normal direction. Unfortunately, we don't have a proof that such an
immersion $H$ exists.

\textbf{Question.} Does an immersion $H$ with the above properties exists?

It is natural also to ask, given some frame $l\subset\mathbb{R}^{3}$, whether
there exists some vector $\xi$ which is realizable as a normal direction of
any membrane spanned by $l$. This would be a kind of 2-dimensional variant of
Lagrange's Theorem. It turns out that, in general, the answer to this question
is negative. However, in some ``non general position'' cases, such directions
do exist. We give below some example of a frame, such that any membrane
spanned by it realizes the principal directions $\mathbf{i},\mathbf{j}%
,\mathbf{k}$.

For a given membrane $M$, let us denote by $N(M)$ the set of all normal
directions to $M$. Now, if $l\subset\mathbb{R}^{3}$, let us define%
\[
\Xi(l)=\cap\left\{  N(M)|~\partial M=l\right\}  \text{,}%
\]

i.e. $\Xi(l)$ is the set of directions realizable by any membrane spanned by
$l$.

\begin{fact}
$\Xi(l)=\emptyset$ for a generic frame $l$.
\end{fact}

We shall not prove this here, but rather will give an example in the opposite direction.

\begin{example}
There is a frame $l$ with $\Xi(l)\supset\left\{  \mathbf{i},\mathbf{j}%
,\mathbf{k}\right\}  $.
\end{example}

We shall sketch the corresponding example, without going into full technical
details. Let us say for brevity, that a function $\varphi:\mathbb{S}%
^{1}\rightarrow\mathbb{R}$ with $n$ local and $s$ absolute extrema is a
\textit{quasi-alternation}, if $s>\frac{n}{2}+1$. Such a function oscillates
sufficiently often between $\min\varphi$ and $\max\varphi$. Let $\pi
_{i}:\mathbb{R}^{3}\rightarrow\mathbb{R}$ be the $i$-th projection. Now, the
construction of the desired frame $l$ proceeds in 3 stages. First, we
construct some map $\lambda_{1}:\mathbb{[}0,1\mathbb{]\rightarrow R}^{3}$ such
that $\pi_{1}\lambda_{1}$ oscillates sufficiently many times between two
values (say, -1 and 1), while $\pi_{2}\lambda_{1}$ and $\pi_{3}\lambda_{1}$
are monotonic. This can be easily provided. Then we repeat the same procedure,
obtaining maps $\lambda_{2}$ and $\lambda_{3}$. Set $l_{i}=\lambda_{i}([0,1])$
and finally glue together the arcs $l_{i}$ into a frame $l$ by 3 simple
``buffer'' segments. Now, it follows from the construction that if
$h:\mathbb{S}^{1}\mathbb{\rightarrow R}^{3}$\ is a parametrization of $l$,
then $\pi_{i}h$ is a quasi-alternation for $i=1,2,3$. But then from
\hyperref[p17]{Proposition~\ref*{p17}} we conclude that the principal
directions $\mathbf{i},\mathbf{j},\mathbf{k}$ are realized by any membrane
spanned by $l$.

This example gives rise to the following general

\textbf{Question.} For any finite set of vectors $\xi_{1},\dots,\xi_{n}$, is
it true that a frame $l$ exists, such that $\Xi(l)\supset\left\{  \xi
_{1},\dots,\xi_{n}\right\}  $?

As following from \hyperref[t12]{Theorem~\ref*{t12}}, a general position
function $\varphi:\mathbb{S}^{1}\rightarrow\mathbb{R}$ has a critical points
free extension. So, it is natural to ask about the number of all such
extensions (distinguished combinatorially). Here we shall give only an
estimation from below of this number.

It turns out that the critical points free extensions are in one-to-one
correspondence with the extension of the correspondent ribbon from
$\mathcal{A}^{-}$, which is realizing the ribbon invariant.

\begin{proposition}
\label{p18}Let $\varphi:\mathbb{S}^{1}\rightarrow\mathbb{R}$ be a general
position (Morse) function and $a=(\varphi,\nu^{-})\in\mathcal{A}^{-}$ be the
ribbon with all nodes marked as negative. Then the critical points free
extensions of $\varphi$ are in a natural one-to-one correspondence with the
extensions of $a$, realizing $\gamma(a)$. (Recall that $\gamma=\frac{n}{2}+1$
in $\mathcal{A}^{-}$.)
\end{proposition}

It is not difficult to see that for a negative ladder $a$ the number of the
extensions of $a$, realizing $\gamma(a)$ equals $2^{\frac{n}{2}-1}$. Now, by
the critical moves technique and \hyperref[p11]{Proposition~\ref*{p11}} one gets

\begin{theorem}
\label{t13}Let $\varphi:\mathbb{S}^{1}\rightarrow\mathbb{R}$ be a general
position smooth function. Then the number of critical points free extensions
of $\varphi$ is $\geq2^{\frac{n}{2}-1}$.
\end{theorem}

So, there is a natural

\textbf{Question.} For a general position function $\varphi:\mathbb{S}%
^{1}\rightarrow\mathbb{R}$, what is the number of critical points free
extensions of $\varphi$? Is it true that this number is the same for all
\textit{alternations} $\varphi$?

\section{\label{s17}Estimate for the number of components of the critical set}

The ribbon invariant $\gamma(a)$ gives an estimate from below of the number of
critical points of the extensions $f\in\mathcal{F}(a)$ of a given ribbon $a$.
However, it is not clear what can be said about the number of components of
the critical set $\mathrm{Crit}(f)$ of an arbitrary extension $f\in
\mathcal{F}(a)$. For example, our previous investigations around $\gamma$ do
not exclude a situation when $\gamma(a)$ is big, but there is some
$f\in\mathcal{F}(a)$ with connected set of critical points. However, we will
show in the present section that this is impossible, since $\gamma(a)$ is an
estimate from below of the number of components of the critical set
$\mathrm{Crit}(f)$ of any extension $f\in\mathcal{F}(a)$.

\begin{lemma}
\label{l9}Let $a=(\varphi,\nu)$ be a ribbon and $f\in\mathcal{F}(a)$. Suppose
that $p$ is a negative node of $a$. Then there is a ribbon $a^{\prime}\sim a$
and some $f_{0}\in\mathcal{F}(a^{\prime})$ such that $\mathrm{Crit}%
(f_{0})=\mathrm{Crit}(f)$ and the level line of $f_{0}$ through $p$ is a
regular touching line.
\end{lemma}

\begin{proof}
By Sard's Theorem, we may find $x^{\prime},x^{\prime\prime}\in\mathbb{S}^{1}$
close to $p$, situated from different sides of $p$ and such that
$\varphi(x^{\prime})=\varphi(x^{\prime\prime})=\alpha$ is a noncritical level
of $f$. Then, by a small perturbation near $p$, we may define the ribbon
$a^{\prime}=(\varphi^{\prime},\nu)$ in such a way that $\varphi^{\prime
}(p^{\prime})=\alpha$, where $p^{\prime}$ is a negative node. Then $f$
perturbs to some $f_{0}$, according to the picture at \hyperref[f36]%
{Fig.~\ref*{f36}}. Thus, the level line of $f_{0}$ through $p^{\prime}$ is
regular. Now, since $a^{\prime}$ is obtained from $a$ by a small perturbation,
it follows that $a^{\prime}\sim a$. Note that the perturbation does not affect
the critical set $\mathrm{Crit}(f)$. From geometrical point of view this
perturbation corresponds to a small vertical move of the ribbon at node $p$,
so that the new value $\varphi(p)$ becomes noncritical.
\end{proof}

\begin{theorem}
\label{t14}Let $a\in\mathcal{A}$ be a ribbon and $f\in\mathcal{F}(a)$ be some
extension. Then the number of components of the critical set $\mathrm{Crit}%
(f)$ is $\geq\gamma(a)$.
\end{theorem}

\begin{proof}
We shall use induction on the lexicographic ordering in $\mathcal{A}$. For
irreducible ribbons the theorem is obvious. Suppose it is true for ribbons
with $\leq n-2$ nodes and let $a\in\mathcal{A}_{n}$. There are 2 cases. 1) The
ribbon $a$ has a negative node $p$. Then, according to \hyperref[f36]%
{Fig.~\ref*{f36}}, we may suppose that the level line $l$ of $f$ passing
through $p$ is a regular touching line. Perform a ternary splitting
$f=f_{1}\circ f_{2}\circ f_{3}$ along $l$, which induces a splitting
$a=a_{1}\ast a_{2}\ast a_{3}$. Then $\gamma(a_{i})\prec\gamma(a)$, $i=1,2,3$.
The line $l$ is dividing $\mathbb{B}^{2}$ into 3 regions $W_{1},W_{2},W_{3}$.
We may suppose that the number of components of $\mathrm{Crit}(f)$ is finite,
say $k$. Let $k_{i}$ be the number of components of $W_{i}\cap\mathrm{Crit}%
(f)$, then $k_{i}$ is the number of components of $\mathrm{Crit}(f_{i})$ and
$k_{1}+k_{2}+k_{3}=k$. By the induction hypothesis $\gamma(a_{i})\leq k_{i}$.
Thus%
\[
\gamma(a)\leq\gamma(a_{1})+\gamma(a_{2})+\gamma(a_{3})\leq k_{1}+k_{2}%
+k_{3}=k.
\]
2) $a\in\mathcal{A}_{+}$. If the cluster number $\delta(a)$ equals 1, then $a$
is an alternation and the theorem is obvious. Suppose that $\delta(a)\geq2$
and let $C_{1},C_{2}$ be two different clusters. Take some regular level $c$
of $f$ between $C_{1}$ and $C_{2}$ (such a level exists according to Sard's
Theorem). Then $f$ has a regular level line $l$ at level $c$. We proceed now
as in 1). Perform a binary splitting $f=f_{1}\vee f_{2}$ along $l$, which
induces a splitting $a=a_{1}\#a_{2}$. Since $c$ is between $C_{1}$ and $C_{2}%
$, we have $\gamma(a_{i})\prec\gamma(a)$, $i=1,2$. The line $l$ is dividing
the disk into 2 regions $W_{1},W_{2}$. Let $k_{i}$ be the number of components
of $W_{i}\cap\mathrm{Crit}(f)$, then $k_{i}$ is the number of components of
$\mathrm{Crit}(f_{i})$ and $k_{1}+k_{2}=k$ is the number of components of
$\mathrm{Crit}(f)$. Now, since $\gamma(a_{i})\leq k_{i}$ by the induction
hypothesis, we have%
\[
\gamma(a)\leq\gamma(a_{1})+\gamma(a_{2})\leq k_{1}+k_{2}=k.
\]
The theorem is proved.
\end{proof}

\begin{center}
\begin{figure}[ptb]
\includegraphics[width=120mm]{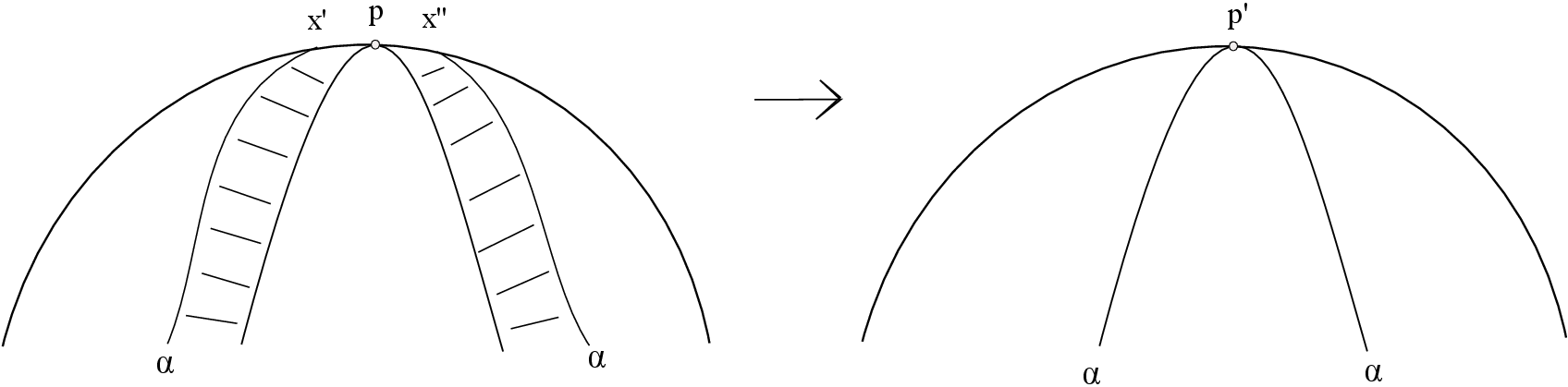}\caption{Regularization at a
negative node.}%
\label{f36}%
\end{figure}
\end{center}

\begin{corollary}
\label{c1}Let $f:\mathbb{R}^{2}\rightarrow\mathbb{R}$ be a smooth function and
$K$ be an isolated component of the critical set $\mathrm{Crit}(f)$. Take a
smooth simple closed curve $\lambda$ such that if $W$ is its interior, then
$\overline{W}\cap\mathrm{Crit}(f)=K$. Let $a_{\lambda}$ denote the ribbon
defined by the restriction of $f$ on $\lambda$. Then
\[
\gamma(a_{\lambda})\leq1\text{.}%
\]

Moreover, $\gamma(a_{\lambda})=0$ if and only if $\sigma(a_{\lambda})=2$
($\sigma$ is the signature).
\end{corollary}

\begin{proof}
This is an immediate corollary from \hyperref[t14]{Theorem~\ref*{t14}}, as
$f|_{\overline{W}}$ is an extension of $a_{\lambda}$ with connected critical
set. Hence, by the theorem $1\geq$ $\gamma(a_{\lambda})$. Note that
$\gamma=0\Rightarrow\sigma=2$, but the converse is not true for arbitrary
ribbons. (Recall that $\sigma=2\Leftrightarrow\deg(\nabla f|_{\lambda})=0$.)
Let $\sigma(a_{\lambda})=2$ and suppose that $\gamma(a_{\lambda})\neq0$. Then
$\gamma(a_{\lambda})=1$, but we know that the case $(\sigma=2,\gamma=1)$ is
impossible (\hyperref[p9]{Proposition~\ref*{p9}}), a contradiction.
\end{proof}

In such a way, we have either $\gamma(a_{\lambda})=0$, or $\gamma(a_{\lambda
})=1$. Of course, both cases are possible in general. However, if
$f|_{\overline{W}}$ is an economic extension of $a_{\lambda}$, then
$\gamma(a_{\lambda})=0$ is impossible, since then $K$ would be a single point
of zero index, but economic extensions don't have such critical points.

Note also that \hyperref[c1]{Corollary~\ref*{c1}} has some geometrical meaning:

\textit{Under the above assumptions, we may transform }$K=\mathrm{Crit}%
(f)$\textit{ into a single point, or into }$\emptyset$\textit{, by a smooth
homotopy which does not disturb a small neighborhood of }$\lambda$\textit{.
}If $\deg(\nabla f|_{\lambda})=0$, then the critical set may be completely
removed by such a homotopy, while in case $\deg(\nabla f|_{\lambda})=k\neq0$,
it can be reduced to a single point of index $k$.

This will de discussed in the next section in a more general setting. In
higher dimensions, it would be interesting whether similar statement holds
true, say in case $K$ is homologically trivial.

\section{\label{s18}The ribbon invariant $\gamma_{\infty}$ for functions on
$\mathbb{R}^{2}$}

Consider a Morse function $f:\mathbb{R}^{2}\rightarrow\mathbb{R}$ with one
local maximum $p$, one (nondegenerate) saddle $q$ and no other critical
points. Let us perturb $f$ by a smooth homotopy $f_{t}$ \textit{with compact
support}, i.e. such that%
\[
f_{t}(x)=f(x)\text{ for }x\in\mathbb{R}^{2}\backslash W\text{, }t\in
\lbrack0,1]\text{,}%
\]

where $W$ is some bounded region in $\mathbb{R}^{2}$. The natural expectation
is that $p$ and $q$ may ``annihilate'' so that the final function $f_{1}(x)$
will not have critical points in $\mathbb{R}^{2}$. This expectation relies on
the fact that such a pair $p,q$ always may be born by an elementary
bifurcation in a critical points free zone. Well, it turns out that this is
not true in general - there are cases when $p$ and $q$ may ``annihilate'', and
other ones, when this cannot be done and each $f_{t}$ will have at least two
critical points. Such examples are given later in this section. It is clear
that the particular case depends on the behaviour of $f$ ``at infinity''. We
shall show, in the present section, that there is a simple ribbon type
invariant $\gamma_{\infty}(f)$ which is controlling the minimal number of
critical points during a homotopy of $f$ with compact support.

Let $\lambda_{1}$ and $\lambda_{2}$ be two simple closed curves in
$\mathbb{R}^{2}$ such that $\lambda_{1}$ lies in the interiour of $\lambda
_{2}$. Then we shall write for brevity $\lambda_{1}\prec\lambda_{2}$. If
$f:\mathbb{R}^{2}\rightarrow\mathbb{R}$ is a smooth function and $\lambda$ is
a closed curve not intersecting $\mathrm{Crit}(f)$, recall that by
$a_{\lambda}(f)$ we denote the ribbon defined by the restriction of $f$ and
$\nabla f$ on $\lambda$. Of course, we implicitly suppose that $\lambda$ is a
generic curve, so that $a_{\lambda}(f)\in\mathcal{A}$.

\begin{proposition}
\label{p19}Let $f:\mathbb{R}^{2}\rightarrow\mathbb{R}$ be a smooth function
and $\lambda_{1}$ and $\lambda_{2}$ be two simple closed curves containing
$\mathrm{Crit}(f)$ in its interior. Then if $\lambda_{1}\prec\lambda_{2}$, we
have
\[
\gamma(a_{\lambda_{1}}(f))\geq\gamma(a_{\lambda_{2}}(f)).
\]

\end{proposition}

\begin{proof}
Let $W_{1}$ be the interior of $\lambda_{1}$ and $W$ be the open annulus
between $\lambda_{1}$ and $\lambda_{2}$. Suppose that $\gamma(a_{\lambda_{1}%
}(f))=k$, then there is an extension $f_{1}:\overline{W}_{1}\rightarrow
\mathbb{R}$ of the ribbon $a_{\lambda_{1}}(f)$ with $k$ critical points. We
may suppose that $f\equiv f_{1}$ in some small neighbourhood of $\lambda_{1}$.
Now one defines an extension $f_{2}:\overline{W}_{1}\cup W\rightarrow
\mathbb{R}$ of $a_{\lambda_{2}}(f)$ by setting%
\[
f_{2}|_{\overline{W}_{1}}=f_{1}\text{ and }f_{2}|_{W}=f\text{.}%
\]
It is now clear that $f_{2}$ has $k$ critical points, as $f$ is critical
points free in $W$. Therefore $\gamma(a_{\lambda_{2}}(f))\leq k$, thus
$\gamma(a_{\lambda_{1}}(f))\geq\gamma(a_{\lambda_{2}}(f)).$
\end{proof}

\begin{proposition}
\label{p20}Let $f:\mathbb{R}^{2}\rightarrow\mathbb{R}$ be a smooth function
and $\lambda_{1}\prec\lambda_{2}\prec\dots$ be an infinite sequence of simple
closed curves containing $\mathrm{Crit}(f)$ in its interior. Then the
following limit exists%
\[
L=\underset{i\rightarrow\infty}{\lim}\gamma(a_{\lambda_{i}}(f))\text{.}%
\]

\end{proposition}

\begin{proof}
By \hyperref[p19]{Proposition~\ref*{p19}}, the sequence $\gamma(a_{\lambda
_{i}}(f))$ is decreasing in $i$, hence the above limit surely exists, since
$\gamma(a_{\lambda_{i}}(f))\in\mathbb{N\cup\{}0\mathbb{\}}$.
\end{proof}

Denote by $B_{r}$ the disk $x^{2}+y^{2}\leq r$. Let $\lambda$ is containing
the origin in its interior $W$, then we shall denote by $d(\lambda)$ the
number%
\[
d(\lambda)=\sup\{r|~B_{r}\subset W\}\text{.}%
\]

\begin{theorem}
\label{t15}Let $f:\mathbb{R}^{2}\rightarrow\mathbb{R}$ be a smooth function
with bounded critical set and $\lambda_{1}\prec\lambda_{2}\prec\dots$ be an
infinite sequence of simple closed curves such that $d(\lambda_{i}%
)\rightarrow\infty$. Suppose that $a_{\lambda_{i}}(f)\in\mathcal{A}$.Then the
following limit exists%
\[
\gamma_{\infty}(f)=\underset{i\rightarrow\infty}{\lim}\gamma(a_{\lambda_{i}%
}(f))\text{,}%
\]

and does not depend on the choice of sequence $\lambda_{i}$. Furthermore, for
any homotopy $f_{t}$ with compact support such that $f_{0}=f$, we have
$\gamma_{\infty}(f_{1})=\gamma_{\infty}(f)$.
\end{theorem}

This is an immediate corollary from \hyperref[p19]{Propositions~\ref*{p19}},
\ref{p20}. We shall call $\gamma_{\infty}$ ribbon number \textit{at infinity}.
Clearly, $\gamma_{\infty}(f)$ controls the minimal number of critical points
under perturbations of $f$ with compact support. For example, $\gamma_{\infty
}=0$ is equivalent to the possibility to remove all critical points by such a
perturbation. From the properties of $\gamma$ we have%
\[
\gamma_{\infty}(f)\geq\deg(\nabla f|_{\lambda})
\]

for any $\lambda$ surrounding $\mathrm{Crit}(f)$. Of course, strict inequality
is possible, as elementary examples further show.

Note that we may consider similarly all other ribbon invariants \textit{at
infinity}, obtaining in such a way some numbers $\gamma_{\infty}^{0}$,
$\gamma_{\infty}^{ext}$, $\gamma_{\infty}^{sad}$. The argumentation about
these is identical with that one about $\gamma$.

There is a simple method for constructing different examples, by extending the
boundary of a ribbon ``at infinity''.

\begin{proposition}
\label{p20.5}Let $a\in\mathcal{A}$ be a ribbon and $f:\mathbb{B}%
^{2}\rightarrow\mathbb{R}$ be its extension, $f\in\mathcal{F}(a)$. Then there
is a smooth function $f_{0}:\mathbb{R}^{2}\rightarrow\mathbb{R}$ such that
$f_{0}|_{\mathbb{B}^{2}}=f$, $\mathrm{Crit}(f_{0})=\mathrm{Crit}(f)$ and
$\gamma_{\infty}(f_{0})=\gamma(a)$.
\end{proposition}

\begin{example}
\label{e2} We shall give an example of a Morse function $f:\mathbb{R}%
^{2}\rightarrow\mathbb{R}$ such that its critical set consists of $k$
nondegenerate extrema and $k$ nondegenerate saddles with $\gamma_{\infty
}(f)=\gamma_{\infty}^{0}(f)=2k$, $\gamma_{\infty}^{ext}(f)=\gamma_{\infty
}^{sad}(f)=k$. This means that the total number of critical points cannot be
reduced under $2k$, moreover, the number of local extrema and of saddle points
cannot be reduced under $k$ by a smooth homotopy $f_{t}$ with compact support.
(Of course, in order this to be true in the part concerning saddles, we have
to suppose that $\mathrm{Crit}(f_{t})$ is finite for each $t$, since otherwise
we may destroy saddles at the price of the birth of infinite number of
critical points.) In such a way, roughly speaking, no one pair extremum-saddle
of $f$ can annihilate under such a homotopy.

Here is the example. Take an arbitrary $a\in\mathcal{A}_{2k}^{-}$ and let
$b\in\mathcal{A}_{2k+2}^{+}$ be a positive ladder. Now let $a_{1}$ be the
ribbon, which is identical to $a$, except for the maximal node, which is made
``positive''. We know that $\gamma(a)=\gamma_{\mathtt{\operatorname{ext}}%
}(a)=\frac{2k}{2}+1=k+1$, $\gamma(b)=\gamma_{\mathtt{\operatorname{sad}}%
}(b)=\frac{2k+2}{2}-1=k$. Furthermore, $\gamma(a_{1})=\gamma(a)-1=k$, as we
know from the opening remarks. Now we may consider the ``connected sum''
$a_{1}\#b$, which is obtained by applying the maximal node of $a_{1}$ to the
minimal one of $b$ and then ``canceling'' them. (This is explained in details
in \hyperref[s20]{Section~\ref*{s20}}.) It is quite easy to show that the
ribbon invariants are additive under such operation: $\gamma(a_{1}%
\#b)=\gamma(a_{1})+\gamma(b)=k+k=2k$, $\gamma_{\mathtt{\operatorname{ext}}%
}(a_{1}\#b)=\gamma_{\mathtt{\operatorname{ext}}}(a_{1})+\gamma
_{\mathtt{\operatorname{ext}}}(b)=k+0=k$, $\gamma_{\mathtt{\operatorname{sad}%
}}(a_{1}\#b)=\gamma_{\mathtt{\operatorname{sad}}}(a_{1})+\gamma
_{\mathtt{\operatorname{sad}}}(b)=0+k=k$. Take the natural Morse extension
$g\in\mathcal{F}(a_{1}\#b)$ which is realizing $\gamma$ and all other ribbon
invariants. We shall finally refer to \hyperref[p20.5]%
{Proposition~\ref*{p20.5}} finding in such a way some $f:\mathbb{R}%
^{2}\rightarrow\mathbb{R}$ with the desired properties.
\end{example}

Note that $\gamma_{\infty}$ is in fact a function from some sheaf
$\mathcal{C}$ into $\mathbb{N\cup\{}0\mathbb{\}}$. The sheaf $\mathcal{C}$ has
for germs the classes of smooth functions $f:\mathbb{R}^{2}\rightarrow
\mathbb{R}$, where $f\sim g$ if there is a disk $D$ such that $f(x)=g(x)$ for
$x\in\mathbb{R}^{2}\backslash D$. It is now clear, that we have some correctly
defined function%
\[
\gamma_{\infty}:\mathcal{C}\rightarrow\mathbb{N\cup\{}0\mathbb{\}}\text{.}%
\]

From this point of view, it would be of some interest to determine whether
$\gamma_{\infty}$ interacts in some manner with the algebraic structure on
$\mathcal{C}$; anyway, we won't go in this direction here.

There is another situation, when $\gamma_{\infty}$ gives some geometrical
information. Let $W\subset\mathbb{R}^{2}$ be some open set, $p_{0}\in W$ and
$f:W\backslash\{p_{0}\}\rightarrow\mathbb{R}$ be a smooth function. Suppose
that $f$ is critical points free in some neighbourhood of $p_{0}$. Let and
$\lambda_{1}$ and $\lambda_{2}$ be two Jordan curves surrounding $p_{0}$ and
sufficiently close to it such that $\lambda_{1}\succ\lambda_{2}$. Consider the
corresponding induced ribbons $a_{\lambda_{1}}(f)$, $a_{\lambda_{2}}(f)$. Then
\hyperref[p19]{Proposition~\ref*{p19}} implies that%
\[
\gamma(\overline{a}_{\lambda_{1}}(f))\geq\gamma(\overline{a}_{\lambda_{2}%
}(f))\text{.}%
\]

This is easily seen by inversion of the plane and observing that a ribbon $a$
passes into $\overline{a}$ by inversion. Take now a sequence $\lambda_{1}%
\succ\lambda_{2}\succ\dots$ of Jordan curves surrounding $p_{0}$ and such that
$\mathrm{diam}(\lambda_{i})\rightarrow0$. Then the following limit exists:%
\[
\overline{\gamma}(f,p_{0})=\underset{i\rightarrow\infty}{\lim}\gamma
(\overline{a}_{\lambda_{i}}(f))\text{.}%
\]

Clearly, $\overline{\gamma}(f,p_{0})$ is an invariant of the ``singularity''
at $p_{0}$ which does not depend on the choice of the sequence $\lambda_{i}$.
On the other hand, one may consider the limit%
\[
\gamma(f,p_{0})=\underset{i\rightarrow\infty}{\lim}\gamma(a_{\lambda_{i}%
}(f))\text{,}%
\]

in case the latter exists, as $\gamma(a_{\lambda_{i}}(f))$ is an increasing
sequence of integers and may diverge. Note that if $f$ may be defined in
$p_{0}$ in such a way, that the extension is smooth, then $\gamma(f,p_{0}%
)\leq1$.\textbf{ }In case $\gamma(f,p_{0})=\infty$, one still may examine the
divergence rate of the sequence $\gamma(a_{\lambda_{i}}(f))$ and this probably
carries some geometrical information about the singularity at $p_{0}$.

\section{\label{s19}Local and global stability of the critical set}

Till now we may claim only the existence of $\gamma$ different critical points
under the corresponding boundary conditions. It is natural to ask whether this
group of points is, in some sense, stable and may be controlled under
perturbations. There are two types of stability results: a local stability
under small perturbations, and a global one - under homotopy.

Our first stability result is a local one. Roughly speaking, it claims that
the $\gamma$ critical points are ``distinguishable'', i.e. distant from each
other under certain natural conditions. Of course, one may concentrate all the
critical points in some small region at the cost of a blow up of the gradient,
so we have to control gradient's norm.

\begin{theorem}
\label{t16}Let $a\in\mathcal{A}$ be a ribbon. Then for any $m>0$ there is
$d>0$ such that if $f\in\mathcal{F}(a)$ is an extension of $a$ with
$\left\Vert \nabla f\right\Vert <m$, then there exist $\gamma=\gamma(a)$
critical points of $f$, $p_{1},\dots,p_{\gamma}$ such that $|p_{i}-p_{j}|\geq
d$ for $i\neq j$.
\end{theorem}

Probably, the distance $d$ in the above theorem may be written as a simple
expression of the gradient $\left\Vert \nabla f\right\Vert $ and then the
constant $m$ is needless. Another local stability result should be a variant
of \hyperref[t16]{Theorem~\ref*{t16}}, controlling the distance between the
critical values of an extension, rather than the distance between the critical
points. Then, clearly, one has to consider the invariant $\beta$ instead of
$\gamma$. Recall that $\beta$\ is controlling the number of critical values
(see \hyperref[s2]{Section~\ref*{s2}}. So, a literal restatement of the
theorem holds true, where ``critical points'' and ``$\gamma$'' are replaced by
``critical values'' and ``$\beta$'', correspondingly.

Let us focus now our attention on the global stability properties of the
critical set under homotopy. We have the following situation: Given some
ribbon $a\in\mathcal{A}$ and a smooth path (homotopy) $f_{t}$ in
$\mathcal{F}(a)$, $t\in\lbrack0,1]$, then are the components of the critical
set $\mathrm{Crit}(f_{t})$ ``stable'' in some sense? We know that
$\mathrm{Crit}(f_{t})$ has at least $\gamma(a)$ essential components. Of
course, we cannot claim that there are $\gamma(a)$ monotonic in $t$ branches
of the critical set, since various bifurcations can occur. However, it turns
out that there are $\gamma(a)$ components of the set
\[
F=\underset{t}{\cup}\mathrm{Crit}(f_{t})\subset\mathbb{B}^{2}\times
\lbrack0,1]
\]

which are essential, in some sense, and therefore are intersecting both
$\mathbb{B}^{2}\times\{0\}$ and $\mathbb{B}^{2}\times\{1\}$.

\begin{lemma}
\label{l10}Let $a\in\mathcal{A}$ be a ribbon and $f_{t}\in\mathcal{F}(a)$ be a
smooth homotopy, $t\in\lbrack0,1]$. Let $K$ be a component of the set
$F=\underset{t}{\cup}\mathrm{Crit}(f_{t})$. Then the index of $K\cap
(\mathbb{B}^{2}\times\{t\})$ with respect to the field $\nabla f_{t}$ is
constant (not depending on $t$). So, if this index is nonzero, then $K$
intersects both $\mathbb{B}^{2}\times\{0\}$ and $\mathbb{B}^{2}\times\{1\}$.
\end{lemma}

A component $K$ of nonzero index will be called \textit{essential.}

\begin{theorem}
\label{t17}For any ribbon $a\in\mathcal{A}$ and any smooth homotopy $f_{t}%
\in\mathcal{F}(a)$ there exist $\gamma(a)$ essential components of the set
$F=\underset{t}{\cup}\mathrm{Crit}(f_{t})$.
\end{theorem}

Probably, the Hausdorff distance between these components may be estimated
from below in terms of the gradient's norm $\left\Vert \nabla f_{t}\right\Vert
$.

\section{\label{s20}The ribbon semigroup}

In \hyperref[s5]{Section~\ref*{s5}} we introduced two splittings of ribbons -
a binary and a ternary one. Broadly speaking, the first one is performed at a
regular value (\hyperref[f20]{Fig.~\ref*{f20}}), while the second one is done
at a negative node (\hyperref[f21]{Fig.~\ref*{f21}}). In the present section
we shall introduce two operations in the class of ribbons, that are, in some
sense, inverse to splittings. This endows the ribbon set $\mathcal{A}$\ with
some algebraic structure, that we shall refer to as ``the ribbon semigroup''.
Of course, like in the definition of fundamental group, where some base point
is selected at the beginning, this can be correctly done only after the choice
of some particular nodes of each ribbon, which serve as application points.

Let us specify from the beginning that we define these operations in the class
$\mathcal{A}$ of rigid ribbons considered up to translations. Moreover, we
shall allow ribbons with coinciding node values, in order to get correctly
defined operations in $\mathcal{A}$. Note also that a discrete variant of the
ribbon semigroup is available, where zig-zag permutations with repetitions are considered.

First, we consider a very natural operation, which \textit{is not} our basic
binary operation, but anyway, has some advantages \textit{per se}.

\textbf{Connected sum of ribbons.}

This is a simple operation, which may be useful for constructing examples,
though it does not define some very interesting algebraic structure in
$\mathcal{A}$.

The connected sum operation is defined for ribbons with positive minimal and
maximal nodes. As we explained at the beginning, this is not a constraint from
point of view of $\gamma$.

\begin{definition}
\label{d21}Let $a,b\in\mathcal{A}$ be ribbons with positive minimal and
maximal nodes. Then their connected sum is the ribbon $a\#b$ obtained by
applying the maximal node of $a$ to the minimal one of $b$ (\hyperref[f37]%
{Fig.~\ref*{f37}}).
\end{definition}

Clearly, this is an associative, but non commutative operation, which is
defining some semigroup $A_{\#}$. Note that there is no a natural unit in
$A_{\#}$, although any ribbon $a$ equivalent to $\alpha_{0}=(1^{+},2^{+})$ may
be thought of as a ``unit'', since $a\#b\sim b\#a\sim b$ for any $b\in A_{\#}%
$. Furthermore, we may add to $A_{\#}$ the elementary ribbons $\alpha
_{1}=(1^{+},2^{-})$ and $\alpha_{2}=(1^{-},2^{+})$ with the convention that
$\alpha_{1}$ may be subject only to left multiplication, whereas $\alpha_{2}$
- to right multiplication. Then we obtain an extended version of $A_{\#}$
which contains all the possible ribbons.

\begin{center}
\begin{figure}[ptb]
\includegraphics[width=120mm]{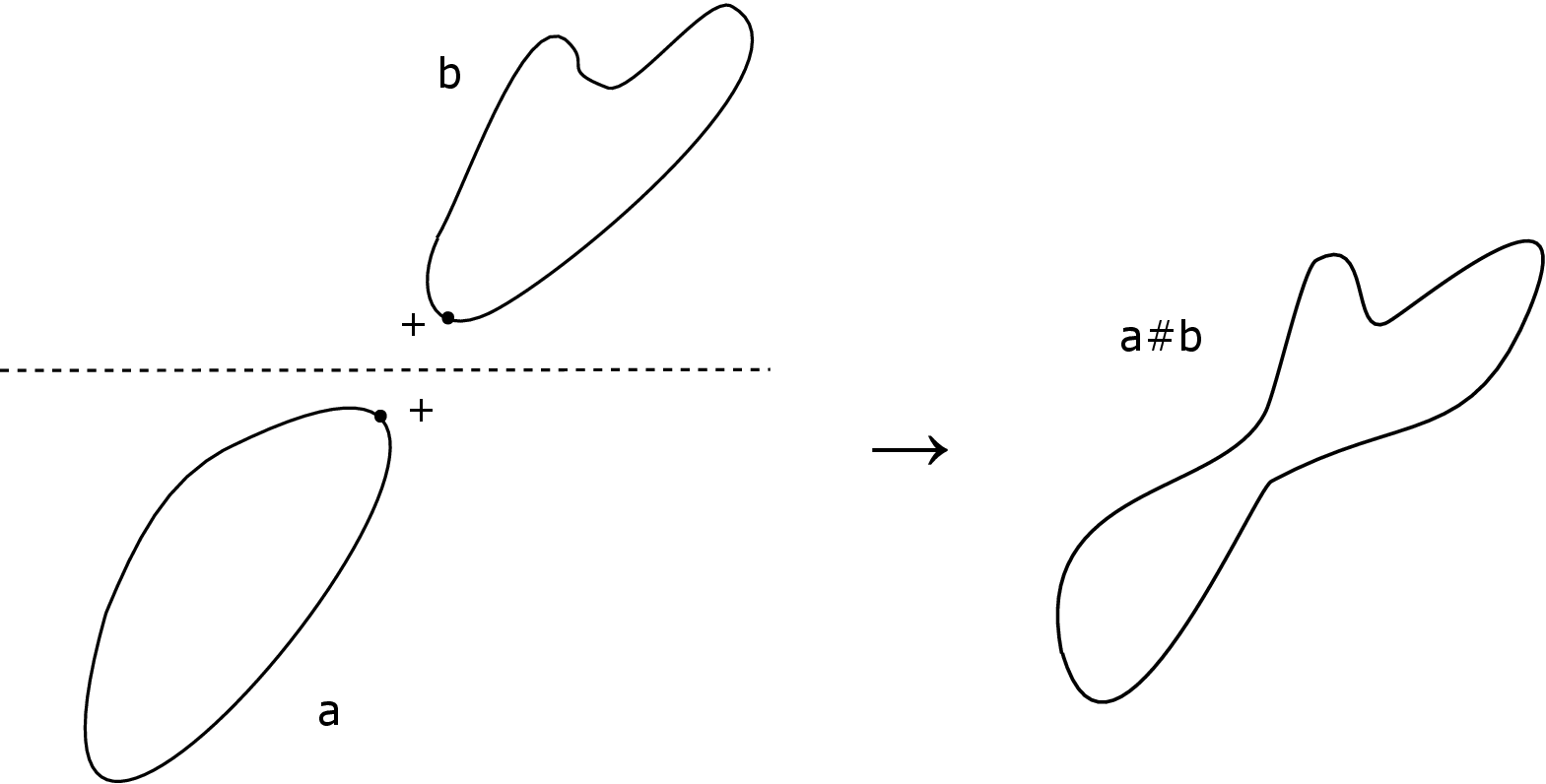}\caption{Connected sum of ribbons.}%
\label{f37}%
\end{figure}
\end{center}

\begin{remark}
In fact, we get some operation in the class of discrete (soft) ribbons,
defining in such a way some countable semigroup $B_{\#}$. This is due to the
obvious fact that%
\[
\text{if }a_{1}\sim a\text{ and }b_{1}\sim b\text{, then }a_{1}\#b_{1}\sim
a\#b\text{.}%
\]

\end{remark}

Note that $B_{\#}$ is a \textit{monoid} (semigroup with unity), where the unit
is the class of any ribbon equivalent to $a_{0}=(1^{+},2^{+})$. Of course,
there is a natural morphism $A_{\#}\rightarrow B_{\#}$.

\begin{example}
If $a\in\mathcal{A}_{n}$ is a ladder, then
\[
a=a_{1}\#a_{2}\#\dots\#a_{n/2-1}\text{,}%
\]

where $a_{i}\in\mathcal{A}_{4}$. Conversely, if $a\in\mathcal{A}_{n}$ may be
written as a connected sum of $\frac{n}{2}-1$ ribbons with 4 nodes, then $a$
is a ladder.
\end{example}

Except the fact that the connected sum is defined for discrete ribbons, it has
the advantage that the ribbon invariants $\gamma_{\ast}$ are additive under
taking connected sum.

\begin{proposition}
\label{p21}For any two $a,b\in A_{\#}$, we have%
\[
\gamma_{\ast}(a\#b)=\gamma_{\ast}(a)+\gamma_{\ast}(b),
\]

where $\gamma_{\ast}$ is any of the ribbon invariants $\gamma$, $\gamma_{0}$,
$\gamma_{\operatorname{ext}}$, $\gamma_{\operatorname{sad}}$.
\end{proposition}

\begin{proof}
This is straightforward, since by definition the ribbon $a\#b=(\varphi,\nu)$
has a ``thin'' level $c_{0}$ such that $\left\vert \varphi^{-1}(c_{0}%
)\right\vert =2$ and $c_{0}$ is separating $a$ from $b$. Then any extension
$f\in\mathcal{F}\left(  a\#b\right)  $ surely has a level line connecting both
points of $\varphi^{-1}(c_{0})$. Note also that we have additivity in the
extended version of $A_{\#}$ too.
\end{proof}

Of course, the degree is also additive under connected sums:
$i(a\#b)=i(a)+i(b)$.

The bad news about $\#$ is that almost any element $a$ of $A_{\#}$ is
\textit{prime}, i.e. if $a=b\#c$, then either $a\sim b$, or $a\sim c$, as
there might not exist a ``thin'' essential level in $a$, and this is the
typical situation. In such a way, there is no any simple base in $A_{\#}$
consisting of prime elements, unlike the case with the main binary and ternary
operations in $\mathcal{A}$ defined below.

\bigskip

\textbf{The main operations.}

As usual, in order to define correctly an operation in the class of objects
under consideration, one has to select some distinguished application point(s)
either in any object, or in the ambient space. Let $A$ denote the class of
rigid ribbons $a$ with two positive nodes selected, say $p_{1}(a)$ and
$p_{2}(a)$, such that $p_{1}(a)$ is of minimal, and $p_{2}(a)$ is of maximal
type. We shall say that $a$\ is a \textit{marked} ribbon and $p_{1}(a)$ is the
\textit{origin}, while $p_{2}(a)$ is the \textit{end} of $a$. Clearly, there
is a \textit{forgetful} map%
\[
q:A\rightarrow\mathcal{A}\text{,}%
\]
whose image $q(A)$\ consists of the ribbons with at least two positive nodes
of opposite type.

Now we may define the superposition of marked ribbons:

\begin{definition}
\label{d22}Let $a,b\in A$, then $c=ab$ is the ribbon obtained by applying
$p_{2}(a)$ to $p_{1}(b)$. For the new ribbon $c$ we set $p_{1}(c)=p_{1}(a)$,
$p_{2}(c)=p_{2}(b)$(\hyperref[f38]{Fig.~\ref*{f38}}). In some sense, this
operation is inverse to a binary splitting.
\end{definition}

It is clear that we get some associative binary operation in $A$:
$(ab)c=a(bc)$, which is, of course, non commutative. Moreover, we don't have a
natural unit in $A$, since the only candidate would be a ribbon $a$ equivalent
to $\alpha_{0}=(1^{+},2^{+})$, but such a ribbon produces different result
when applied to an arbitrary other ribbon, depending on the difference between
the minimal and the maximal value of $a$. So, $A$ is not a \textit{natural}
monoid. Of course, ``small'' ribbons equivalent to $\alpha_{0}=(1^{+},2^{+})$
may be thought of as ``local'' units with respect to some fixed ribbon. Their
application does not change any of the ribbon invariants, as the result is
equivalent to the original ribbon.

\begin{center}
\begin{figure}[ptb]
\includegraphics[width=100mm]{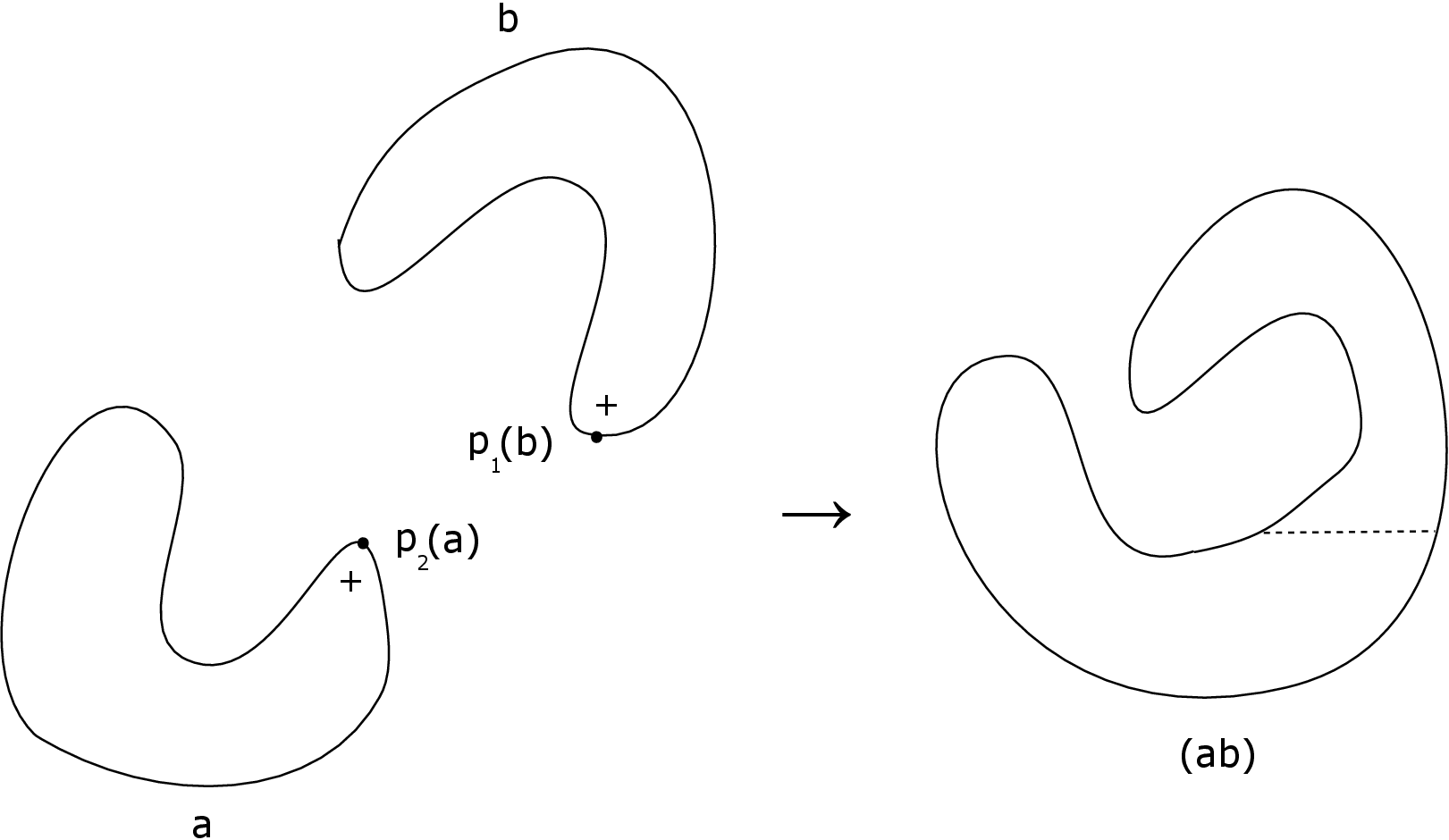}\caption{The main binary operation
$c=ab$.}%
\label{f38}%
\end{figure}

\begin{figure}[ptb]
\includegraphics[width=100mm]{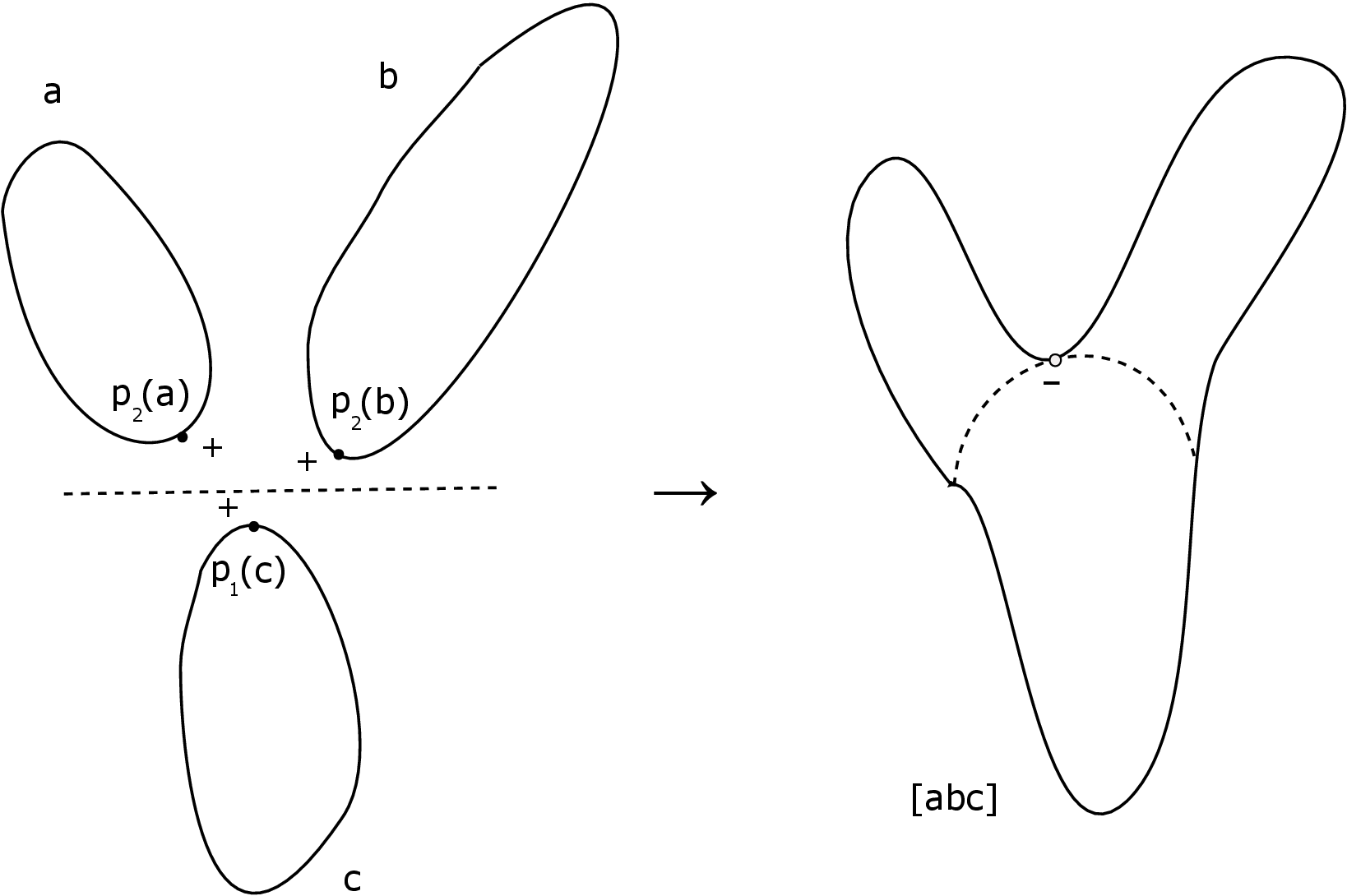}\caption{The ternary operation
$[abc]$.}%
\label{f39}%
\end{figure}
\end{center}

Let us emphasize the difference between the connected sum $a\#b$ and the
operation $ab$, the latter being defined in a larger class of marked ribbons.

Now we define the ternary operation in $A$, which is necessary for full
characterization of the ribbon invariant $\gamma$.

\begin{definition}
\label{d23}Let $a,b,c\in A$, then $x=\left[  abc\right]  $ is the ribbon
obtained by the triple application of $p_{2}(a)$, $p_{2}(b)$ and $p_{1}(c)$ at
one and the same level $l_{0}$ (\hyperref[f39]{Fig.~\ref*{f39}}). Then a
newborn negative node appears at level $l_{0}$ and we set $p_{1}(x)=p_{1}(a)$
and $p_{2}(x)=p_{2}(c)$.
\end{definition}

\begin{center}
\begin{figure}[ptb]
\includegraphics[width=94mm]{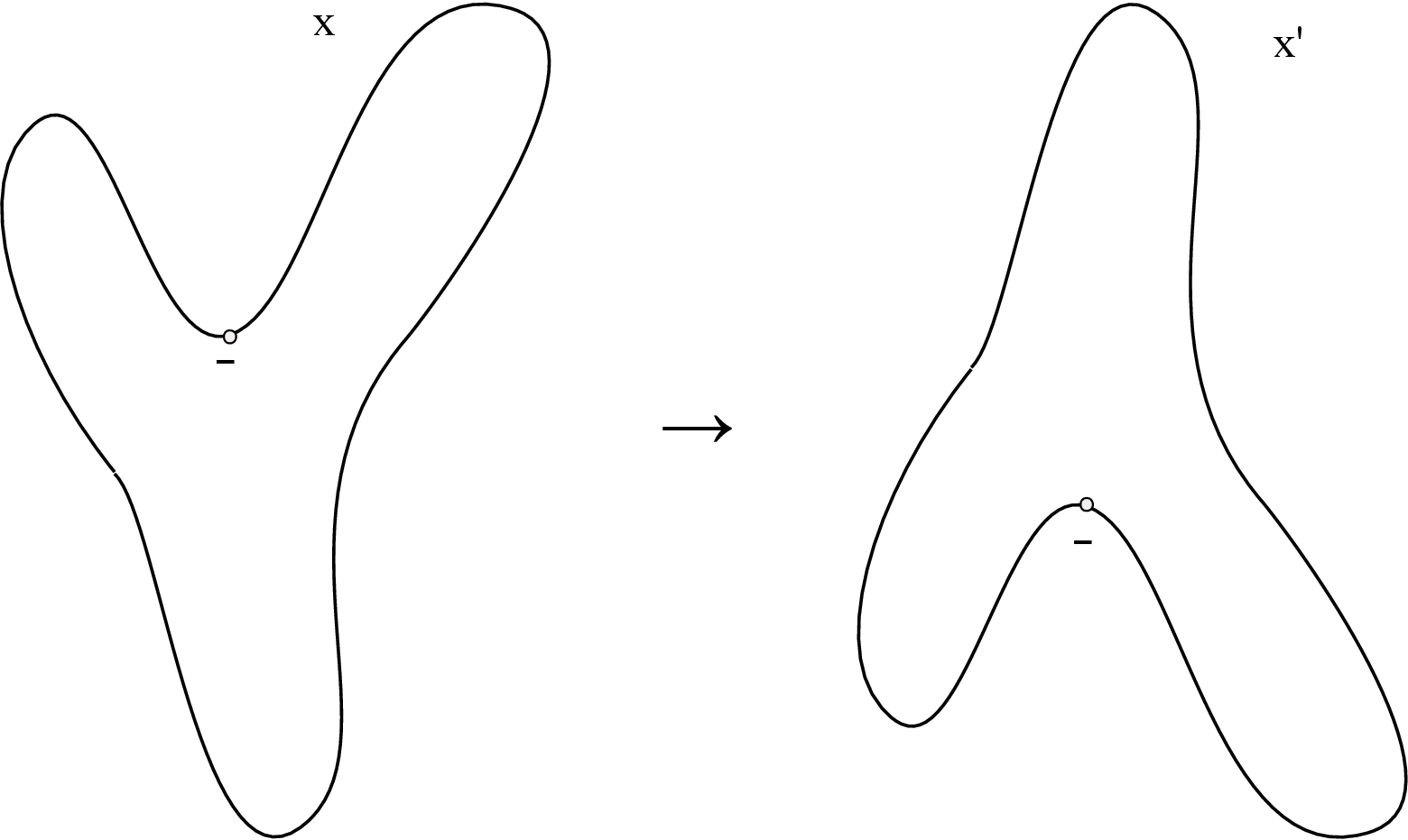}\caption{The inversion.}%
\label{f40}%
\end{figure}
\end{center}

Furthermore, by $y=x^{\prime}$ we denote the image of $x$ via the
\textit{inversion} from \hyperref[f40]{Fig.~\ref*{f40}} We set $p_{1}%
(x^{\prime})=p_{2}(x)$ and $p_{2}(x^{\prime})=p_{1}(x)$.\ This involution is
necessary in order to obtain all kind of ribbons from the elementary ones,
since only the binary and ternary operations are not enough. The reason is
that the ternary operation $\left[  abc\right]  $, as defined, gives birth to
a negative node which is a local maximum, but, of course, there are ribbons
with negative nodes at a local minimum.\ For example, the ribbon $x^{\prime}$
from \hyperref[f40]{Fig.~\ref*{f40}} cannot be obtained without inversion from
the elementary ribbons. Clearly, inversion does not affect any of the ribbon invariants.

It turns out that there are several natural relations involving binary and
ternary operations.

\begin{proposition}
\label{p22}The following relations hold true in $A$:

1) $(ab)c=a(bc)$, $\left(  ab\right)  ^{\prime}=b^{\prime}a^{\prime}$

2) $\left[  \left[  abc\right]  de\right]  =\left[  ab\left[  cde\right]
\right]  $

3) $\left[  abc\right]  d=\left[  ab\left(  cd\right)  \right]  $, $a\left[
bcd\right]  =\left[  \left(  ab\right)  cd\right]  $

4) $\sigma(ab)=\sigma(a)+\sigma(b)-2$

5) $\sigma\left(  \left[  abc\right]  \right)  =\sigma(a)+\sigma
(b)+\sigma(c)-4$.
\end{proposition}

Here $\sigma(a)$ is the signature of ribbon $a$. Then 4) and 5) easily imply
that the index $i=1-\frac{\sigma}{2}$ is additive under both operations in
$A$:%
\[
i(ab)=i(a)+i(b)\text{, \ }i\left(  \left[  abc\right]  \right)
=i(a)+i(b)+i(c)\text{.}%
\]

Conversely to the above equalities, the ribbon invariant $\gamma$ turns out to
be subbaditive function in $A$. Of course, this holds true for all the other
ribbon invariants $\gamma_{\ast}$:

\begin{theorem}
\label{t18}The invariants $\gamma_{\ast}$ are subadditive in $A$:%
\[
\gamma_{\ast}(ab)\leq\gamma_{\ast}(a)+\gamma_{\ast}(b)\text{, \ }\gamma_{\ast
}\left(  \left[  abc\right]  \right)  \leq\gamma_{\ast}(a)+\gamma_{\ast
}(b)+\gamma_{\ast}(c)\text{.}%
\]

\end{theorem}

This was proved in fact in \hyperref[s5]{Section~\ref*{s5}}, where these
inequalities were established for binary and ternary splittings. Note also
that the above inequalities are basic in the general algebraic definition of a
\textit{ribbon invariant }and this will be subject of the next section.

\begin{lemma}
\label{l11}Every $a\in A$ may be represented as a composition of irreducible
elements of $A$ via the binary and ternary operations and inversion.
\end{lemma}

The proof is by splittings and induction on the number of nodes. Of course,
the above representation is not unique at all, as simple examples show.

We shall say that an element $a\in A$ is \textit{essential}, if it is not
equivalent to $\left(  1^{+},2^{+}\right)  $. Clearly, for any essential
\textit{irreducible }ribbon $x\in A$ we have $\gamma(x)=1$, otherwise
$\gamma(x)=0$.

\begin{definition}
\label{d24}Let $a\in A$ be a ribbon. Then its representations are all possible
expressions of $a$ as a superposition of irreducible\ ribbons via the binary
and ternary operations ($ab$, $\left[  abc\right]  $) and inversion
($a^{\prime}$). (E.g. $a=\left[  \left[  xyz\right]  ^{\prime}uv\right]  $). A
representation of $a$ will be denoted by $r(a)$ and their collection by
$R(a)$. The weight of a representation $r(a)$ is the number of its essential
elements and is denoted by $w(r(a))$.
\end{definition}

\begin{lemma}
\label{l12}Any ribbon $a\in A$ has a representation.
\end{lemma}

Now we give an algebraic description of $\gamma$ in terms of the semigroup $A$.

\begin{theorem}
\label{t19}Let $a\in A$. Then the ribbon invariant $\gamma(a)$ equals the
minimal weight $w(r(a))$, when $r\in R(a)$ runs over all representations of
ribbon $a$:%
\[
\gamma(a)=\underset{r\in R(a)}{\min}w(r(a))\text{.}%
\]

\end{theorem}

Furthermore, let us call a representation of a ribbon $a$ \textit{simple}, if
it is a superposition \textit{only} of binary operations. Clearly, any
positive ribbon $a\in A^{+}$ has a simple representation by positive
alternations (ternary operations give birth to a negative node), so $A^{+}$ is
an ordinary semigroup. In such a way, the problem of calculating $\gamma$ in
$A^{+}$ is settled, from algebraic point of view, in a much more simple situation.

Consider, on the other hand the class $\Gamma_{0}$ of ribbons with $\gamma=0$.
Then it may be algebraically characterized as follows:

\begin{proposition}
The class $\Gamma_{0}$ is generated by the ribbons equivalent to the minimal
one, $\alpha_{0}=\left(  1^{+},2^{+}\right)  $, subject only to ternary
multiplication and inversion.
\end{proposition}

In such a way, the class $\Gamma_{0}$ of ribbons with $\gamma=0$ may be
inductively constructed \textit{bottom-up} and may probably be geometrically
well described by \textit{marked} trees (in the sense that each vertex is
assigned its \textit{level}).

Let us see now what happens with the other \textit{geometric} ribbon
invariants $\gamma_{\ast}=\gamma_{0}$, $\gamma_{\operatorname{ext}}$,
$\gamma_{\operatorname{sad}}$. It is not difficult to see that in these cases
it is convenient to define the weight $w_{\ast}$ of a representation $r(a)$ as
the sum of the ribbon invariants $\gamma_{\ast}$ of its elements (the latter
being obvious), and then the following natural improvement of \hyperref[t19]%
{Theorem~\ref*{t19}} holds:

\begin{theorem}
\label{t20}Let $a\in A$ be a ribbon. Then%
\[
\gamma_{\ast}(a)=\underset{r\in R(a)}{\min}w_{\ast}(r(a))\text{,}%
\]

where $\gamma_{\ast}=\gamma$, $\gamma_{0}$, $\gamma_{\operatorname{ext}}$,
$\gamma_{\operatorname{sad}}$.
\end{theorem}

\bigskip Recall the weight systems $w_{\ast}$ for the elementary ribbons
$\alpha_{0}=\left(  1^{+},2^{+}\right)  $, $\alpha_{1}=\left(  1^{+}%
,2^{-}\right)  $, $\alpha_{2}=\left(  1^{-},2^{+}\right)  $, $\beta_{n}%
\in\mathcal{A}_{n}^{+}$ - the positive alternations. These are depicted at
\hyperref[table1]{Table~\ref*{table1}}.

\begin{table}[ptb]
\begin{center}%
\begin{tabular}
[c]{|l|c|c|c|c|}\hline
& $\alpha_{0}$ & $\alpha_{1}$ & $\alpha_{2}$ & $\beta_{n}$\\\hline
$w$ & 0 & 1 & 1 & 1\\\hline
$w_{0}$ & 0 & 1 & 1 & $n/2-1$\\\hline
$w_{\operatorname{ext}}$ & 0 & 1 & 1 & 0\\\hline
$w_{\operatorname{sad}}$ & 0 & 0 & 0 & 1\\\hline
\end{tabular}
\end{center}
\caption{Weights of the ribbon invariants.}%
\label{table1}%
\end{table}

Here by $\beta_{n}$ we mean \textit{any} positive alternation with $n$ nodes.

If $A^{\ast}$ denotes the set of all types of ribbons, yet in this case any
$a\in A^{\ast}$ has a representation $r(a)$, which is a composition of
elementary ribbons. For example, $\left(  1^{-},2^{-}\right)  =\left(
1^{-},2^{+}\right)  \left(  1^{+},2^{-}\right)  $. Then \hyperref[t20]%
{Theorem~\ref*{t20}} still holds true (though $A^{\ast}$\ has no longer
completely defined internal algebraic operations).

The following observation is about the fact that any representation $r(a)$ of
a given ribbon $a\in A^{\ast}$ corresponds to an economic extension
$f\in\mathcal{F}^{e}(a)$ of $a$, and vice versa.

\begin{theorem}
\label{t21}For any ribbon $a\in A^{\ast}$ the set of economic extensions
$\mathcal{F}^{e}(a)$ is in one-to-one correspondence with the set of algebraic
representations $R(a)$ of ribbon $a$.
\end{theorem}

The proof is geometrically self-evident, as any procedure of splitting of an
economic extension $f\in\mathcal{F}^{e}(a)$ into elementary regions
corresponds to a procedure of constructing some representation $r\in R(a)$ of
ribbon $a$.

Consider now the set%
\[
\mathcal{F}^{\ast}\mathcal{=}\cup\left\{  f\in\mathcal{F}^{e}(a)|\ a\in
A^{\ast}\right\}
\]

of all economic extensions. It is easy to see that both algebraic operations
(and inversion) may be correctly defined in $\mathcal{F}^{\ast}$. We shall
preserve the same notation about these: $fg$, $\left[  fgh\right]  $,
$f^{\prime}$. Let $f\in\mathcal{F}^{e}(a)$, then we shall write $a_{f}=a$. In
such a way, $f\rightarrow a_{f}$ is a \textit{forgetful} functor. Clearly, one
has%
\[
a_{fg}=a_{f}a_{g}\text{, }a_{\left[  fgh\right]  }=\left[  a_{f}a_{g}%
a_{h}\right]  \text{, }a_{f^{\prime}}=\left(  a_{f}\right)  ^{\prime}\text{,}%
\]

so, if we define $j(f)=a_{f}$, then we get a natural \textit{forgetful}
morphism\newline$j:\mathcal{F}^{\ast}\mathcal{\rightarrow}A^{\ast}$. In the
functors-morphisms terminology, the functor defined by $j$ is \textit{fully
faithful}. Obviously, $j^{-1}(a)$ is the set of all economic extensions of $a$.

\begin{remark}
\textbf{1) }In order to finitize the situation, it is convenient to consider
``truncated'' versions of the above semigroups. For a given $n$, denote by
$A_{(n)}^{\ast}$ the elements of $A^{\ast}$ with $\leq n$ nodes. We
furthermore add to $A_{(n)}^{\ast}$ an ``artificial'' unit $\left\{
1\right\}  $, making in such a way $A_{(n)}^{\ast}$ a ``monoid''. Now we
define $ab=\left\{  1\right\}  $ and $\left[  abc\right]  =\left\{  1\right\}
$, in case the number of nodes of $ab$, respectively $\left[  abc\right]  $
exceeds $n$. Then working in the class of discrete ribbons (with coinciding
levels allowed), we get some finite algebraic object $A_{(n)}^{\ast}$, which
is a monoid with respect to binary multiplication, though having in addition
ternary multiplication and inversion. The same way is defined $\mathcal{F}%
_{(n)}^{\ast}$.\ Clearly, $A^{\ast}$ and $\mathcal{F}^{\ast}$\ are direct
limits of systems $A_{(n)}^{\ast}$ and $\mathcal{F}_{(n)}^{\ast}$,
respectively.\ This truncation will be exploited in the next section for
finitization of algebraic ribbon invariants.

\textbf{2) }In the class of positive ribbons the ternary operation is missing
(neither need we involution), so we get some finite semigroup $A_{(n)}^{+}$,
and the situation is pretty simplified.
\end{remark}

\section{\label{s21}Partial ordering of the ribbon space}

The above results raise the natural question whether we may define some
equivalence relation in class $\mathcal{A}$ of ribbons following the common
general idea:

\textit{Two ribbons }$a$\textit{ and }$b$\textit{ are equivalent, if there is
a noncritical membrane }$C$\textit{ which ``connects'' them.}

So, in some sense, there is a critical points free function $F:C\rightarrow
\mathbb{R}$ and $\partial C=a\cup b$. Unfortunately, it turns out that all our
efforts to define correctly such an equivalence relation are failing. One
reason is that, in order to provide transitivity, the membrane $C$ has to
connect $a$ with the complementary ribbon $\overline{b}$, but then we are
loosing symmetry. Another reason is that we can hardly expect that a ribbon
can be connected to itself by a noncritical membrane, so we don't have
reflexivity.\ This idea leads only to some partial ordering of $\mathcal{A}$,
which anyway gives useful information about the structure of the ribbon
space\ and agrees with the ribbon invariants theory. Furthermore, we define
the same way a partial order in the class of gradient (nonzero) vector fields
on $\mathbb{S}^{1}$. Both partial orderings turn out to be continuous in the
corresponding natural topology. Here we face up with the problem of
determining when $a\prec b$, and it turns out again that this problem may be
solved only algorithmically (at least by the author).

Note also that this ordering \textquotedblleft$\prec$\textquotedblright%
\ should not be confused with the lexicographic ordering of the ribbon space
defined at the beginning of the article in \hyperref[s4]{Section~\ref*{s4}}.

Let us make the agreement that in this section all ribbons are \textit{rigid}
in the strongest possible sense, and these cannot be subject to translation
and rotation without obtaining a different ribbon. Later we shall discuss the
possibility to define such an ordering in the class of discrete ribbons.

\begin{definition}
\label{d25}Let $C=\mathbb{S}^{1}\times\lbrack1,2]$ be an annulus and
$F:C\rightarrow\mathbb{R}$ be a smooth critical points free function. If
$\partial C=C_{1}\cup C_{2}$, where $C_{1}=\mathbb{S}^{1}\times\{1\}$ and
$C_{2}=\mathbb{S}^{1}\times\{2\}$, consider the boundary ribbons $a_{1}$ and
$a_{2}$ defined by the restrictions $F|_{C_{1}}$ and $F|_{C_{2}}$,
respectively. Then we shall write $a_{1}\prec a_{2}$. Clearly, we define in
such a way some relation in the ribbon space $\mathcal{A}$, which turns out to
be a partial ordering.
\end{definition}

Here we allow ribbons to have any Jordan curve for a domain, and won't
restrict ourselves to the case of the unit circle $\mathbb{S}^{1}$. Note also
that the membrane $C$ is ``connecting'' $a_{2}$ with the complementary ribbon
$\overline{a_{1}}$, rather than $a_{2}$ with $a_{1}$.

\begin{lemma}
\label{l13}The relation ``$\prec$'' is an anti-reflexive partial ordering of
$\mathcal{A}$, which is continuous with respect to the natural ($C^{1}$)
topology in the ribbon space.
\end{lemma}

\begin{proof}
The transitivity is straightforward, the anti-reflexivity, $a\nprec a$, is
obtained as follows. Suppose that $a\prec a$ for some ribbon $a$, so, it can
be connected to its inverse by some membrane $F:C\rightarrow\mathbb{R}$. But
then, if $p$ is the maximal node of $a$ and supposing it being positive (for
example), then $p$ is a maximal node of $\overline{a}$ which is negative and
then it is obvious that function $F$ defining the membrane should have an
absolute maximum in $C\backslash\partial C$, which contradicts the condition
that $F$ is critical points free. Continuity means that if $a_{i}\rightarrow
a$, $b_{i}\rightarrow b$ in $\mathcal{A}$ and $a_{i}\prec b_{i}$, then $a\prec
b$, but this is not that difficult to be done as well.
\end{proof}

The following simple fact will be proved later, during the description of the
algorithm for recognizing when $a\prec b$.

\begin{fact}
If $a\prec b$, then $\sigma(a)=\sigma(b)$.
\end{fact}

In such a way, two elements $a,b\in\mathcal{A}$ are comparable via the
relation ``$\prec$'', only if these lay in one and the same component of
$\mathcal{A}$. Of course, the converse is not true, i.e. the components of
$\mathcal{A}$ are not \textit{totally} ordered by ``$\prec$''.

It is not difficult to see that the ribbon invariant $\gamma$ is an
anti-monotone function on $\mathcal{A}$ with respect to the partial ordering.

\begin{proposition}
\label{p23}If $a_{1}$, $a_{2}\in\mathcal{A}$ and $a_{1}\prec a_{2}$, then
$\gamma(a_{1})\geq\gamma(a_{2})$.
\end{proposition}

\begin{proof}
Let $f\in\mathcal{F(}a_{1}\mathcal{)}$ be an extension of ribbon $a_{1}$ with
$k=\gamma(a_{1})$ critical points. Since $a_{1}\prec a_{2}$, there is a
noncritical membrane $F:C\rightarrow\mathbb{R}$ connecting $a_{2}$ with
$\overline{a_{1}}$. But then it is clear that $f$ and $F$ may be arranged so
that we get an extension $f^{\prime}\in\mathcal{F(}a_{2}\mathcal{)}$ of
$a_{2}$ with $k$ critical points, thus $\gamma(a_{1})\geq\gamma(a_{2})$. Easy
examples show that strict inequality is possible.
\end{proof}

Of course, the above proposition holds true for the other ribbon invariants
$\gamma_{0}$, $\gamma_{\operatorname{ext}}$, $\gamma_{\operatorname{sad}}$: If
$a_{1}\prec a_{2}$, then $\gamma_{\ast}(a_{1})\geq\gamma_{\ast}(a_{2})$.

Note that in fact the function $F:C\rightarrow\mathbb{R}$ defines some
\textit{generalized} ribbon in a small neighbourhood of $\partial C=C_{1}\cup
C_{2}$, which has ribbon invariant 0. We allow here coinciding critical values
of $a_{1}$, $a_{2}$. As we pointed out at the beginning, all the definitions
and invariants remain correctly defined and relevant in this \textit{non
general position} situation.

Now, one may consider the general situation as follows.

\begin{definition}
\label{d26}In the settings of \hyperref[d25]{Definition~\ref*{d25}} we shall
write%
\[
d(a_{1},a_{2})=k\text{,}%
\]

if $k$ is the minimal cardinality of critical points of a membrane
$F:C\rightarrow\mathbb{R}$ connecting $a_{2}$ with $\overline{a_{1}}$.
\end{definition}

Clearly, $d(a_{1},a_{2})=0$ if and only if $a_{1}\prec a_{2}$.\ It turns that
$d$ is a kind of a ``distance'' function. In fact, the triangle inequality%
\[
d(a,b)\leq d(a,c)+d(c,b)
\]

follows immediately from the definition. On the other hand, simple examples
show that it is not symmetric, neither is it reflexive. Moreover, every ribbon
is ``distanced'' from itself.

\begin{proposition}
\label{p24}For any ribbon $a\in\mathcal{A}$, the next inequality holds%
\[
d(a,a)\geq3\text{.}%
\]

\end{proposition}

\begin{proof}
Suppose that the membrane $F:C\rightarrow\mathbb{R}$ is connecting $a$ with
$\overline{a}$. Then it is easy to see that $F$ may be considered in fact as a
function on the torus $\mathbb{T}^{2}$. Now, it is a common fact that each
such function has at least 3 critical points, as the Lusternik-Schnirelmann
category of the $\mathbb{T}^{2}$ equals 3 (see \cite{b5}).
\end{proof}

Now we shall see what can be done about the problem of recognition of the
relation $a_{1}\prec a_{2}$. As we pointed at the beginning, only an
algorithmic approach seems to be relevant here. This problem is closely
related to the problem of recognition of $\gamma(a)=0$ (\hyperref[s12]%
{Section~\ref*{s12}}), which is also only algorithmically solved therein. In
fact, we shall present two different algorithms for solving the problem whose
complexity is close to that for $\gamma=0$.

\begin{proposition}
\label{p25}Let $a_{1}$, $a_{2}\in\mathcal{A}$, $a_{1}\prec a_{2}$ and
$F:C\rightarrow\mathbb{R}$ be the corresponding membrane. Then there is
$c\in\mathbb{R}$ such that $F^{-1}(c)$ is a regular Jordan curve separating
$C_{1}=\mathbb{S}^{1}\times\{1\}$ from $C_{2}=\mathbb{S}^{1}\times\{2\}$, if
and only if $F(C_{1})\cap F(C_{2})=\varnothing$. Then $\sigma(a_{1}%
)=\sigma(a_{2})=0$ and moreover, if $F(C_{1})\cap F(C_{2})=\varnothing$ and,
say $F(C_{1})$ is situated ``under'' $F(C_{2})$, then the maximal node of
$a_{1}$ is positive and the minimal node of $a_{2}$ is negative.
\end{proposition}

In \hyperref[s12]{Section~\ref*{s12}} we defined \textit{cancellation} of a
pair of nodes and this concept was crucial for the algorithm for recognizing
$\gamma=0$. According to the order and type of the canceled pair, we shall
call this cancellation of $(-,+)$ type. Now, changing the types of the nodes
to the opposite ones, we define cancellations of $(+,-)$ type. (Note that
there may be situations, when two consecutive nodes of opposite type cannot be
canceled by any type cancellation.) Both types are needed for one of the
algorithms discussed later.

First we deal with some particular case of disposition of ribbons
$a_{1}=(\varphi_{1},\nu_{1})$ and $a_{2}=(\varphi_{2},\nu_{2})$.

\medskip

\textbf{Case 1.} $\varphi_{1}(C_{1})\cap\varphi_{2}(C_{2})=\varnothing$. Then
the following lemma solves the problem.

\begin{lemma}
\label{l14}Let case 1 be present for $a_{1}$, $a_{2}\in\mathcal{A}$ and
suppose, without loss of generality, that $\varphi_{1}(C_{1})$ is situated
``under'' $\varphi_{2}(C_{2})$. Consider the ribbon $b_{1}$ which is obtained
from $\overline{a_{1}}$ by ``making'' its maximal node positive (the latter
being negative), and ribbon $b_{2}$ which is obtained from $a_{2}$ by
``making'' its maximal node positive as well. Then $a_{1}\prec a_{2}$ if and
only if $\gamma(b_{1})=\gamma(b_{2})=0$.
\end{lemma}

\textbf{Case 2.} $\varphi_{1}(C_{1})\cap\varphi_{2}(C_{2})\neq\varnothing
$.\medskip

\textbf{Method 1.} Consider all (combinatorially different) pairs $x\in
C_{1},\ y\in C_{2}$ such that $\varphi_{1}(x)=\varphi_{2}(y)$. Then perform a
``cut'' of the annulus $C$ along the pair $x,y$ obtaining in such a way some
ribbon $a\in\mathcal{A}$, for which the two new-born nodes are marked as positive.

\begin{lemma}
\label{l15}In case 2, we have $a_{1}\prec a_{2}$ if and only if there is a cut
as above such that $\gamma(a)=0$ for the corresponding ribbon.
\end{lemma}

Note that ribbon $a$ is not a general position ribbon, yet the ribbon
invariant $\gamma$ being defined for such ribbons and the algorithm from
\hyperref[s12]{Section~\ref*{s12}} for detecting $\gamma=0$ works as
well.\medskip

\textbf{Method 2. }This algorithm relies on appropriate cancellations of
ribbons $a_{1}$, $a_{2}$ and reducing them finally to ribbons $b_{1},b_{2}$
which are bounding an \textit{elementary band}.

\begin{definition}
\label{d27}The ribbons $b_{1},b_{2}$ with the same number of nodes are
bounding an elementary band, if there is an order preserving bijection between
their nodes $j:P_{1}\rightarrow P_{2}$ such that

1) $p$ and $j(p)$ have the same marking

2) $\varphi_{1}^{\prime\prime}(p)(\varphi_{2}(j(p))-\varphi_{1}(p))<0$ for any
node $p$ of $b_{1}$.
\end{definition}

It is easy to see that if $b_{1}$ and $b_{2}$ are bounding an elementary band,
then there is a noncritical membrane $C$ between them. Roughly speaking,
elementary bands are sufficiently thin noncritical membranes.

\begin{lemma}
\label{l16}In the presence of case 2, we have $a_{1}\prec a_{2}$ if and only
if there is a cancellation of $a_{1}$ of type $(+,-)$ and a cancellation of
$a_{2}$ of type $(-,+)$ which reduce them to ribbons $b_{1}$ and $b_{2}$,
respectively, which are bounding an \textit{elementary band}.
\end{lemma}

Clearly, the number of cancellations of $a_{1}$ and $a_{2}$ may differ.

Note finally that in both cases it is not difficult to present more detailed
algorithms for establishing the relation $a_{1}\prec a_{2}$ between two given
ribbons. However these algorithms cannot be faster than the algorithm for
detecting $\gamma=0$, the latter being ramifying by nature and thus slow for
large number of nodes.

In our concluding remarks in this section, we shall comment some partial
ordering of \textit{gradient} vector fields on the circle. First of all, by a
\textit{gradient} \textit{field on} $\mathbb{S}^{1}$ we mean a (smooth)
nonzero function $v:\mathbb{S}^{1}\rightarrow\mathbb{R}^{2}$ such that%
\[
\int_{\mathbb{S}^{1}}\left\langle v,\tau\right\rangle ds=0\text{,}%
\]

where $\tau$ is the unit tangent vector field on $\mathbb{S}^{1}$. Of course,
this is not the usual definition of \textquotedblleft gradient
field\textquotedblright, which has to be tangent to the underlying manifold.
It is not difficult to see that for a gradient vector field $v$ on
$\mathbb{S}^{1}$\ there is a critical points free function $F$ defined in a
small neighbourhood $U$ of $\mathbb{S}^{1}$ such that%
\[
\nabla F|_{\mathbb{S}^{1}}=v\text{.}%
\]

Of course, the restriction $\varphi=F|_{\mathbb{S}^{1}}$ is defined by $v$ up
to translation. Moreover, some ribbon $a_{v}=(\varphi,\nu)$ is defined up to
translation $(\varphi,\nu)\rightarrow(\varphi+C,\nu)$.

\begin{definition}
\label{d28}If $v_{1}$ and $v_{2}$ are gradient fields on $\mathbb{S}^{1}$ and
$a_{v_{1}}=(\varphi_{1},\nu_{1})$, $a_{v_{2}}=(\varphi_{2},\nu_{2})$ are the
corresponding ribbons, we shall write%
\[
v_{1}\prec v_{2}\text{,}%
\]

if there is a constant $C$ such that $(\varphi_{1}+C,\nu_{1})\prec(\varphi
_{2},\nu_{2})$.
\end{definition}

It is easy to see that this is a partial ordering in the class of gradient
fields. On the other hand, examples show that the relation $v\prec v$ is
possible, unlike in the case of rigid ribbons (it suffices to consider the
linear field).

\section{\label{s22}Algebraic ribbon invariants}

In the present section we consider the ribbon invariants from purely algebraic
point of view. There is a natural definition of a ribbon invariant as a
subadditive function on the ribbon space, with some initial values on the
elementary ribbons defined. We shall see that the class of algebraic
invariants is quite large. Then, considering a natural partial ordering of
ribbon invariants, we establish the following fact:

\begin{center}
$\gamma$\textit{ is the (unique) maximum in the ordered set of ribbon
invariants.}
\end{center}

The same is equally true for the other 3 invariants $\gamma_{0}$,
$\gamma_{\operatorname{ext}}$, $\gamma_{\operatorname{sad}}$ in the space of
invariants with the corresponding normalization on elementary ribbons.

\begin{definition}
\label{d29}We say that the function $\eta:A\rightarrow\mathbb{N\cup
\{}0\mathbb{\}}$ is an algebraic $\gamma_{\ast}$-ribbon invariant, if

1) $\eta(ab)\leq\eta(a)+\eta(b)$

2) $\eta(\left[  abc\right]  )\leq\eta(a)+\eta(b)+\eta(c)$

3) $\eta(a^{\prime})=\eta(a)$

4) $\eta(x)=\gamma_{\ast}(x)$ for any elementary ribbon $x$

5) $\eta(a\#b)=\eta(a)+\eta(b)$

6) if $q(a)=q(b)$, then $\eta(a)=\eta(b)$ ($q:A\rightarrow\mathcal{A}$ is the
forgetful map).
\end{definition}

Note that 6) implies that $\eta$ may be considered as an invariant on the set
$q(A)\subset\mathcal{A}$, which is \textquotedblleft almost\textquotedblright%
\ the whole $\mathcal{A}$. Also, the additivity of $\eta$ under connected sums
(property 5)) is somewhat consistent from geometric point of view; it is
available for all the $\gamma_{\ast}$-invariants and their derivatives
described below.

Of course, $\gamma_{\ast}$ is an algebraic $\gamma_{\ast}$-ribbon invariant
itself, but we shall show later that this class is quite larger.

Recall that the values of the 4 ribbon invariants $\gamma_{\ast}$ on
elementary ribbons are given at \hyperref[table1]{Table~\ref*{table1}}. Of
course, one may consider algebraic ribbon invariants with different
normalization on elementary ribbons, by excluding 4) from the definition.
Clearly, $\eta\equiv0$ is a ribbon invariant in this sense.

We shall focus our attention mainly on the principal case $\gamma_{\ast
}=\gamma$ with the usual normalization: $\gamma(\alpha_{0})=0$, $\gamma
(\alpha_{1})=\gamma(\alpha_{2})=1$, $\gamma(\beta_{n})=1$.

\textbf{Examples.} Here we give some examples of algebraic ribbon invariants.

\textbf{0.} It is clear that $\gamma_{\operatorname{ext}}+\gamma
_{\operatorname{sad}}$ is a $\gamma$-ribbon invariant, which is different from
$\gamma$ itself, since we showed in \hyperref[s9]{Section~\ref*{s9}} that
$\gamma\geq\gamma_{\operatorname{ext}}+\gamma_{\operatorname{sad}}$ and there
are examples of strict inequality.

Another example of a $\gamma$-ribbon invariant is $\beta$ - the minimal number
of critical values of an extension of the ribbon (\hyperref[s2]%
{Section~\ref*{s2}}). Surely, $\beta\neq\gamma$, since for example $\beta
\ll\gamma$ in $\mathcal{A}^{+}$. Note also that the cluster number $\delta$ is
a $\gamma_{\operatorname{sad}}$-ribbon invariant, since it has the same
normalization as $\gamma_{\operatorname{sad}}$ on elementary ribbons.

\textbf{1.} \textit{Truncation of} \textit{an invariant }$\eta$.

Let $\eta$ be a ribbon invariant and $m$ be a natural number, then define%
\[
\eta_{m}(a)=\left\{
\begin{array}
[c]{ll}%
\eta(a) & \text{if }a\in\mathcal{A}_{m}\\
0 & \text{otherwise}%
\end{array}
\right.
\]

Recall that $\mathcal{A}_{m}$ is the space of ribbons with $\leq m$ nodes.
Then it is easily seen that $\eta_{m}$ is a ribbon invariant defined on
$\mathcal{A}_{m}$. Furthermore, as we shall see later, it is not difficult to
see that the number of such invariants is finite, so various questions of
combinatorial character may be raised here, for example:

\begin{center}
\textit{What is the number of }$\gamma_{m}$ \textit{- invariants?}
\end{center}

The same may be asked about the other $(\gamma_{\ast})_{m}$ - invariants.

Note that setting in addition $\eta_{m}(a)=\gamma_{\ast}(a)$ for elementary
ribbons, we get some $\gamma_{\ast}$-invariant of \textit{finite type}, i.e.
it equals $0$ on each \textit{non elementary} ribbon with sufficiently many nodes.

\textbf{2.} \textit{Divisor invariants}.

Let $a\in\mathcal{A}$ be a ribbon. We shall say that the ribbon $b$ is a
\textit{divisor} of $a$, if there is a (nontrivial) representation of $a$,
where $b$ takes part. Clearly, this defines yet another partial ordering of
the ribbon space. Let $D(a)$ denote the set of all divisors of $a$. Now we may
consider the following \textit{divisor~invariant} $d_{a}$ associated to $a$:%

\[
d_{a}(x)=\left\{
\begin{array}
[c]{ll}%
1 & \text{if }x\in D(a)\\
0 & \text{otherwise}%
\end{array}
\right.
\]

It is easily seen that $d_{a}$ is a ribbon invariant. Furthermore, we may
modify it in a way, that it takes the usual $\gamma$-normalization values on
the elementary ribbons, obtaining in such a way some $\gamma$-invariant
$\overline{d_{a}}$ of finite type. Note also that this construction may be
generalized in the following way.

Take some set $M\subset\mathcal{A}$ and denote by $D(M)$ the set of all
divisors of elements of $M$. Then we may define the ribbon invariants $d_{M}$
and $\overline{d_{M}}$ as above (simply replacing $a$ by $M$).

\textbf{3.} \textit{Arithmetically generated invariants}.

Let $\eta$ be some $\gamma-$ribbon invariant and $k$ be a natural number, then
define%
\[
\eta^{(k)}=\left[  \frac{\eta+k}{k+1}\right]  \text{.}%
\]

It is easy to see that $\eta^{(k)}$ is also a $\gamma$-ribbon invariant. In
general, it differs from $\eta$, since asymptotically $\eta^{(k)}\sim
\eta/(k+1)$. Note also that $\eta^{(k)}$ is not of \textit{finite type,} if
$\eta$ is not so.

Now, repeating inductively this procedure, for any finite sequence of integers
$k_{1},\dots,k_{m}$ we get some invariant $\eta^{(k_{1}...k_{m})}$. It is
clear that%
\[
\eta^{(k_{1}...k_{m})}\underset{m\rightarrow\infty}{\rightarrow}0\text{,}%
\]

in the sense that for a given ribbon $a$ we have $\eta^{(k_{1}...k_{m})}(a)=0$
for sufficiently large $m$.

\textbf{Partial ordering of ribbon invariants.} Let us denote by $RI$ the set
of ribbon invariants. There is a natural partial order in $RI$:%
\[
\eta\leq\eta^{\prime}\text{, if }\eta(a)\leq\eta^{\prime}(a)\text{ for any
}a\in\mathcal{A}\text{.}%
\]

Then various questions about the structure of the \textit{poset} $RI$ may be
asked. (Recall that \textquotedblleft poset\textquotedblright\ = partially
ordered set.)

First of all, it turns out that $RI$ is a \textit{complete semi-lattice} in
the sense that every non empty subset has a supremum (in particular, $RI$ has
a global maximum).\ It becomes evident from the following fact:

\begin{center}
If $A\subset RI$, then $\eta_{\max}(a)=\max\left\{  \eta(a)|\ \eta\in
A\right\}  $ is the supremum of family $A$.
\end{center}

The above maximum exists by the inequality $\eta(a)\leq\frac{n}{2}+1$, which
is proved later (here $n$ is the number of nodes of $a$). Note that the
minimum of a set of invariants \textit{may not be} a ribbon invariant itself,
as simple examples show. In other words, $RI$ \textit{is not} \textit{a
lattice} with respect to the usual order.

Recall that the \textit{width} of a poset is the maximal cardinality if its
\textit{antichains}, the subsets of mutually incomparable elements.

\begin{proposition}
\label{p26}The width of $RI$ is infinite.
\end{proposition}

\begin{proof}
It is clear that one can find an infinite system of mutually incomparable
ribbons $\left\{  a_{i}\right\}  $, i.e. such that $a_{i}$ is not a divisor of
$a_{j}$ for $i\neq j$. Then the system $\left\{  d_{a_{i}}\right\}  $ of the
corresponding divisor-invariants is an antichain in $RI$, which becomes
evident by testing it on the system of ribbons $\left\{  a_{i}\right\}  $.
\end{proof}

It is natural to consider the space $RI_{m}$ of truncated $\gamma-$ribbon
invariants, which is a finite poset. Now, one may ask the following

\textbf{Question.} What is the width of $RI_{m}$? More generally - what can be
said about its \textit{structure}?

Of course, the same may be asked about the other truncated $\gamma_{\ast}%
$-ribbon invariants.

Now we prove the central fact in this section.

\begin{theorem}
\label{t22}Consider the poset $RI$ of ribbon invariants with the usual
$\gamma$-normalization. Then its global maximum is $\gamma$.
\end{theorem}

\begin{proof}
It may be done by induction on the lexicographic order $\prec$ in
$\mathcal{A}$. Let $\eta\in RI$ and $a\in\mathcal{A}$ be a non-elementary
ribbon. It suffices to show that $\eta(a)\leq\gamma(a)$. For a non-elementary
ribbon $a$, we proved previously that at least one of the following two cases
happens a) there is a binary splitting $a=a_{1}a_{2}$, such that
$\gamma(a)=\gamma(a_{1})+\gamma(a_{2})$ and $a_{i}\prec a$, $i=1,2$ b) there
is a ternary splitting $a=\left[  a_{1}a_{2}a_{3}\right]  $, such that
$\gamma(a)=\gamma(a_{1})+\gamma(a_{2})+\gamma(a_{3})$ and $a_{i}\prec a$,
$i=1,2,3$. Suppose now that $\eta(x)\leq\gamma(x)$ for any $x\prec a$. Then in
case a) we have%
\[
\eta(a)\leq\eta(a_{1})+\eta(a_{2})\leq\gamma(a_{1})+\gamma(a_{2}%
)=\gamma(a)\text{,}%
\]
while in case b)%
\[
\eta(a)\leq\eta(a_{1})+\eta(a_{2})+\eta(a_{3})\leq\gamma(a_{1})+\gamma
(a_{2})+\gamma(a_{3})=\gamma(a)\text{.}%
\]

\end{proof}

\begin{corollary}
If $\eta$ is a $\gamma$-ribbon invariant, then $\eta\leq\frac{n}{2}+1$.
\end{corollary}

This follows from the basic inequality $\gamma\leq\frac{n}{2}+1$ proved in
\hyperref[s14]{Section~\ref*{s14}}. Similarly, all estimates from above for
the other invariants $\gamma_{\ast}$ remain valid for any ribbon invariant
with the same normalization on elementary ribbons.

\begin{corollary}
Let $\eta$ be a $\gamma$-ribbon invariant different from $\gamma$. Then there
is some ribbon $a\in\mathcal{A}$ such that for any nontrivial splitting
$a=a_{1}a_{2}$ or $a=\left[  a_{1}a_{2}a_{3}\right]  $ into ``smaller''
ribbons, we have
\[
\eta(a)<\eta(a_{1})+\eta(a_{2})~~~and~~~\eta(a)<\eta(a_{1})+\eta(a_{2}%
)+\eta(a_{3})\text{.}%
\]

\end{corollary}

It is not that easy to find such a ``defective'' ribbon for a given $\gamma
$-invariant $\eta\neq\gamma$. For example for $\eta=\gamma_{\operatorname{ext}%
}+\gamma_{\operatorname{sad}}$, see \hyperref[e1]{Example~\ref*{e1}}.

\begin{remark}
\hyperref[t22]{Theorem~\ref*{t22}} remains equally true for the other three
$\gamma_{\ast}$-ribbon invariants $\gamma_{0},\gamma_{\operatorname{ext}%
},\gamma_{\operatorname{sad}}$. So, we arrive at some curious
observation:\medskip

The absolute maximum of all \textbf{algebraic} invariants is
\textbf{geometric} in nature.
\end{remark}

\section{Final remarks and open questions}

Let us summarize in this final section some open problems and trace the routes
for further investigations.

First we list some general problems and speculations about ribbons.\medskip

\textbf{Q1.} Find ``good'' algorithms for the calculation of $\gamma$ and the
other geometric ribbon invariants. Is there a polynomial one for the
calculation of $\gamma$? Here by ``good'' algorithm we mean ``fast''
algorithm. In Part~II of the present paper we describe an algorithm for
calculation of $\gamma$ in $\mathcal{A}^{+}$\ which is based on the reduction
of a ribbon to a ladder by a sequence of elementary moves. It is by no means
better than the general one described in \hyperref[s11]{Section~\ref*{s11}}.

\textbf{Q2.} Note that the algebraic methods exposed here for managing ribbon
invariants are not quite ``algebraic'' in nature. For example, we get
\textit{semigroups} instead of groups, \textit{partial ordering} instead of
equivalence relation classes, etc. In this sense, it would be interesting to
find some more algebraically calculable aspects of the ribbon invariants.

\textbf{Q3.} Investigate the ribbon invariants for 2-surfaces with boundary
(connected or not), instead of only for the disk $\mathbb{B}^{2}$. In
\hyperref[s21]{Section~\ref*{s21}} we presented some specifical results about
the annulus.

\textbf{Q4.} Multidimensional ribbons. In \hyperref[s15]{Section~\ref*{s15}}
we speculated about these by defining some \textit{ribbon-term} in Morse and
Lusternik-Schnirelmann inequalities. Of course, there is a bunch of problems
here. For example, one of them is whether there is a class of good extensions
analogous to the \textit{economic} ones defined in \hyperref[s6]%
{Section~\ref*{s6}}. Another problem is about the structure of the boundary
data of a multidimensional ribbon. One way is to consider only Morse functions
on the boundary. By the way, L. Nicolaescu \cite{b6} counted (and described)
all Morse function on the sphere $\mathbb{S}^{2}$. On the other hand, we find
reasonable to consider as boundary conditions some larger class of
\textquotedblleft economic\textquotedblright\ functions, admitting degenerated
saddles. Of course, from point of view of the applications, we should be able
to provide, by a little move, Morse type boundary conditions and then
considering$~\gamma$.

Also, it would be interesting to see whether ribbon type estimates may be
obtained in the \textit{symplectic} category.

\textbf{Q5.} Infinite dimensional ribbons, functionals. As we know, there are
multiple variants of Morse theory where the ambient manifold $M$ is infinite
dimensional and the number of critical points of some functional are estimated
from below (of course, under some compactness conditions), see \cite{b7}. This
provides results about the existence of solutions of the corresponding
differential equation, associated with this functional. In this setting, it
would be interesting to see whether some technique of \textit{infinite
dimensional ribbons} may be applied in order to obtain multiplicity results
about the number of solutions of differential equations.

\textbf{Q6.} PL-ribbons and extensions. In fact all the ribbon theory may be
settled in the PL-category instead of the differential one. Here the ribbons
are defined by piecewise linear maps on polygonal Jordan curves, the
extensions are also PL-maps, etc. The \textit{critical points }of a
2-variables PL-function may be correctly defined and almost all the results
remain the same. From this point of view,

\begin{center}
\textit{Ribbons and ribbon invariants are more a combinatorial
phenomenon,\newline rather than a differential one.}
\end{center}

\textbf{Q7.} Applications, computer simulations. As to applications, it is
clear that this technique may be very useful for establishing various
multiplicity results. For example, it is natural to consider the following
situation: $f(x,y)$ is a smooth function defined in some region $D\subset
\mathbb{R}^{2}$ and we want to count its critical points. Take some rectangle
$Q\subset D$ and estimate the number of critical points of $f$ in $Q$ in two steps

1) Find the induced ribbon $a_{Q}(f)$ on $\partial Q$. This is not that
difficult to be done, since the problem is in fact 1-dimensional for a rectangle.

2) Calculate $\gamma(a_{Q}(f))$, then $f$ has at least $\gamma(a_{Q}(f))$
critical points inside $Q$. Surely, this step depends on the available
algorithm and may slow down the task, anyway, for small number of nodes it may
be done in a reasonable time. Another way is to make use of the general
estimates from below for $\gamma$, such as the \textit{cluster number} or the
\textit{index} of a ribbon, which are much easier calculable. Note also that
this method is appropriate for the localization of the critical set
$\mathrm{Crit}(f)$ by moving rectangle $Q$ inside $D$. Of course, one may use
some different shape instead of rectangles, e.g. circles or even a free shape form.

As for computer simulations, maybe the PL-approach is the most convenient one,
since it deals in fact only with finite numerical data concentrated at the
nodes of the PL-subdivision.\medskip

Now we list some minor open problems that appear here and there in the
text.\medskip

\textbf{q0.} What can be said about non general position ribbons? There are
two cases of degeneracy of a ribbon - 1) coinciding critical values and 2)
appearance of a critical point at a node of the ribbon itself. These ribbons
are of importance when considering elementary moves in the ribbon space
$\mathcal{A}$. Another question about such ribbons is whether there exists a
nice generating function for these, analogous to the formula obtained in
\hyperref[s4]{Section~\ref*{s4}}, p.~\pageref{andre} for ordinary (general
position) ribbons.

\textbf{q1.} For a given ribbon $a\in\mathcal{A}$,\ what is the cardinality of
the set $\mathcal{F}^{e}(a)$ of its economic extensions? Here extensions are
combinatorially distinguished. Another, more difficult question, is about
counting topologically different economic extensions.

\textbf{q2.} What is the cardinality of $\mathcal{F}_{n}^{e}$ - the set of
economic extensions of ribbons with $\leq n$ nodes?

\textbf{q3.} Are there any consistent estimates of the number $\phi_{r}(a)$ of
economic extensions of a given ribbon $a$ with exactly $r$ critical points?

\textbf{q4.} For which pairs $(\sigma,\gamma)$ do a ribbon exist with
signature $\sigma$ and ribbon invariant $\gamma$? In \hyperref[s9]%
{Section~\ref*{s9}}, p.~\pageref{realiz} we showed that the basic inequality
involving $\sigma$ and $\gamma$ does not guarantee the existence of such a
ribbon. Let $\Sigma_{m}$ be the space of (rigid) ribbons of signature
$\sigma=m$ and $\Gamma_{k}$ be the space of ribbons with $\gamma=k$. Then the
above question takes the form ``When $\Sigma_{m}\cap\Gamma_{k}\neq\varnothing
$?'' Another natural question is ``When does $\Sigma_{m}\cap\Gamma_{k}$ have a
finite number of components?''

\textbf{q5.} What are the homologies (the homotopy type) of the space
$\Sigma_{m}\cap\Gamma_{k}$? This is the most general question of this sort.

\textbf{q6.} In Part~II we prove that the space $\Gamma_{0}$ of ribbons with
$\gamma=0$ is connected. So, it is natural to ask:

Is $\Gamma_{0}$ contractible? If not, then what is the fundamental group
$\pi_{1}(\Gamma_{0})$?

\textbf{q7.} Is it true that $\gamma_{0}(a)=\gamma(a)$ implies that the ribbon
$a$ is a ladder?

\textbf{q8.} Is there a winning strategy for player B in \textit{the ribbon
game} described in \hyperref[s12]{Section~\ref*{s12}}, p.~\pageref{r-game}?
The same question about the \textit{Jordan} variant of the game.

\textbf{q9.} Consider in class $\mathcal{A}^{+}$ the number $\nu$ which is
changing by $-1$ at meetings and by $+1$ at separations (and has some initial
values on ladders). Then is this number $\nu$ an \textit{invariant}, i.e. is
it independent of the path selected in $\mathcal{A}^{+}$? If so, does it have
some intrinsic definition in terms of the corresponding zig-zag permutation?

\textbf{q10.} What is the \textit{mathematical expectation} of the quantity
$\frac{\gamma}{n}$?\ What is the distribution of the sequence $\left\{
\frac{\gamma(a_{i})}{n(a_{i})}\right\}  $ in $\mathbb{[}0,1/2\mathbb{]}$,
where the ribbons $a_{i}\in\mathcal{A}$ are lexicographically ordered? For
example, is it uniformly distributed or has it some other peculiar behaviour?
The same may be asked under condition of fixed signature: $\sigma=\sigma_{0}$.

\textbf{q11.} Let $P\subset\mathbb{S}^{1}$ be a set of $n$ points
(\textit{nodes}), $n$ is even, and $\nu:P\rightarrow\left\{  +,-\right\}  $ be
a \textit{marking} with signature $\sigma$. Then, under what conditions is it
true that for a given number $k$, satisfying the general inequality
$1-\frac{\sigma}{2}\leq k\leq n-1-\frac{\sigma}{2}$, there is a smooth
function $\varphi:\mathbb{S}^{1}\rightarrow\mathbb{R}$ with node set $P$, such
that for the ribbon $a=(\varphi,\nu)$ we have $\gamma(a)=k$?

\textbf{q12.} For any finite set of vectors $\xi_{1},\dots,\xi_{n}$ in
$\mathbb{R}^{3}$, is it true that a frame $l$ exists, such that any membrane
$M$ with $\partial M=l$ realizes all $\xi_{i}~,i=1,\dots,n$ as normal directions?

\textbf{q13.} For a general position function $\varphi:\mathbb{S}%
^{1}\rightarrow\mathbb{R}$, what is the number of critical points free
extensions of $\varphi$? (see p.~\pageref{t12}) Is it true that this number is
the same for all \textit{alternations} $\varphi$? Here the extensions are
combinatorially distinguished. Recall that for a ladder-type $\varphi$ this
number equals $2^{\frac{n}{2}-1}$ (where $n$ is the number of critical points
of $\varphi$) and this is the minimal possible value among all functions
$\varphi$. Another natural question is about the maximal value of the number
of critical points free extensions (for fixed $n$). It is not difficult to see
that this value is attained at some alternation.

The same may be asked about functions $\varphi$ with $n$ local and $s$
absolute extrema such that $s\leq\frac{n}{2}+1$. In case $s>\frac{n}{2}+1$ we
have 1 critical point guaranteed and one may be interested in the number of
extensions with \textit{exactly} 1 critical point. (Note that such an
extension always exists - it suffices to take a \textit{cone} over the graph
of $\varphi$ and then to perform smoothening at the vertex.)

\textbf{q14.} If $M^{n}$ is a smooth $n$-dimensional manifold and
$f:M^{n}\rightarrow\mathbb{R}$ is a Morse function, does there exist an
immersion $\phi:M^{n}\rightarrow\mathbb{R}^{n+1}$, such that $\pi\phi=f$,
where $\pi$ is the projection on some fixed line $l\approx\mathbb{R}$?

\textbf{q15.} What is the cardinality and the width of the partially ordered
set $RI_{m}$? More generally - what can be said about its \textit{structure}?
The same question about the other truncated $\gamma_{\ast}-$ribbon invariants.

\textbf{q16.} Describe the class of \textit{harmonic} ribbons. A ribbon
$a=(\varphi,\nu)\in\mathcal{A}$ is \textit{harmonic} if the (unique) solution
$f$ of the Dirichlet problem $\Delta f=0,~$ $f|_{\mathbb{S}^{1}}=\varphi$ is
inducing ribbon $a$ itself on $\mathbb{S}^{1}$.

\textbf{q17.} Describe the class of \textit{Jordan} ribbons. A ribbon
$a=(\varphi,\nu)\in\mathcal{A}$ is Jordan, if it has a representation by a
Jordan curve in the plane (\hyperref[s14]{Section~\ref*{s14}},
p.~\pageref{jordan}). A straightforward necessary condition for this is
$\gamma(a)=0$, however, various obstructions for such a ribbon to be
\textit{Jordan}\ may be formulated. The interconnection between ribbons and
the theory of immersed curves in the plane will be discussed in more detail in
Part~II of the present article.

Note finally that a part of the material in the paper was communicated at the
conference \cite{b8}.


\begin{thebibliography}{99}                                                                                               %


\bibitem {b1}J. W. Milnor. Topology from the Differentiable Viewpoint.
Princeton University Press, 1997.

\bibitem {b2}D. Andr\'{e}. Sur les permutations altern\'{e}es. Journal de
math\'{e}matiques pures et appliqu\'{e}es, 3e s\'{e}rie, tome 7 (1881), 167--184.

\bibitem {b3}P. W. Shor and C. J. Van Wyk. Detecting and decomposing
self-overlapping curves. Computational Geometry: Theory and Applications 2
(1992), Elsevier, 31--50.

\bibitem {b4}V. I. Arnold. Plane curves. Their invariants, perestroikas and
classifications. Adv. in Soviet Math. 21 (1994), 33--91.

\bibitem {b5}F. Takens, The minimal number of critical points of a function on
compact manifolds and the Lusternik-Schnirelmann category, Inventiones
Mathematicae 6 (1968), 197--244.

\bibitem {b6}L. I. Nicolaescu. Counting Morse functions on the 2-sphere.
Compositio Math. 144 (2008), 1081--1106.

\bibitem {b7}A. Floer. Witten's complex and infinite dimensional Morse theory.
J. Differ. Geom. 30(1989), 207--221.

\bibitem {b8}S. Stefanov. An invariant counting critical points. II Intern.
Conf. \textquotedblleft Math. Days in Sofia\textquotedblright, July 10--14,
2017, Sofia.

\bibitem {b9}A. O. Prishlyak. Topological equivalence of smooth functions with
isolated critical points on a closed surface. arXiv:math/9912004v1 [math.GT] 1
Dec 1999.

\bibitem {b10}A. Parusi\'{n}ski. Gradient homotopies of gradient vector
fields. Studia Mathematica, v. XCVI (1990), 73 - 80.
\end{thebibliography}
\end{document}